\documentclass[a4paper,12pt,reqno]{amsart}

\usepackage[utf8]{inputenc}
\usepackage[top=25truemm,bottom=25truemm,left=20truemm,right=20truemm]{geometry}

\usepackage{amsmath}
\usepackage{amsthm,amsfonts}
\usepackage{amssymb}
\usepackage{stmaryrd}
\usepackage{mathtools}
\usepackage{mathrsfs} 
\usepackage[bbgreekl]{mathbbol} 
\usepackage{tikz-cd}
\usetikzlibrary{decorations.markings,decorations.pathmorphing}
\usepackage{thmtools} 
\usepackage{ascmac}
\usepackage{enumitem}
\usepackage[nolist,smaller]{acronym}

\DeclareSymbolFontAlphabet{\mathbb}{AMSb}
\DeclareSymbolFontAlphabet{\mathbbl}{bbold}

\DeclareEmphSequence{\bfseries\itshape}

\renewcommand{\epsilon}{\varepsilon}
\renewcommand{\phi}{\varphi}

\setlist{nosep}
\setlist[enumerate,1]{label*=(\roman*)}

\makeatletter
\newcommand{\labeleditem}[1]{
\item[\text{#1}]\protected@edef\@currentlabel{\text{#1}}\phantomsection
}
\makeatother

\usepackage[%
    backend=biber,%
    style=alphabetic,%
    sorting=nyt,%
    giveninits=true,
    backref=true,%
    isbn=false,%
    doi=true,%
    url=true,%
]{biblatex}
\makeatletter
\renewrobustcmd*{\mkbibemph}{\mkbibitalic}
\protected\long\def\blx@imc@mkbibemph#1{\blx@imc@mkbibitalic{#1}}
\makeatother

\definecolor{darkblue}{rgb}{0.2,0,0.6}
\definecolor{darkgreen}{rgb}{0.2,0.5,0.2}
\usepackage[%
    colorlinks=true,%
    linkcolor=darkblue,%
    citecolor=green,%
    urlcolor=darkgreen,%
    naturalnames=true,
    pdfusetitle=true,
]{hyperref}

\usepackage[%
    nameinlink,%
]{cleveref}

\crefname{equation}{}{}
\crefname{enumi}{}{}
\crefname{section}{Section}{Sections}
\crefname{appendix}{Appendix}{Appendices}

\declaretheoremstyle[
    spaceabove=6pt,%
    spacebelow=6pt,%
    headfont=\normalfont\bfseries,%
    notefont=\mdseries,%
    notebraces={(}{)},%
    bodyfont=\normalfont,%
    postheadspace=1em,%
]{mystyle}
\declaretheoremstyle[
    spaceabove=0pt,%
    spacebelow=0pt,%
    headfont={\normalfont\ttfamily},%
    notefont={\normalfont},%
    bodyfont={\addtolength{\leftskip}{.8em}},%
    postheadspace=.5em,%
    headpunct={.},%
]{claim}
\declaretheoremstyle[
    spaceabove=0pt,%
    spacebelow=5pt,%
    headfont={\normalfont},%
    notefont={\normalfont},%
    bodyfont={\addtolength{\leftskip}{.8em}},%
    postheadspace=.5em,%
    headpunct={},%
]{since}

\declaretheorem[%
    style=mystyle,%
    numberwithin=section,%
    qed=$\blacklozenge$,%
]{definition}
\declaretheorem[%
    style=mystyle,%
    numberwithin=section,%
    sibling=definition,%
    qed=$\blacklozenge$,%
]{remark,example,notation,fact,construction}
\declaretheorem[%
    style=mystyle,%
    numberwithin=section,%
    sibling=definition,%
]{theorem,lemma,corollary,proposition}
\declaretheorem[%
    style=claim,%
]{claim}
\declaretheorem[%
    style=since,%
    numbered=no,%
    qed=$\lozenge$,%
    name=$\because$,%
]{since}

\crefname{theorem}{Theorem}{Theorems}
\crefname{definition}{Definition}{Definitions}
\crefname{lemma}{Lemma}{Lemmas}
\crefname{corollary}{Corollary}{Corollaries}
\crefname{proposition}{Proposition}{Propositions}
\crefname{remark}{Remark}{Remarks}
\crefname{example}{Example}{Examples}
\crefname{notation}{Notation}{Notations}
\crefname{construction}{Construction}{Construction}
\crefname{claim}{Claim}{Claims}

\usepackage{xpatch}
\makeatletter   
\xpatchcmd{\@tocline}
{\hfil\hbox to\@pnumwidth{\@tocpagenum{#7}}\par}
{\ifnum#1<2\hfill\else\dotfill\fi\hbox to\@pnumwidth{\@tocpagenum{#7}}\par}
{}{}
\def\l@subsection{\@tocline{2}{0pt}{2pc}{6pc}{}}
\makeatother    

\makeatletter
    \def\tikzcd@myar#1{\relax\pgfutil@ifnextchar[{\tikzcd@handle@shortcuts@next\tikzcd@@myar{#1}}{\tikzcd@arrow[#1]}}
    \def\tikzcd@@myar#1[#2]{\tikzcd@arrow[#1,#2]}%
\def\tikzcd@[#1]{%
    \tikzpicture[/tikz/commutative diagrams/.cd,every diagram,#1]%
    \ifx\arrow\tikzcd@arrow%
        \pgfutil@packageerror{tikz-cd}{Diagrams cannot be nested}{}%
    \fi%
    \let\arrow\tikzcd@arrow%
    \let\ar\tikzcd@arrow%
    \def\larrow{\tikzcd@myar{vslash}}
    \def\lar{\tikzcd@myar{vslash}}
    \def\lpath{\tikzcd@myar{vslash,path}}
    \def\lphan{\tikzcd@myar{vslash,phan}}
    \global\let\tikzcd@savedpaths\pgfutil@empty%
    \matrix[%
        /tikz/matrix of \iftikzcd@mathmode math \fi nodes,%
        /tikz/every cell/.append code={\tikzcdset{every cell}},%
        /tikz/commutative diagrams/.cd,every matrix]%
    \bgroup}
\makeatother

\tikzcdset{
path/.code = {\tikzcdset{dashed};},
phan/.code = {\tikzcdset{dotted};}
}

\tikzstyle{vslash} = [decoration={markings, mark=at position 0.5 with {\draw[-,solid] (0,-1.5pt) -- (0,1.5pt);}}, postaction ={decorate}]

\tikzcdset{equal/.append style={double distance=1.0pt,-}}
\tikzcdset{Rightarrow/.append style={double distance=1.0pt}}

\tikzcdset{%
    arrow style=tikz,
    diagrams={>={Straight Barb[scale=0.8]}}
}

\tikzstyle{largetri} = [row sep={2.8em}, column sep={1.4em}]
\tikzstyle{tri} = [row sep={1.6em}, column sep={0.8em}]
\tikzstyle{tinytri} = [row sep={0.8em}, column sep={0.1em}]

\tikzstyle{huge} = [row sep={3.6em}, column sep={3.6em}]
\tikzstyle{large} = [row sep={2.7em}, column sep={2.7em}]
\tikzstyle{normal} = [row sep={1.8em}, column sep={1.8em}]
\tikzstyle{scriptsize} = [row sep={1.35em}, column sep={1.35em}]
\tikzstyle{small} = [row sep={0.9em}, column sep={0.9em}]
\tikzstyle{tiny} = [row sep={0.45em}, column sep={0.45em}]

\tikzstyle{xhugecolumn} = [column sep={4.5em}]
\tikzstyle{hugecolumn} = [column sep={3.6em}]
\tikzstyle{largecolumn} = [column sep={2.7em}]
\tikzstyle{normalcolumn} = [column sep={1.8em}]
\tikzstyle{scriptsizecolumn} = [column sep={1.35em}]
\tikzstyle{smallcolumn} = [column sep={0.9em}]
\tikzstyle{tinycolumn} = [column sep={0.45em}]

\tikzstyle{hugerow} = [row sep={3.6em}]
\tikzstyle{largerow} = [row sep={2.7em}]
\tikzstyle{normalrow} = [row sep={1.8em}]
\tikzstyle{scriptsizerow} = [row sep={1.35em}]
\tikzstyle{smallrow} = [row sep={0.9em}]
\tikzstyle{tinyrow} = [row sep={0.45em}]

\renewcommand{\-}{\mathchar`-}

\newcommand{\N}{\mathscr{N}}

\newcommand{\U}{\mathscr{U}}

\newcommand{\bE}{\dbl{E}}

\newcommand{\bJ}{\dbl{J}}
\newcommand{\bK}{\dbl{K}}
\newcommand{\bL}{\dbl{L}}

\newcommand{\bX}{\dbl{X}}

\newcommand{\bfempty}{{\pmb{\varnothing}}}

\newcommand{\Set}{\one{Set}}

\newcommand{\SET}{\one{SET}}
\newcommand{\AVDC}{{\bi{AVDC}}}
\newcommand{\UAVDC}{{\bi{UAVDC}}}
\newcommand{\UVDCn}{{\bi{UVDC}_\mathrm{n}}}

\newcommand{\op}{\mathrm{op}}
\newcommand{\co}{\mathrm{co}}
\newcommand{\Ob}{\zero{Ob}}

\newcommand{\vvect}[2]{
\begin{psmallmatrix*}[c]
   #1 \\
   #2
\end{psmallmatrix*}
}
\newcommand{\abs}[1]{{|#1|}}
\newcommand{\const}[1]{\ulcorner #1\urcorner}
\newcommand{\tup}[1]{{\vec{#1}}}

\newcommand{\incat}[1]{\hspace{1em}\text{in }{#1}}

\newcommand{\id}{\mathsf{id}}
\newcommand{\Id}{\mathsf{Id}}

\newcommand{\Vline}{\quad\vline\,\vline\quad}

\newcommand{\proofdirection}[2]{{[#1$\implies$#2]}}

\newcommand{\fgrpd}{{\pi_1}}

\NewDocumentCommand{\arr}{D(){} O{} O{2}}{\mathrel{%
    \begin{tikzcd}[column sep={#3em}, ampersand replacement=\&]
        \hspace{-2ex} \& \hspace{-2ex}
        \arrow[from=1-1, to=1-2, "{%
            \IfValueTF{#1}{%
                \raisebox{0pt}[3pt][\depth]{$\scriptstyle #1$}%
            }{%
                \scriptstyle #1%
            }%
        }", #2]
    \end{tikzcd}
}}
\NewDocumentCommand{\rra}{D(){} O{} O{2}}{\mathrel{%
    \begin{tikzcd}[column sep={#3em}, ampersand replacement=\&]
        \hspace{-2ex} \& \hspace{-2ex}
        \arrow[from=1-2, to=1-1, "{%
            \IfValueTF{#1}{%
                \raisebox{0pt}[3pt][\depth]{$\scriptstyle #1$}%
            }{%
                \scriptstyle #1%
            }%
        }"', #2]
    \end{tikzcd}
}}

\NewDocumentCommand{\larr}{D(){} O{} O{2}}{\mathrel{%
    \begin{tikzcd}[column sep={#3em}, ampersand replacement=\&]
        \hspace{-2ex} \& \hspace{-2ex}
        \arrow[from=1-1, to=1-2, vslash, "{%
            \IfValueTF{#1}{%
                \raisebox{0pt}[3pt][\depth]{$\scriptstyle #1$}%
            }{%
                \scriptstyle #1%
            }%
        }", #2]
    \end{tikzcd}
}}

\NewDocumentCommand{\adjoint}{D(){\perp} O{1} m m m m }{
    \begin{tikzcd}[column sep={#2em},ampersand replacement=\&]
        {#3}\ar[rr,"{#5}",shift left=5pt,bend left=10] \& {#1} \& {#4}\ar[ll,"{#6}",shift left=5pt,bend left=10]
    \end{tikzcd}
}

\newcommand*{\vcong}{\rotatebox{-90}{$\,\cong$}}

\newcommand{\linv}{\vcong}
\newcommand{\cart}{\mathsf{cart}}
\newcommand{\cocart}{\mathsf{cocart}}
\newcommand{\VDcocart}{\scalebox{0.8}{$\mathsf{VD.cocart}$}}
\newcommand{\cocartorVD}{\scalebox{0.8}{$\mathsf{(VD.)cocart}$}}
\newcommand{\extending}{\mathsf{ext}}
\newcommand{\lifting}{\mathsf{lift}}
\newcommand{\heq}{%
    \mathord{%
        \rule[3.5pt]{4pt}{0.6pt}\hspace{-4pt}\rule[2pt]{4pt}{0.6pt}
    }%
}
\newcommand{\veq}{%
    \mathord{%
        \mkern4mu\rule[0.6pt]{0.6pt}{5pt}\hspace{-2pt}\rule[0.6pt]{0.6pt}{5pt}\mkern4mu
    }%
}

\NewDocumentCommand{\cellsymb}{D(){} O{} m m}{
    \arrow[from=#3,to=#4,phantom,"{#1}"{#2}]
}

\NewDocumentCommand{\utwocell}{D(){} O{right=0pt} m m O{}}{
    \arrow[from=#3, to=#4, phantom, "{#1}"{#2}, "\rotatebox{90}{$\Rightarrow$}", #5]
}
\NewDocumentCommand{\dtwocell}{D(){} O{right=0pt} m m O{}}{
    \arrow[from=#3, to=#4, phantom, "{\scriptstyle #1}"{#2}, "\rotatebox{90}{$\Leftarrow$}", #5]
}
\NewDocumentCommand{\ltwocell}{D(){} O{above=2pt} m m O{}}{
    \arrow[from=#3, to=#4, phantom, "{\scriptstyle #1}"{#2}, "\Leftarrow", #5]
}
\NewDocumentCommand{\rtwocell}{D(){} O{above=2pt} m m O{}}{
    \arrow[from=#3, to=#4, phantom, "{\scriptstyle #1}"{#2}, "\Rightarrow", #5]
}
\NewDocumentCommand{\urtwocell}{D(){} O{above left=-1pt} m m O{}}{
    \arrow[from=#3, to=#4, phantom, "{\scriptstyle #1}"{#2}, "\rotatebox{45}{$\Rightarrow$}", #5]
}
\NewDocumentCommand{\rutwocell}{D(){} O{above left=-1pt} m m O{}}{
    \arrow[from=#3, to=#4, phantom, "{\scriptstyle #1}"{#2}, "\rotatebox{45}{$\Rightarrow$}", #5]
}
\NewDocumentCommand{\drtwocell}{D(){} O{above right=0pt} m m O{}}{
    \arrow[from=#3, to=#4, phantom, "{\scriptstyle #1}"{#2}, "\rotatebox{-45}{$\Rightarrow$}", #5]
}
\NewDocumentCommand{\rdtwocell}{D(){} O{above right=0pt} m m O{}}{
    \arrow[from=#3, to=#4, phantom, "{\scriptstyle #1}"{#2}, "\rotatebox{-45}{$\Rightarrow$}", #5]
}
\NewDocumentCommand{\ultwocell}{D(){} O{above right=-1pt} m m O{}}{
    \arrow[from=#3, to=#4, phantom, "{\scriptstyle #1}"{#2}, "\rotatebox{-45}{$\Leftarrow$}", #5]
}
\NewDocumentCommand{\lutwocell}{D(){} O{above right=-1pt} m m O{}}{
    \arrow[from=#3, to=#4, phantom, "{\scriptstyle #1}"{#2}, "\rotatebox{-45}{$\Leftarrow$}", #5]
}
\NewDocumentCommand{\dltwocell}{D(){} O{above left=0pt} m m O{}}{
    \arrow[from=#3, to=#4, phantom, "{\scriptstyle #1}"{#2}, "\rotatebox{45}{$\Leftarrow$}", #5]
}
\NewDocumentCommand{\ldtwocell}{D(){} O{above left=0pt} m m O{}}{
    \arrow[from=#3, to=#4, phantom, "{\scriptstyle #1}"{#2}, "\rotatebox{45}{$\Leftarrow$}", #5]
}

\newcommand{\lcomp}{\mathord{\odot}}
\newcommand{\tcomp}{\mathord{\fatsemi}}
\newcommand{\comp}[1]{{{#1}_{\ast}}}
\newcommand{\conj}[1]{{{#1}^{\ast}}}
\newcommand{\Unit}{\mathsf{U}}
\newcommand{\compcell}[1]{{{#1}_{\dagger}}}
\newcommand{\conjcell}[1]{{{#1}^{\dagger}}}

\NewDocumentCommand{\Cells}{O{} m m m m}{
    {\zero{Cell}_{#1}\mkern-2mu\bigl(
        \begin{smallmatrix}
        & #4 & \\
        #2 & & #3 \\
        & #5 &
        \end{smallmatrix}
    \bigl)}
}
\newcommand{\Homcat}[1][]{\one{Hom}_{#1}}
\newcommand{\Homset}[1][]{\zero{Hom}_{#1}}
\newcommand{\Tcat}{\one{T}}
\newcommand{\Ttwocat}{\bi{T}}
\newcommand{\Lbicat}{\bi{L}}
\newcommand{\Larcat}{{\Tcat^1}}
\newcommand{\Lphancat}{{\Tcat^{\le 1}}}

\NewDocumentCommand{\Prof}{o}{%
    {%
        \IfValueT{#1}{#1\-}%
        \dbl{P}\mathrm{rof}
    }%
}
\NewDocumentCommand{\Mat}{o}{%
    {%
        \IfValueT{#1}{#1\-}%
        \dbl{M}\mathrm{at}
    }%
}
\NewDocumentCommand{\biProf}{o}{
    {%
        \IfValueT{#1}{#1\-}%
        \bi{P}\mathrm{rof}
    }%
}
\newcommand{\Mod}{{\dbl{M}\mathrm{od}}}
\renewcommand{\dim}[1]{{{#1}^{\flat}}} 

\newcommand{\Rel}{{\dbl{R}\mathrm{el}}}

\NewDocumentCommand{\extension}{m o m}{%
    \mathord{%
        #1\mkern-2mu\vartriangleright\mkern-5mu%
        \IfValueT{#2}{%
            {}^{#2}\mkern-3mu
        }%
        #3
    }%
}
\NewDocumentCommand{\lift}{m o m}{%
    \mathord{%
        \IfValueTF{#2}{#1^{#2}\mkern-6mu}{#1\mkern-4mu}%
        \blacktriangleleft\mkern-2mu%
        #3
    }%
}

\acrodef{AVDC}{augmented virtual double category}
\acrodefplural{AVDC}[AVDCs]{augmented virtual double categories}

\acrodef{VDC}{virtual double category}
\acrodefplural{VDC}[VDCs]{virtual double categories}

\acrodef{UVDC}{unitary virtual double category}
\acrodefplural{UVDC}[UVDCs]{unitary virtual double categories}

\acrodef{AVD}{augmented virtual double}

\acrodef{VD}{virtual double}

\newcommand{\zero}[1]{\mathrm{#1}}
\newcommand{\one}[1]{\mathbf{#1}}
\newcommand{\bi}[1]{\mathcal{#1}}
\newcommand{\dbl}[1]{\mathbb{#1}}

\newcommand{\Mdl}[2]{%
    \mathord{%
        \mathop{\one{Mdl}}({#1},{#2})
    }%
}
\newcommand{\Cone}{%
    \mathord{%
        \one{Cone}
    }%
}

\newcommand{\Ldbl}{\mathord{\dbl{V}}}
\newcommand{\Ddbl}{\mathord{\dbl{D}}}
\newcommand{\Idbl}{\mathord{\dbl{I}}}
\newcommand{\Idimdbl}{\mathord{\dim{\dbl{I}}}}

\newcommand{\slice}[2]{{\Tcat{#1}/{#2}}}
\newcommand{\Max}{{\one{Max}}}
\newcommand{\Fam}{{\one{Fam}}}

\newcommand{\dblslash}{/\mkern-6mu/}

\newcommand{\ie}[2]{{(#1{;}#2)}}
\newcommand{\hq}{\mathfrak{q}}
\newcommand{\hp}{\mathfrak{p}}
\newcommand{\coslice}[2]{{\one{S}\vvect{#1}{#2}}}
\newcommand{\fs}{\mathfrak{s}}
\newcommand{\fl}{\mathfrak{l}}
\newcommand{\fm}{\mathfrak{m}}

\newcommand{\Alg}{\mathop{\one{Alg}}}
\newcommand{\mate}[1]{\overline{#1}}

\title{Double categories of profunctors}
\author{Yuto Kawase}
\address{Research Institute for Mathematical Sciences, Kyoto University, Kyoto 606-8502, Japan}
\email{ykawase@kurims.kyoto-u.ac.jp}
\date{\today}
\thanks{%
    The author wishes to express his thanks to his supervisor, Prof.\ Masahito Hasegawa, for his support.
    The author is also grateful to Keisuke Hoshino, Yuki Imamura, Nathanael Arkor, and the anonymous referees for several valuable comments, and grateful to Hayato Nasu for providing the idea for the proof of the strongness theorem and for suggesting the term ``versatile colimits.''
    This work was supported by JSPS KAKENHI Grant Numbers JP24KJ1462.
}

\begin{document}
\begin{abstract}
    We characterize virtual double categories of enriched categories, functors, and profunctors by introducing a new notion of double-categorical colimits.
    Our characterization is strict in the sense that it is up to equivalence between virtual double categories and, at the level of objects, up to isomorphism of enriched categories.
    Throughout the paper, we treat enrichment in a unital virtual double category rather than in a bicategory or a monoidal category, and, for consistency and better visualization of pasting diagrams, we adopt augmented virtual double categories as a fundamental language for double-categorical concepts.
\end{abstract}

\maketitle
\tableofcontents

\section{Introduction}\label{sec:intro}
One key aspect of formal category theory is the study of profunctors.
Their behavior has classically been studied through \textit{proarrow equipments} \cite{Wood1982proarrowI,Wood1985proarrowII}, as introduced by Wood.
However, recently, \textit{(augmented) virtual double categories} have begun to be used instead with the expectation that they are a better refinement of proarrow equipments \cite{Koudenburg2024formal,ArkorMcdermott2024formal,Arkor2024nervetheorem,ArkorMcdermott2025relative}.
Enriched category theory is a prototypical stage for applying formal category theory.
For each monoidal category $\bi{V}$, we can obtain a virtual double category of $\bi{V}$-enriched profunctors, where we can do $\bi{V}$-enriched category theory.
The aim of the paper is to characterize such virtual double categories of enriched profunctors.

We will treat enrichment not only in a monoidal category but also in a bicategory \cite{Walters1982sheaves} or, moreover, in a virtual double category \cite{Leinster1999enrichment,Leinster2002enrichment,Leinster2004higher}.
Since it is more general and contains the other cases, we will focus on enrichment in a virtual double category.
Analogous to the monoidally enriched case, for each virtual double category $\bX$, we can obtain a new virtual double category $\Prof[\bX]$ of profunctors enriched in $\bX$.

In the paper, we give a characterization of virtual double categories equivalent to $\Prof[\bX]$ for some virtual double category $\bX$.
Besides $\Prof[\bX]$, we also characterize virtual double categories of the forms $\Mod(\bX)$ and $\Mat[\bX]$, arising from other constructions on virtual double categories $\bX$.
The first one is known as the \textit{module} construction \cite{Leinster1999enrichment,Leinster2004higher,CruttwellShulman2010unified}, and the latter is known as the \textit{matrix} construction in the bicategorical context \cite{BettiCarboniStreetWalters1983variation}.
Since the profunctor construction can be decomposed into them as $\Prof[\bX]=\Mod(\Mat[\bX])$, those three constructions are strongly related to each other.

Our strategy for characterization is parallel to the characterization of cocompletions in ordinary category theory.
Recall that, in ordinary category theory, cocompletions of a category under specific colimits are characterized by the following properties \cite[4.3.\ Proposition]{KellySchmitt2005notes}:
\begin{itemize}
    \item
        It has all colimits in mind;
    \item
        Every object can be written as such a colimit of ``atomic'' objects.
\end{itemize}
For instance, when considering filtered colimits, the atomic objects are precisely finitely presentable objects, while for coproducts, the atomic objects are connected objects.
If the base virtual double category $\bX$ is \textit{unital}, i.e., all objects admit a \textit{unit}, then it can be fully embedded into the virtual double category $\Prof[\bX]$ (\cref{thm:Z_is_embedding}).
Objects in $\bX$ are regarded as single-object $\bX$-categories there.
Then, every $\bX$-category is constructed by pasting single-object $\bX$-categories together and behaves like a ``colimit'' whose universal property extends in three directions in the virtual double category $\Prof[\bX]$ (\cref{thm:every_obj_is_versatile_colim}).
This observation leads us to a new notion of double-categorical colimits called \textit{versatile collages}, refining Street's \textit{collage} construction for profunctors \cite{Street1981cauchy}.
Then, $\Prof[\bX]$ can be regarded as a cocompletion of the enriching base $\bX$ under such colimits, and we obtain a cocompletion-like characterization of $\Prof[\bX]$ (\cref{thm:characterization_prof}).
That is, $\Prof[\bX]$ is determined by the following properties:
\begin{itemize}
    \item
        It has all versatile collages;
    \item
        Every object can be written as a versatile collage of \textit{collage-atomic} objects.
\end{itemize}
In addition to the unitality, our characterization theorem also requires \textit{iso-fibrancy} on the enriching base.
However, when we are considering enrichment in a bicategory, these conditions are satisfied automatically.
This indicates that our theorem is new even in the case of enrichment in a bicategory.

The same strategy works not only for the profunctor construction but also for the module and matrix constructions.
The notion of colimits corresponding to the module construction is called \textit{versatile collapses}, and the corresponding notion of colimits to the matrix construction is called \textit{versatile coproducts}.
These kinds of colimits are unified under a more general notion called \textit{versatile colimits}, which is also a new notion of double-categorical colimits and recovers the notion of colimits that Wood studied in \cite{Wood1985proarrowII}.

\begin{remark}
    A central ingredient in the theory of versatile colimits is the notion of ``cocones,'' which should be defined as a family of 0-coary cells, i.e., cells whose bottom boundary is of length 0.
    However, virtual double categories cannot naturally deal with 0-coary cells unless they are unital.
    While our virtual double categories $\Prof[\bX]$ and $\Mod(\bX)$ are unital, unfortunately, virtual double categories $\Mat[\bX]$ of matrices are not.
    To address this limitation, we adopt \textit{augmented virtual double categories} as a framework for developing the general theory of versatile colimits.
    Furthermore, we will regard every virtual double category as an augmented virtual double category and will consistently use the language of augmented virtual double categories throughout the paper.
    Although this is an experimental approach, the author believes that it enhances the coherence of the paper and unifies our treatment of double-categorical concepts.
\end{remark}

\vspace{1em}
\paragraph{\textbf{Related work}}
Bicategories of profunctors enriched in a bicategory were characterized by Street \cite{Street1981cauchy} and by Carboni et al.\ \cite{CarboniKasangianWalters1987axiomatics}.
Our characterization is a double-categorical refinement of Street's but differs considerably from these previous work with respect to ``strictness.''
In fact, our characterization is strict in the sense that it is up to equivalence between (augmented) virtual double categories and, at the level of objects, up to isomorphism of enriched categories.
On the other hand, the previous characterizations are up to biequivalence between bicategories and thus, at the level of objects, up to Morita equivalence of categories.

Related results also appear in the work of Garner and Shulman \cite{GarnerShulman2016enriched}, which characterizes, under certain assumptions, the profunctor construction as the cocompletion under a suitable notion of colimits.
Without going into the technical framework used there, their characterization can be reinterpreted as concerning pseudo double categories of profunctors enriched in a pseudo double category with companions and with appropriate colimits in the loose hom-categories.
Their characterization also differs from ours with respect to strictness: it is, at the level of objects, up to equivalence of categories.

Another related work appears in \cite[Proposition 4.23]{SpivakGarnerFairbanks2025functorial}, which gives a characterization of the (co)module construction from a monoidal category with suitable colimits.

\vspace{1em}
\paragraph{\textbf{Outline}}
In \cref{sec:prelim}, we first introduce basic concepts of augmented virtual double categories as the fundamental language of the paper.
We next recall enrichment in a virtual double category from \cite{Leinster1999enrichment,Leinster2002enrichment} and discuss its double-categorical aspects.

In \cref{sec:colim}, we develop the general theory of \textit{versatile colimits}, which is a new concept of double-categorical colimits.
We will show the \textit{unitality theorem} (\cref{thm:unitality_theorem}) and the \textit{strongness theorem} (\cref{thm:strongness_theorem}), which depict the behavior of versatile colimits.
In particular, the latter will play a crucial role in our characterization theorem.

\cref{sec:axioma} is devoted to the characterization of virtual double categories of enriched profunctors (\cref{thm:characterization_prof}), of modules (\cref{thm:characterization_mod}), and of matrices (\cref{thm:characterization_mat}).
These are the main theorems in the paper.
We will also apply them to the slice virtual double categories.

\cref{sec:proarrow_equipments} is devoted to a detailed investigation of how versatile colimits relate to the notion of colimits that Wood studied in \cite{Wood1985proarrowII}.

We also explore the notion of \textit{finality} with respect to versatile colimits (\cref{sec:final_functors}), which brings us a natural insight, especially when we are in a virtual equipment (\cref{sec:density_ve}).
However, since we can reach the main theorems without finality, it is driven to the appendix.

\vspace{1em}
\paragraph{\textbf{Notation and terminology}}
\begin{remark}
    In this paper, we will use the terms ``left'' and ``right'' relative to the direction of tight arrows in an augmented virtual double category.
    That is, ``left'' and ``right'' refer to the sides that lie to the left and right when facing in the direction of the tight arrow.
    For example, the terminologies \textit{left/right modules} (\cref{def:left_and_right_modules}) and \textit{left/right-pullingness} (\cref{def:pulling}) follow this convention.
    Since tight arrows are often written in the downward direction, our convention is opposite to the natural visual perception of left and right when viewing diagrams.
\end{remark}

\begin{remark}
    For clarity, let us declare the sizes of the categories we treat.
    We fix three Grothendieck universes $\U_0\in\U_1\in\U_2$.
    Elements in $\U_0$ are called \emph{small}, elements in $\U_1$ are called \emph{large}, elements in $\U_2$ are called \emph{huge}.
    Arbitrary sets (not necessarily in $\U_0$ nor $\U_1$ nor $\U_2$) are called \emph{classes}.
    However, we do not distinguish between small (resp.\ large; huge) sets and ``essentially'' small (resp.\ large; huge) sets, i.e., sets that are bijective to some small (resp.\ large; huge) set.
\end{remark}

\begin{remark}
    In order to distinguish different categorical structures by their dimensional levels, we use a systematic notation using distinct font styles:
    \begin{itemize}
        \item
            0-dimensional structures (such as sets) are written in roman font, e.g., $\zero{X}$, $\zero{Y}$, $\zero{Z}$.
        \item
            1-dimensional structures (such as ordinary categories and enriched categories) are written in bold font, e.g., $\one{X}$, $\one{Y}$, $\one{Z}$.
        \item
            2-dimensional structures (such as 2-categories, bicategories, and monoidal categories) are written in script font, e.g., $\bi{X}$, $\bi{Y}$, $\bi{Z}$.
        \item
            Double-dimensional structures (such as double categories) are written in blackboard font, e.g., $\dbl{X}$, $\dbl{Y}$, $\dbl{Z}$.\qedhere
    \end{itemize}
\end{remark}
\section{Preliminaries}\label{sec:prelim}
\subsection{Augmented virtual double categories}
As explained in the introduction, the main language of this paper is that of augmented virtual double categories, whose definition we recall below.
\subsubsection{The 2-category of augmented virtual double categories}
\begin{definition}[{\cite[1.2.\ Definition]{Koudenburg2020aug}}]
    An \emph{\ac{AVDC}} $\bL$ consists of the following data:
    \begin{itemize}
        \item
            A class $\Ob\bL$, whose elements are called \emph{objects} in $\bL$.
            We write $A\in\bL$ to mean $A\in\Ob\bL$.
        \item
            For $A,B\in\bL$, a class $\Homset[\bL]\vvect{A}{B}$, whose elements are called \emph{tight arrows} from $A$ to $B$ in $\bL$.
            The objects and the tight arrows are required to form a category $\Tcat\bL$, which is called the \emph{tight category} of $\bL$.
            We write $\id_A$ for the identity on an object $A\in\bL$.
            The composite of $A\arr(f)[][1]B\arr(g)[][1]C$ in $\Tcat\bL$ is denoted by $f\tcomp g$.
            Tight arrows are often written vertically:
            \begin{equation*}
                \begin{tikzcd}[scriptsize]
                    A\ar[d,"f"'] \\
                    B
                \end{tikzcd}
                \qquad
                \begin{tikzcd}[scriptsize]
                    A\ar[d,equal,"\id_A"] \\
                    A
                \end{tikzcd}\incat{\bL}
            \end{equation*}
        \item
            For $A,B\in\bL$, a class $\Homset[\bL](A,B)$, whose elements are called \emph{loose arrows} from $A$ to $B$ in $\bL$.
            A loose arrow is denoted by $\larr$ and is often written horizontally.
            A path of loose arrows $A_0\larr(u_1)A_1\larr(u_2)\cdots\larr(u_n)A_n$ is called a \emph{loose path} of length $n$ and is often denoted by a dashed arrow $A_0\larr(\tup{u})[path]A_n$.
            A loose path $v$ of length 0 or 1 is denoted by a dotted arrow $A\larr(v)[phan]B$.
            Note that $A=B$ is required when the loose path $v$ is of length 0.
        \item
            A class $\Cells[\bL]{f}{g}{\tup{u}}{v}$, whose elements are called \emph{cells}, for each ``boundary'' formed by loose arrows and tight arrows in the following way:
            \begin{equation*}
                \begin{tikzcd}
                    A_0\ar[d,"f"']\lar[r,path,"\tup{u}"] & A_n\ar[d,"g"] \\
                    B\lar[r,phan,"v"'] & C
                \end{tikzcd}\incat{\bL}.
            \end{equation*}
            Cells where $v$ is of length 1 (resp.\ 0) are called \emph{1-coary} (resp.\ \emph{0-coary}).
        \item
            Two kinds of special cells:
            \begin{equation*}
                \begin{tikzcd}
                    A\ar[d,equal]\lar[r,"u"] & B\ar[d,equal] \\
                    A\lar[r,"u"'] & B
                    \cellsymb(\veq_u){1-1}{2-2}
                \end{tikzcd}
                \qquad
                \begin{tikzcd}
                    A\ar[d,"f"{left},bend right=30]\ar[d,"f"{right},bend left=30] \\
                    B
                    \cellsymb(\heq_f){1-1}{2-1}
                \end{tikzcd}\incat{\bL}.
            \end{equation*}
            The cells $\veq_u$ on the left are called \emph{loose identity cells}.
            The cells $\heq_f$ on the right are called \emph{tight identity cells}.
        \item
            For cells $\alpha_1,\dots,\alpha_n,\beta$ on the left below, a cell $\tup{\alpha}\tcomp\beta$ of the following form:
            \begin{equation*}
                \begin{tikzcd}
                    A_0\ar[d,"f_0"']\lar[r,path,"\tup{u}_1"] & A_1\ar[d,"f_1"]\lar[r,path,"\tup{u}_2"] & \cdots\lar[r,path,"\tup{u}_n"] & A_n\ar[d,"f_n"] \\
                    B_0\ar[d,"g"']\lar[r,phan,"v_1"] & B_1\lar[r,phan,"v_2"] & \cdots\lar[r,phan,"v_n"] & B_n\ar[d,"h"] \\
                    C\lar[rrr,phan,"w"'] &&& D
                    \cellsymb(\alpha_1){1-1}{2-2}
                    \cellsymb(\alpha_2){1-2}{2-3}
                    \cellsymb(\alpha_n){2-3}{1-4}
                    \cellsymb(\beta){2-1}{3-4}
                \end{tikzcd}
                \mapsto
                \begin{tikzcd}
                    A_0\ar[d,"f_0\tcomp g"']\lar[r,path,"\tup{u}_1"] & A_1\lar[r,path,"\tup{u}_2"] & \cdots\lar[r,path,"\tup{u}_n"] & A_n\ar[d,"f_n\tcomp h"] \\
                    C\lar[rrr,phan,"w"'] &&& D
                    \cellsymb(\tup{\alpha}\tcomp\beta){1-1}{2-4}
                \end{tikzcd}
            \end{equation*}
            The composition defined by the assignments $(\alpha_1,\dots,\alpha_n,\beta)\mapsto\tup{\alpha}\tcomp\beta$ is required to satisfy a suitable associative law and a unit law with identity cells.
            See \cite[1.2.\ Definition]{Koudenburg2020aug} for more detail.\qedhere
    \end{itemize}
\end{definition}

\begin{notation}
    Let $A_0\larr(\tup{u})[path]A_n$ be a loose path of length $n$ in an \ac{AVDC}.
    We extend the notation for the loose identity cells as follows:
    \begin{equation}\label{eq:notaion_for_loose_identities}
        \begin{tikzcd}
            A_0\ar[d,equal]\lar[r,path,"\tup{u}"] & A_n\ar[d,equal] \\
            A_0\lar[r,path,"\tup{u}"'] & A_n
            \cellsymb(\veq_\tup{u}){1-1}{2-2}
        \end{tikzcd}
    \end{equation}
    When $n\ge 1$, the notation \cref{eq:notaion_for_loose_identities} means the path $(\veq_{u_1},\dots,\veq_{u_n})$ of loose identity cells.
    When $n=0$, the notation \cref{eq:notaion_for_loose_identities} means the tight identity cell $\heq_{\id_{A_0}}$, where $A_0=A_n$.
\end{notation}

\begin{notation}
    Let $\alpha_1,\dots,\alpha_n$ be cells in an \ac{AVDC} of the following form:
    \begin{equation}\label{eq:path_of_cells}
        \begin{tikzcd}
            A_0\ar[d,"f_0"']\lar[r,path,"\tup{u}_1"] & A_1\ar[d,"f_1"]\lar[r,path,"\tup{u}_2"] & \cdots\lar[r,path,"\tup{u}_n"] & A_n\ar[d,"f_n"] \\
            B_0\lar[r,phan,"v_1"'] & B_1\lar[r,phan,"v_2"'] & \cdots\lar[r,phan,"v_n"'] & B_n
            \cellsymb(\alpha_1){1-1}{2-2}
            \cellsymb(\alpha_2){1-2}{2-3}
            \cellsymb(\alpha_n){2-3}{1-4}
        \end{tikzcd}
    \end{equation}
    When the composite path $\tup{v}$ of $v_1,\dots,v_n$ is of length $\le 1$, we use the same notation \cref{eq:path_of_cells} for the composite of the following cells:
    \begin{equation*}
        \begin{tikzcd}
            A_0\ar[d,"f_0"']\lar[r,path,"\tup{u}_1"] & A_1\ar[d,"f_1"]\lar[r,path,"\tup{u}_2"] & \cdots\lar[r,path,"\tup{u}_n"] & A_n\ar[d,"f_n"] \\
            B_0\ar[d,equal]\lar[r,phan,"v_1"] & B_1\lar[r,phan,"v_2"] & \cdots\lar[r,phan,"v_n"] & B_n\ar[d,equal] \\
            B_0\lar[rrr,phan,"\tup{v}"'] &&& B_n
            \cellsymb(\alpha_1){1-1}{2-2}
            \cellsymb(\alpha_2){1-2}{2-3}
            \cellsymb(\alpha_n){2-3}{1-4}
            \cellsymb(\veq){2-1}{3-4}
        \end{tikzcd}
    \end{equation*}
    For example, the following exhibits a cell given by composition:
    \begin{equation}\label{eq:example_0-coary_composite}
        \begin{tikzcd}
            A_0\ar[dr,"f_0"']\lar[r,path,"\tup{u}_1"] & A_1\ar[d,"f_1"{right=-1}]\lar[r,path,"\tup{u}_2"] & A_2\ar[dl,equal]\ar[d,"f_3"] \\
            & A_2\lar[r,phan,"v_3"'] & B_3
            \cellsymb(\alpha_1)[above right]{1-1}{2-2}
            \cellsymb(\alpha_2)[above left]{1-3}{2-2}
            \cellsymb(\alpha_3)[below right]{1-3}{2-2}
        \end{tikzcd}
    \end{equation}
    Note that, by the laws of cell composition in \acp{AVDC}, the cell \cref{eq:example_0-coary_composite} coincides with another composite of the following cells.
    \begin{equation*}
        \begin{tikzcd}[smallcolumn]
            A_0\ar[drr,"f_0"']\lar[rr,path,"\tup{u}_1"] &[-20pt] & A_1\ar[d,"f_1"{right=-1}]\lar[rr,path,"\tup{u}_2"] & &[-20pt] A_2\ar[dll,equal] \\
            && A_2\ar[dl,equal]\ar[dr,"f_3"] && \\
            & A_2\lar[rr,phan,"v_3"'] &{}& B_3 &
            \cellsymb(\alpha_1)[above right]{1-1}{2-3}
            \cellsymb(\alpha_2)[above left]{1-5}{2-3}
            \cellsymb(\alpha_3){2-3}{3-3}
        \end{tikzcd}
    \end{equation*}
\end{notation}

\begin{notation}
    Let $\bL$ be an \ac{AVDC}.
    We write $\Ttwocat\bL$ for the 2-category defined as follows:
    The underlying category is $\Tcat\bL$; 2-cells from $f$ to $g$ are cells whose top and bottom boundaries are of length 0 and whose left and right boundaries are $f$ and $g$, respectively.
    The 2-category $\Ttwocat\bL$ is called the \emph{tight 2-category} of $\bL$.
\end{notation}

\begin{notation}
    Let $\bL$ be an \ac{AVDC}, and let $A,B\in\bL$.
    We write $\Homcat[\bL]\vvect{A}{B}$ for the category of tight arrows from $A$ to $B$, i.e., the hom-category of the 2-category $\Ttwocat\bL$.
    In addition, loose arrows from $A$ to $B$ also form a category $\Homcat[\bL](A,B)$.
    Here, the hom-set from $A\larr(u)[][1]B$ to $A\larr(v)[][1]B$ is defined as $\Cells{\id_A}{\id_B}{u}{v}$.
\end{notation}

\begin{example}\label{eg:avdc_of_relations}
    The \ac{AVDC} $\Rel$ is defined as follows:
    \begin{itemize}
        \item
            An object is a (large) set.
        \item
            A tight arrow is a map.
        \item
            A loose arrow $\zero{X}\larr \zero{Y}$ is a relation $R\subseteq\zero{X}\times\zero{Y}$.
        \item
            $\Rel$ has at most one cell for every boundary.
            A 1-coary cell on the left below exists if and only if, for any $x_0\in\zero{X}_0,\dots,x_n\in\zero{X}_n$, the conjunction of $(x_0,x_1)\in R_1,\dots,(x_{n-1},x_n)\in\nolinebreak R_n$ implies $(f(x_0),g(x_n))\in S$.
            A 0-coary cell on the right below exists if and only if, for any $x_0\in\zero{X}_0,\dots,x_n\in\zero{X}_n$, the conjunction of $(x_0,x_1)\in R_1,\dots,(x_{n-1},x_n)\in R_n$ implies $f(x_0)=g(x_n)$.
            \begin{equation*}
                \begin{tikzcd}
                    \zero{X}_0\ar[d,"f"']\lar[r,path,"\tup{R}"] & \zero{X}_n\ar[d,"g"] \\
                    \zero{Y}\lar[r,"S"'] & \zero{Z}
                    \cellsymb(\cdot){1-1}{2-2}
                \end{tikzcd}
                \qquad
                \begin{tikzcd}[tri]
                    \zero{X}_0\ar[dr,"f"']\lar[rr,path,"\tup{R}"] &{}& \zero{X}_n\ar[dl,"g"] \\
                    & \zero{Y} &
                    \cellsymb(\cdot)[above=-1]{1-2}{2-2}
                \end{tikzcd}\incat{\Rel}
            \end{equation*}\qedhere
    \end{itemize}
\end{example}

\begin{definition}[{\cite[3.1.\ Definition]{Koudenburg2020aug}}]
    Let $\bK$ and $\bL$ be \acp{AVDC}.
    An \emph{\ac{AVD}-functor} $\bK\arr(F)\bL$ consists of the following data:
    \begin{itemize}
        \item
            A functor $F\colon\Tcat\bK\to\Tcat\bL$, which is also denoted by $\Tcat F$.
        \item
            Assignments of loose arrows
            \[
                A\larr(u)B\incat{\bK}
                \quad\mapsto\quad
                FA\larr(Fu)FB\incat{\bL}.
            \]
            In what follows, we extend the assignments from loose arrows to loose paths.
            Specifically, $F\tup{u}=F(u_1,\dots,u_n)\coloneq(Fu_1,\dots,Fu_n)$.
        \item
            Assignments of cells
            \[
                \begin{tikzcd}
                    A\ar[d,"f"']\lar[r,path,"\tup{u}"] & B\ar[d,"g"] \\
                    X\lar[r,phan,"v"'] & Y
                    \cellsymb(\alpha){1-1}{2-2}
                \end{tikzcd}\incat{\bK}
                \quad\mapsto\quad
                \begin{tikzcd}
                    FA\ar[d,"Ff"']\lar[r,path,"F\tup{u}"] & FB\ar[d,"Fg"] \\
                    FX\lar[r,phan,"Fv"'] & FY
                    \cellsymb(F\alpha){1-1}{2-2}
                \end{tikzcd}\incat{\bL}.
            \]
    \end{itemize}
    These are required to satisfy the following:
    \begin{itemize}
        \item
            For any composable cells
            \begin{equation*}
                \begin{tikzcd}
                    A_0\ar[d,"f_0"']\lar[r,path,"\tup{u}_1"] & A_1\ar[d,"f_1"']\lar[r,path,"\tup{u}_2"] & \cdots\lar[r,path,"\tup{u}_n"] & A_n\ar[d,"f_n"] \\
                    B_0\ar[d,"g"']\lar[r,phan,"v_1"'] & B_1\lar[r,phan,"v_2"'] & \cdots\lar[r,phan,"v_n"'] & B_n\ar[d,"h"] \\
                    X\lar[rrr,phan,"w"'] & & & Y
                    \cellsymb(\alpha_1){1-1}{2-2}
                    \cellsymb(\alpha_2){1-2}{2-3}
                    \cellsymb(\alpha_n){1-4}{2-3}
                    \cellsymb(\beta){2-1}{3-4}
                \end{tikzcd}
                =
                \begin{tikzcd}
                    A_0\ar[d,"f_0"']\lar[r,path,"\tup{u}_1"] & A_1\lar[r,path,"\tup{u}_2"] & \cdots\lar[r,path,"\tup{u}_n"] & A_n\ar[d,"f_n"] \\
                    B_0\ar[d,"g"'] &  &  & B_n\ar[d,"h"] \\
                    X\lar[rrr,phan,"w"'] & & & Y
                    \cellsymb(\tup{\alpha}\tcomp\beta){1-1}{3-4}
                \end{tikzcd}
                \incat{\bK},
            \end{equation*}
            the equality $F\tup{\alpha}\tcomp F\beta = F(\tup{\alpha}\tcomp\beta)$ holds.
            \begin{equation*}
                \begin{tikzcd}
                    FA_0\ar[d,"Ff_0"']\lar[r,path,"F\tup{u}_1"] & FA_1\ar[d,"Ff_1"']\lar[r,path,"F\tup{u}_2"] & \cdots\lar[r,path,"F\tup{u}_n"] & FA_n\ar[d,"Ff_n"] \\
                    FB_0\ar[d,"Fg"']\lar[r,phan,"Fv_1"'] & FB_1\lar[r,phan,"Fv_2"'] & \cdots\lar[r,phan,"Fv_n"'] & FB_n\ar[d,"Fh"] \\
                    FX\lar[rrr,phan,"Fw"'] & & & FY
                    \cellsymb(F\alpha_1){1-1}{2-2}
                    \cellsymb(F\alpha_2){1-2}{2-3}
                    \cellsymb(F\alpha_n){1-4}{2-3}
                    \cellsymb(F\beta){2-1}{3-4}
                \end{tikzcd}
                =
                \begin{tikzcd}
                    FA_0\ar[d,"Ff_0"']\lar[r,path,"F\tup{u}_1"] & FA_1\lar[r,path,"F\tup{u}_2"] & \cdots\lar[r,path,"F\tup{u}_n"] & FA_n\ar[d,"Ff_n"] \\
                    FB_0\ar[d,"Fg"'] & & & FB_n\ar[d,"Fh"] \\
                    FX\lar[rrr,phan,"Fw"'] & & & FY
                    \cellsymb(F(\tup{\alpha}\tcomp\beta)){1-1}{3-4}
                \end{tikzcd}
                \incat{\bL}
            \end{equation*}
        \item
            For any $A\larr(u)B$ in $\bK$, the equality $F\veq_u=\veq_{Fu}$ holds.
            \begin{equation*}
                \begin{tikzcd}
                    A\ar[d,equal]\lar[r,"u"] & B\ar[d,equal] \\
                    A\lar[r,"u"'] & B
                    \cellsymb(\veq_u){1-1}{2-2}
                \end{tikzcd}
                \mapsto
                \begin{tikzcd}
                    FA\ar[d,equal]\lar[r,"Fu"] & FB\ar[d,equal] \\
                    FA\lar[r,"Fu"'] & FB
                    \cellsymb(F\veq_u){1-1}{2-2}
                \end{tikzcd}
                =
                \begin{tikzcd}
                    FA\ar[d,equal]\lar[r,"Fu"] & FB\ar[d,equal] \\
                    FA\lar[r,"Fu"'] & FB
                    \cellsymb(\veq_{Fu}){1-1}{2-2}
                \end{tikzcd}
            \end{equation*}
        \item
            For any $A\arr(f)B$ in $\bK$, the equality $F\heq_f=\heq_{Ff}$ holds.
            \begin{equation*}
                \begin{tikzcd}
                    A\ar[d,bend right=30,"f"{left}]\ar[d,bend left=30,"f"{right}] \\
                    B
                    \cellsymb(\heq_f){1-1}{2-1}
                \end{tikzcd}
                \mapsto
                \begin{tikzcd}
                    FA\ar[d,bend right=30,shift right=2,"Ff"{left}]\ar[d,bend left=30,shift left=2,"Ff"{right}] \\
                    FB
                    \cellsymb(F\heq_f){1-1}{2-1}
                \end{tikzcd}
                =
                \begin{tikzcd}
                    FA\ar[d,bend right=30,shift right=2,"Ff"{left}]\ar[d,bend left=30,shift left=2,"Ff"{right}] \\
                    FB
                    \cellsymb(\heq_{Ff}){1-1}{2-1}
                \end{tikzcd}
            \end{equation*}
    \end{itemize}
\end{definition}

\begin{definition}[{\cite[3.2.\ Definition]{Koudenburg2020aug}}]
    Let $F,G\colon\bK\to\bL$ be \ac{AVD}-functors between \acp{AVDC}.
    A \emph{tight \ac{AVD}-transformation} $F\arr(\rho)[Rightarrow]G$ consists of:
    \begin{itemize}
        \item
            for each $A\in\bK$, a tight arrow
            $
            \begin{tikzcd}[small]
                FA\ar[d,"\rho_A"']\\
                GA
            \end{tikzcd}
            $
            in $\bL$;
        \item
            for each $A\larr(u)B$ in $\bK$, a cell
            $
            \begin{tikzcd}
                FA\ar[d,"\rho_A"']\lar[r,"Fu"] & FB\ar[d,"\rho_B"] \\
                GA\lar[r,"Gu"'] & GB
                \cellsymb(\rho_u){1-1}{2-2}
            \end{tikzcd}
            $
            in $\bL$
    \end{itemize}
    satisfying the following:
    \begin{itemize}
        \item
            $\rho$ yields a natural transformation
            $
            \begin{tikzcd}[scriptsize]
                \Tcat\bK\ar[rr,"F",bend left=20]\ar[rr,"G"',bend right=20] & & \Tcat\bL,
                \dtwocell(\rho){1-1}{1-3}
            \end{tikzcd}
            $
            i.e., for any $A\arr(f)B$ in $\bK$,
            \begin{equation*}
                \begin{tikzcd}[tinytri]
                    & FA\ar[dl,"\rho_A"']\ar[dr,"Ff"] & \\
                    GA\ar[dr,"Gf"'] && FB\ar[dl,"\rho_B"] \\
                    & GB &
                    \cellsymb(\heq){2-1}{2-3}
                \end{tikzcd}\incat{\bL}.
            \end{equation*}
        \item
            For any 1-coary cell
            \begin{equation*}
                \begin{tikzcd}[small]
                    A_0\ar[d,"f"']\lar[r,"u_1"] & A_1\lar[r,"u_2"] & \cdots\lar[r,"u_n"] & A_n\ar[d,"g"] \\
                    X\lar[rrr,"v"'] &&& Y
                    \cellsymb(\alpha){1-1}{2-4}
                \end{tikzcd}\incat{\bK},
            \end{equation*}
            the following equality holds.
            \begin{equation*}
                \begin{tikzcd}
                    FA_0\ar[d,"\rho_{A_0}"']\lar[r,"Fu_1"] & FA_1\ar[d,"\rho_{A_1}"']\lar[r,"Fu_2"] & \cdots\lar[r,"Fu_n"] & FA_n\ar[d,"\rho_{A_n}"] \\
                    GA_0\ar[d,"Gf"']\lar[r,"Gu_1"'] & GA_1\lar[r,"Gu_2"'] & \cdots\lar[r,"Gu_n"'] & GA_n\ar[d,"Gg"] \\
                    GX\lar[rrr,"Gv"'] & & & GY
                    \cellsymb(\rho_{u_1}){1-1}{2-2}
                    \cellsymb(\rho_{u_2}){1-2}{2-3}
                    \cellsymb(\rho_{u_n}){1-4}{2-3}
                    \cellsymb(G\alpha){2-1}{3-4}
                \end{tikzcd}
                =
                \begin{tikzcd}
                    FA_0\ar[d,"Ff"']\lar[r,"Fu_1"] & FA_1\lar[r,"Fu_2"] & \cdots\lar[r,"Fu_n"] & FA_n\ar[d,"Fg"] \\
                    FX\ar[d,"\rho_X"']\lar[rrr,"Fv"'] &&& FY\ar[d,"\rho_Y"] \\
                    GX\lar[rrr,"Gv"'] &&& GY
                    \cellsymb(F\alpha){1-1}{2-4}
                    \cellsymb(\rho_v){2-1}{3-4}
                \end{tikzcd}
            \end{equation*}
            When $n=0$, the left-hand cell above is $\heq_{\rho_{A_0}}\tcomp G\alpha$.
        \item
            For any 0-coary cell
            \begin{equation*}
                \begin{tikzcd}[tri]
                    A_0\ar[dr,"f"']\lar[r,"u_1"] & \cdots\lar[r,"u_n"] & A_n\ar[dl,"g"] \\
                    & X &
                    \cellsymb(\alpha)[above=-1]{1-2}{2-2}
                \end{tikzcd}\incat{\bK},
            \end{equation*}
            the following equality holds.
            \begin{equation*}
                \begin{tikzcd}
                    FA_0\ar[d,"\rho_{A_0}"']\lar[r,"Fu_1"] & \cdots\lar[r,"Fu_n"] & FA_n\ar[d,"\rho_{A_n}"] \\
                    GA_0\ar[dr,"Gf"']\lar[r,"Gu_1"'] & \cdots\lar[r,"Gu_n"'] & GA_n\ar[dl,"Gg"] \\
                    & GX &
                    \cellsymb(\rho_{u_1}){1-1}{2-2}
                    \cellsymb(\rho_{u_n}){1-3}{2-2}
                    \cellsymb(G\alpha){2-2}{3-2}
                \end{tikzcd}
                =
                \begin{tikzcd}
                    FA_0\ar[dr,"Ff"']\lar[r,"Fu_1"] & \cdots\lar[r,"Fu_n"] & FA_n\ar[dl,"Fg"] \\
                    & FX\ar[d,"\rho_X"{left},bend right=30]\ar[d,"\rho_X"{right},bend left=30] & \\
                    & GX &
                    \cellsymb(F\alpha)[above=-4]{1-2}{2-2}
                    \cellsymb(\heq){2-2}{3-2}
                \end{tikzcd}
            \end{equation*}
            When $n=0$, the left-hand cell above is $\heq_{\rho_{A_0}}\tcomp G\alpha$.\qedhere
    \end{itemize}
\end{definition}

\begin{notation}
    The huge \acp{AVDC}, \ac{AVD}-functors, and tight \ac{AVD}-transformations form a 2-category \cite[3.3.\ Proposition]{Koudenburg2020aug}, which is denoted by $\AVDC$.
\end{notation}

\begin{definition}
    Let $\bL$ be an \ac{AVDC}.
    A \emph{full sub-\ac{AVDC}} of $\bL$ is an \ac{AVDC} whose class of objects is a subclass of $\Ob\bL$ and whose ``local'' classes of tight arrows, loose arrows, and cells are identical to those of $\bL$.
    Additionally, all compositions and identities in the full sub-\ac{AVDC} are required to be inherited directly from $\bL$.
\end{definition}

The following is convenient to treat virtual-double-categorical concepts in the augmented-virtual-double-categorical setting.
\begin{definition}
    An \ac{AVDC} is called \emph{diminished} if all 0-coary cells are tight identity cells, that is, $\heq_f$ for some tight arrow $f$.
\end{definition}

\begin{notation}
    Let $\bL$ be an \ac{AVDC}.
    We write $\dim{\bL}$ for the diminished \ac{AVDC} obtained by removing all 0-coary cells, except for tight identity cells, from $\bL$.
\end{notation}

\begin{remark}
    Diminished \acp{AVDC} are essentially the same concept as \emph{\acp{VDC}} \cite[2.1.\ Definition]{CruttwellShulman2010unified}, which are also called \emph{$\mathsf{fc}$-multicategories} \cite{Leinster1999enrichment,Leinster2002enrichment,Leinster2004higher} and were originally introduced in \cite{Burroni1971tcategories}.
    Indeed, the \ac{AVD}-functors between diminished \acp{AVDC} correspond to the \emph{\ac{VD}-functors} between \acp{VDC}.
\end{remark}

\subsubsection{Equivalences in the 2-category $\AVDC$}
\begin{notation}
    For an \ac{AVDC} $\bL$, let $\Lphancat\bL$ denote a category defined as follows:
    \begin{itemize}
        \item
            An object is a loose path $A^0\larr(A)[phan]A^1$ in $\bL$ of length $\le 1$.
        \item
            A morphism from $A^0\larr(A)[phan][1]A^1$ to $B^0\larr(B)[phan][1]B^1$ is a tuple $(\alpha^0,\alpha^1,\alpha)$ of the following form:
            \begin{equation*}
                \begin{tikzcd}
                    A^0\ar[d,"\alpha^0"']\lar[r,phan,"A"] & A^1\ar[d,"\alpha^1"] \\
                    B^0\lar[r,phan,"B"'] & B^1
                    \cellsymb(\alpha){1-1}{2-2}
                \end{tikzcd}\incat{\bL}.
            \end{equation*}
    \end{itemize}
    We write $\Larcat\bL$ for the full subcategory of $\Lphancat\bL$ consisting of paths of length 1, i.e., loose arrows.
\end{notation}

\begin{definition}[Loosewise invertible cells]
    Let $\bL$ be an \ac{AVDC}.
    Isomorphisms in the category $\Tcat\bL$ are called \emph{invertible tight arrows}.
    Two objects in $\bL$ are called \emph{tightwise isomorphic} if there is an invertible tight arrow between them.
    Isomorphisms in the category $\Lphancat\bL$ are called \emph{loosewise invertible cells} and are often denoted by the symbol ``$\cong$'' as follows:
    \begin{equation*}
        \begin{tikzcd}
            \cdot\ar[d,"f"']\lar[r,phan] & \cdot\ar[d,"g"] \\
            \cdot\lar[r,phan] & \cdot
            \cellsymb(\linv){1-1}{2-2}
        \end{tikzcd}\incat{\bL}
    \end{equation*}
    For a loosewise invertible cell of the above form, the tight arrows $f$ and $g$ automatically become invertible.
\end{definition}

\begin{theorem}[{\cite[3.8.\ Proposition]{Koudenburg2020aug}}]\label{thm:equiv_in_AVDC}
    An \ac{AVD}-functor $F\colon\bK\to\bL$ is a part of an equivalence in the 2-category $\AVDC$ if and only if it satisfies the following conditions:
    \begin{itemize}
        \item
            The assignments $\alpha\mapsto F\alpha$ induce bijections $\Cells[\bK]{f}{g}{\tup{u}}{v}\cong\Cells[\bL]{Ff}{Fg}{F\tup{u}}{Fv}$;
        \item
            The assignments $f\mapsto Ff$ induce bijections $\Homset[\bK]\vvect{A}{B}\cong\Homset[\bL]\vvect{FA}{FB}$;
        \item
            We can simultaneously make the following choices:
            \begin{itemize}
                \item
                    for each $A\in\bL$, an object $A'\in\bK$ and an invertible tight arrow $FA'\arr(\epsilon_A)[][1]A$ in $\bL$;
                \item
                    for each $A\larr(u)B$ in $\bL$, a loose arrow $A'\larr(u')B'$ in $\bK$ and a loosewise invertible cell
                    \begin{equation*}
                        \begin{tikzcd}
                            FA'\ar[d,"\epsilon_A"']\lar[r,"Fu'"] & FB'\ar[d,"\epsilon_B"] \\
                            A\lar[r,"u"'] & B
                            \cellsymb(\linv){1-1}{2-2}
                        \end{tikzcd}\incat{\bL}.
                    \end{equation*}
            \end{itemize}
    \end{itemize}
\end{theorem}

\subsubsection{Cartesian cells}
\begin{definition}[{Cartesian cells \cite[4.1.\ Definition]{Koudenburg2020aug}}]
    A cell
    \begin{equation}\label{eq:phantom_cell}
        \begin{tikzcd}
            X^0\ar[d,"\alpha^0"']\lar[r,phan,"X"] & X^1\ar[d,"\alpha^1"] \\
            Y^0\lar[r,phan,"Y"'] & Y^1
            \cellsymb(\alpha){1-1}{2-2}
        \end{tikzcd}
    \end{equation}
    in an \ac{AVDC} is called \emph{cartesian} if it satisfies the following condition:
    Suppose that we are given a loose path $A\larr(\tup{u})[path]B$, tight arrows $A\arr(f)[][1]X^0$ and $B\arr(g)[][1]X^1$, and a cell $\beta$ on the right below; then there uniquely exists a cell $\gamma$ satisfying the following equation.
    \begin{equation*}
        \begin{tikzcd}
            A\ar[d,"f"']\lar[r,path,"\tup{u}"] & B\ar[d,"g"] \\
            X^0\ar[d,"\alpha^0"']\lar[r,phan,"X"] & X^1\ar[d,"\alpha^1"] \\
            Y^0\lar[r,phan,"Y"'] & Y^1
            \cellsymb(\gamma){1-1}{2-2}
            \cellsymb(\alpha){2-1}{3-2}
        \end{tikzcd}
        =
        \begin{tikzcd}
            A\ar[d,"f"']\lar[r,path,"\tup{u}"] & B\ar[d,"g"] \\
            X^0\ar[d,"\alpha^0"'] & X^1\ar[d,"\alpha^1"] \\
            Y^0\lar[r,phan,"Y"'] & Y^1
            \cellsymb(\beta){1-1}{3-2}
        \end{tikzcd}
    \end{equation*}
    We will use the symbol ``$\cart$'' to represent a cartesian cell:
    \begin{equation*}
        \begin{tikzcd}[largecolumn]
            \cdot\ar[d]\lar[r,phan] & \cdot\ar[d] \\
            \cdot\lar[r,phan] & \cdot
            \cellsymb(\cart){1-1}{2-2}
        \end{tikzcd}
    \end{equation*}
\end{definition}

\begin{proposition}\label{prop:cart_and_loosewise_inv}
    Let $\alpha$ be a cell of the form \cref{eq:phantom_cell} in an \ac{AVDC}, and suppose that $\alpha^0$ and $\alpha^1$ are invertible.
    Then, the cell $\alpha$ is cartesian if and only if it is loosewise invertible.
    In particular, every loosewise invertible cell is cartesian.
\end{proposition}
\begin{proof}
    Straightforward.
\end{proof}

\begin{proposition}[{Pasting lemma \cite[4.15.\ Lemma]{Koudenburg2020aug}}]\label{prop:pasting_lemma_cartesian}
    Let $\alpha$ and $\beta$ be cells of the following forms in an \ac{AVDC}.
    \begin{equation*}
        \begin{tikzcd}
            X^0\ar[d,"\alpha^0"']\lar[r,phan,"X"] & X^1\ar[d,"\alpha^1"] \\
            Y^0\ar[d,"\beta^0"']\lar[r,phan,"Y"] & Y^1\ar[d,"\beta^1"] \\
            Z^0\lar[r,phan,"Z"'] & Z^1
            \cellsymb(\alpha){1-1}{2-2}
            \cellsymb(\beta){2-1}{3-2}
        \end{tikzcd}
    \end{equation*}
    Suppose that $\beta$ is cartesian.
    Then, $\alpha$ is cartesian if and only if the composite $\alpha\tcomp\beta$ is cartesian.
\end{proposition}

\begin{definition}[Restrictions]
    Suppose that we are given a cartesian cell in an \ac{AVDC} of the following form:
    \begin{equation*}
        \begin{tikzcd}
            \cdot\ar[d,"f"']\lar[r,"p"] & \cdot\ar[d,"g"] \\\
            X\lar[r,phan,"u"'] & Y
            \cellsymb(\cart){1-1}{2-2}
        \end{tikzcd}
    \end{equation*}
    \begin{enumerate}
        \item
            Since the loose arrow $p$ is unique up to a loosewise invertible cell, we call $p$ the \emph{restriction} of $u$ along $f$ and $g$ and write $u(f,g)$ for it.
            When $u$ is of length 0 (hence $X=Y$), we also write $X(f,g)$ for $p$.
            To emphasize that $u$ is of length 1 (resp.\ 0), we sometimes call $u(f,g)$ the \emph{1-coary restriction} (resp.\ \emph{0-coary restriction}).
            \begin{equation*}
                \begin{tikzcd}
                    \cdot\ar[d,"f"']\lar[r,"{u(f,g)}"] & \cdot\ar[d,"g"] \\\
                    X\lar[r,phan,"u"'] & Y
                    \cellsymb(\cart){1-1}{2-2}
                \end{tikzcd}
                \quad
                \begin{tikzcd}[tri]
                    \cdot\ar[dr,"f"']\lar[rr,"{X(f,g)}"] &{}& \cdot\ar[dl,"g"] \\
                    & X &
                    \cellsymb(\cart)[above=-2]{1-2}{2-2}
                \end{tikzcd}
            \end{equation*}
        \item
            When $g=\id$ and $u$ is of length 0, we call $p$ the \emph{companion} of $f$ and write $\comp{f}$ for it.
            When $f=\id$ and $u$ is of length 0, we call $p$ the \emph{conjoint} of $g$ and write $\conj{g}$ for it.
            We write $\compcell{f}$ and $\conjcell{g}$ for the associated cartesian cells as follows:
            \begin{equation*}
                \begin{tikzcd}[tri]
                    \cdot\ar[dr,"f"']\lar[rr,"\comp{f}"] &{}& X\ar[dl,equal] \\
                    & X &
                    \cellsymb(\compcell{f})[above=-5]{1-2}{2-2}
                \end{tikzcd}\colon\cart
                \qquad
                \begin{tikzcd}[tri]
                    X\ar[dr,equal]\lar[rr,"\conj{g}"] &{}& \cdot\ar[dl,"g"] \\
                    & X &
                    \cellsymb(\conjcell{g})[above=-4]{1-2}{2-2}
                \end{tikzcd}\colon\cart
            \end{equation*}
        \item
            When $f=g=\id$ and $u$ is of length 0, we call $p$ the \emph{loose unit} on $X$ and write $\Unit_X$ for it.
            Note that the associated cartesian cell is loosewise invertible automatically:
            \begin{equation*}
                \begin{tikzcd}[tri]
                    X\ar[dr,equal]\lar[rr,"\Unit_X"] &{}& X\ar[dl,equal] \\
                    & X &
                    \cellsymb(\linv)[above=-4]{1-2}{2-2}
                \end{tikzcd}\colon\cart
            \end{equation*}\qedhere
    \end{enumerate}
\end{definition}

\begin{definition}
    Let $\bL$ be an \ac{AVDC}.
    We say $\bL$ \emph{has restrictions} (resp.\ \emph{1-coary restrictions}) if the restriction $u(f,g)$ exists for any $f$, $g$, and $u$ of length $\le 1$ (resp.\ length 1).
    We say $\bL$ \emph{has companions} (resp.\ \emph{conjoints}) if the companion $\comp{f}$ (resp.\ conjoint $\conj{f}$) exists for any $f$.
    We say $\bL$ \emph{has loose units} if the loose unit $\Unit_X$ exists for any $X$.
    We refer to such an $\bL$ as an \ac{AVDC} with restrictions, companions, etc.
\end{definition}

\begin{proposition}[{\cite[5.4.\ Lemma]{Koudenburg2020aug}}]
    Let $A\arr(f)[][1]X$ be a tight arrow in an \ac{AVDC}.
    Then, the following data correspond bijectively to each other:
    \begin{enumerate}
        \item
            A pair $(p,\epsilon)$ of a loose arrow $A\larr(p)X$ and a cartesian cell
            \begin{equation*}
                \begin{tikzcd}[tri]
                    A\ar[dr,"f"']\lar[rr,"p"] &{}& X\ar[dl,equal] \\
                    & X &
                    \cellsymb(\epsilon)[above=-2]{1-2}{2-2}
                \end{tikzcd}\colon\cart,
            \end{equation*}
            which defines $p$ as the companion of $f$.
        \item
            A tuple $(p,\eta,\epsilon)$ of a loose arrow $A\larr(p)X$ and cells $\eta,\epsilon$ satisfying the following equations:
            \begin{equation*}
                \begin{tikzcd}
                    & A\ar[dl,equal]\ar[d,"f"] \\
                    A\ar[d,"f"']\lar[r,"p"] & X\ar[dl,equal] \\
                    X &
                    \cellsymb(\eta)[below right]{1-2}{2-1}
                    \cellsymb(\epsilon)[above left]{2-2}{3-1}
                \end{tikzcd}
                =
                \begin{tikzcd}
                    A\ar[d,"f"{left},bend right=20]\ar[d,"f"{right},bend left=20] \\
                    X
                    \cellsymb(\heq){1-1}{2-1}
                \end{tikzcd}
                \quad
                \begin{tikzcd}
                    & A\ar[dl,equal]\ar[d,"f"]\lar[r,"p"] & X\ar[dl,equal] \\
                    A\lar[r,"p"'] & X &
                    \cellsymb(\eta)[below right]{1-2}{2-1}
                    \cellsymb(\epsilon)[above left]{1-3}{2-2}
                \end{tikzcd}
                =
                \begin{tikzcd}
                    A\ar[d,equal]\lar[r,"p"] & X\ar[d,equal] \\
                    A\lar[r,"p"'] & X
                    \cellsymb(\veq){1-1}{2-2}
                \end{tikzcd}
            \end{equation*}
    \end{enumerate}
\end{proposition}

\begin{corollary}[{\cite[5.5.\ Corollary]{Koudenburg2020aug}}]
    Companions, conjoints, and loose units are preserved by any \ac{AVD}-functor.
\end{corollary}

\begin{remark}\label{rem:avdc_with_units}
    An \ac{AVDC} with loose units, called a \emph{unital \ac{AVDC}} in \cite{Koudenburg2020aug}, can be identified with a \emph{unital \ac{VDC}}, i.e., \ac{VDC} with ``units'' in the sense of \cite[5.1.\ Definition]{CruttwellShulman2010unified}.
    When we regard \acp{AVDC} with loose units as unital \acp{VDC}, the {AVD}-functors between them correspond to the \emph{normal} \ac{VD}-functors \cite[Section 5]{CruttwellShulman2010unified}.
    Indeed, there is a 2-equivalence \cite[10.1.\ Theorem]{Koudenburg2020aug}:
    \begin{equation}\label{eq:2-equiv_betwen_uavdc_and_uvdcn}
        \UAVDC\simeq\UVDCn.
    \end{equation}
    Here, $\UAVDC$ denotes the 2-category of (huge) unital \acp{AVDC} and \ac{AVD}-functors, and $\UVDCn$ denotes the 2-category of (huge) unital \acp{VDC} and normal \ac{VD}-functors.

    An \ac{AVDC} with 1-coary restrictions is called an \emph{augmented virtual equipment}, and \ac{AVDC} with restrictions is called a \emph{unital virtual equipment} in \cite{Koudenburg2020aug}.
    The latter can be identified with a \emph{virtual equipment} \cite{CruttwellShulman2010unified} by the 2-equivalence \cref{eq:2-equiv_betwen_uavdc_and_uvdcn}.
\end{remark}

\begin{remark}
    We now have two ways to regard unital \acp{VDC} as \acp{AVDC}.
    The first one is to regard them as diminished \acp{AVDC}, where the \ac{AVD}-functors between them correspond to the \ac{VD}-functors.
    The second one is to regard them as \acp{AVDC} with loose units, where the \ac{AVD}-functors between them correspond to the normal \ac{VD}-functors.
    Depending on which types of \ac{VD}-functors are considered, we will use either way.
\end{remark}

\begin{proposition}\label{prop:radj_preserves_cart}
    Every right adjoint in the 2-category $\AVDC$ preserves cartesian cells.
\end{proposition}
\begin{proof}
    Consider an adjunction in $\AVDC$ of the following form:
    \begin{equation*}
        \adjoint(\hspace{11pt}\perp_{\eta,\epsilon}){\bL}{\bL'}{\Phi}{\Psi}\incat{\AVDC},
    \end{equation*}
    where $\eta$ and $\epsilon$ are the unit and counit, respectively.
    Let $\alpha$ be a cartesian cell in $\bL'$ of the form \cref{eq:phantom_cell}.
    Then, we have bijective correspondences among the cells of the following forms:
    \begin{equation*}
        \begin{tikzcd}
            A\ar[d,"f"']\lar[r,path,"\tup{u}"] & B\ar[d,"g"] \\
            \Psi X^0\ar[d,"\Psi\alpha^0"'] & \Psi X^1\ar[d,"\Psi\alpha^1"] \\
            \Psi Y^0\lar[r,phan,"\Psi Y"'] & \Psi Y^1
            \cellsymb(\cdot){1-1}{3-2}
        \end{tikzcd}\incat{\bL}
        \Vline
        \begin{tikzcd}
            \Phi A\ar[d,"\hat{f}"']\lar[r,path,"\Phi\tup{u}"] & \Phi B\ar[d,"\hat{g}"] \\
            X^0\ar[d,"\alpha^0"'] & X^1\ar[d,"\alpha^1"] \\
            Y^0\lar[r,phan,"Y"'] & Y^1
            \cellsymb(\cdot){1-1}{3-2}
        \end{tikzcd}\incat{\bL'}
    \end{equation*}
    \begin{equation*}
        \Vline
        \begin{tikzcd}
            \Phi A\ar[d,"\hat{f}"']\lar[r,path,"\Phi\tup{u}"] & \Phi B\ar[d,"\hat{g}"] \\
            X^0\lar[r,phan,"X"'] & X^1
            \cellsymb(\cdot){1-1}{2-2}
        \end{tikzcd}\incat{\bL'}
        \Vline
        \begin{tikzcd}
            A\ar[d,"f"']\lar[r,path,"\tup{u}"] & B\ar[d,"g"] \\
            \Psi X^0\lar[r,phan,"\Psi X"'] & \Psi X^1
            \cellsymb(\cdot){1-1}{2-2}
        \end{tikzcd}\incat{\bL}.
    \end{equation*}
    The first and third correspondences come from the adjunction, and the second one follows from the universal property of the cartesian cell $\alpha$.
    This shows that the cell $\Psi\alpha$ is cartesian.
\end{proof}

\subsubsection{Cocartesian cells}
\begin{definition}[Cocartesian cells]
    A cell
    \begin{equation*}
        \begin{tikzcd}
            A\ar[d,equal]\lar[r,path,"\tup{u}"] & B\ar[d,equal] \\
            A\lar[r,phan,"v"'] & B
            \cellsymb(\alpha){1-1}{2-2}
        \end{tikzcd}
    \end{equation*}
    in an \ac{AVDC} is called \emph{cocartesian} if the following assignment induces a bijection $\Cells{f}{g}{\tup{p}v\tup{q}}{w}\cong\Cells{f}{g}{\tup{p}\tup{u}\tup{q}}{w}$ for any $f,g,\tup{p},\tup{q},w$:
    \begin{equation*}
        \begin{tikzcd}
            \cdot\ar[d,"f"']\lar[r,path,"\tup{p}"] & A\lar[r,phan,"v"] & B\lar[r,path,"\tup{q}"] & \cdot\ar[d,"g"] \\
            \cdot\lar[rrr,phan,"w"'] &&& \cdot
            \cellsymb(\cdot){1-1}{2-4}
        \end{tikzcd}
        \mapsto
        \begin{tikzcd}
            \cdot\ar[d,equal]\lar[r,path,"\tup{p}"] & A\ar[d,equal]\lar[r,path,"\tup{u}"] & B\ar[d,equal]\lar[r,path,"\tup{q}"] & \cdot\ar[d,equal] \\
            \cdot\ar[d,"f"']\lar[r,path,"\tup{p}"] & A\lar[r,phan,"v"] & B\lar[r,path,"\tup{q}"] & \cdot\ar[d,"g"] \\
            \cdot\lar[rrr,phan,"w"'] &&& \cdot
            \cellsymb(\veq){1-1}{2-2}
            \cellsymb(\alpha){1-2}{2-3}
            \cellsymb(\veq){1-3}{2-4}
            \cellsymb(\cdot){2-1}{3-4}
        \end{tikzcd}
    \end{equation*}
    The cell $\alpha$ is called \emph{\ac{VD}-cocartesian} if it induces the above bijection only for $w$ of length 1.
    Cocartesian cells and \ac{VD}-cocartesian cells are often denoted by the symbol ``$\cocart$'' and ``$\mathsf{VD.cocart}$,'' respectively:
    \begin{equation*}
        \begin{tikzcd}[largecolumn]
            \cdot\ar[d,equal]\lar[r,path] & \cdot\ar[d,equal] \\
            \cdot\lar[r,phan] & \cdot
            \cellsymb(\cocart){1-1}{2-2}
        \end{tikzcd}
        \qquad
        \begin{tikzcd}[largecolumn]
            \cdot\ar[d,equal]\lar[r,path] & \cdot\ar[d,equal] \\
            \cdot\lar[r,phan] & \cdot
            \cellsymb(\VDcocart){1-1}{2-2}
        \end{tikzcd}
    \end{equation*}
\end{definition}

\begin{remark}
    We can also consider cocartesian cells with an arbitrary boundary rather than identity tight arrows.
    See \cite[Section 7]{Koudenburg2020aug} for details.
\end{remark}

\begin{remark}
    The \ac{VD}-cocartesian cells recover the concept of ``cocartesian cells in \acp{VDC}'' introduced in \cite[5.1.\ Definition]{CruttwellShulman2010unified}, where a different term ``opcartesian'' is used.
    Indeed, \ac{VD}-cocartesian cells in a diminished \ac{AVDC} are nothing but opcartesian cells, in the sense of \cite[5.1.\ Definition]{CruttwellShulman2010unified}, in the corresponding \ac{VDC}.
\end{remark}

\begin{definition}
    Let $\bL$ be an \ac{AVDC}, and let $X\in\bL$.
    A loose arrow $u$ in a \ac{VD}-cocartesian cell of the following form is called the \emph{loose \ac{VD}-unit} on $X$.
    \begin{equation}\label{eq:loose_VD_unit}
        \begin{tikzcd}
            & X\ar[dl,equal]\ar[dr,equal] & \\
            X\lar[rr,"u"'] &{}& X
            \cellsymb(\VDcocart)[below=-2]{1-2}{2-2}
        \end{tikzcd}\incat{\bL}
    \end{equation}
    Note that the loose \ac{VD}-unit on $X$ is, if it exists, unique up to loosewise invertible cell.
\end{definition}

\begin{remark}
    If the cell \cref{eq:loose_VD_unit} is cocartesian rather than \ac{VD}-cocartesian, the loose cell $u$ in \cref{eq:loose_VD_unit} becomes the loose unit on $X$.
    Indeed, every cocartesian cell of the form \cref{eq:loose_VD_unit} is loosewise invertible.
    Thus, the loose \ac{VD}-units are a weaker concept than the loose units.
    Clearly, loose \ac{VD}-units in diminished \acp{AVDC} are the same concept as (loose) ``units'' in \acp{VDC} in the sense of \cite[5.1.\ Definition]{CruttwellShulman2010unified}.
\end{remark}

\begin{definition}[Loose composites]
    Suppose that we are given a (\ac{VD}-)cocartesian cell in an \ac{AVDC} of the following form:
    \begin{equation*}
        \begin{tikzcd}[hugecolumn]
            A\ar[d,equal]\lar[r,path,"\tup{u}"] & B\ar[d,equal] \\
            A\lar[r,"p"'] & B
            \cellsymb(\cocartorVD){1-1}{2-2}
        \end{tikzcd}
    \end{equation*}
    Since the loose arrow $p$ is unique up to isomorphism, we call $p$ the \emph{loose (\ac{VD}-)composite} of $\tup{u}$.
    The loose composite of a loose path $\tup{u}$ is often denoted by $\lcomp\tup{u}$.
    (We do not assign a specific symbol to loose \ac{VD}-composites.)
    Note that the loose (\ac{VD}-)composite of a loose path of length 0 is the same as the loose (\ac{VD}-)unit.
    An \ac{AVDC} is said to \emph{have loose (\ac{VD}-)composites} if the loose (\ac{VD}-)composite exists for every loose path.
    Clearly, an \ac{AVDC} has loose composites if and only if it has loose units and loose \ac{VD}-composites.
\end{definition}

\begin{definition}
    Let $\bL$ be an \ac{AVDC}.
    An object $A\in\bL$ is called \emph{(\ac{VD}-)composable} in $\bL$ if:
    \begin{itemize}
        \item
            For any loose arrows $\cdot\larr(u_1)A\larr(u_2)\cdot$ in $\bL$, there exists the loose (\ac{VD}-)composite of them:
            \begin{equation}\label{eq:VD-composition}
                \begin{tikzcd}
                    \cdot\ar[d,equal]\lar[r,"u_1"] & A\lar[r,"u_2"] & \cdot\ar[d,equal] \\
                    \cdot\lar[rr] & & \cdot
                    \cellsymb(\cocartorVD){1-1}{2-3}
                \end{tikzcd}\incat{\bL};
            \end{equation}
        \item
            $A$ has the loose (\ac{VD}-)unit:
            \begin{equation}\label{eq:VD-identity}
                \begin{tikzcd}
                    & A\ar[dl,equal]\ar[dr,equal] & \\
                    A\lar[rr] &{}& A
                    \cellsymb(\cocartorVD)[below=-2]{1-2}{2-2}
                \end{tikzcd}\incat{\bL}.
            \end{equation}
    \end{itemize}
\end{definition}

\begin{notation}\label{note:diminished_avdc_from_bicategory}
    Given a bicategory $\bi{W}$, we can obtain a diminished \ac{AVDC} $\Ldbl\bi{W}$ as follows.
    The tight category $\Tcat(\Ldbl\bi{W})$ is the discrete category of objects in $\bi{W}$.
    A loose arrow in $\Ldbl\bi{W}$ is a 1-cell in $\bi{W}$.
    A cell from $\tup{f}$ to $g$ in $\Ldbl\bi{W}$ is a 2-cell from $\lcomp\tup{f}$ to $g$ in $\bi{W}$:
    \begin{equation*}
        \begin{tikzcd}
            c\ar[d,equal]\lar[r,path,"\tup{f}"] & c'\ar[d,equal] \\
            c\lar[r,"g"'] & c'
            \cellsymb(\alpha){1-1}{2-2}
        \end{tikzcd}\incat{\Ldbl\bi{W}}
        \Vline
        \begin{tikzcd}
            c\ar[rr,bend left=30,"\lcomp\tup{f}"{above}]\ar[rr,bend right=30,"g"{below}] & & c'
            \dtwocell(\alpha){1-1}{1-3}
        \end{tikzcd}\incat{\bi{W}}
    \end{equation*}
    Here, $\lcomp\tup{f}$ denotes the composite of $\tup{f}$ in $\bi{W}$.
\end{notation}

\begin{theorem}\label{thm:avd_functor_is_lax}
    For bicategories $\bi{W}$ and $\bi{W}'$, the lax-functors $\bi{W}\to\bi{W}'$ are the same as the \ac{AVD}-functors $\Ldbl\bi{W}\to\Ldbl\bi{W}'$.
    Moreover, the pseudo-functors $\bi{W}\to\bi{W}'$ are the same as the \ac{AVD}-functors that preserve all \ac{VD}-cocartesian cells.
\end{theorem}
\begin{proof}
    See \cite[3.5.\ Example]{CruttwellShulman2010unified}.
\end{proof}

\begin{remark}\label{rem:pseudo_double_category}
    Diminished \acp{AVDC} with loose \ac{VD}-composites are essentially the same concept as \emph{pseudo double categories}; \ac{AVD}-functors between them are the same as \emph{lax double functors}.
    \Acp{AVDC} with loose composites are also essentially the same concept as pseudo double categories, but \ac{AVD}-functors between them are the same as \emph{normal lax double functors}.
    See \cite[5.2.\ Theorem]{CruttwellShulman2010unified} or \cite[2.8.\ Proposition]{DawsonParePronk2006paths} for details.
\end{remark}

\begin{notation}
    Let $\bL$ be an \ac{AVDC}.
    Then, all of the \ac{VD}-composable objects yield a bicategory $\Lbicat\bL$, called the \emph{loose bicategory} of $\bL$, where 1-cells are loose arrows and compositions and identities are defined by the \ac{VD}-cocartesian cells \cref{eq:VD-composition,eq:VD-identity}.
    This can be justified as follows.
    Consider the full sub-\ac{AVDC} of $\bL$ consisting of all \ac{VD}-composable objects.
    Since the diminished \ac{AVDC} obtained by forgetting non-trivial 0-coary cells from the full sub-\ac{AVDC} still has loose \ac{VD}-composites, it can be regarded as a pseudo double category, as explained in \cref{rem:pseudo_double_category}.
\end{notation}

\begin{proposition}\label{prop:special_sandwich}
    Suppose that we are given the following data in an \ac{AVDC}:
    \begin{equation*}
        \begin{tikzcd}
            A\ar[d,"f"']\lar[r,phan,"p"] & X\ar[dl,equal]\lar[r,phan,"u"] & Y\ar[dr,equal]\lar[r,phan,"q"] & B\ar[d,"g"] \\
            X\lar[rrr,phan,"u"'] &&& Y
            \cellsymb(\cart)[above left]{1-2}{2-1}
            \cellsymb(\veq){1-1}{2-4}
            \cellsymb(\cart)[above right]{1-3}{2-4}
        \end{tikzcd}
        =
        \begin{tikzcd}
            A\ar[d,equal]\lar[r,phan,"p"] & X\lar[r,phan,"u"] & Y\lar[r,phan,"q"] & B\ar[d,equal] \\
            A\ar[d,"f"']\lar[rrr,phan,"r"] &&& B\ar[d,"g"] \\
            X\lar[rrr,phan,"u"'] &&& Y
            \cellsymb(\alpha){1-1}{2-4}
            \cellsymb(\beta){2-1}{3-4}
        \end{tikzcd}
    \end{equation*}
    Then, the cell $\alpha$ is cocartesian if and only if the cell $\beta$ is cartesian.
    In particular, we have the following cocartesian cells whenever all restrictions, companions, and conjoints exist in each diagram:
    \begin{equation*}
        \begin{tikzcd}[scriptsizecolumn]
            A\ar[d,equal]\lar[r,"\comp{f}"] & X\lar[r,phan,"u"] & Y\ar[d,equal] \\
            A\lar[rr,"{u(f,\id)}"'] && Y
            \cellsymb(\cocart){1-1}{2-3}
        \end{tikzcd}
        \qquad
        \begin{tikzcd}[scriptsizecolumn]
            A\ar[d,equal]\lar[r,"\comp{f}"] & X\lar[r,phan,"u"] & Y\lar[r,"\conj{g}"] & B\ar[d,equal] \\
            A\lar[rrr,"{u(f,g)}"'] &&& B
            \cellsymb(\cocart){1-1}{2-4}
        \end{tikzcd}
        \qquad
        \begin{tikzcd}[scriptsizecolumn]
            X\ar[d,equal]\lar[r,phan,"u"] & Y\lar[r,"\conj{g}"] & B\ar[d,equal] \\
            X\lar[rr,"{u(\id,g)}"'] && B
            \cellsymb(\cocart){1-1}{2-3}
        \end{tikzcd}
    \end{equation*}
\end{proposition}
\begin{proof}
    The special case where both $p$ and $q$ are of length 1 is exactly \cite[8.1.\ Lemma]{Koudenburg2020aug}, and this can be proved in the same way.
\end{proof}

\subsubsection{Extensions and lifts}
\begin{definition}[Extending cells]
    Let $\bL$ be an \ac{AVDC}.
    A cell
    \begin{equation}\label{eq:extending_cell}
        \begin{tikzcd}
            A\ar[d,"f"']\lar[r,path,"\tup{u}"] & B\lar[r,"p"] & C\ar[d,equal] \\
            X\lar[rr,phan,"v"'] & & C
            \cellsymb(\alpha){1-1}{2-3}
        \end{tikzcd}\incat{\bL}
    \end{equation}
    is called \emph{extending} if, for any $Y$, $\tup{q}$, $g$, and a cell $\beta$ of the following form, there exists a unique cell $\gamma$ satisfying the following equation:
    \begin{equation*}
        \begin{tikzcd}
            A\ar[d,"f"']\lar[r,path,"\tup{u}"] & B\lar[r,path,"\tup{q}"] & Y\ar[d,"g"] \\
            X\lar[rr,phan,"v"'] & & C
            \cellsymb(\beta){1-1}{2-3}
        \end{tikzcd}
        =
        \begin{tikzcd}
            A\ar[d,equal]\lar[r,path,"\tup{u}"] & B\ar[d,equal]\lar[r,path,"\tup{q}"] & Y\ar[d,"g"] \\
            A\ar[d,"f"']\lar[r,path,"\tup{u}"] & B\lar[r,"p"] & C\ar[d,equal] \\
            X\lar[rr,phan,"v"'] & & C
            \cellsymb(\veq){1-1}{2-2}
            \cellsymb(\gamma){1-2}{2-3}
            \cellsymb(\alpha){2-1}{3-3}
        \end{tikzcd}\incat{\bL}.
    \end{equation*}
    By the universal property, a loose arrow $p$ in the extending cell \cref{eq:extending_cell} is unique up to isomorphism, hence we write $\extension{\tup{u}}[f]{v}$ for $p$.
    When $f$ is the identity, we also use the notation $\extension{\tup{u}}{v}$.
    When $v$ is of length 0 (hence $X=C$), we also use the notation $\extension{\tup{u}}[f]{X}$.
    An extending cell is often denoted by the symbol ``$\extending$'' as follows:
    \begin{equation*}
        \begin{tikzcd}
            A\ar[d,"f"']\lar[r,path,"\tup{u}"] & B\lar[r,"{\extension{\tup{u}}[f]{v}}"] & C\ar[d,equal] \\
            X\lar[rr,phan,"v"'] & & C
            \cellsymb(\extending){1-1}{2-3}
        \end{tikzcd}
    \end{equation*}
    An \ac{AVDC} is said to \emph{have extensions} (resp.\ \emph{1-coary extensions}) if $\extension{\tup{u}}[f]{v}$ exists for any $\tup{u}$, $f$, and $v$ of length $\le 1$ (resp.\ length 1).
\end{definition}

\begin{definition}[Lifting cells]
    \emph{Lifting} cells are defined to be the loosewise dual of extending cells.
    A lifting cell is often denoted by the symbol ``$\lifting$'' as follows:
    \begin{equation*}
        \begin{tikzcd}
            C\ar[d,equal]\lar[r,"{\lift{v}[f]{\tup{u}}}"] & B\lar[r,path,"\tup{u}"] & A\ar[d,"f"] \\
            C\lar[rr,phan,"v"'] && X
            \cellsymb(\lifting){1-1}{2-3}
        \end{tikzcd}
    \end{equation*}
    An \ac{AVDC} is said to \emph{have lifts} (resp.\ \emph{1-coary lifts}) if $\lift{v}[f]{\tup{u}}$ exists for any $\tup{u}$, $f$, and $v$ of length $\le 1$ (resp.\ length 1).
\end{definition}

\begin{remark}
    Let $\bi{W}$ be a bicategory.
    Then, the diminished \ac{AVDC} $\Ldbl\bi{W}$ as in \cref{note:diminished_avdc_from_bicategory} has 1-coary extensions (resp.\ 1-coary lifts) if and only if the bicategory $\bi{W}$ has right Kan extensions (resp.\ right Kan lifts) in the usual sense.
\end{remark}

\begin{remark}
    In the context of pseudo double categories, the concept corresponding to 1-coary extensions (resp.\ lifts) is studied in \cite{Pare2024retrocells} under the name \textit{strong right homs} (resp.\ \textit{strong left homs}), where the tight arrow $f$ in \cref{eq:extending_cell} is supposed to be the identity.
\end{remark}

\pagebreak
\begin{proposition}\label{prop:consequence_from_ext_lift}
    Let $\bL$ be an \ac{AVDC}.
    \begin{enumerate}
        \item\label{prop:consequence_from_ext_lift-1}
            If $\bL$ has extensions, then $\bL$ has companions.
        \item\label{prop:consequence_from_ext_lift-2}
            If $\bL$ has extensions and conjoints, then, for any $A\arr(f)[][1]X\larr(u)Y$ in $\bL$, the restriction $u(f,\id)$ exists.
        \item\label{prop:consequence_from_ext_lift-3}
            If $\bL$ has extensions and lifts, then $\bL$ has restrictions.
    \end{enumerate}
\end{proposition}
\begin{proof}\quad
    \begin{enumerate}
        \item
            Let $A\arr(f)[][1]X$ in $\bL$.
            By the universal property of extending cells, we have a unique cell $\eta$ satisfying the following:
            \begin{equation*}
                \begin{tikzcd}[large]
                    & A\ar[dl,equal]\ar[d,"f"] \\
                    A\ar[d,"f"']\lar[r,"{\extension{}[f]{X}}"] & X\ar[dl,equal] \\
                    X &
                    \cellsymb(\eta)[below right=1]{1-2}{2-1}
                    \cellsymb(\extending)[above left]{2-2}{3-1}
                \end{tikzcd}
                =
                \begin{tikzcd}[large]
                    A\ar[d,bend right=20,"f"{left}]\ar[d,bend left=20,"f"{right}] \\
                    X
                    \cellsymb(\heq){1-1}{2-1}
                \end{tikzcd}\incat{\bL}.
            \end{equation*}
            By the universal property of $\extension{}[f]{X}$ again, the other equation for $\extension{}[f]{X}$ being a companion follows, hence $\comp{f}=\extension{}[f]{X}$.
        \item
            Consider the following extending cell:
            \begin{equation*}
                \begin{tikzcd}
                    A\ar[d,"f"']\lar[r,"{\extension{}[f]{u}}"] & Y\ar[d,equal] \\
                    X\lar[r,"u"'] & Y
                    \cellsymb(\extending){1-1}{2-2}
                \end{tikzcd}\incat{\bL}.
            \end{equation*}
            Then, for any $A\rra(s)[][1]B_0\larr(\tup{v})[path]B_n\arr(t)[][1]Y$ in $\bL$, there are bijective correspondences among the cells of the following forms in $\bL$:
            \begin{equation*}
                \begin{tikzcd}
                    B_0\ar[d,"s"']\lar[r,path,"\tup{v}"] & B_n\ar[d,"t"] \\
                    A\ar[d,"f"'] & Y\ar[d,equal] \\
                    X\lar[r,"u"'] & Y
                    \cellsymb(\cdot){1-1}{3-2}
                \end{tikzcd}
                \Vline
                \begin{tikzcd}[scriptsizecolumn]
                    A\ar[d,"f"']\lar[r,"\conj{s}"] & B_0\lar[r,path,"\tup{v}"] & B_n\ar[d,"t"] \\
                    X\lar[rr,"u"'] && Y
                    \cellsymb(\cdot){1-1}{2-3}
                \end{tikzcd}
                \Vline
                \begin{tikzcd}[scriptsizecolumn]
                    A\ar[d,equal]\lar[r,"\conj{s}"] & B_0\lar[r,path,"\tup{v}"] & B_n\ar[d,"t"] \\
                    A\lar[rr,"{\extension{}[f]{u}}"'] && Y
                    \cellsymb(\cdot){1-1}{2-3}
                \end{tikzcd}
                \Vline
                \begin{tikzcd}
                    B_0\ar[d,"s"']\lar[r,path,"\tup{v}"] & B_n\ar[d,"t"] \\
                    A\ar[r,"{\extension{}[f]{u}}"'] & Y
                    \cellsymb(\cdot){1-1}{2-2}
                \end{tikzcd}
            \end{equation*}
            This shows that $\extension{}[f]{u}$ gives the desired restriction $u(f,\id)$.
        \item
            Combining \cref{prop:consequence_from_ext_lift-1}, \cref{prop:consequence_from_ext_lift-2}, and their loosewise duals, we can show that $\bL$ has companions, conjoints, and 1-coary restrictions.
            Here, we use the fact that tightwise composition preserves cartesian cells (\cref{prop:pasting_lemma_cartesian}).
            Since loose units exist in this case, $\bL$ also has 0-coary restrictions.\qedhere
    \end{enumerate}
\end{proof}

\subsubsection{The module construction}
We recall the $\Mod$-construction from \cite{Leinster1999enrichment,Leinster2004higher,CruttwellShulman2010unified}, which is a construction of a \ac{VDC} ``$\Mod(\bX)$'' from a \ac{VDC} $\bX$.
Since the resulting \acp{VDC} are always unital and normal \ac{VD}-functors between them are often considered, we redefine ``$\Mod(\bX)$'' as an \ac{AVDC} with loose units.
Such a redefinition is also considered in \cite[2.2.\ Example]{Koudenburg2020aug}.
\pagebreak
\begin{definition}[\cite{Leinster1999enrichment,Leinster2004higher,CruttwellShulman2010unified,Koudenburg2020aug}]
    Let $\bX$ be an \ac{AVDC}.
    The \ac{AVDC} $\Mod(\bX)$ is defined as follows:
    \begin{itemize}
        \item
            An object is a \emph{monoid}, which consists of the following data $A\coloneq(A^0,A^1,A^e,A^m)$:
            \begin{equation*}
                \begin{tikzcd}[tri]
                    & A^0\ar[dl,equal]\ar[dr,equal] & \\
                    A^0\lar[rr,"A^1"'] &{}& A^0
                    \cellsymb(A^e)[above=-5]{1-2}{2-2}
                \end{tikzcd}
                \qquad
                \begin{tikzcd}
                    A^0\ar[d,equal]\lar[r,"A^1"] & A^0\lar[r,"A^1"] & A^0\ar[d,equal] \\
                    A^0\lar[rr,"A^1"'] && A^0
                    \cellsymb(A^m){1-1}{2-3}
                \end{tikzcd}\incat{\bX}.
            \end{equation*}
            The data $(A^0,A^1,A^e,A^m)$ are required to satisfy monoid-like axioms.
            The cells $A^e$ and $A^m$ are called the \emph{unit} and the \emph{multiplication} of the monoid $A$, respectively.
        \item
            A tight arrow $A\arr(f)B$ is called a \emph{monoid homomorphism}.
            It consists of the following data $(f^0,f^1)$:
            \begin{equation*}
                \begin{tikzcd}
                    A^0\ar[d,"f^0"']\lar[r,"A^1"] & A^0\ar[d,"f^0"] \\
                    B^0\lar[r,"B^1"'] & B^0
                    \cellsymb(f^1){1-1}{2-2}
                \end{tikzcd}\incat{\bX}
            \end{equation*}
            that is required to be compatible with units and multiplications.
        \item
            A loose arrow $A\larr(M)B$ is called a \emph{(bi)module}.
            It consists of the following data $(M^1,M^l,M^r)$:
            \begin{equation*}
                \begin{tikzcd}
                    A^0\ar[d,equal]\lar[r,"A^1"] & A^0\lar[r,"M^1"] & B^0\ar[d,equal] \\
                    A^0\lar[rr,"M^1"'] && B^0
                    \cellsymb(M^l){1-1}{2-3}
                \end{tikzcd}
                \qquad
                \begin{tikzcd}
                    A^0\ar[d,equal]\lar[r,"M^1"] & B^0\lar[r,"B^1"] & B^0\ar[d,equal] \\
                    A^0\lar[rr,"M^1"'] && B^0
                    \cellsymb(M^r){1-1}{2-3}
                \end{tikzcd}\incat{\bX}
            \end{equation*}
            that is required to satisfy module-like axioms.
        \item
            A 1-coary cell $\alpha$ in $\Mod(\bX)$ on the left below is a cell in $\bX$ on the right below
            \begin{equation*}
                \begin{tikzcd}
                    A_0\ar[d,"f"']\lar[r,path,"\tup{M}"] & A_n\ar[d,"g"] \\
                    B\lar[r,"N"'] & C
                    \cellsymb(\alpha){1-1}{2-2}
                \end{tikzcd}\incat{\Mod(\bX)}
                \qquad
                \begin{tikzcd}
                    A_0^0\ar[d,"f^0"']\lar[r,"M_1^1"] & \cdots\lar[r,"M_n^1"] & A_n^0\ar[d,"g^0"] \\
                    B^0\lar[rr,"N^1"'] && C^0
                    \cellsymb(\alpha){1-1}{2-3}
                \end{tikzcd}\incat{\bX}
            \end{equation*}
            such that, for each $0\le i\le n$, the two canonical ways to fill the following boundary give the same cell in $\bX$:
            \begin{equation*}
                \begin{tikzcd}[hugecolumn]
                    A_0^0\ar[d,"f^0"']\lar[r,path,"(M_j^1)_{0<j\le i}"] & A_i^0\lar[r,"A_i^1"] & A_i^0\lar[r,path,"(M_j^1)_{i<j\le n}"] & A_n^0\ar[d,"g^0"] \\
                    B^0\lar[rrr,"N^1"'] &&& C^0
                \end{tikzcd}\incat{\bX}.
            \end{equation*}
        \item
            A 0-coary cell $\beta$ in $\Mod(\bX)$ on the left below is a cell in $\bX$ on the right below
            \begin{equation*}
                \begin{tikzcd}[tri]
                    A_0\ar[dr,"f"']\lar[rr,path,"\tup{M}"] &{}& A_n\ar[dl,"g"] \\
                    & B &
                    \cellsymb(\beta)[above=-3]{1-2}{2-2}
                \end{tikzcd}\incat{\Mod(\bX)}
                \qquad
                \begin{tikzcd}
                    A_0^0\ar[d,"f^0"']\lar[r,"M_1^1"] & \cdots\lar[r,"M_n^1"] & A_n^0\ar[d,"g^0"] \\
                    B^0\lar[rr,"B^1"'] && B^0
                    \cellsymb(\beta){1-1}{2-3}
                \end{tikzcd}\incat{\bX}
            \end{equation*}
            such that, for each $0\le i\le n$, the two canonical ways to fill the following boundary give the same cell in $\bX$:
            \begin{equation*}
                \begin{tikzcd}[hugecolumn]
                    A_0^0\ar[d,"f^0"']\lar[r,path,"(M_j^1)_{0<j\le i}"] & A_i^0\lar[r,"A_i^1"] & A_i^0\lar[r,path,"(M_j^1)_{i<j\le n}"] & A_n^0\ar[d,"g^0"] \\
                    B^0\lar[rrr,"B^1"'] &&& B^0
                \end{tikzcd}\incat{\bX}.
            \end{equation*}
    \end{itemize}
\end{definition}

\begin{remark}
    In the construction of $\Mod(\bX)$, no 0-coary cell in $\bX$ is used except for identities.
    In particular, we have $\Mod(\bX)=\Mod(\dim{\bX})$.
\end{remark}

\begin{theorem}[\cite{CruttwellShulman2010unified}]\label{thm:universal_property_of_Mod}
    Let $\bL$ be an \ac{AVDC} with loose units and let $\bX$ be an \ac{AVDC}.
    Then, the following data correspond to each other up to isomorphism:
    \begin{enumerate}
        \item
            An \ac{AVD}-functor $\bL\to\Mod(\bX)$.
        \item
            An \ac{AVD}-functor $\dim{\bL}\to\bX$.\qedhere
    \end{enumerate}
\end{theorem}
\begin{proof}
    An \ac{AVD}-functor $\dim{\bL}\to\bX$ is nothing but a \ac{VD}-functor $\dim{\bL}\to\dim{\bX}$.
    By the universal property of the $\Mod$-construction \cite[5.14.\ Proposition]{CruttwellShulman2010unified}, it corresponds to a normal \ac{VD}-functor $\dim{\bL}\to\dim{\Mod(\dim{\bX})}$ in the sense of \cite{CruttwellShulman2010unified}.
    Since $\Mod(\dim{\bX})=\Mod(\bX)$ and since both $\bL$ and $\Mod(\bX)$ have loose units, it also corresponds to an \ac{AVD}-functor $\bL\to\Mod(\bX)$.
\end{proof}

\begin{remark}
    We now give an explicit description of the above construction.
    Let $F\colon\dim{\bL}\to\bX$ be an \ac{AVD}-functor in the situation of \cref{thm:universal_property_of_Mod}.
    Consider loose units $\Unit_L$ on objects $L\in\bL$.
    Then, $\Unit_L$ is no longer a loose unit in the diminished \ac{AVDC} $\dim{\bL}$, but it is a monoid in $\dim{\bL}$.
    Thus, $F\Unit_L$ is still a monoid in $\bX$, to which the corresponding \ac{AVD}-functor $\bL\to\Mod(\bX)$ sends each object $L\in\bL$.

    Furthermore, if loose units are chosen for each object in $\bL$, the correspondence of \cref{thm:universal_property_of_Mod} actually becomes a bijection, which gives a (strict) 2-adjunction.
\end{remark}

\begin{notation}\label{note:AVD_functor_U}
    For an \ac{AVDC} $\bX$ with loose units, we write $U\colon\bX\to\Mod(\bX)$ for the \ac{AVD}-functor corresponding to the inclusion $\dim{\bX}\to\bX$.
    This \ac{AVD}-functor sends each object $c\in\bX$ to the trivial monoid, denoted by $U_c$, which is induced by the loose unit $\Unit_c$ on $c$.
\end{notation}

\begin{remark}
    It follows straightforwardly that $U$ locally induces bijections on the classes of tight arrows, loose arrows, and cells. Thus, we can regard $\bX$ as a full sub-\ac{AVDC} of $\Mod(\bX)$ by the inclusion $U$.
\end{remark}

\begin{proposition}[\cite{CruttwellShulman2010unified}]\label{prop:restriction_in_Mod}
    Let $\bX$ be an \ac{AVDC}.
    \begin{enumerate}
        \item
            $\Mod(\bX)$ has loose units.
        \item\label{prop:restriction_in_Mod-1coaryrest}
            If $\bX$ has 1-coary restrictions, then $\Mod(\bX)$ has restrictions.
    \end{enumerate}
\end{proposition}
\begin{proof}\quad
    \begin{enumerate}
        \item
            By \cite[5.5.\ Proposition]{CruttwellShulman2010unified}, the diminished \ac{AVDC} $\dim{\Mod(\bX)}$ has loose \ac{VD}-units.
            Those units automatically become loose units in $\Mod(\bX)$ since all 0-coary cells are inherited from them.
        \item
            By \cite[7.4.\ Proposition]{CruttwellShulman2010unified}, 1-coary restrictions in $\bX$ give those in $\Mod(\bX)$.\qedhere
    \end{enumerate}
\end{proof}

\subsubsection{Loosewise indiscreteness}
We now introduce several notions of indiscreteness and discreteness for \acp{AVDC}.
\Acp{AVDC} satisfying these properties will play an important role as diagram shapes for the various notions of colimits introduced later.
\begin{definition}
    An \ac{AVDC} $\bK$ is called \emph{loosewise discrete} if:
    \begin{itemize}
        \item
            It has no loose arrows.
        \item
            It has no cells except for tight identity cells.\qedhere
    \end{itemize}
\end{definition}

\begin{definition}
    An \ac{AVDC} $\bK$ is called \emph{loosewise \ac{AVD}-indiscrete} if:
    \begin{itemize}
        \item
            For any objects $A,B\in\bK$, there is a unique loose arrow from $A$ to $B$, denoted by $A\larr(!_{AB})B$.
        \item
            For any boundary for cells, there is a unique cell filling it.\qedhere
    \end{itemize}
\end{definition}

\begin{definition}
    An \ac{AVDC} $\bK$ is called \emph{loosewise \ac{VD}-indiscrete} if:
    \begin{itemize}
        \item
            For any objects $A,B\in\bK$, there is a unique loose arrow from $A$ to $B$, denoted by $A\larr(!_{AB})B$.
        \item
            For any $A_0,A_1,\dots,A_n,X,Y\in\bK$ $(n\ge 0)$ and any tight arrows $A_0\arr(f)[][1]X,A_n\arr(g)[][1]Y$ in $\bK$, there is a unique cell of the following form:
            \begin{equation*}
                \begin{tikzcd}
                    A_0\ar[d,"f"']\lar[r,"!_{A_0A_1}"] & A_1\lar[r,"!_{A_1A_2}"] & \cdots\lar[r,"!_{A_{n-1}A_n}"] & A_n\ar[d,"g"] \\
                    X\lar[rrr,"!_{XY}"'] & & & Y
                    \cellsymb(!){1-1}{2-4}
                \end{tikzcd}\incat{\bK}.
            \end{equation*}
        \item
            $\bK$ is diminished.\qedhere
    \end{itemize}
\end{definition}

\begin{notation}\label{note:loosely_indiscrete_avdc}
    Let $\one{C}$ be a category.
    Let $\Ddbl\one{C}$ (resp.\ $\Idbl\one{C}$; $\Idimdbl\one{C}$) denote the loosewise discrete (resp.\ \ac{AVD}-indiscrete; \ac{VD}-indiscrete) \ac{AVDC} uniquely determined by $\Tcat(\Ddbl\one{C})=\one{C}$ (resp.\ $\Tcat(\Idbl\one{C})=\one{C}$; $\Tcat(\Idimdbl\one{C})=\one{C}$).
    Then, $\Idimdbl\one{C}=\dim{(\Idbl\one{C})}$ follows immediately.
    Note that every loosewise discrete (resp.\ \ac{AVD}-indiscrete; \ac{VD}-indiscrete) \ac{AVDC} is of the form $\Ddbl\one{C}$ (resp.\ $\Idbl\one{C}$; $\Idimdbl\one{C}$) for some $\one{C}$.
\end{notation}

\begin{notation}
    For a large set $\zero{S}$, we write $\Ddbl\zero{S}$ (resp.\ $\Idbl\zero{S}$; $\Idimdbl\zero{S}$) for the loosewise discrete (resp.\ \ac{AVD}-indiscrete; \ac{VD}-indiscrete) large \ac{AVDC} of \cref{note:loosely_indiscrete_avdc} obtained from the discrete category $\zero{S}$.
\end{notation}

\begin{remark}
    Let $\zero{1}$ denote the singleton, and let $\bL$ be an \ac{AVDC}.
    \begin{enumerate}
        \item
            An \ac{AVD}-functor $\Ddbl\zero{1}\to\bL$ is the same as an object in $\bL$.
        \item
            An \ac{AVD}-functor $\Idbl\zero{1}\to\bL$ is the same as an object with a chosen loose unit in $\bL$.
        \item
            An \ac{AVD}-functor $\Idimdbl\zero{1}\to\bL$ is the same as a monoid in $\bL$.\qedhere
    \end{enumerate}
\end{remark}

The following definition is useful for dealing with two types of ``indiscreteness'' simultaneously.
\begin{definition}\label{def:loosewise_indiscrete}
    An \ac{AVDC} $\bK$ is called \emph{loosewise indiscrete} if:
    \begin{itemize}
        \item
            For any objects $A,B\in\bK$, there is a unique loose arrow from $A$ to $B$.
        \item
            For any boundary for 1-coary cells, there is a unique cell filling it.
        \item
            For any boundary for 0-coary cells, there is at most one cell filling it.
    \end{itemize}
    Note that either loosewise \ac{AVD}-indiscreteness or loosewise \ac{VD}-indiscreteness implies loosewise indiscreteness.
\end{definition}

Surprisingly, in a loosewise indiscrete \ac{AVDC}, all cells whose top and bottom boundaries are of length 1 become cartesian for a diagrammatic reason.
To show this, we introduce a special type of ``absolutely'' cartesian cells.
\begin{definition}\label{def:split_cell}
    A cell
    \begin{equation*}
        \begin{tikzcd}
            A_0\ar[d,"f_0"']\lar[r,phan,"u"] & A_1\ar[d,"f_1"] \\
            B_0\lar[r,phan,"v"'] & B_1
            \cellsymb(\alpha){1-1}{2-2}
        \end{tikzcd}
    \end{equation*}
    in an \ac{AVDC} is called \emph{split} if there are data $(p_0,p_1,q_0,q_1,\beta_0,\beta_1,\gamma,\delta_0,\delta_1,\sigma,\eta_0,\eta_1)$ of the following forms:
    \begin{equation*}
        \begin{tikzcd}
            & A_0\ar[dl,equal]\ar[d,"f_0"] \\
            A_0\lar[r,phan,"p_0"'] & B_0
            \cellsymb(\beta_0)[below right]{1-2}{2-1}
        \end{tikzcd}
        \quad
        \begin{tikzcd}
            A_1\ar[d,"f_1"']\ar[dr,equal] & \\
            B_1\lar[r,phan,"p_1"'] & A_1
            \cellsymb(\beta_1)[below left]{1-1}{2-2}
        \end{tikzcd}
        \quad
        \begin{tikzcd}
            A_0\ar[d,equal]\lar[r,phan,"p_0"] & B_0\lar[r,phan,"v"] & B_1\lar[r,phan,"p_1"] & A_1\ar[d,equal] \\
            A_0\lar[rrr,phan,"u"'] &&& A_1
            \cellsymb(\gamma){1-1}{2-4}
        \end{tikzcd}
    \end{equation*}
    \begin{equation*}
        \begin{tikzcd}
            A_0\ar[d,"f_0"']\lar[r,phan,"p_0"] & B_0\ar[d,equal] \\
            B_0\lar[r,phan,"q_0"'] & B_0
            \cellsymb(\delta_0){1-1}{2-2}
        \end{tikzcd}
        \quad
        \begin{tikzcd}
            B_1\ar[d,equal]\lar[r,phan,"p_1"] & A_1\ar[d,"f_1"] \\
            B_1\lar[r,phan,"q_1"'] & B_1
            \cellsymb(\delta_1){1-1}{2-2}
        \end{tikzcd}
        \quad
        \begin{tikzcd}
            B_0\ar[d,equal]\lar[r,phan,"q_0"] & B_0\lar[r,phan,"v"] & B_1\lar[r,phan,"q_1"] & B_1\ar[d,equal] \\
            B_0\lar[rrr,phan,"v"'] & & & B_1
            \cellsymb(\sigma){1-1}{2-4}
        \end{tikzcd}
    \end{equation*}
    \begin{equation*}
        \begin{tikzcd}[tri]
            & B_0\ar[dl,equal]\ar[dr,equal] & \\
            B_0\lar[rr,phan,"q_0"'] &{}& B_0
            \cellsymb(\eta_0){1-2}{2-2}
        \end{tikzcd}
        \quad
        \begin{tikzcd}[tri]
            & B_1\ar[dl,equal]\ar[dr,equal] & \\
            B_1\lar[rr,phan,"q_1"'] &{}& B_1
            \cellsymb(\eta_1){1-2}{2-2}
        \end{tikzcd}
    \end{equation*}
    These are required to satisfy the following equations:
    \begin{equation*}
        \begin{tikzcd}
            & A_0\ar[dl,equal]\ar[d,"f_0"]\lar[r,phan,"u"] & A_1\ar[d,"f_1"']\ar[dr,equal] & \\
            A_0\ar[d,equal]\lar[r,phan,"p_0"'] & B_0\lar[r,phan,"v"'] & B_1\lar[r,phan,"p_1"'] & A_1\ar[d,equal] \\
            A_0\lar[rrr,phan,"u"'] &&& A_1
            \cellsymb(\beta_0)[below right]{1-2}{2-1}
            \cellsymb(\alpha){1-2}{2-3}
            \cellsymb(\beta_1)[below left]{1-3}{2-4}
            \cellsymb(\gamma){2-1}{3-4}
        \end{tikzcd}
        =
        \begin{tikzcd}
            A_0\ar[d,equal]\lar[r,phan,"u"] & A_1\ar[d,equal] \\
            A_0\lar[r,phan,"u"'] & A_1
            \cellsymb(\veq){1-1}{2-2}
        \end{tikzcd}
    \end{equation*}
    \begin{equation*}
        \begin{tikzcd}
            A_0\ar[d,"f_0"']\lar[r,phan,"p_0"] & B_0\ar[d,equal]\lar[r,phan,"v"] & B_1\ar[d,equal]\lar[r,phan,"p_1"] & A_1\ar[d,"f_1"] \\
            B_0\ar[d,equal]\lar[r,phan,"q_0"'] & B_0\lar[r,phan,"v"'] & B_1\lar[r,phan,"q_1"'] & B_1\ar[d,equal] \\
            B_0\lar[rrr,phan,"v"'] & & & B_1
            \cellsymb(\delta_0){1-1}{2-2}
            \cellsymb(\veq){1-2}{2-3}
            \cellsymb(\delta_1){1-3}{2-4}
            \cellsymb(\sigma){2-1}{3-4}
        \end{tikzcd}
        =
        \begin{tikzcd}
            A_0\ar[d,equal]\lar[r,phan,"p_0"] & B_0\lar[r,phan,"v"] & B_1\lar[r,phan,"p_1"] & A_1\ar[d,equal] \\
            A_0\ar[d,"f_0"']\lar[rrr,phan,"u"'] &&& A_1\ar[d,"f_1"] \\
            B_0\lar[rrr,phan,"v"'] &&& B_1
            \cellsymb(\gamma){1-1}{2-4}
            \cellsymb(\alpha){2-1}{3-4}
        \end{tikzcd}
    \end{equation*}
    \begin{equation*}
        \begin{tikzcd}
            & A_0\ar[dl,equal]\ar[d,"f_0"] \\
            A_0\ar[d,"f_0"']\lar[r,phan,"p_0"'] & B_0\ar[d,equal] \\
            B_0\lar[r,phan,"q_0"'] & B_0
            \cellsymb(\beta_0)[below right]{1-2}{2-1}
            \cellsymb(\delta_0){2-1}{3-2}
        \end{tikzcd}
        =
        \begin{tikzcd}[tri]
            & A_0\ar[d,"f_0"{left},bend right=20]\ar[d,"f_0"{right},bend left=20] & \\
            & B_0\ar[dl,equal]\ar[dr,equal] & \\
            B_0\lar[rr,phan,"q_0"'] &{}& B_0
            \cellsymb(\heq){1-2}{2-2}
            \cellsymb(\eta_0){2-2}{3-2}
        \end{tikzcd}
        \quad
        \begin{tikzcd}[tri]
            & A_1\ar[d,"f_1"{left},bend right=20]\ar[d,"f_1"{right},bend left=20] & \\
            & B_1\ar[dl,equal]\ar[dr,equal] & \\
            B_1\lar[rr,phan,"q_1"'] &{}& B_1
            \cellsymb(\heq){1-2}{2-2}
            \cellsymb(\eta_1){2-2}{3-2}
        \end{tikzcd}
        =
        \begin{tikzcd}
            A_1\ar[dr,equal]\ar[d,"f_1"'] & \\
            B_1\ar[d,equal]\lar[r,phan,"p_1"'] & A_1\ar[d,"f_1"] \\
            B_1\lar[r,phan,"q_1"'] & B_1
            \cellsymb(\beta_1)[below left]{1-1}{2-2}
            \cellsymb(\delta_1){2-1}{3-2}
        \end{tikzcd}
    \end{equation*}
    \begin{equation*}
        \begin{tikzcd}
            & B_0\ar[dl,equal]\ar[d,equal]\lar[r,phan,"v"] & B_1\ar[d,equal]\ar[dr,equal] & \\
            B_0\ar[d,equal]\lar[r,phan,"q_0"'] & B_0\lar[r,phan,"v"'] & B_1\lar[r,phan,"q_1"'] & B_1\ar[d,equal] \\
            B_0\lar[rrr,phan,"v"'] &&& B_1
            \cellsymb(\eta_0)[below right]{1-2}{2-1}
            \cellsymb(\veq){1-2}{2-3}
            \cellsymb(\eta_1)[below left]{1-3}{2-4}
            \cellsymb(\sigma){2-1}{3-4}
        \end{tikzcd}
        =
        \begin{tikzcd}
            B_0\ar[d,equal]\lar[r,phan,"v"] & B_1\ar[d,equal] \\
            B_0\lar[r,phan,"v"'] & B_1
            \cellsymb(\veq){1-1}{2-2}
        \end{tikzcd}
    \end{equation*}
\end{definition}

\begin{lemma}\label{lem:split_cell_is_cartesian}
    Every split cell is cartesian.
    In particular, every split cell is \emph{absolutely cartesian}; that is, it is a cartesian cell preserved by any \ac{AVD}-functor.
\end{lemma}
\begin{proof}
    Let $\alpha$ be a split cell as in \cref{def:split_cell}.
    Take an arbitrary cell $\theta$ on the left below:
    \begin{equation}\label{eq:universal_property_of_alpha}
        \begin{tikzcd}
            X_0\ar[d,"x_0"']\lar[r,path,"\tup{w}"] & X_1\ar[d,"x_1"] \\
            A_0\ar[d,"f_0"'] & A_1\ar[d,"f_1"] \\
            B_0\lar[r,phan,"v"'] & B_1
            \cellsymb(\theta){1-1}{3-2}
        \end{tikzcd}
        =
        \begin{tikzcd}
            X_0\ar[d,"x_0"']\lar[r,path,"\tup{w}"] & X_1\ar[d,"x_1"] \\
            A_0\ar[d,"f_0"']\lar[r,phan,"u"] & A_1\ar[d,"f_1"] \\
            B_0\lar[r,phan,"v"'] & B_1
            \cellsymb(\bar{\theta}){1-1}{2-2}
            \cellsymb(\alpha){2-1}{3-2}
        \end{tikzcd}
    \end{equation}
    If there exists a cell $\bar{\theta}$ satisfying the above equation, then $\bar{\theta}$ must be given by the following:
    \begin{equation*}
        \bar{\theta}
        =
        \begin{tikzcd}
            & X_0\ar[d,"x_0"']\lar[r,path,"\tup{w}"] & X_1\ar[d,"x_1"] & \\
            & A_0\ar[dl,equal]\ar[d,"f_0"]\lar[r,phan,"u"] & A_1\ar[d,"f_1"']\ar[dr,equal] & \\
            A_0\ar[d,equal]\lar[r,phan,"p_0"'] & B_0\lar[r,phan,"v"'] & B_1\lar[r,phan,"p_1"'] & A_1\ar[d,equal] \\
            A_0\lar[rrr,phan,"u"'] &&& A_1
            \cellsymb(\bar{\theta}){1-2}{2-3}
            \cellsymb(\alpha){2-2}{3-3}
            \cellsymb(\beta_0)[below right]{2-2}{3-1}
            \cellsymb(\beta_1)[below left]{2-3}{3-4}
            \cellsymb(\gamma){3-1}{4-4}
        \end{tikzcd}
        =
        \begin{tikzcd}
            & X_0\ar[d,"x_0"']\lar[r,path,"\tup{w}"] & X_1\ar[d,"x_1"] & \\
            & A_0\ar[dl,equal]\ar[d,"f_0"] & A_1\ar[d,"f_1"']\ar[dr,equal] & \\
            A_0\ar[d,equal]\lar[r,phan,"p_0"'] & B_0\lar[r,phan,"v"'] & B_1\lar[r,phan,"p_1"'] & A_1\ar[d,equal] \\
            A_0\lar[rrr,phan,"u"'] &&& A_1
            \cellsymb(\theta){1-2}{3-3}
            \cellsymb(\beta_0)[below right]{2-2}{3-1}
            \cellsymb(\beta_1)[below left]{2-3}{3-4}
            \cellsymb(\gamma){3-1}{4-4}
        \end{tikzcd}
    \end{equation*}
    Conversely, let us define $\bar{\theta}$ by the above equation.
    Then, the following calculation shows that $\bar{\theta}$ satisfies the desired equation \cref{eq:universal_property_of_alpha}:
    \begin{equation*}
        \begin{tikzcd}
            & X_0\ar[d,"x_0"']\lar[r,path,"\tup{w}"] & X_1\ar[d,"x_1"] & \\
            & A_0\ar[dl,equal]\ar[d,"f_0"] & A_1\ar[d,"f_1"']\ar[dr,equal] & \\
            A_0\ar[d,equal]\lar[r,phan,"p_0"'] & B_0\lar[r,phan,"v"'] & B_1\lar[r,phan,"p_1"'] & A_1\ar[d,equal] \\
            A_0\ar[d,"f_0"']\lar[rrr,phan,"u"'] &&& A_1\ar[d,"f_1"] \\
            B_0\lar[rrr,phan,"v"'] &&& B_1
            \cellsymb(\theta){1-2}{3-3}
            \cellsymb(\beta_0)[below right]{2-2}{3-1}
            \cellsymb(\beta_1)[below left]{2-3}{3-4}
            \cellsymb(\gamma){3-1}{4-4}
            \cellsymb(\alpha){4-1}{5-4}
        \end{tikzcd}
        =
        \begin{tikzcd}
            & X_0\ar[d,"x_0"']\lar[r,path,"\tup{w}"] & X_1\ar[d,"x_1"] & \\
            & A_0\ar[dl,equal]\ar[d,"f_0"] & A_1\ar[d,"f_1"']\ar[dr,equal] & \\
            A_0\ar[d,"f_0"']\lar[r,phan,"p_0"'] & B_0\ar[d,equal]\lar[r,phan,"v"'] & B_1\ar[d,equal]\lar[r,phan,"p_1"'] & A_1\ar[d,"f_1"] \\
            B_0\ar[d,equal]\lar[r,phan,"q_0"'] & B_0\lar[r,phan,"v"'] & B_1\lar[r,phan,"q_1"'] & B_1\ar[d,equal] \\
            B_0\lar[rrr,phan,"v"'] &&& B_1
            \cellsymb(\theta){1-2}{3-3}
            \cellsymb(\beta_0)[below right]{2-2}{3-1}
            \cellsymb(\beta_1)[below left]{2-3}{3-4}
            \cellsymb(\delta_0){3-1}{4-2}
            \cellsymb(\delta_1){3-3}{4-4}
            \cellsymb(\veq){3-2}{4-3}
            \cellsymb(\sigma){4-1}{5-4}
        \end{tikzcd}
    \end{equation*}
    \begin{equation*}
        =
        \begin{tikzcd}
            & X_0\ar[d,"x_0"']\lar[r,path,"\tup{w}"] & X_1\ar[d,"x_1"] & \\
            & A_0\ar[d,"f_0"'] & A_1\ar[d,"f_1"] & \\
            & B_0\ar[dl,equal]\ar[d,equal]\lar[r,phan,"v"] & B_1\ar[d,equal]\ar[dr,equal] & \\
            B_0\ar[d,equal]\lar[r,phan,"q_0"'] & B_0\lar[r,phan,"v"'] & B_1\lar[r,phan,"q_1"'] & B_1\ar[d,equal] \\
            B_0\lar[rrr,phan,"v"'] &&& B_1
            \cellsymb(\theta){1-2}{3-3}
            \cellsymb(\eta_0)[below right]{3-2}{4-1}
            \cellsymb(\veq){3-2}{4-3}
            \cellsymb(\eta_1)[below left]{3-3}{4-4}
            \cellsymb(\sigma){4-1}{5-4}
        \end{tikzcd}
        =\theta.
    \end{equation*}
    This shows that $\alpha$ is cartesian.
\end{proof}

\begin{corollary}\label{cor:abs_cart_in_loosewise_indiscrete}
    Let $\bK$ be a loosewise indiscrete \ac{AVDC}.
    Then, every cell of the following form is absolutely cartesian.
    \begin{equation*}
        \begin{tikzcd}
            A\ar[d,"f"']\lar[r,"!_{AB}"] & B\ar[d,"g"] \\
            X\lar[r,"!_{XY}"'] & Y
            \cellsymb(!_{fg}){1-1}{2-2}
        \end{tikzcd}\incat{\bK}.\qedhere
    \end{equation*}
\end{corollary}
\begin{proof}
    By the loosewise indiscreteness, it immediately follows that the cell $!_{fg}$ is split.
    Then, \cref{lem:split_cell_is_cartesian} shows that it is absolutely cartesian.
\end{proof}
\subsection{Categories enriched in a virtual double category}
In this subsection, we will recall the notion of enriched categories in a \ac{VDC} from \cite{Leinster1999enrichment,Leinster2002enrichment}.
We first define the diminished \ac{AVDC} of \emph{matrices}, whose special cases have appeared in the literature: \cite{BettiCarboniStreetWalters1983variation} for bicategories, and \cite[Example 5.1.9]{Leinster2004higher} for multicategories.
\begin{definition}
    Let $\bX$ be an \ac{AVDC}.
    By an \emph{$\bX$-colored large set}, we mean a large set $\zero{A}$ equipped with a map $\zero{A}\arr(\abs{\cdot}_\zero{A})\Ob\bX$.
\end{definition}

\begin{definition}
    Let $\bX$ be an \ac{AVDC}.
    Let $\zero{A}$ and $\zero{B}$ be $\bX$-colored large sets.
    A \emph{morphism of families} $F$ from $\zero{A}$ to $\zero{B}$ consists of:
    \begin{itemize}
        \item
            for $x\in\zero{A}$, an element $F^0x\in\zero{B}$;
        \item
            for $x\in\zero{A}$, a tight arrow $\abs{x}_\zero{A}\arr(F^1x)\abs{F^0x}_\zero{B}$ in $\bX$.\qedhere
    \end{itemize}
\end{definition}

\begin{definition}
    Let $\bX$ be an \ac{AVDC}.
    Let $\zero{A}$ and $\zero{B}$ be $\bX$-colored large sets.
    An \emph{$(\zero{A}\times\zero{B})$-matrix} $M$ over $\bX$ is defined to be a family of loose arrows $\abs{x}_\zero{A}\larr(M(x,y))[][3]\abs{y}_\zero{B}$ in $\bX$ for $x\in\zero{A}$ and $y\in\zero{B}$.
\end{definition}

\begin{definition}
    Let $\bX$ be an \ac{AVDC}.
    The \emph{\ac{AVDC} of matrices over $\bX$}, denoted by $\Mat[\bX]$, is defined as follows: it is diminished, its objects are $\bX$-colored large sets, its tight arrows are morphisms of families, its loose arrows $\zero{A}\larr\zero{B}$ are $(\zero{A}\times\zero{B})$-matrices over $\bX$, and a 1-coary cell of the form
    \begin{equation*}
        \begin{tikzcd}
            \zero{A}_0\ar[d,"F"']\lar[r,"M_1"] & \zero{A}_1\lar[r,"M_2"] & \cdots\lar[r,"M_n"] & \zero{A}_n\ar[d,"G"] \\
            \zero{B}\lar[rrr,"N"'] & & & \zero{C}
            \cellsymb(\alpha){1-1}{2-4}
        \end{tikzcd}\incat{\Mat[\bX]}
    \end{equation*}
    consists of a family of cells
    \begin{equation*}
        \begin{tikzcd}[hugecolumn]
            \abs{x_0}_{\zero{A}_0}\ar[d,"F^1x_0"']\lar[r,"{M_1(x_0,x_1)}"] & \abs{x_1}_{\zero{A}_1}\lar[r,"{M_2(x_1,x_2)}"] & \cdots\lar[r,"{M_n(x_{n-1},x_n)}"] & \abs{x_n}_{\zero{A}_n}\ar[d,"G^1x_n"] \\
            \abs{F^0x_0}_\zero{B}\lar[rrr,"{N(F^0x_0,G^0x_n)}"'] &&& \abs{G^0x_n}_\zero{C}
            \cellsymb(\alpha_{x_0,x_1,\dots,x_n}){1-1}{2-4}
        \end{tikzcd}\incat{\bX},
    \end{equation*}
    one for each tuple of $x_0\in\zero{A}_0,x_1\in\zero{A}_1,\dots,x_n\in\zero{A}_n$.
\end{definition}

\begin{remark}
    In the above definition of $\Mat[\bX]$, we do not use any 0-coary cell in $\bX$, hence $\Mat[\bX]=\Mat[\dim{\bX}]$.
\end{remark}

\begin{remark}
    The tight category $\Tcat(\Mat[\bX])$ is isomorphic to $\Fam(\Tcat\bX)$, known as the \emph{category of families} or the (large) \emph{coproduct cocompletion of $\Tcat\bX$}.
\end{remark}

\begin{example}
    Let $\bi{V}$ be a monoidal category.
    Regarding $\bi{V}$ as a single-object bicategory, we have a diminished \ac{AVDC} $\Mat[(\Ldbl\bi{V})]$, which is also denoted by $\Mat[\bi{V}]$, whose objects are (large) sets, whose tight arrows are maps, and whose loose arrows $\zero{X}\larr\zero{Y}$ are families $(M(x,y))_{x\in\zero{X},y\in\zero{Y}}$ of objects in $\bi{V}$.
    When $\bi{V}$ is the two element chain, we have $\Mat[\bi{V}]\cong\dim{\Rel}$.
\end{example}

\begin{proposition}\label{prop:Mat_has_1-coary_restrictions}
    If an \ac{AVDC} $\bX$ has all 1-coary restrictions, so does $\Mat[\bX]$.
\end{proposition}
\begin{proof}
    Suppose that we are given the following data:
    \begin{equation*}
        \begin{tikzcd}
            \zero{A}'\ar[d,"F"'] & \zero{B}'\ar[d,"G"] \\
            \zero{A}\lar[r,"N"'] & \zero{B}
        \end{tikzcd}\incat{\Mat[\bX]}.
    \end{equation*}
    For $x\in\zero{A}'$ and $y\in\zero{B}'$, let $N(F,G)(x,y)$ denote the following loose arrow:
    \begin{equation*}
        \begin{tikzcd}[xhugecolumn]
            \abs{x}\ar[d,"F^1x"']\lar[r,"{N(F,G)(x,y)}"] & \abs{y}\ar[d,"G^1y"] \\
            \abs{F^0x}\lar[r,"{N(F^0x,G^0y)}"'] & \abs{G^0y}
            \cellsymb(\cart){1-1}{2-2}
        \end{tikzcd}\incat{\bX}.
    \end{equation*}
    Then, the matrix $N(F,G)$ over $\bX$ gives the desired restriction.
\end{proof}

\begin{definition}[Enrichment in a virtual double category]
    Let $\bX$ be an \ac{AVDC}.
    The \emph{\ac{AVDC} of $\bX$-enriched profunctors}, denoted by  $\Prof[\bX]$, is defined to be $\Mod(\Mat[\bX])$.
    Objects in $\Prof[\bX]$ are called \emph{$\bX$-enriched (large) categories}, tight arrows are called \emph{$\bX$-functors}, and loose arrows are called \emph{$\bX$-profunctors}.
    Note that $\Prof[\bX]$ has restrictions whenever $\bX$ has all 1-coary restrictions, which follows from \cref{prop:restriction_in_Mod,prop:Mat_has_1-coary_restrictions}.
\end{definition}

\begin{remark}
    Our $\bX$-enriched categories and $\bX$-functors coincide with Leinster's \cite{Leinster1999enrichment,Leinster2002enrichment}.
    For a bicategory $\bi{W}$, the \ac{AVDC} $\Prof[(\Ldbl\bi{W})]$ recovers the classical notion of enrichment in a bicategory, which includes ordinary enrichment in a monoidal category as a special case.
    Indeed, the tight 2-category $\Ttwocat(\Prof[(\Ldbl\bi{W})])$ is isomorphic to the 2-category of $\bi{W}$-enriched categories and $\bi{W}$-functors defined by Walters \cite{Walters1982sheaves}.
    Moreover, the loose bicategory $\Lbicat(\Prof[(\Ldbl\bi{W})])$ of \ac{VD}-composable objects coincides with the bicategory of sufficiently small $\bi{W}$-enriched categories and $\bi{W}$-profunctors (sometimes called \emph{$\bi{W}$-modules}).
    The \ac{AVDC} $\Prof[(\Ldbl\bi{W})]$ is also denoted by $\Prof[\bi{W}]$.
\end{remark}

\begin{remark}
    If an \ac{AVDC} $\bX$ is huge, then the \acp{AVDC} $\Mat[\bX]$, $\Mod(\bX)$, and $\Prof[\bX]$ are also huge.
\end{remark}

We now unpack the definition.
\begin{remark}
    Let $\bX$ be an \ac{AVDC}.
    An $\bX$-enriched (large) category $\one{A}$ consists of:
    \begin{itemize}
        \item (\emph{Colored objects})
            An $\bX$-colored large set $\Ob\one{A}$.
            For $x\in\Ob\one{A}$, its color is denoted by $\abs{x}_\one{A}$ or simply $\abs{x}$.
            When $\abs{x}=c$, we call $x$ an \emph{object colored with $c$}.
        \item (\emph{Hom-loose arrows})
            For $x,y\in\Ob\one{A}$, a loose arrow $\abs{x}\larr(\one{A}(x,y))[][3]\abs{y}$ in $\bX$.
        \item (\emph{Compositions})
            For $x,y,z\in\Ob\one{A}$, a cell $\mu_{x,y,z}$ of the following form:\vspace{-0.4em}
            \begin{equation*}
                \begin{tikzcd}[largecolumn]
                    \abs{x}\ar[d,equal]\lar[r,"{\one{A}(x,y)}"] & \abs{y}\lar[r,"{\one{A}(y,z)}"] & \abs{z}\ar[d,equal] \\
                    \abs{x}\lar[rr,"{\one{A}(x,z)}"'] & & \abs{z}
                    \cellsymb(\mu_{x,y,z}){1-1}{2-3}
                \end{tikzcd}\incat{\bX}.
            \end{equation*}
        \item (\emph{Identities})
            For each $x\in\Ob\one{A}$, a cell $\eta_x$ of the following form:\vspace{-0.4em}
            \begin{equation*}
                \begin{tikzcd}[tri]
                    & \abs{x}\ar[dl,equal]\ar[dr,equal] & \\
                    \abs{x}\lar[rr,"{\one{A}(x,x)}"'] &{}& \abs{x}
                    \cellsymb(\eta_x){1-2}{2-2}
                \end{tikzcd}\incat{\bX}.
            \end{equation*}
    \end{itemize}\vspace{-0.4em}
    The above data are required to satisfy suitable axioms.
\end{remark}

\begin{proposition}\label{prop:enriched_cat_is_AVD_functor}
    Let $\bX$ be an \ac{AVDC}.
    Then, an $\bX$-enriched (large) category is the same as the following data:
    \begin{itemize}
        \item
            A (large) set $\zero{S}$;
        \item
            An \ac{AVD}-functor $\Idimdbl\zero{S}\to\bX$.\qedhere
    \end{itemize}
\end{proposition}
\begin{proof}
    Let $\one{A}$ be an $\bX$-enriched large category.
    Then, the following assignments yield an \ac{AVD}-functor $\Idimdbl\Ob\one{A}\to\bX$:
    \begin{equation*}
        x\mapsto \abs{x}_\one{A},
        \qquad\qquad
        x\larr(!_{xy})y \quad\mapsto\quad \abs{x}\larr(\one{A}(x,y))[][3]\abs{y},
    \end{equation*}
    \begin{equation*}
        \begin{tikzcd}[tri]
            & x\ar[dl,equal]\ar[dr,equal] & \\
            x\lar[rr,"!_{xx}"'] &{}& x
            \cellsymb(!){1-2}{2-2}
        \end{tikzcd}
        \mapsto
        \begin{tikzcd}[tri]
            & \abs{x}\ar[dl,equal]\ar[dr,equal] & \\
            \abs{x}\lar[rr,"{\one{A}(x,x)}"'] &{}& \abs{x}
            \cellsymb(\eta_x){1-2}{2-2}
        \end{tikzcd}
        \qquad
        \begin{tikzcd}
            x\ar[d,equal]\lar[r,"!_{xy}"] & y\lar[r,"!_{yz}"] & z\ar[d,equal] \\
            x\lar[rr,"!_{xz}"'] && z
            \cellsymb(!){1-1}{2-3}
        \end{tikzcd}
        \mapsto
        \begin{tikzcd}[largecolumn]
            \abs{x}\ar[d,equal]\lar[r,"{\one{A}(x,y)}"] & \abs{y}\lar[r,"{\one{A}(y,z)}"] & \abs{z}\ar[d,equal] \\
            \abs{x}\lar[rr,"{\one{A}(x,z)}"'] & & \abs{z}
            \cellsymb(\mu_{x,y,z}){1-1}{2-3}
        \end{tikzcd}
    \end{equation*}
    Furthermore, we can reconstruct $\one{A}$ from the \ac{AVD}-functor $\Idimdbl\Ob\one{A}\to\bX$.
\end{proof}

\begin{notation}\label{note:AVD_functor_Y}
    Let $\bX$ be an \ac{AVDC}.
    For $c\in\bX$, let $\zero{Y}_c$ denote the $\bX$-colored set $\zero{Y}_c\coloneq\{\ast\}$ containing a unique element $\ast$ colored with $c$.
    It easily follows that the full sub-\ac{AVDC} of $\Mat[\bX]$ spanned by the objects $\zero{Y}_c$ is isomorphic to $\dim{\bX}$.
    We write $Y\colon \dim{\bX}\to\Mat[\bX]$ for the corresponding inclusion.
\end{notation}

\begin{notation}\label{note:AVD_functor_Z}
    Let $\bX$ be an \ac{AVDC} with loose units.
    We write $Z\colon\bX\to\Prof[\bX]$ for an \ac{AVD}-functor corresponding to $Y\colon \dim{\bX}\to\Mat[\bX]$ by \cref{thm:universal_property_of_Mod}.
    We write $\one{Z}_c$ for the $\bX$-enriched category assigned to each $c\in\bX$ by $Z$.
\end{notation}

\begin{lemma}\label{rem:monoid_structure_of_Zc}
    Let $\bX$ be an \ac{AVDC} with loose units, and let $c\in\bX$.
    Then, the unit cell associated with the monoid $\one{Z}_c$ is \ac{VD}-cocartesian in $\Mat[\bX]$.   
\end{lemma}
\begin{proof}
     Let
    \begin{equation*}
        \begin{tikzcd}[tri]
            & c\ar[dl,equal]\ar[dr,equal] & \\
            c\lar[rr,"\Unit_c"'] &{}& c
            \cellsymb(\gamma){1-2}{2-2}
        \end{tikzcd}\incat{\bX}
    \end{equation*}
    be the loosewise invertible (cocartesian) cell associated with the loose unit $\Unit_c$ of $c$.
    In the diminished \ac{AVDC} $\dim{\bX}$, the cell $\gamma$ is no longer cocartesian but \ac{VD}-cocartesian.
    Moreover, we see at once that the \ac{VD}-cocartesian cell $\gamma$ is preserved by the \ac{AVD}-functor $Y\colon \dim{\bX}\to\Mat[\bX]$.
    Thus, the monoid structure of $\one{Z}_c$ is induced by the \ac{VD}-cocartesian cell $Y\gamma$.
\end{proof}

\begin{definition}
    Let $\one{A}$ be an $\bX$-enriched category.
    A \emph{preobject} in $\one{A}$ colored with $c\in\bX$ is a pair $x=(x^0,x^1)$ of an object $x^0\in\Ob\one{A}$ and a tight arrow $c\arr(x^1)\abs{x^0}$ in $\bX$.
\end{definition}

\begin{remark}
    In this terminology, we regard objects as preobjects $x$ whose underlying tight arrow $x^1$ is the identity.
    In addition, if $\bX$ has loose units, then the preobjects of an $\bX$-enriched category $\one{A}$ form the category $\Tcat\bX/\one{A}$ defined later in \cref{note:comma_category}, whose \textit{maximal} objects (\cref{def:maximal_objects}) are the same as the objects in $\one{A}$.
    The term ``preobject'' comes from this fact.
\end{remark}

We call $\one{Z}_c$ the \emph{preobject classifier} because it classifies the preobjects colored with $c$ in the following sense:
\begin{theorem}\label{thm:object_classifiers_vdc-enriched}
    Let $\bX$ be an \ac{AVDC} with loose units, and let $c\in\bX$.
    Then, there is a bijective correspondence between the $\bX$-functors $\one{Z}_c\to\one{A}$ and the preobjects in $\one{A}$ colored with $c$.
\end{theorem}
\begin{proof}
    By \cref{rem:monoid_structure_of_Zc}, a monoid homomorphism $\one{Z}_c\to\one{A}$ is simply a tight arrow $\zero{Y}_c\to\Ob\one{A}$ in $\Mat[\bX]$.
    Indeed, a monoid homomorphism $\one{Z}_c\arr((f^0,f^1))[][3]\one{A}$ must be compatible with units as follows:
    \begin{equation*}
        \begin{tikzcd}[tri]
            & \zero{Y}_c\ar[d,bend right=20,"f^0"{left}]\ar[d,bend left=20,"f^0"{right}] & \\
            & A^0\ar[dl,equal]\ar[dr,equal] & \\
            A^0\lar[rr,"A^1"'] &{}& A^0
            \cellsymb(\heq){1-2}{2-2}
            \cellsymb(A^e){2-2}{3-2}
        \end{tikzcd}
        =
        \begin{tikzcd}
            & \zero{Y}_c\ar[dl,equal]\ar[dr,equal] & \\
            \zero{Y}_c\ar[d,"f^0"']\lar[rr,"Y\Unit_c"'] &{}& \zero{Y}_c\ar[d,"f^0"] \\
            A^0\lar[rr,"A^1"'] && A^0
            \cellsymb(\VDcocart){1-2}{2-2}
            \cellsymb(f^1){2-1}{3-3}
        \end{tikzcd}\incat{\Mat[\bX]}.
    \end{equation*}
    Here, $\one{A}$ is regarded as a monoid $(\Ob\one{A}=A^0,A^1,A^e,A^m)$ in $\Mat[\bX]$.
    By the universal property of the \ac{VD}-cocartesian cell, $f^1$ can be reconstructed uniquely from $f^0$.
    Since the compatibility of $f^1$ with multiplications is automatically satisfied, the monoid homomorphism $(f^0,f^1)$ is the same as the tight arrow $f^0$.
    Since $f^0$ is simply a choice of a preobject in $\one{A}$ colored with $c$, this finishes the proof.
\end{proof}

\begin{theorem}\label{thm:Z_is_embedding}
    For an \ac{AVDC} $\bX$ with loose units, the \ac{AVD}-functor $Z\colon\bX\to\Prof[\bX]$ makes $\bX$ into a full sub-\ac{AVDC} of $\Prof[\bX]$.
\end{theorem}
\begin{proof}
    Let $c,d$ be objects in $\bX$.
    By \cref{thm:object_classifiers_vdc-enriched}, the $\bX$-functors $\one{Z}_c\to\one{Z}_d$ are the same as the tight arrows $c\to d$ in $\bX$.
    The same is true for loose arrows.
    Indeed, an $\bX$-profunctor $\one{Z}_c\larr((P^1,P^l,P^r))[][3.5]\one{Z}_d$ must be compatible with the unit of $\one{Z}_c$ for example:
    \begin{equation*}
        \begin{tikzcd}
            & \zero{Y}_c\ar[dl,equal]\ar[dr,equal]\lar[rr,"P^1"] && \zero{Y}_d\ar[dr,equal] & \\
            \zero{Y}_c\ar[d,equal]\lar[rr,"Y\Unit_c"'] &{}& \zero{Y}_c\lar[rr,"P^1"'] &{}& \zero{Y}_d\ar[d,equal] \\
            \zero{Y}_c\lar[rrrr,"P^1"'] &&&& \zero{Y}_d
            \cellsymb(\VDcocart){1-2}{2-2}
            \cellsymb(\veq){1-4}{2-3}
            \cellsymb(P^l){2-1}{3-5}
        \end{tikzcd}
        =
        \begin{tikzcd}
            \zero{Y}_c\ar[d,equal]\lar[r,"P^1"] & \zero{Y}_d\ar[d,equal] \\
            \zero{Y}_c\lar[r,"P^1"'] & \zero{Y}_d
            \cellsymb(\veq){1-1}{2-2}
        \end{tikzcd}\incat{\Mat[\bX]}.
    \end{equation*}
    By the universal property of the \ac{VD}-cocartesian cell, $P^l$ can be reconstructed uniquely from $P^1$, and so can $P^r$.
    Since the compatibility with the multiplications of $\one{Z}_c$ and $\one{Z}_d$ is automatically satisfied, the $\bX$-profunctor $(P^1,P^l,P^r)$ is the same as a loose arrow $P^1$.
    Since $Y\colon\dim{\bX}\to\Mat[\bX]$ is a full inclusion, the loose arrow $P^1$ is simply a loose arrow $c\larr[][1] d$ in $\bX$.

    Similarly, we can establish between $\bX$ and $\Prof[\bX]$, a bijective correspondence of 1-coary cells.
    Furthermore, since both $\bX$ and $\Prof[\bX]$ have loose units, the same is true also for 0-coary cells.
    This finishes the proof.
\end{proof}

We give a remark on idempotency of the $\Prof$-construction.
\begin{remark}
    It is known that, at the level of bicategories, the profunctor construction is idempotent up to biequivalence.
    Indeed, for any bicategory $\bi{W}$ admitting a suitable cocompleteness condition, the bicategory $\biProf[\bi{W}]$ of small $\bi{W}$-categories and $\bi{W}$-profunctors is biequivalent to the bicategory $\biProf[{(\biProf[\bi{W}])}]$ of small $(\biProf[\bi{W}])$-categories and $(\biProf[\bi{W}])$-profunctors \cite[Proposition 2.2]{CarboniKasangianWalters1987axiomatics}.
    Note that this idempotency result was later refined in \cite{GarnerShulman2016enriched}.
    However, at the level of double categories, such idempotency fails (cf.\ \cite[p.\ 7]{GarnerShulman2016enriched}).
    Indeed, there is a bicategory $\bi{B}$ with no equivalence in $\AVDC$ between $\Prof[\bi{B}]$ and $\Prof[(\Prof[\bi{B}])]$.
    Moreover, such an equivalence does still not exist even if we weaken the notion of equivalences into, at the level of objects, up to ``tightwise equivalence'' rather than up to invertible tight arrow.
    For such a counterexample, we can take $\bi{B}$ to be the trivial one, as shown in the following lemma.
\end{remark}

\begin{lemma}
    Let us consider the diminished \ac{AVDC} $\bX\coloneq\Idimdbl\zero{1}$, where $\zero{1}\coloneq\{c\}$ is the singleton.
    Then, there is no \ac{AVD}-functor $\Prof[\bX]\arr(K)\Prof[(\Prof[\bX])]$ satisfying the following:
    \begin{itemize}
        \item
            It is surjective on objects up to \emph{tightwise equivalence}, i.e., equivalence in the tight 2-categories.
        \item
            It is ``full'' on tight arrows.
            Equivalently, it induces surjections between the local classes of tight arrows.
    \end{itemize}
\end{lemma}
\begin{proof}
    Suppose that such an \ac{AVD}-functor $K$ exists.
    Let $\bfempty$ be the empty $\bX$-category.
    Consider the following three $(\Prof[\bX])$-categories and unique $(\Prof[\bX])$-functors between them:
    \begin{equation}\label{eq:sequence_in_profprof}
        \bar\bfempty \arr \one{Z}_\bfempty \arr \one{Z}_{\one{Z}_c}
        \incat{\Prof[(\Prof[\bX])]}.
    \end{equation}
    Here, $\bar\bfempty$ denotes the empty $(\Prof[\bX])$-category.
    Then, we can observe that, in the sequence \cref{eq:sequence_in_profprof}, there is no $(\Prof[\bX])$-functor in the opposite direction.
    Now, from the first condition for $K$, there are three $\bX$-categories $\one{A},\one{B},\one{C}$ and tightwise equivalences $K\one{A}\simeq\bar\bfempty, K\one{B}\simeq\one{Z}_\bfempty, K\one{C}\simeq\one{Z}_{\one{Z}_c}$ in $\Prof[(\Prof[\bX])]$.
    Then, from the second condition for $K$, we have a sequence of $\bX$-functors
    \begin{equation*}
        \one{A}\arr\one{B}\arr\one{C}
        \incat{\Prof[\bX]},
    \end{equation*}
    and there is still no $\bX$-functor in the opposite direction.
    However, since the tight category $\Tcat(\Prof[\bX])$ is isomorphic to the category of large sets, such a sequence cannot exist.
    This is a contradiction.
\end{proof}
\section{Colimits in augmented virtual double categories}\label{sec:colim}
\subsection{Cocones, modules, and modulations}
To give a notion of ``colimits'' in an \ac{AVDC}, we consider ``cocones'' for each of the three directions: left, right, and downward.
The ``cocones'' for the downward direction are called \emph{tight cocones}, and the ``cocones'' for the left and right directions are called left and right \emph{modules}, respectively.
In addition, we also consider several types of morphisms between them, called \emph{modulations}.
The terms ``module'' and ``modulation'' come from similar concepts in \cite{Pare2011yoneda}.
Although modulations can be defined in greater generality, we will only consider certain special cases, which we call types 0, 1, 2, and 3.
See \cref{rem:general_modulations} for further comments.
\begin{definition}[Tight cocones]
    Let $F\colon\bK\to\bL$ be an \ac{AVD}-functor between \acp{AVDC}.
    A \emph{tight cocone} $l$ (from $F$) consists of:
    \begin{itemize}
        \item
            an object $L\in\bL$ (the \emph{vertex} of $l$);
        \item
            for each $A\in\bK$, a tight arrow
            $
            \begin{tikzcd}[small]
                FA\ar[d,"l_A"']\\
                L
            \end{tikzcd}
            $
            in $\bL$;
        \item
            for each $A\larr(u)B$ in $\bK$, a cell
            $
            \begin{tikzcd}[tri]
                FA\ar[dr,"l_A"']\lar[rr,"Fu"] &{}& FB\ar[dl,"l_B"] \\
                & L &
                \cellsymb(l_u)[above=-2]{1-2}{2-2}
            \end{tikzcd}\incat{\bL}
            $
    \end{itemize}
    satisfying the following conditions:
    \begin{itemize}
        \item
            For any tight arrow $A\arr(f)B$ in $\bK$, $(Ff)\tcomp l_B=l_A$;
        \item
            For any cell
            \begin{equation*}
                \begin{tikzcd}[small]
                    A_0\ar[d,"f"']\lar[r,"u_1"] & A_1\lar[r,"u_2"] & \cdots\lar[r,"u_n"] & A_n\ar[d,"g"] \\
                    X\lar[rrr,phan,"v"'] &&& Y
                    \cellsymb(\alpha){1-1}{2-4}
                \end{tikzcd}\incat{\bK},
            \end{equation*}
            \begin{equation*}
                \begin{tikzcd}[scriptsize]
                    FA_0\ar[d,"Ff"']\lar[rr,path,"F\tup{u}"] && FA_n\ar[d,"Fg"] \\
                    FX\ar[dr,"l_X"']\lar[rr,phan,"Fv"] &{}& FY\ar[dl,"l_Y"] \\
                    & L &
                    \cellsymb(F\alpha){1-1}{2-3}
                    \cellsymb(l_v)[above=-3]{2-2}{3-2}
                \end{tikzcd}
                =
                \begin{tikzcd}[scriptsize]
                    FA_0\ar[dr,"l_{A_0}"']\lar[rr,path,"F\tup{u}"] &{}& FA_n\ar[dl,"l_{A_n}"] \\
                    & L &
                    \cellsymb(l_{\tup{u}})[above=-3]{1-2}{2-2}
                \end{tikzcd}\incat{\bL}.
            \end{equation*}
            Here $l_\tup{u}$ denotes the composite of the following cells:
            \begin{equation*}
                \begin{tikzcd}
                    FA_0\ar[drr,"l_{A_0}"']\lar[r,"Fu_1"] & FA_1\ar[dr,"l_{A_1}"{description}]\lar[r,"Fu_2"] & \cdots\lar[r,"Fu_{n-1}"] & FA_{n-1}\ar[dl,"l_{A_{n-1}}"'{description}]\lar[r,"Fu_n"] & FA_n\ar[dll,"l_{A_n}"] \\[10pt]
                    && L &&
                    \cellsymb(l_{u_1})[above=2]{1-1}{2-3}
                    \cellsymb(\cdots)[above=2]{1-3}{2-3}
                    \cellsymb(l_{u_n})[above=2]{1-5}{2-3}
                \end{tikzcd}\incat{\bL}.
            \end{equation*}
            When $\tup{u}$ is of length 0, the cell $l_\tup{u}$ is defined to be the tight identity.\qedhere
    \end{itemize}
\end{definition}

\begin{definition}
    A tight cocone $l$ is called \emph{strong} if $l_u$ is cartesian for any loose arrow $u$.
\end{definition}

\begin{remark}\label{rem:cocone_grandis_pare}
    Let $\bK$ be a diminished \ac{AVDC} with loose \ac{VD}-composites, and let $\bL$ be an \ac{AVDC} with loose composites.
    Then, both $\bK$ and $\bL$ can be regarded as pseudo double categories, and in fact, an \ac{AVD}-functor $F\colon\bK\to\bL$ coincides with a lax double functor between them, as well as described in \cref{rem:pseudo_double_category}.
    Then, a tight cocone from $F$ is the same thing as a \textit{(strict) cocone} of $F$ in the sense of \cite[5.2.1]{Grandis2020higher}, whose tightwise dual notion is originally introduced in \cite[4.1]{GrandisPare1999limits}.
\end{remark}

\begin{definition}[Left/right modules]\label{def:left_and_right_modules}
    Let $F\colon\bK\to\bL$ be an \ac{AVD}-functor between \acp{AVDC}.
    A \emph{left $F$-module} $m$ consists of:
    \begin{itemize}
        \item
            an object $M\in\bL$ (the \emph{vertex} of $m$);
        \item
            for each $A\in\bK$, a loose arrow $FA\larr(m_A)M$ in $\bL$;
        \item
            for each $A\arr(f)B$ in $\bK$, a cartesian cell
            \begin{equation*}
                \begin{tikzcd}
                    FA\ar[d,"Ff"']\lar[r,"m_A"] & M\ar[d,equal] \\
                    FB\lar[r,"m_B"'] & M
                    \cellsymb(m_f\colon\cart){1-1}{2-2}
                \end{tikzcd}\incat{\bL};
            \end{equation*}
        \item
            for each $A\larr(u)B$ in $\bK$,  a cell
            \begin{equation*}
                \begin{tikzcd}
                    FA\ar[d,equal]\lar[r,"Fu"] & FB\lar[r,"m_B"] & M\ar[d,equal] \\
                    FA\lar[rr,"m_A"'] & & M
                    \cellsymb(m_u){1-1}{2-3}
                \end{tikzcd}\incat{\bL}
            \end{equation*}
    \end{itemize}
    satisfying the following conditions:
    \begin{itemize}
        \item
            For any $A\arr(f)B\arr(g)C$ in $\bK$,
            \begin{equation*}
                \begin{tikzcd}
                    FA\ar[d,"Ff"']\lar[r,"m_A"] & M\ar[d,equal] \\
                    FB\ar[d,"Fg"']\lar[r,"m_B"] & M\ar[d,equal] \\
                    FC\lar[r,"m_C"'] & M
                    \cellsymb(m_f){1-1}{2-2}
                    \cellsymb(m_g){2-1}{3-2}
                \end{tikzcd}
                =
                \begin{tikzcd}
                    FA\ar[d,"Ff"']\lar[r,"m_A"] & M\ar[dd,equal] \\
                    FB\ar[d,"Fg"'] & \\
                    FC\lar[r,"m_C"'] & M
                    \cellsymb(m_{f\tcomp g}){1-1}{3-2}
                \end{tikzcd}\incat{\bL}.
            \end{equation*}
        \item
            For any $A\in\bK$,
            \begin{equation*}
                \begin{tikzcd}
                    FA\ar[d,equal,"F\id_A"']\lar[r,"m_A"] & M\ar[d,equal] \\
                    FA\lar[r,"m_A"'] & M
                    \cellsymb(m_{\id_A}){1-1}{2-2}
                \end{tikzcd}
                =
                \begin{tikzcd}
                    FA\ar[d,equal]\lar[r,"m_A"] & M\ar[d,equal] \\
                    FA\lar[r,"m_A"'] & M
                    \cellsymb(\veq){1-1}{2-2}
                \end{tikzcd}\incat{\bL}.
            \end{equation*}
        \item
            For any cell
            \begin{equation*}
                \begin{tikzcd}[small]
                    A_0\ar[d,"f"']\lar[r,"u_1"] & A_1\lar[r,"u_2"] & \cdots\lar[r,"u_n"] & A_n\ar[d,"g"] \\
                    X\lar[rrr,phan,"v"'] &&& Y
                    \cellsymb(\alpha){1-1}{2-4}
                \end{tikzcd}\incat{\bK},
            \end{equation*}
            \begin{equation*}
                \begin{tikzcd}
                    FA_0\ar[d,"Ff"']\lar[r,path,"F\tup{u}"] & FA_n\ar[d,"Fg"']\lar[r,"m_{A_n}"] & M\ar[d,equal] \\
                    FX\ar[d,equal]\lar[r,phan,"Fv"'] & FY\lar[r,"m_Y"'] & M\ar[d,equal] \\
                    FX\lar[rr,"m_X"'] & & M
                    \cellsymb(F\alpha){1-1}{2-2}
                    \cellsymb(m_g){1-2}{2-3}
                    \cellsymb(m_v){2-1}{3-3}
                \end{tikzcd}
                =
                \begin{tikzcd}
                    FA_0\ar[d,equal]\lar[r,path,"F\tup{u}"] & FA_n\lar[r,"m_{A_n}"] & M\ar[d,equal] \\
                    FA_0\ar[d,"Ff"']\lar[rr,"m_{A_0}"'] & & M\ar[d,equal] \\
                    FX\lar[rr,"m_X"'] & & M
                    \cellsymb(m_{\tup{u}}){1-1}{2-3}
                    \cellsymb(m_f){2-1}{3-3}
                \end{tikzcd}\incat{\bL}.
            \end{equation*}
            Here, $m_{\tup{u}}$ denotes the composite of the following cells:
            \begin{equation*}
                \begin{tikzcd}
                    FA_0\ar[d,equal]\lar[r,"Fu_1"] & FA_1\ar[d,equal]\lar[r,"Fu_2"] & \cdots\lar[r,"Fu_{n-1}"] & FA_{n-1}\ar[d,equal]\lar[r,"Fu_n"] & FA_n\lar[r,"m_{A_n}"] & M\ar[d,equal] \\
                    FA_0\ar[d,equal]\lar[r,"Fu_1"'] & FA_1\ar[d,equal]\lar[r,"Fu_2"'] & \cdots\lar[r,"Fu_{n-1}"'] & FA_{n-1}\lar[rr,"m_{A_{n-1}}"'] & & M\ar[d,equal] \\
                    \vdots\ar[d,equal] & \vdots\ar[d,equal] & & & & \vdots\ar[d,equal] \\
                    FA_0\ar[d,equal]\lar[r,"Fu_1"] & FA_1\lar[rrrr,"m_{A_1}"] & & & & M\ar[d,equal] \\
                    FA_0\lar[rrrrr,"m_{A_0}"'] & & & & & M
                    \cellsymb(\veq){1-1}{2-2}
                    \cellsymb(\veq){1-2}{2-3}
                    \cellsymb(\cdots){1-3}{2-3}
                    \cellsymb(\veq){2-3}{1-4}
                    \cellsymb(m_{u_n}){1-4}{2-6}
                    \cellsymb(m_{u_1}){4-1}{5-6}
                \end{tikzcd}\incat{\bL}.
            \end{equation*}
            When $\tup{u}$ (resp.\ $v$) is of length 0, the cell $m_\tup{u}$ (resp.\ $m_v$) is defined to be the loose identity.
    \end{itemize}
    Moreover, \emph{right $F$-modules} are also defined as the loosewise dual of the left $F$-modules.
\end{definition}

\begin{remark}
    The most significant difference between our definition of modules and that of \cite[3.2.\ Definition]{Pare2011yoneda} lies in the requirement of the cells $m_f$ being cartesian.
    This requirement is necessary in order for the axioms \ref{axiom:loose_esssurj_left} and \ref{axiom:loose_esssurj_right}, introduced later in the definition of versatile colimits, to be meaningful.
    Indeed, the modules constructed in \cref{const:canonical_functor_loose_to_mdl} automatically satisfy this cartesianness condition by the Pasting Lemma (\cref{prop:pasting_lemma_cartesian}).
\end{remark}

\begin{notation}
    A tight cocone from $F$ with a vertex $L$ is denoted by a double arrow $F\arr[Rightarrow][1]L$.
    A left (resp.\ right) $F$-module with a vertex $M$ is denoted by a slashed double arrow $F\larr[Rightarrow]M$ (resp.\ $M\larr[Rightarrow]F$).
\end{notation}

\begin{definition}[Modulations of type 0]
    Let $F\colon\bK\to\bL$ be an \ac{AVD}-functor between \acp{AVDC}.
    Let $m,m'$ be left $F$-modules whose vertices are $M,M'\in\bL$, respectively.
    Consider\linebreak \mbox{$M\larr(\tup{p})[path] M''\arr(j) M'$} in $\bL$.
    A \emph{modulation (of type 0)} $\rho$, denoted by
    \begin{equation}\label{eq:modulation_of_type0}
        \begin{tikzcd}
            F\ar[d,equal]\lar[r,"m",Rightarrow] & M\lar[r,path,"\tup{p}"] & M''\ar[d,"j"] \\
            F\lar[rr,"m'"',Rightarrow] & & M'
            \cellsymb(\rho){1-1}{2-3}
        \end{tikzcd}
    \end{equation}
    consists of:
    \begin{itemize}
        \item
            for each $A\in\bK$, a cell
            \begin{equation*}
                \begin{tikzcd}
                    FA\ar[d,equal]\lar[r,"m_A"] & M\lar[r,path,"\tup{p}"] & M''\ar[d,"j"] \\
                    FA\lar[rr,"m'_A"'] & & M'
                    \cellsymb(\rho_A){1-1}{2-3}
                \end{tikzcd}\incat{\bL}
            \end{equation*}
    \end{itemize}
    satisfying the following conditions:
    \begin{itemize}
        \item
            For any $A\arr(f)B$ in $\bK$,
            \begin{equation*}
                \begin{tikzcd}
                    FA\ar[d,"Ff"']\lar[r,"m_A"] & M\ar[d,equal]\lar[r,path,"\tup{p}"] & M''\ar[d,equal] \\
                    FB\ar[d,equal]\lar[r,"m_B"] & M\lar[r,path,"\tup{p}"] & M''\ar[d,"j"] \\
                    FB\lar[rr,"m'_B"'] && M'
                    \cellsymb(m_f){1-1}{2-2}
                    \cellsymb(\veq){1-2}{2-3}
                    \cellsymb(\rho_B){2-1}{3-3}
                \end{tikzcd}
                =
                \begin{tikzcd}
                    FA\ar[d,equal]\lar[r,"m_A"] & M\lar[r,path,"\tup{p}"] & M''\ar[d,"j"] \\
                    FA\ar[d,"Ff"']\lar[rr,"m'_A"] && M'\ar[d,equal] \\
                    FB\lar[rr,"m'_B"'] && M'
                    \cellsymb(\rho_A){1-1}{2-3}
                    \cellsymb(m'_f){2-1}{3-3}
                \end{tikzcd}\incat{\bL}.
            \end{equation*}
        \item
            For any $A\larr(u)B$ in $\bK$,
            \begin{equation*}
                \begin{tikzcd}
                    FA\ar[d,equal]\lar[r,"Fu"] & FB\lar[r,"m_B"] & M\ar[d,equal]\lar[r,path,"\tup{p}"] & M''\ar[d,equal] \\
                    FA\ar[d,equal]\lar[rr,"m_A"'] & & M\lar[r,path,"\tup{p}"] & M''\ar[d,"j"] \\
                    FA\lar[rrr,"m'_A"'] & & & M'
                    \cellsymb(m_u){1-1}{2-3}
                    \cellsymb(\veq){1-3}{2-4}
                    \cellsymb(\rho_A){2-1}{3-4}
                \end{tikzcd}
                =
                \begin{tikzcd}
                    FA\ar[d,equal]\lar[r,"Fu"] & FB\ar[d,equal]\lar[r,"m_B"] & M\lar[r,path,"\tup{p}"] & M''\ar[d,"j"] \\
                    FA\ar[d,equal]\lar[r,"Fu"'] & FB\lar[rr,"m'_B"'] & & M'\ar[d,equal] \\
                    FA\lar[rrr,"m'_A"'] & & & M'
                    \cellsymb(\veq){1-1}{2-2}
                    \cellsymb(\rho_B){1-2}{2-4}
                    \cellsymb(m'_u){2-1}{3-4}
                \end{tikzcd}\incat{\bL}.
            \end{equation*}
    \end{itemize}
\end{definition}

\begin{notation}
    For a functor $F\colon\bK\to\bL$ between \acp{AVDC} and $M\in\bL$, let $\Mdl{F}{M}$ denote the category of left $F$-modules with the vertex $M$.
    A morphism $m\to m'$ in $\Mdl{F}{M}$ is defined as a modulation of type 0 such that $\tup{p}$ is of length 0 and $j$ is the identity in \cref{eq:modulation_of_type0}.
    Similarly, we write $\Mdl{M}{F}$ for the category of right $F$-modules with the vertex $M$.
\end{notation}

\begin{remark}
    A modulation (of type 0) $\rho\colon m\to m'$ in $\Mdl{F}{M}$ is called \emph{invertible} if every component $\rho_A$ is loosewise invertible.
    Such a modulation (of type 0) is the same as an isomorphism in $\Mdl{F}{M}$.
\end{remark}

\begin{definition}[Modulations of type 1]\label{def:modulation_of_type1}
    Let $F\colon\bK\to\bL$ be an \ac{AVD}-functor between \acp{AVDC}.
    Let $F\arr(l)[Rightarrow]L\in\bL$ be a tight cocone and let $F\larr(m)[Rightarrow]M\in\bL$ be a left $F$-module.
    Consider $M\larr(\tup{p})[path]M'$, $M'\arr(j)L'$, and $L\larr(q)[phan]L'$ in $\bL$.
    A \emph{modulation (of type 1)} $\sigma$, denoted by
    \begin{equation*}
        \begin{tikzcd}
            F\ar[d,Rightarrow,"l"']\lar[r,Rightarrow,"m"] & M\lar[r,path,"\tup{p}"] & M'\ar[d,"j"] \\
            L\lar[rr,phan,"q"'] & & L'
            \cellsymb(\sigma){1-1}{2-3}
        \end{tikzcd}
    \end{equation*}
    consists of:
    \begin{itemize}
        \item
            for each $A\in\bK$, a cell
            \begin{equation*}
                \begin{tikzcd}
                    FA\ar[d,"l_A"']\lar[r,"m_A"] & M\lar[r,path,"\tup{p}"] & M'\ar[d,"j"] \\
                    L\lar[rr,phan,"q"'] & & L'
                    \cellsymb(\sigma_A){1-1}{2-3}
                \end{tikzcd}\incat{\bL}
            \end{equation*}
    \end{itemize}
    satisfying the following conditions:
    \begin{itemize}
        \item
            For any $A\arr(f)B$ in $\bK$,
            \begin{equation*}
                \begin{tikzcd}
                    FA\ar[d,"Ff"']\lar[r,"m_A"] & M\ar[d,equal]\lar[r,path,"\tup{p}"] & M'\ar[d,equal] \\
                    FB\ar[d,"l_B"']\lar[r,"m_B"'] & M\lar[r,path,"\tup{p}"'] & M'\ar[d,"j"] \\
                    L\lar[rr,phan,"q"'] & & L'
                    \cellsymb(m_f){1-1}{2-2}
                    \cellsymb(\veq){1-2}{2-3}
                    \cellsymb(\sigma_B){2-1}{3-3}
                \end{tikzcd}
                =
                \begin{tikzcd}[scriptsize]
                    FA\ar[d,"l_A"']\lar[r,"m_A"] & M\lar[r,path,"\tup{p}"] & M'\ar[d,"j"] \\
                    L\lar[rr,phan,"q"'] & & L'
                    \cellsymb(\sigma_A){1-1}{2-3}
                \end{tikzcd}\incat{\bL}.
            \end{equation*}
        \item
            For any $A\larr(u)B$ in $\bK$,
            \begin{equation*}
                \begin{tikzcd}
                    FA\ar[d,equal]\lar[r,"Fu"] & FB\lar[r,"m_B"] & M\ar[d,equal]\lar[r,path,"\tup{p}"] & M'\ar[d,equal] \\
                    FA\ar[d,"l_A"']\lar[rr,"m_A"'] & & M\lar[r,path,"\tup{p}"'] & M'\ar[d,"j"] \\
                    L\lar[rrr,phan,"q"'] & & & L'
                    \cellsymb(m_u){1-1}{2-3}
                    \cellsymb(\veq){1-3}{2-4}
                    \cellsymb(\sigma_A){2-1}{3-4}
                \end{tikzcd}
                =
                \begin{tikzcd}
                    FA\ar[d,"l_A"']\lar[r,"Fu"] & FB\lar[r,"m_B"]\ar[dl,"l_B"{right=3}] & M\lar[r,path,"\tup{p}"] & M'\ar[d,"j"] \\
                    L\lar[rrr,phan,"q"'] & & & L'
                    \cellsymb(l_u)[above left]{1-2}{2-1}
                    \cellsymb(\sigma_B)[left=0.2]{1-2}{2-4}
                \end{tikzcd}\incat{\bL}.
            \end{equation*}
    \end{itemize}
\end{definition}

\begin{remark}
    Suppose that, in the situation of \cref{def:modulation_of_type1}, we are alternatively given a right $F$-module $M\larr(m)[Rightarrow][1.5]F$, loose paths $M'\larr(\tup{p})[path][1.3]M$ and $L'\larr(q)[phan][1.3]L$ in $\bL$.
    Then, we can also define the loosewise dual concept, which is called modulations of type 1 as well and is denoted by
    \begin{equation*}
        \begin{tikzcd}
            M'\ar[d,"j"']\lar[r,path,"\tup{p}"] & M\lar[r,Rightarrow,"m"] & F\ar[d,Rightarrow,"l"] \\
            L'\lar[rr,phan,"q"'] && L
            \cellsymb(\sigma){1-1}{2-3}
        \end{tikzcd}
    \end{equation*}
\end{remark}

\begin{definition}[Modulations of type 2]
    Let $F\colon\bK\to\bL$ be an \ac{AVD}-functor between \acp{AVDC}.
    Let $F\arr(l)[Rightarrow]L\in\bL$ and $F\arr(l')[Rightarrow]L'\in\bL$ be tight cocones.
    Consider $L\larr(q)[phan]L'$ in $\bL$.
    A \emph{modulation (of type 2)} $\tau$, denoted by
    \begin{equation}\label{eq:modulation_of_type2}
        \begin{tikzcd}[tri]
            & F\ar[dl,"l"',Rightarrow]\ar[dr,"l'",Rightarrow] & \\
            L\lar[rr,phan,"q"'] & {} & L'
            \cellsymb(\tau)[above=-5]{1-2}{2-2}
        \end{tikzcd}
    \end{equation}
    consists of:
    \begin{itemize}
        \item
            for each $A\in\bK$, a cell
            \begin{equation*}
                \begin{tikzcd}[tri]
                    & FA\ar[dl,"l_A"']\ar[dr,"l'_A"] & \\
                    L\lar[rr,phan,"q"'] & {} & L'
                    \cellsymb(\tau_A)[above=-5]{1-2}{2-2}
                \end{tikzcd}\incat{\bL}
            \end{equation*}
    \end{itemize}
    satisfying the following conditions:
    \begin{itemize}
        \item
            For any $A\arr(f)B$ in $\bK$,
            \begin{equation*}
                \begin{tikzcd}[tri]
                    & FA\ar[d,"Ff"{left},bend right=20]\ar[d,"Ff"{right},bend left=20] & \\
                    & FB\ar[dl,"l_B"']\ar[dr,"l'_B"] & \\
                    L\lar[rr,phan,"q"'] &{}& L'
                    \cellsymb(\heq){1-2}{2-2}
                    \cellsymb(\tau_B){2-2}{3-2}
                \end{tikzcd}
                =
                \begin{tikzcd}[tri]
                    & FA\ar[dl,"l_A"']\ar[dr,"l'_A"] & \\
                    L\lar[rr,phan,"q"'] &{}& L'
                    \cellsymb(\tau_A){1-2}{2-2}
                \end{tikzcd}\incat{\bL}.
            \end{equation*}
        \item
            For any $A\larr(u)B$ in $\bK$,
            \begin{equation*}
                \begin{tikzcd}[large]
                    FA\ar[d,"l_A"']\ar[dr,"l'_A"{left=0.1}]\lar[r,"Fu"] & FB\ar[d,"l'_B"] \\
                    L\lar[r,phan,"q"'] & L'
                    \cellsymb(l'_u)[above right=2]{1-1}{2-2}
                    \cellsymb(\tau_A)[below left=3]{1-1}{2-2}
                \end{tikzcd}
                =
                \begin{tikzcd}[large]
                    FA\ar[d,"l_A"']\lar[r,"Fu"] & FB\ar[dl,"l_B"{right=2}]\ar[d,"l'_B"] \\
                    L\lar[r,phan,"q"'] & L'
                    \cellsymb(l_u)[above left=2]{1-2}{2-1}
                    \cellsymb(\tau_B)[below right=4]{1-2}{2-1}
                \end{tikzcd}\incat{\bL}.
            \end{equation*}
    \end{itemize}
\end{definition}

\begin{notation}
    Let $\Cone\vvect{F}{L}$ denote the category of tight cocones from $F$ with a vertex $L$.
    A morphism $l\to l'$ in $\Cone\vvect{F}{L}$ is defined as a modulation of type 2 such that $q$ is of length 0 in \cref{eq:modulation_of_type2}.
\end{notation}

\begin{definition}[Modulations of type 3]
    Let $F\colon\bK\to\bL$ be an \ac{AVD}-functor between \acp{AVDC}.
    Let $N\larr(n)[Rightarrow]F\larr(m)[Rightarrow]M$ be a right $F$-module and a left $F$-module, respectively.
    Consider $N'\larr(\tup{q})[path]N$, $M\larr(\tup{p})[path]M'$, $N'\arr(j)N''$, $M'\arr(i)M''$, and $N''\larr(r)[phan]M''$ in $\bL$.
    A \emph{modulation (of type 3)} $\omega$, denoted by
    \begin{equation*}
        \begin{tikzcd}
            N'\ar[d,"j"']\lar[r,path,"\tup{q}"] & N\lar[r,Rightarrow,"n"] & F\lar[r,Rightarrow,"m"] & M\lar[r,path,"\tup{p}"] & M'\ar[d,"i"] \\
            N''\lar[rrrr,phan,"r"'] &&&& M''
            \cellsymb(\omega){1-1}{2-5}
        \end{tikzcd}
    \end{equation*}
    consists of:
    \begin{itemize}
        \item
            for each $A\in\bK$, a cell
            \begin{equation*}
                \begin{tikzcd}
                    N'\ar[d,"j"']\lar[r,path,"\tup{q}"] & N\lar[r,"n_A"] & FA\lar[r,"m_A"] & M\lar[r,path,"\tup{p}"] & M'\ar[d,"i"] \\
                    N''\lar[rrrr,phan,"r"'] &&&& M''
                    \cellsymb(\omega_A){1-1}{2-5}
                \end{tikzcd}
            \end{equation*}
    \end{itemize}
    satisfying the following conditions:
    \begin{itemize}
        \item
            For any $A\arr(f)B$ in $\bK$,
            \begin{equation*}
                \begin{tikzcd}
                    N'\ar[d,equal]\lar[r,path,"\tup{q}"] & N\ar[d,equal]\lar[r,"n_A"] & FA\ar[d,"Ff"{right=-2}]\lar[r,"m_A"] & M\ar[d,equal]\lar[r,path,"\tup{p}"] & M'\ar[d,equal] \\
                    N'\ar[d,"j"']\lar[r,path,"\tup{q}"] & N\lar[r,"n_B"] & FB\lar[r,"m_B"] & M\lar[r,path,"\tup{p}"] & M'\ar[d,"i"] \\
                    N''\lar[rrrr,phan,"r"'] &&&& M''
                    \cellsymb(\veq){1-1}{2-2}
                    \cellsymb(n_f){1-2}{2-3}
                    \cellsymb(m_f){1-3}{2-4}
                    \cellsymb(\veq){1-4}{2-5}
                    \cellsymb(\omega_B){2-1}{3-5}
                \end{tikzcd}\quad
                =\quad\omega_A \incat{\bL}.
            \end{equation*}
        \item
            For any $A\larr(u)B$ in $\bK$,
            \begin{equation*}
                \begin{tikzcd}
                    N'\ar[d,equal]\lar[r,path,"\tup{q}"] & N\ar[d,equal]\lar[r,"n_A"] & FA\ar[d,equal]\lar[r,"Fu"] & FB\lar[r,"m_B"] & M\ar[d,equal]\lar[r,path,"\tup{p}"] & M'\ar[d,equal] \\
                    N'\ar[d,"j"']\lar[r,path,"\tup{q}"] & N\lar[r,"n_A"] & FA\lar[rr,"m_A"] & & M\lar[r,path,"\tup{p}"] & M'\ar[d,"i"] \\
                    N''\lar[rrrrr,phan,"r"'] &&&&& M''
                    \cellsymb(\veq){1-1}{2-2}
                    \cellsymb(\veq){1-2}{2-3}
                    \cellsymb(m_u){1-3}{2-5}
                    \cellsymb(\veq){1-5}{2-6}
                    \cellsymb(\omega_A){2-1}{3-6}
                \end{tikzcd}
            \end{equation*}
            \begin{equation*}
                =
                \begin{tikzcd}
                    N'\ar[d,equal]\lar[r,path,"\tup{q}"] & N\ar[d,equal]\lar[r,"n_A"] & FA\lar[r,"Fu"] & FB\ar[d,equal]\lar[r,"m_B"] & M\ar[d,equal]\lar[r,path,"\tup{p}"] & M'\ar[d,equal] \\
                    N'\ar[d,"j"']\lar[r,path,"\tup{q}"] & N\lar[rr,"n_B"] & & FB\lar[r,"m_B"] & M\lar[r,path,"\tup{p}"] & M'\ar[d,"i"] \\
                    N''\lar[rrrrr,phan,"r"'] &&&&& M''
                    \cellsymb(\veq){1-1}{2-2}
                    \cellsymb(n_u){1-2}{2-4}
                    \cellsymb(\veq){1-4}{2-5}
                    \cellsymb(\veq){1-5}{2-6}
                    \cellsymb(\omega_B){2-1}{3-6}
                \end{tikzcd}\incat{\bL}.
            \end{equation*}
    \end{itemize}
\end{definition}

\begin{construction}\label{const:canonical_functor_tight_to_cocone}
    Let $F\colon\bK\to\bL$ be an \ac{AVD}-functor between \acp{AVDC} and let $L\in\bL$.
    Let $F\arr(\xi)[Rightarrow]\Xi\in\bL$ be a tight cocone.
    For a tight arrow $\Xi\arr(k)L$ in $\bL$, we have a tight cocone $F\arr(\xi\tcomp k)[Rightarrow]L$ as follows:
    \begin{itemize}
        \item
            For any $A\in\bK$,
            \begin{equation*}
                \begin{tikzcd}[tinytri]
                    & FA\ar[dl,"\xi_A"']\ar[dd,"(\xi\tcomp k)_A"] \\
                    \Xi\ar[dr,"k"'] & {} \\
                    & L
                    \cellsymb(\heq{\vcentcolon}){2-1}{2-2}
                \end{tikzcd}\incat{\bL}.
            \end{equation*}
        \item
            For any $A\larr(u)B$ in $\bK$,
            \begin{equation*}
                \begin{tikzcd}[tri]
                    FA\ar[dr,"\xi_A"']\lar[rr,"Fu"] & {} & FB\ar[dl,"\xi_B"] \\
                    & \Xi\ar[d,"k"{left},bend right=30]\ar[d,"k"{right},bend left=30] & \\
                    & L &
                    \cellsymb(\xi_u)[above=-3]{1-2}{2-2}
                    \cellsymb(\heq){2-2}{3-2}
                \end{tikzcd}
                \eqcolon
                \begin{tikzcd}[tri]
                    FA\ar[dr,"(\xi\tcomp k)_A"']\lar[rr,"Fu"] & {} & FB\ar[dl,"(\xi\tcomp k)_B"] \\
                    & L &
                    \cellsymb((\xi\tcomp k)_u)[above=-3]{1-2}{2-2}
                \end{tikzcd}\incat{\bL}.
            \end{equation*}
    \end{itemize}
    Furthermore, the assignment $k\mapsto\xi\tcomp k$ extends to a functor $\Homcat[\bL]\vvect{\Xi}{L}\arr(\xi\tcomp -)\Cone\vvect{F}{L}$.
\end{construction}

\begin{definition}\label{def:pulling}
    A tight arrow $A\arr(f)[][1]B$ in an \ac{AVDC} is called \emph{left-pulling} if every loose arrow $B\larr(p)\cdot$ has a restriction $p(f,\id)$ along $f$:
    \begin{equation*}
        \begin{tikzcd}
            A\ar[d,"f"']\lar[r,"{p(f,\id)}"] & \cdot\ar[d,equal] \\
            B\lar[r,"p"'] & \cdot
            \cellsymb(\cart){1-1}{2-2}
        \end{tikzcd}
    \end{equation*}
    Moreover, \emph{right-pulling} tight arrows are also defined in the loosewise dual way.
    Tight arrows that are left-pulling and right-pulling are simply called \emph{pulling}.
\end{definition}

\begin{lemma}\label{lem:pulling_and_restriction}
    Suppose that we are given the following data in an \ac{AVDC} such that $f$ is left-pulling and $g$ is right pulling.
    \begin{equation*}
        \begin{tikzcd}
            A\ar[d,"f"'] & B\ar[d,"g"] \\
            X\lar[r,"u"'] & Y
        \end{tikzcd}
    \end{equation*}
    Then, the 1-coary restriction $u(f,g)$ exists.
\end{lemma}
\begin{proof}
    Since cartesianness of cells is preserved under tightwise composition (\cref{prop:pasting_lemma_cartesian}), the desired restriction can be given as $u(f,\id)(\id,g)$, or equivalently $u(\id,g)(f,\id)$.
\end{proof}

\begin{construction}\label{const:canonical_functor_loose_to_mdl}
    Let $F\colon\bK\to\bL$ be an \ac{AVD}-functor between \acp{AVDC} and let $L\in\bL$.
    Let $\xi$ be a tight cocone from $F$ with a vertex $\Xi\in\bL$.
    Assume that $\xi_A$ is left-pulling for any $A\in\bK$.
    Then, depending on a choice of cartesian cells
    \begin{equation*}
        \begin{tikzcd}
            FA\ar[d,"\xi_A"']\lar[r,"{p(\xi_A,\id)}"] & L\ar[d,equal] \\
            \Xi\lar[r,"p"'] & L
            \cellsymb(\tilde{p}_A\colon\cart){1-1}{2-2}
        \end{tikzcd}\incat{\bL}
    \end{equation*}
    for each loose arrow $p$, the following assignments yield a functor $\Homcat[\bL](\Xi,L)\arr(\comp{\xi}-)\Mdl{F}{L}$ between categories.
    \begin{itemize}
        \item
            For each $\Xi\larr(p)L$ in $\bL$, a left $F$-module $\comp{\xi}p$ with the vertex $L$ is defined as follows:
            \begin{itemize}
                \item
                    For each $A\in\bK$, $(\comp{\xi}p)_A\coloneq p(\xi_A,\id)$.
                \item
                    For each $A\arr(f)B$ in $\bK$, $(\comp{\xi}p)_f$ is a unique cell such that
                    \begin{equation*}
                        \begin{tikzcd}
                            FA\ar[d,"Ff"']\lar[r,"(\comp{\xi}p)_A"] & L\ar[d,equal] \\
                            FB\ar[d,"\xi_B"{left=-2}]\lar[r,"(\comp{\xi}p)_B"] & L\ar[d,equal] \\
                            \Xi\lar[r,"p"'] & L
                            \cellsymb((\comp{\xi}p)_f){1-1}{2-2}
                            \cellsymb(\tilde{p}_B\colon\cart){2-1}{3-2}
                        \end{tikzcd}
                        =
                        \begin{tikzcd}
                            FA\ar[d,"\xi_A"']\lar[r,"(\comp{\xi}p)_A"] & L\ar[d,equal] \\
                            \Xi\lar[r,"p"'] & L
                            \cellsymb(\tilde{p}_A\colon\cart){1-1}{2-2}
                        \end{tikzcd}\incat{\bL}.
                    \end{equation*}
                \item
                    For each $A\larr(u)B$ in $\bK$, $(\comp{\xi}p)_u$ is a unique cell such that
                    \begin{equation*}
                        \begin{tikzcd}
                            FA\ar[d,equal]\lar[r,"Fu"] & FB\lar[r,"(\comp{\xi}p)_B"] & L\ar[d,equal] \\
                            FA\ar[d,"\xi_A"']\lar[rr,"(\comp{\xi}p)_A"] & & L\ar[d,equal] \\
                            \Xi\lar[rr,"p"'] & & L
                            \cellsymb((\comp{\xi}p)_u){1-1}{2-3}
                            \cellsymb(\tilde{p}_A\colon\cart){2-1}{3-3}
                        \end{tikzcd}
                        =
                        \begin{tikzcd}
                            FA\ar[d,"\xi_A"']\lar[r,"Fu"] & FB\ar[dl,"\xi_B"{below right=-2}]\lar[r,"(\comp{\xi}p)_B"] & L\ar[d,equal] \\
                            \Xi\lar[rr,"p"'] & & L
                            \cellsymb(\xi_u)[above left]{1-2}{2-1}
                            \cellsymb(\tilde{p}_B\colon\cart)[left=-5]{1-2}{2-3}
                        \end{tikzcd}\incat{\bL}.
                    \end{equation*}
            \end{itemize}
        \item
            For each cell
            \begin{equation*}
                \begin{tikzcd}
                    \Xi\ar[d,equal]\lar[r,"p"] & L\ar[d,equal] \\
                    \Xi\lar[r,"q"'] & L
                    \cellsymb(\delta){1-1}{2-2}
                \end{tikzcd}\incat{\bL},
            \end{equation*}
            a modulation $\comp{\xi}\delta\colon\comp{\xi}p\to\comp{\xi}q$ is defined as follows:
            \begin{itemize}
                \item
                    For each $A\in\bK$, $(\comp{\xi}\delta)_A$ is a unique cell such that
                    \begin{equation*}
                        \begin{tikzcd}
                            FA\ar[d,equal]\lar[r,"(\comp{\xi}p)_A"] & L\ar[d,equal] \\
                            FA\ar[d,"\xi_A"']\lar[r,"(\comp{\xi}q)_A"] & L\ar[d,equal] \\
                            \Xi\lar[r,"q"'] & L
                            \cellsymb((\comp{\xi}\delta)_A){1-1}{2-2}
                            \cellsymb(\tilde{q}_A\colon\cart){2-1}{3-2}
                        \end{tikzcd}
                        =
                        \begin{tikzcd}
                            FA\ar[d,"\xi_A"']\lar[r,"(\comp{\xi}p)_A"] & L\ar[d,equal] \\
                            \Xi\ar[d,equal]\lar[r,"p"] & L\ar[d,equal] \\
                            \Xi\lar[r,"q"'] & L
                            \cellsymb(\tilde{p}_A\colon\cart){1-1}{2-2}
                            \cellsymb(\delta){2-1}{3-2}
                        \end{tikzcd}\incat{\bL}.
                    \end{equation*}
            \end{itemize}
    \end{itemize}
\end{construction}

\begin{notation}\label{note:canonical_modulation_of_type1}
    In \cref{const:canonical_functor_loose_to_mdl}, the cartesian cells $(\tilde{p}_A)_{A\in\bK}$ yield a modulation of type 1 below.
    We write $\compcell{\xi}p$ for such a modulation.
    \begin{equation*}
        \begin{tikzcd}
            F\ar[d,"\xi"',Rightarrow]\lar[r,"\comp{\xi}p",Rightarrow] & L\ar[d,equal] \\
            \Xi\lar[r,"p"'] & L
            \cellsymb(\compcell{\xi}p){1-1}{2-2}
        \end{tikzcd}
    \end{equation*}
\end{notation}

\begin{remark}\label{rem:companion_of_tight_cocone}
    By an argument similar to \cref{const:canonical_functor_loose_to_mdl}, we can show that every tight cocone $F\arr(l)[Rightarrow]L$ induces a left $F$-module $F\larr(\comp{l})[Rightarrow]L$ whenever the companions $\comp{l_A}$ $(A\in\bK)$ exist.
\end{remark}

\begin{notation}
    In \cref{const:canonical_functor_loose_to_mdl}, if we alternatively assume that $\xi_A$ is right-pulling for any $A\in\bK$, then we can construct in the same way a functor $\Homcat[\bL](L,\Xi)\arr(-\conj{\xi})\Mdl{L}{F}$, which sends $q$ to a right $F$-module $q\conj{\xi}$.
    As well as \cref{note:canonical_modulation_of_type1}, we can get a modulation of type 1, denoted by $q\conjcell{\xi}$, of the following form:
    \begin{equation*}
        \begin{tikzcd}
            L\ar[d,equal]\lar[r,"q\conj{\xi}",Rightarrow] & F\ar[d,"\xi",Rightarrow] \\
            L\lar[r,"q"'] & \Xi
            \cellsymb(q\conjcell{\xi}){1-1}{2-2}
        \end{tikzcd}
    \end{equation*}
\end{notation}

\begin{remark}\label{rem:general_modulations}
    We have defined the modulations of the following types:
    \begin{equation*}
        \begin{tikzcd}
            \cdot\ar[d]\lar[r,path] & \cdot\lar[r,Rightarrow] & F\ar[d,equal] \\
            \cdot\lar[rr,Rightarrow] && F
            \cellsymb(\text{type 0}){1-1}{2-3}
        \end{tikzcd}
        \begin{tikzcd}
            F\ar[d,equal]\lar[r,Rightarrow] & \cdot\lar[r,path] & \cdot\ar[d] \\
            F\lar[rr,Rightarrow] && \cdot
            \cellsymb(\text{type 0}){1-1}{2-3}
        \end{tikzcd}
        \quad
        \begin{tikzcd}
            \cdot\ar[d]\lar[r,path] & \cdot\lar[r,Rightarrow] & F\ar[d,Rightarrow] \\
            \cdot\lar[rr,phan] && \cdot
            \cellsymb(\text{type 1}){1-1}{2-3}
        \end{tikzcd}
        \begin{tikzcd}
            F\ar[d,Rightarrow]\lar[r,Rightarrow] & \cdot\lar[r,path] & \cdot\ar[d] \\
            \cdot\lar[rr,phan] & & \cdot
            \cellsymb(\text{type 1}){1-1}{2-3}
        \end{tikzcd}
    \end{equation*}
    \begin{equation*}
        \begin{tikzcd}[scriptsizecolumn]
            & F\ar[dl,Rightarrow]\ar[dr,Rightarrow] & \\
            \cdot\lar[rr,phan] &{}& \cdot
            \cellsymb(\text{type 2})[below=-3]{1-2}{2-2}
        \end{tikzcd}
        \quad
        \begin{tikzcd}[scriptsizecolumn]
            \cdot\ar[d]\lar[r,path] & \cdot\lar[r,Rightarrow] & F\lar[r,Rightarrow] & \cdot\lar[r,path] & \cdot\ar[d] \\
            \cdot\lar[rrrr,phan] &&&& \cdot
            \cellsymb(\text{type 3}){1-1}{2-5}
        \end{tikzcd}
    \end{equation*}
    We may consider another type of ``modulation.''
    For example:
    \begin{equation*}
        \begin{tikzcd}[tri]
            & F\ar[dl,equal]\ar[dr,Rightarrow] & \\
            F\lar[rr,Rightarrow] &{}& \cdot
        \end{tikzcd}
    \end{equation*}
    However, for the purpose of defining the notion of \emph{versatile colimits} in the next subsection, no further types are required, and it would be difficult to give a fully general definition of modulations, for technical reasons described below.

    In order to consider general boundaries for modulations, one would first need to generalize the notion of vertices of cocones or modules from objects to \ac{AVD}-functors, which yields notions of \textit{multicocones} and \textit{bimodules}.
    In the author's view, multicocones should be defined by using a notion of ``tight \ac{AVD}-profunctors,'' while bimodules should be defined by using ``loose \ac{AVD}-profunctors.''
    Then, it seems most natural to define general modulations as ``3-cells'' in the ``(augmented) virtual triple category'' of \acp{AVDC}, \ac{AVD}-functors, tight \ac{AVD}-profunctors, and loose \ac{AVD}-profunctors.
    However, developing the theory of such a virtual triple category would be highly nontrivial and would exceed the scope of the paper.
\end{remark}
\subsection{Versatile colimits}
In this subsection, we fix an \ac{AVD}-functor $F\colon\bK\to\bL$ between \acp{AVDC} and a tight cocone $\xi$ from $F$ with a vertex $\Xi\in\bL$.
\begin{definition}
    We consider the following conditions for $\xi$:
    \begin{itemize}
        \labeleditem{(T)}\label{axiom:tight_bij}
            The canonical functor $\Homcat[\bL]\vvect{\Xi}{L}\arr(\xi\tcomp-)\Cone\vvect{F}{L}$ of \cref{const:canonical_functor_tight_to_cocone} is bijective on objects for any $L\in\bL$.
        \labeleditem{(L-l)}\label{axiom:loose_esssurj_left}
            $\xi_A$ is left-pulling for any $A\in\bK$, and the canonical functor $\Homcat[\bL](\Xi,L)\arr(\comp{\xi}-)\Mdl{F}{L}$ of \cref{const:canonical_functor_loose_to_mdl} is essentially surjective for any $L\in\bL$.
        \labeleditem{(L-r)}\label{axiom:loose_esssurj_right}
            The loosewise dual of \ref{axiom:loose_esssurj_left} holds.
        \labeleditem{(M0-l)}\label{axiom:cell_type0_left}
            $\xi_A$ is left-pulling for any $A\in\bK$, and the following holds.
            Take objects $M,M'\in\bL$ and $\Xi\larr(p)[][1]M,\Xi\larr(p')[][1]M'$ in $\bL$ arbitrarily.
            Then, for any modulation $\rho$ of type 0
            \begin{equation*}
                \begin{tikzcd}
                    F\ar[d,equal]\lar[r,"\comp{\xi}p",Rightarrow] & M\lar[r,path,"\tup{q}"] & M''\ar[d,"j"] \\
                    F\lar[rr,"\comp{\xi}p'"',Rightarrow] & & M',
                    \cellsymb(\rho){1-1}{2-3}
                \end{tikzcd}
            \end{equation*} 
            there exists a unique cell $\hat{\rho}$ such that for any $A\in\bK$,
            \begin{equation*}
                \begin{tikzcd}
                    FA\ar[d,equal]\lar[r,"(\comp{\xi}p)_A"] & M\lar[r,path,"\tup{q}"] & M''\ar[d,"j"] \\
                    FA\ar[d,"\xi_A"']\lar[rr,"(\comp{\xi}p')_A"'] & & M'\ar[d,equal] \\
                    \Xi\lar[rr,"p'"'] & & M'
                    \cellsymb(\rho_A){1-1}{2-3}
                    \cellsymb((\compcell{\xi}p')_A\colon\cart){2-1}{3-3}
                \end{tikzcd}
                =
                \begin{tikzcd}[hugecolumn]
                    FA\ar[d,"\xi_A"']\lar[r,"(\comp{\xi}p)_A"] & M\ar[d,equal]\lar[r,path,"\tup{q}"] & M''\ar[d,equal] \\
                    \Xi\ar[d,equal]\lar[r,"p"'] & M\lar[r,path,"\tup{q}"] & M''\ar[d,"j"] \\
                    \Xi\lar[rr,"p'"'] & & M'
                    \cellsymb((\compcell{\xi}p)_A\colon\cart){1-1}{2-2}
                    \cellsymb(\veq){1-2}{2-3}
                    \cellsymb(\hat{\rho}){2-1}{3-3}
                \end{tikzcd}\incat{\bL}.
            \end{equation*}
        \labeleditem{(M0-r)}\label{axiom:cell_type0_right}
            The loosewise dual of \ref{axiom:cell_type0_left} holds.
        \labeleditem{(M1-l)}\label{axiom:cell_type1_left}
            $\xi_A$ is left-pulling for any $A\in\bK$, and the following holds.
            Take objects $L,M\in\bL$ and $\Xi\arr(k)[][1]L,\Xi\larr(p)[][1]M$ in $\bL$ arbitrarily.
            Then, for any modulation $\sigma$ of type 1
            \begin{equation*}
                \begin{tikzcd}
                    F\ar[d,Rightarrow,"\xi\tcomp k"']\lar[r,Rightarrow,"\comp{\xi}p"] & M\lar[r,path,"\tup{q}"] & M'\ar[d,"j"] \\
                    L\lar[rr,phan,"r"'] & & L',
                    \cellsymb(\sigma){1-1}{2-3}
                \end{tikzcd}
            \end{equation*}
            there exists a unique cell $\hat{\sigma}$ such that for any $A\in\bK$,
            \begin{equation*}
                \begin{tikzcd}
                    FA\ar[d,"(\xi\tcomp k)_A"']\lar[r,"(\comp{\xi}p)_A"] & M\lar[r,path,"\tup{q}"] & M'\ar[d,"j"] \\
                    L\lar[rr,phan,"r"'] & & L'
                    \cellsymb(\sigma_A){1-1}{2-3}
                \end{tikzcd}
                =
                \begin{tikzcd}[hugecolumn]
                    FA\ar[d,"\xi_A"']\lar[r,"(\comp{\xi}p)_A"] & M\ar[d,equal]\lar[r,path,"\tup{q}"] & M'\ar[d,equal] \\
                    \Xi\ar[d,"k"']\lar[r,"p"'] & M\lar[r,path,"\tup{q}"'] & M'\ar[d,"j"] \\
                    L\lar[rr,phan,"r"'] & & L'
                    \cellsymb((\compcell{\xi}p)_A\colon\cart){1-1}{2-2}
                    \cellsymb(\veq){1-2}{2-3}
                    \cellsymb(\hat{\sigma}){2-1}{3-3}
                \end{tikzcd}\incat{\bL}.
            \end{equation*}
        \labeleditem{(M1-r)}\label{axiom:cell_type1_right}
            The loosewise dual of \ref{axiom:cell_type1_left} holds.
        \labeleditem{(M2)}\label{axiom:cell_type2}
            Take $L,L'\in\bL$ and $\Xi\arr(k)L,\Xi\arr(k')L'$ in $\bL$ arbitrarily.
            Then, for any modulation $\tau$ of type 2
            \begin{equation*}
                \begin{tikzcd}[tri]
                    & F\ar[dl,"\xi\tcomp k"',Rightarrow]\ar[dr,"\xi\tcomp k'",Rightarrow] & \\
                    L\lar[rr,phan,"q"'] & {} & L',
                    \cellsymb(\tau)[above=-5]{1-2}{2-2}
                \end{tikzcd}
            \end{equation*}
            there exists a unique cell $\hat{\tau}$ such that for any $A\in\bK$,
            \begin{equation*}
                \begin{tikzcd}[tri]
                    & FA\ar[dl,"(\xi\tcomp k)_A"']\ar[dr,"(\xi\tcomp k')_A"] & \\
                    L\lar[rr,phan,"q"'] & {} & L'
                    \cellsymb(\tau_A)[above=-5]{1-2}{2-2}
                \end{tikzcd}
                =
                \begin{tikzcd}[tri]
                    & FA\ar[d,"\xi_A"{left},bend right=20]\ar[d,"\xi_A"{right},bend left=20] & \\
                    & \Xi\ar[dl,"k"']\ar[dr,"k'"] & \\
                    L\lar[rr,phan,"q"'] & {} & L'
                    \cellsymb(\heq){1-2}{2-2}
                    \cellsymb(\hat{\tau})[above=-5]{2-2}{3-2}
                \end{tikzcd}\incat{\bL}.
            \end{equation*}
        \labeleditem{(M3)}\label{axiom:cell_type3}
            $\xi_A$ is pulling for any $A\in\bK$, and the following holds.
            Take $N,M\in\bL$ and\linebreak $N\larr(t)\Xi\larr(s)M$ in $\bL$ arbitrarily.
            Then, for any modulation $\omega$ of type 3
            \begin{equation*}
                \begin{tikzcd}
                    N'\ar[d,"j"']\lar[r,path,"\tup{q}"] & N\lar[r,Rightarrow,"t\conj{\xi}"] & F\lar[r,Rightarrow,"\comp{\xi}s"] & M\lar[r,path,"\tup{p}"] & M'\ar[d,"i"] \\
                    N''\lar[rrrr,phan,"r"'] &&&& M'',
                    \cellsymb(\omega){1-1}{2-5}
                \end{tikzcd}
            \end{equation*}
            there exists a unique cell $\hat{\omega}$ such that for any $A\in\bK$,
            \begin{equation*}
                \omega_A
                =
                \begin{tikzcd}[hugecolumn]
                    N'\ar[d,equal]\lar[r,path,"\tup{q}"] & N\ar[d,equal]\lar[r,"(t\conj{\xi})_A"] &[5pt] FA\ar[d,"\xi_A"{description}]\lar[r,"(\comp{\xi}s)_A"] &[5pt] M\ar[d,equal]\lar[r,path,"\tup{p}"] & M'\ar[d,equal] \\
                    N'\ar[d,"j"']\lar[r,path,"\tup{q}"] & N\lar[r,"t"] & \Xi\lar[r,"s"] & M\lar[r,path,"\tup{p}"] & M'\ar[d,"i"] \\
                    N''\lar[rrrr,phan,"r"'] &&&& M''
                    \cellsymb(\veq){1-1}{2-2}
                    \cellsymb((t\conjcell{\xi})_A\colon\cart){1-2}{2-3}
                    \cellsymb((\compcell{\xi}s)_A\colon\cart){1-3}{2-4}
                    \cellsymb(\veq){1-4}{2-5}
                    \cellsymb(\hat{\omega}){2-1}{3-5}
                \end{tikzcd}\incat{\bL}.
            \end{equation*}
    \end{itemize}
\end{definition}

\begin{remark}\label{rem:variation_axiom_hor_esssurj}
    The above conditions are independent of the construction of the functors $\comp{\xi}-$ and $-\conj{\xi}$.
    In particular, the condition \ref{axiom:loose_esssurj_left} can be rephrased as follows:
    \begin{itemize}[topsep=6pt]
        \labeleditem{(L-l)'}\label{axiom:loose_esssurj_left-dash}
            $\xi_A$ is left-pulling for any $A\in\bK$.
            Furthermore, for any left $F$-module $m\colon F\larr[Rightarrow][1]L$, there exist a loose arrow $\Xi\larr(p)[][1]L$ in $\bL$ and a modulation $\sigma$ of type 1
            \begin{equation*}
                \begin{tikzcd}
                    F\ar[d,Rightarrow,"\xi"']\lar[r,Rightarrow,"m"] & L\ar[d,equal] \\
                    \Xi\lar[r,"p"'] & L
                    \cellsymb(\sigma){1-1}{2-2}
                \end{tikzcd}
            \end{equation*}
            such that every component $\sigma_A$ $(A\in\bK)$ is cartesian.\qedhere
    \end{itemize}
\end{remark}

\begin{proposition}\label{prop:fully_faithfulness_from_M2_M0}\quad
    \begin{enumerate}
        \item
            \ref{axiom:cell_type2} implies that the functor $\Homcat[\bL]\vvect{\Xi}{L}\arr(\xi\tcomp-)\Cone\vvect{F}{L}$ is fully faithful for any $L\in\bL$.
        \item\label{prop:fully_faithfulness_from_M2_M0-module}
            \ref{axiom:cell_type0_left} implies that the functor $\Homcat[\bL](\Xi,L)\arr(\comp{\xi}-)\Mdl{F}{L}$ is fully faithful for any $L\in\bL$.
    \end{enumerate}
\end{proposition}
\begin{proof}
    This follows from the fact that morphisms between tight cocones or modules are a special case of modulations of type 2 or 0. 
\end{proof}

\begin{proposition}\label{prop:M1_and_M0}\quad
    \begin{enumerate}
        \item
            \ref{axiom:cell_type1_left} implies \ref{axiom:cell_type0_left}.
        \item
            If $\bL$ has loose units and every tight arrow is left-pulling in $\bL$, then \ref{axiom:cell_type1_left} and \ref{axiom:cell_type0_left} are equivalent.
    \end{enumerate}
\end{proposition}
\begin{proof}\quad
    \begin{enumerate}
        \item
            We have bijective correspondences among the cells and the modulations of the following forms:
            \begin{equation*}
                \begin{tikzcd}
                    F\ar[d,equal]\lar[r,Rightarrow,"\comp{\xi}p"] & M\lar[r,path,"\tup{q}"] & M''\ar[d,"j"] \\
                    F\lar[rr,Rightarrow,"\comp{\xi}p'"'] && M'
                    \cellsymb(\cdot){1-1}{2-3}
                \end{tikzcd}
                \Vline
                \begin{tikzcd}
                    F\ar[d,Rightarrow,"\xi"']\lar[r,Rightarrow,"\comp{\xi}p"] & M\lar[r,path,"\tup{q}"] & M''\ar[d,"j"] \\
                    \Xi\lar[rr,"p'"'] && M'
                    \cellsymb(\cdot){1-1}{2-3}
                \end{tikzcd}
                \Vline
                \begin{tikzcd}
                    \Xi\ar[d,equal]\lar[r,"p"] & M\lar[r,path,"\tup{q}"] & M''\ar[d,"j"] \\
                    \Xi\lar[rr,"p'"'] && M'
                    \cellsymb(\cdot){1-1}{2-3}
                \end{tikzcd}
            \end{equation*}
            The first correspondence is given by tightwise composition with the modulation $\compcell{\xi}p'$ as in \cref{note:canonical_modulation_of_type1} whose components are cartesian.
            The second one is given by \ref{axiom:cell_type1_left}.
            Hence \ref{axiom:cell_type0_left} follows.
        \item
            Suppose \ref{axiom:cell_type0_left} and that $\bL$ has loose units and every tight arrow is left-pulling in $\bL$.
            Then, we have bijective correspondences among the cells and the modulations of the following forms:
            \begin{equation*}
                \begin{tikzcd}[scriptsize]
                    F\ar[d,Rightarrow,"\xi"']\lar[r,Rightarrow,"\comp{\xi}p"] & M\lar[r,path,"\tup{q}"] & M'\ar[dd,"j"] \\
                    \Xi\ar[d,"k"'] && \\
                    L\lar[rr,phan,"r"'] && L'
                    \cellsymb(\cdot){1-1}{3-3}
                \end{tikzcd}
                \Vline
                \begin{tikzcd}
                    F\ar[d,Rightarrow,"\xi"']\lar[r,Rightarrow,"\comp{\xi}p"] & M\lar[r,path,"\tup{q}"] & M'\ar[d,"j"] \\
                    \Xi\lar[rr,"{r(k,\id)}"'] && L'
                    \cellsymb(\cdot){1-1}{2-3}
                \end{tikzcd}
                \Vline
                \begin{tikzcd}
                    F\ar[d,equal]\lar[r,Rightarrow,"\comp{\xi}p"] & M\lar[r,path,"\tup{q}"] & M'\ar[d,"j"] \\
                    F\lar[rr,Rightarrow,"{\comp{\xi}r(k,\id)}"'] && L'
                    \cellsymb(\cdot){1-1}{2-3}
                \end{tikzcd}
            \end{equation*}
            \begin{equation*}
                \Vline
                \begin{tikzcd}
                    \Xi\ar[d,equal]\lar[r,"p"] & M\lar[r,path,"\tup{q}"] & M'\ar[d,"j"] \\
                    \Xi\lar[rr,"{r(k,\id)}"'] && L'
                    \cellsymb(\cdot){1-1}{2-3}
                \end{tikzcd}
                \Vline
                \begin{tikzcd}
                    \Xi\ar[d,"k"']\lar[r,"p"] & M\lar[r,path,"\tup{q}"] & M'\ar[d,"j"] \\
                    L\lar[rr,phan,"r"'] && L'
                    \cellsymb(\cdot){1-1}{2-3}
                \end{tikzcd}
            \end{equation*}
            Here, the restriction $r(k,\id)$ exists by the assumption, whose universal property implies the first and the last correspondences.
            The second correspondence follows from the fact that every components of the modulation $\compcell{\xi}r(k,\id)$ is cartesian, and the third one follows from \ref{axiom:cell_type0_left}.
            This shows \ref{axiom:cell_type1_left}.\qedhere
    \end{enumerate}
\end{proof}

\begin{proposition}\label{prop:M3_to_M1_to_M2}\quad
    \begin{enumerate}
        \item
            If $\Xi$ has a loose unit, then \ref{axiom:cell_type1_left} implies \ref{axiom:cell_type2}.
        \item
            If $\Xi$ has a loose unit, then \ref{axiom:cell_type3} implies \ref{axiom:cell_type1_left}.
    \end{enumerate}
\end{proposition}
\begin{proof}\quad
    \begin{enumerate}
        \item
            Suppose \ref{axiom:cell_type1_left} and that the loose unit $\Unit_\Xi$ on $\Xi$ exists.
            Since every $\xi_A$ is left-pulling, it has a companion, which is given by the restriction $\Unit_\Xi(\xi_A,\id)$.
            Consider the canonical cells associated with the companion $\comp{\xi_A}$:
            \begin{equation}\label{eq:canonial_cells_companion_xi_A}
                \begin{tikzcd}
                    FA\ar[d,"\xi_A"']\lar[r,"\comp{\xi_A}"] & \Xi\ar[dl,equal] \\
                    \Xi &
                    \cellsymb(\cdot)[above left=3]{1-2}{2-1}
                \end{tikzcd}
                \begin{tikzcd}
                    & FA\ar[dl,equal]\ar[d,"\xi_A"] \\
                    FA\lar[r,"\comp{\xi_A}"'] & \Xi
                    \cellsymb(\cdot)[below right=3]{1-2}{2-1}
                \end{tikzcd}\incat{\bL}\quad (A\in\bK).
            \end{equation}
            Let $\comp{\xi}$ denote the left $F$-module given by the companions $\comp{\xi_A}$.
            Then, we have bijective correspondences among the cells and the modulations of the following forms:
            \begin{equation*}
                \begin{tikzcd}[tri]
                    & F\ar[dl,Rightarrow,"\xi"']\ar[dr,Rightarrow,"\xi"] & \\
                    \Xi\ar[d,"k"'] &{}& \Xi\ar[d,"k'"] \\
                    L\lar[rr,phan,"q"'] & & L'
                    \cellsymb(\cdot)[below]{2-1}{2-3}
                \end{tikzcd}
                \Vline
                \begin{tikzcd}
                    F\ar[d,Rightarrow,"\xi"']\lar[r,Rightarrow,"\comp{\xi}"] & \Xi\ar[dd,"k'"] \\
                    \Xi\ar[d,"k"'] & \\
                    L\lar[r,phan,"q"'] & L'
                    \cellsymb(\cdot){1-1}{3-2}
                \end{tikzcd}
                \Vline
                \begin{tikzcd}
                    \Xi\ar[d,"k"']\lar[r,equal] & \Xi\ar[d,"k'"] \\
                    L\lar[r,phan,"q"'] & L'
                    \cellsymb(\cdot){1-1}{2-2}
                \end{tikzcd}
                \Vline
                \begin{tikzcd}[tri]
                    & \Xi\ar[dl,"k"']\ar[dr,"k'"] & \\
                    L\lar[rr,phan,"q"'] &{}& L'
                    \cellsymb(\cdot){1-2}{2-2}
                \end{tikzcd}
            \end{equation*}
            Here, the first correspondence is given by component-wise pasting with the cells \cref{eq:canonial_cells_companion_xi_A}.
            The second one is given by \ref{axiom:cell_type1_left}.
            The third one is given by the universal property of the loose unit.
            Therefore \ref{axiom:cell_type2} follows.
        \item
            Suppose \ref{axiom:cell_type3} and that the loose unit on $\Xi$ exists.
            Since every $\xi_A$ is, in particular, right-pulling, it has a conjoint.
            Then, we have bijective correspondences among the cells and the modulations of the following forms:
            \begin{equation*}
                \begin{tikzcd}[scriptsize]
                    F\ar[d,Rightarrow,"\xi"']\lar[r,Rightarrow,"\comp{\xi}p"] & M\lar[r,path,"\tup{q}"] & M'\ar[dd,"j"] \\
                    \Xi\ar[d,"k"'] & & \\
                    L\lar[rr,phan,"r"'] & & L'
                    \cellsymb(\cdot){1-1}{3-3}
                \end{tikzcd}
                \Vline
                \begin{tikzcd}
                    \Xi\ar[d,"k"']\lar[r,Rightarrow,"\conj{\xi}"] & F\lar[r,Rightarrow,"\comp{\xi}p"] & M\lar[r,path,"\tup{q}"] & M'\ar[d,"j"] \\
                    L\lar[rrr,phan,"r"'] &&& L'
                    \cellsymb(\cdot){1-1}{2-4}
                \end{tikzcd}
            \end{equation*}
            \begin{equation*}
                \Vline
                \begin{tikzcd}
                    \Xi\ar[d,"k"']\lar[r,equal] & \Xi\lar[r,"p"] & M\lar[r,path,"\tup{q}"] & M'\ar[d,"j"] \\
                    L\lar[rrr,phan,"r"'] &&& L'
                    \cellsymb(\cdot){1-1}{2-4}
                \end{tikzcd}
                \Vline
                \begin{tikzcd}[scriptsize]
                    \Xi\ar[d,"k"']\lar[r,"p"] & M\lar[r,path,"\tup{q}"] & M'\ar[d,"j"] \\
                    L\lar[rr,phan,"r"'] & & L'
                    \cellsymb(\cdot){1-1}{2-3}
                \end{tikzcd}
            \end{equation*}
            The first correspondence is given by component-wise pasting with the canonical cells associated with the conjoints $\conj{\xi_A}$.
            The second one is given by \ref{axiom:cell_type3}.
            The third one is given by the universal property of the loose unit.
            Therefore \ref{axiom:cell_type1_left} follows.\qedhere
    \end{enumerate}
\end{proof}

\begin{proposition}\label{prop:M1_to_M3}
    Suppose that $\bL$ has extensions, $\xi_A$ is right-pulling for every $A\in\bK$.
    Then, \ref{axiom:cell_type1_left} implies \ref{axiom:cell_type3}.
\end{proposition}
\begin{proof}
    Suppose \ref{axiom:cell_type1_left} and that $\bL$ has extensions and $\xi_A$ are right-pulling.
    By \cref{prop:consequence_from_ext_lift}, $\bL$ has companions (hence loose units).
    Thus, every $\xi_A$ has a conjoint.
    Then, we have bijective correspondences among the cells and the modulations of the following forms:
    \begin{equation*}
        \begin{tikzcd}[scriptsizecolumn]
            N'\ar[d,"j"']\lar[r,path,"\tup{q}"] & N\lar[r,Rightarrow,"t\conj{\xi}"] & F\lar[r,Rightarrow,"\comp{\xi}s"] & M\lar[r,path,"\tup{p}"] & M'\ar[d,"i"] \\
            N''\lar[rrrr,phan,"r"'] &&&& M''
            \cellsymb(\cdot){1-1}{2-5}
        \end{tikzcd}
        \Vline
        \begin{tikzcd}[scriptsizecolumn]
            N'\ar[d,"j"']\lar[r,path,"\tup{q}"] & N\lar[r,"t"] & \Xi\lar[r,Rightarrow,"\conj{\xi}"] & F\lar[r,Rightarrow,"\comp{\xi}s"] & M\lar[r,path,"\tup{p}"] & M'\ar[d,"i"] \\
            N''\lar[rrrrr,phan,"r"'] &&&&& M''
            \cellsymb(\cdot){1-1}{2-6}
        \end{tikzcd}
    \end{equation*}
    \begin{equation*}
        \Vline
        \begin{tikzcd}
            \Xi\ar[d,equal]\lar[r,Rightarrow,"\conj{\xi}"] & F\lar[r,Rightarrow,"\comp{\xi}s"] & M\lar[r,path,"\tup{p}"] & M'\ar[d,"i"] \\
            \Xi\lar[rrr,"{\extension{(\tup{q},t)}[j]{r}}"'] &&& M''
            \cellsymb(\cdot){1-1}{2-4}
        \end{tikzcd}
        \Vline
        \begin{tikzcd}
            F\ar[d,Rightarrow,"\xi"']\lar[r,Rightarrow,"\comp{\xi}s"] & M\lar[r,path,"\tup{p}"] & M'\ar[d,"i"] \\
            \Xi\lar[rr,"{\extension{(\tup{q},t)}[j]{r}}"'] && M''
            \cellsymb(\cdot){1-1}{2-3}
        \end{tikzcd}
    \end{equation*}
    \begin{equation*}
        \Vline
        \begin{tikzcd}
            \Xi\ar[d,equal]\lar[r,"s"] & M\lar[r,path,"\tup{p}"] & M'\ar[d,"i"] \\
            \Xi\lar[rr,"{\extension{(\tup{q},t)}[j]{r}}"'] && M''
            \cellsymb(\cdot){1-1}{2-3}
        \end{tikzcd}
        \Vline
        \begin{tikzcd}[scriptsizecolumn]
            N'\ar[d,"j"']\lar[r,path,"\tup{q}"] & N\lar[r,"t"] & \Xi\lar[r,"s"] & M\lar[r,path,"\tup{p}"] & M'\ar[d,"i"] \\
            N''\lar[rrrr,phan,"r"'] &&&& M''
            \cellsymb(\cdot){1-1}{2-5}
        \end{tikzcd}
    \end{equation*}
    At the level of components, the first correspondence follows from the canonical cocartesian cells as in \cref{prop:special_sandwich}.
    The second and the last ones follow from the universal property of the extension.
    The third one is given by component-wise pasting with the canonical cells associated with the conjoints $\conj{\xi_A}$.
    The fourth one is given by \ref{axiom:cell_type1_left}.
    Therefore \ref{axiom:cell_type3} follows.
\end{proof}

\begin{definition}[Versatile colimits]\label{def:versatile_colim}
    $\xi$ is called a \emph{versatile colimit} of $F$ if it satisfies the conditions \ref{axiom:tight_bij}\ref{axiom:loose_esssurj_left}\ref{axiom:loose_esssurj_right}\ref{axiom:cell_type1_left}\ref{axiom:cell_type1_right}\ref{axiom:cell_type2}\ref{axiom:cell_type3}.
\end{definition}

\begin{remark}
    Whenever we consider pseudo double categories, our notion of versatile colimits can be viewed as Grandis--Par\'{e}'s notion of \textit{(strict) double colimits} (\cite[5.2.3]{Grandis2020higher} or \cite[4.2]{GrandisPare1999limits}) equipped with extra conditions.
    Indeed, in the situation of \cref{rem:cocone_grandis_pare}, $\xi$ becomes a (strict) double colimit if and only if it satisfies the conditions \ref{axiom:tight_bij} and \ref{axiom:cell_type2}.
\end{remark}

\begin{corollary}\label{cor:simplification_to_M3}
    Suppose that $\Xi$ has a loose unit.
    Then, $\xi$ becomes a versatile colimit if and only if it satisfies \ref{axiom:tight_bij}\ref{axiom:loose_esssurj_left}\ref{axiom:loose_esssurj_right}\ref{axiom:cell_type3}.
\end{corollary}
\begin{proof}
    This follows from \cref{prop:M3_to_M1_to_M2}.
\end{proof}

\begin{corollary}\label{cor:simplification_to_M0}
    Suppose that $\bL$ has extensions and conjoints.
    Then, $\xi$ becomes a versatile colimit if and only if it satisfies \ref{axiom:tight_bij}\ref{axiom:loose_esssurj_left}\ref{axiom:loose_esssurj_right}\ref{axiom:cell_type0_left}.
\end{corollary}
\begin{proof}
    Since the existence of either extensions or conjoints implies that $\bL$ has loose units, we can apply \cref{cor:simplification_to_M3}.
    By \cref{prop:consequence_from_ext_lift}\cref{prop:consequence_from_ext_lift-2}, every tight arrow is left-pulling in $\bL$.
    Then, since \ref{axiom:loose_esssurj_right} includes the right-pullingness of $\xi_A$, \cref{prop:M1_and_M0,prop:M1_to_M3} show that \ref{axiom:cell_type3} follows from \ref{axiom:cell_type0_left}, which finishes the proof.
\end{proof}

\begin{theorem}\label{thm:strongest_simplification}
    Suppose that $\bL$ has extensions, lifts, and loose composites.
    Then, $\xi$ becomes a versatile colimit if and only if it satisfies the following conditions:
    \begin{itemize}
        \item
            The functor $\Homcat[\bL]\vvect{\Xi}{L}\arr(\xi\tcomp-)\Cone\vvect{F}{L}$ is an isomorphism of categories for any $L\in\bL$;
        \item
            The functors $\Homcat[\bL](\Xi,L)\arr(\comp{\xi}-)\Mdl{F}{L}$ and $\Homcat[\bL](L,\Xi)\arr(-\conj{\xi})\Mdl{L}{F}$ are equivalences of categories for any $L\in\bL$.
    \end{itemize}
\end{theorem}
\begin{proof}
    Let $\xi$ satisfy the conditions above.
    We will show that $\xi$ is a versatile colimit.
    By \cref{cor:simplification_to_M0}, it suffices to show \ref{axiom:cell_type0_left}.
    Now, we have bijective correspondences among the cells and the modulations of the following forms.
    In what follows, $w$ denotes the loose composite of $\Xi\larr(p)[][1]M\larr(\tup{q})[path]M''\larr(\comp{j})[][1]M'$.
    \begin{equation*}
        \begin{tikzcd}
            F\ar[d,equal]\lar[r,Rightarrow,"\comp{\xi}p"] & M\lar[r,path,"\tup{q}"] & M''\ar[d,"j"] \\
            F\lar[rr,Rightarrow,"\comp{\xi}p'"'] && M'
            \cellsymb(\cdot){1-1}{2-3}
        \end{tikzcd}
        \Vline
        \begin{tikzcd}[scriptsizecolumn]
            F\ar[d,equal]\lar[r,Rightarrow,"\comp{\xi}p"] & M\lar[r,path,"\tup{q}"] & M''\lar[r,"\comp{j}"] & M'\ar[d,equal] \\
            F\lar[rrr,Rightarrow,"\comp{\xi}p'"'] &&& M'
            \cellsymb(\cdot){1-1}{2-4}
        \end{tikzcd}
        \Vline
        \begin{tikzcd}
            F\ar[d,equal]\lar[r,Rightarrow,"\comp{\xi}w"] & M'\ar[d,equal] \\
            F\lar[r,Rightarrow,"\comp{\xi}p'"'] & M'
            \cellsymb(\cdot){1-1}{2-2}
        \end{tikzcd}
    \end{equation*}
    \begin{equation*}
        \Vline
        \begin{tikzcd}
            \Xi\ar[d,equal]\lar[r,"w"] & M'\ar[d,equal] \\
            \Xi\lar[r,"p'"'] & M'
            \cellsymb(\cdot){1-1}{2-2}
        \end{tikzcd}
        \Vline
        \begin{tikzcd}[scriptsizecolumn]
            \Xi\ar[d,equal]\lar[r,"p"] & M\lar[r,path,"\tup{q}"] & M''\lar[r,"\comp{j}"] & M'\ar[d,equal] \\
            \Xi\lar[rrr,"p'"'] &&& M'
            \cellsymb(\cdot){1-1}{2-4}
        \end{tikzcd}
        \Vline
        \begin{tikzcd}
            \Xi\ar[d,equal]\lar[r,"p"] & M\lar[r,path,"\tup{q}"] & M''\ar[d,"j"] \\
            \Xi\lar[rr,"p'"'] && M'
            \cellsymb(\cdot){1-1}{2-3}
        \end{tikzcd}
    \end{equation*}
    Here, the first and the last correspondences are given by component-wise pasting with the canonical cells associated with the companion $\comp{j}$.
    The second one follows from the universal property of the loose composite $\lcomp((\comp{\xi}p)_A,\tup{q},\comp{j})$ and from the fact that $(\comp{\xi}w)_A=w(\xi_A,\id)\cong\lcomp(p(\xi_A,\id),\tup{q},\comp{j})=\lcomp((\comp{\xi}p)_A,\tup{q},\comp{j})$ (see \cite[9.8.\ Lemma]{Koudenburg2020aug}).
    The third one follows from the full faithfulness of the functor $\comp{\xi}-$.
    The fourth one follows from the universal property of the loose composite $w$.
    Therefore \ref{axiom:cell_type0_left} follows.
\end{proof}

\begin{remark}
    A variant of \cref{thm:strongest_simplification} can be obtained by weakening the isomorphisms of categories to equivalences (\cref{thm:special_case_versbicolim}).
    Moreover, the variant enables us to compare our versatile colimits with the notion of colimits that Wood studied in \cite{Wood1985proarrowII}.
    The precise statement can be found in \cref{thm:proarr_equip_in_terms_of_versbicolim}.
\end{remark}

\begin{theorem}[Unitality theorem]\label{thm:unitality_theorem}
    Suppose \ref{axiom:loose_esssurj_left}\ref{axiom:cell_type1_left}\ref{axiom:cell_type2} and that $\xi_A$ has a companion for every $A\in\bK$.
    Then, $\Xi$ has a loose unit.
\end{theorem}
\begin{proof}
    Let $\comp{\xi}$ denote the left $F$-module given by the companions $\comp{\xi_A}$.
    Then, the canonical cartesian cells $\compcell{\xi_A}$ on the right below form a modulation $\compcell{\xi}$ of type 1 on the left below:
    \begin{equation*}
        \begin{tikzcd}
            F\ar[d,Rightarrow,"\xi"']\lar[r,Rightarrow,"\comp{\xi}"] & \Xi\ar[dl,equal] \\
            \Xi
            \cellsymb(\compcell{\xi})[above left]{1-2}{2-1}
        \end{tikzcd}
        \Vline
        \begin{tikzcd}
            FA\ar[d,"\xi_A"']\lar[r,"\comp{\xi_A}"] & \Xi\ar[dl,equal] \\
            \Xi
            \cellsymb(\compcell{\xi_A})[above left=-2]{1-2}{2-1}
        \end{tikzcd}
        \colon\cart\incat{\bL}\quad (A\in\bK)
    \end{equation*}
    By \ref{axiom:loose_esssurj_left-dash}, we have a loose arrow $\Xi\larr(u)\Xi$ in $\bL$ and a modulation $\compcell{\xi}u$ of type 1 whose components are cartesian:
    \begin{equation*}
        \begin{tikzcd}
            F\ar[d,Rightarrow,"\xi"']\lar[r,Rightarrow,"\comp{\xi}"] & \Xi\ar[d,equal] \\
            \Xi\lar[r,"u"'] & \Xi
            \cellsymb(\compcell{\xi}u){1-1}{2-2}
        \end{tikzcd}
        \Vline
        \begin{tikzcd}
            FA\ar[d,"\xi_A"']\lar[r,"\comp{\xi_A}"] & \Xi\ar[d,equal] \\
            \Xi\lar[r,"u"'] & \Xi
            \cellsymb(\cart){1-1}{2-2}
        \end{tikzcd}\incat{\bL}\quad (A\in\bK)
    \end{equation*}
    By \ref{axiom:cell_type1_left}, there is a unique cell $\epsilon$ corresponding to the modulation $\compcell{\xi}$.
    The cell $\epsilon$ is uniquely determined by the following equations:
    \begin{equation*}
        \begin{tikzcd}[largecolumn]
            FA\ar[d,"\xi_A"']\lar[r,"\comp{\xi_A}"] & \Xi\ar[d,equal] \\
            \Xi\ar[d,equal]\lar[r,"u"] & \Xi\ar[dl,equal] \\
            \Xi &
            \cellsymb((\compcell{\xi}u)_A){1-1}{2-2}
            \cellsymb(\epsilon)[above left=1]{2-2}{3-1}
        \end{tikzcd}
        =
        \begin{tikzcd}
            FA\ar[d,"\xi_A"']\lar[r,"\comp{\xi_A}"] & \Xi\ar[dl,equal] \\
            \Xi
            \cellsymb(\compcell{\xi_A})[above left=-2]{1-2}{2-1}
        \end{tikzcd}\incat{\bL}\quad (A\in\bK).
    \end{equation*}
    Let us consider a modulation $\tau$ of type 2 given by the following:
    \begin{equation*}
        \begin{tikzcd}[tri]
            & F\ar[dl,Rightarrow,"\xi"']\ar[dr,Rightarrow,"\xi"] & \\
            \Xi\lar[rr,"u"'] &{}& \Xi
            \cellsymb(\tau){1-2}{2-2}
        \end{tikzcd}
        \Vline
        \begin{tikzcd}[tri]
            & FA\ar[dl,equal]\ar[dr,"\xi_A"] & \\
            FA\ar[d,"\xi_A"']\lar[rr,"\comp{\xi_A}"'] &{}& \Xi\ar[d,equal] \\
            \Xi\lar[rr,"u"'] & & \Xi
            \cellsymb(\delta_A){1-2}{2-2}
            \cellsymb((\compcell{\xi}u)_A){2-1}{3-3}
        \end{tikzcd}\incat{\bL}\quad (A\in\bK),
    \end{equation*}
    where $\delta_A$ denotes the canonical cell associated with the companion $\comp{\xi_A}$.
    By \ref{axiom:cell_type2}, there is a unique cell $\eta$ corresponding to $\tau$.
    The cell $\eta$ is uniquely determined by the following equations:
    \begin{equation*}
        \begin{tikzcd}[tri]
            & FA\ar[d,"\xi_A"{left},bend right=20]\ar[d,"\xi_A"{right},bend left=20] & \\
            & \Xi\ar[dl,equal]\ar[dr,equal] & \\
            \Xi\lar[rr,"u"'] &{}& \Xi
            \cellsymb(\heq){1-2}{2-2}
            \cellsymb(\eta){2-2}{3-2}
        \end{tikzcd}
        =
        \begin{tikzcd}[tri]
            & FA\ar[dl,equal]\ar[dr,"\xi_A"] & \\
            FA\ar[d,"\xi_A"']\lar[rr,"\comp{\xi_A}"'] &{}& \Xi\ar[d,equal] \\
            \Xi\lar[rr,"u"'] & & \Xi
            \cellsymb(\delta_A){1-2}{2-2}
            \cellsymb((\compcell{\xi}u)_A){2-1}{3-3}
        \end{tikzcd}\incat{\bL}\quad (A\in\bK).
    \end{equation*}
    Then, \ref{axiom:cell_type1_left}\ref{axiom:cell_type2} and the following calculations conclude that $u$ becomes a loose unit on $\Xi$:
    \begin{equation*}
        \begin{tikzcd}[tri]
            FA\ar[d,"\xi_A"']\lar[rr,"\comp{\xi_A}"] && \Xi\ar[d,equal] \\
            \Xi\ar[dr,equal]\lar[rr,"u"] &{}& \Xi\ar[dl,equal] \\
            & \Xi\ar[dl,equal]\ar[dr,equal] & \\
            \Xi\lar[rr,"u"'] &{}& \Xi
            \cellsymb((\compcell{\xi}u)_A){1-1}{2-3}
            \cellsymb(\epsilon)[above=-3]{2-2}{3-2}
            \cellsymb(\eta){3-2}{4-2}
        \end{tikzcd}
        =
        \begin{tikzcd}
            & FA\ar[d,"\xi_A"']\lar[r,"\comp{\xi_A}"] & \Xi\ar[dl,equal] \\
            & \Xi\ar[dl,equal]\ar[d,equal] & \\
            \Xi\lar[r,"u"'] & \Xi &
            \cellsymb(\compcell{\xi_A})[above left=-2]{1-3}{2-2}
            \cellsymb(\eta)[below right]{2-2}{3-1}
        \end{tikzcd}
        =
        \begin{tikzcd}
            & FA\ar[dl,equal]\ar[d,"\xi_A"{left=-3}]\lar[r,"\comp{\xi_A}"] & \Xi\ar[dl,equal] \\
            FA\ar[d,"\xi_A"']\lar[r,"\comp{\xi_A}"'] & \Xi\ar[d,equal] & \\
            \Xi\lar[r,"u"'] & \Xi &
            \cellsymb(\delta_A)[below right]{1-2}{2-1}
            \cellsymb(\compcell{\xi_A})[above left=-2]{1-3}{2-2}
            \cellsymb((\compcell{\xi}u)_A){2-1}{3-2}
        \end{tikzcd}
        =
        \begin{tikzcd}
            FA\ar[d,"\xi_A"']\lar[r,"\comp{\xi_A}"] & \Xi\ar[d,equal] \\
            \Xi\lar[r,"u"'] & \Xi
            \cellsymb((\compcell{\xi}u)_A){1-1}{2-2}
        \end{tikzcd}
    \end{equation*}
    \begin{equation*}
        \begin{tikzcd}[tri]
            & FA\ar[d,"\xi_A"{left},bend right=20]\ar[d,"\xi_A"{right},bend left=20] & \\
            & \Xi\ar[dl,equal]\ar[dr,equal] & \\
            \Xi\ar[dr,equal]\lar[rr,"u"] &{}& \Xi\ar[dl,equal] \\
            & \Xi &
            \cellsymb(\heq){1-2}{2-2}
            \cellsymb(\eta){2-2}{3-2}
            \cellsymb(\epsilon)[above=-3]{3-2}{4-2}
        \end{tikzcd}
        =
        \begin{tikzcd}
            & FA\ar[dl,equal]\ar[d,"\xi_A"] \\
            FA\ar[d,"\xi_A"']\lar[r,"\comp{\xi_A}"'] & \Xi\ar[d,equal] \\
            \Xi\ar[d,equal]\lar[r,"u"] & \Xi\ar[dl,equal] \\
            \Xi &
            \cellsymb(\delta_A)[below right]{1-2}{2-1}
            \cellsymb((\compcell{\xi}u)_A){2-1}{3-2}
            \cellsymb(\epsilon)[above left=1]{3-2}{4-1}
        \end{tikzcd}
        =
        \begin{tikzcd}
            & FA\ar[dl,equal]\ar[d,"\xi_A"] \\
            FA\ar[d,"\xi_A"']\lar[r,"\comp{\xi_A}"] & \Xi\ar[dl,equal] \\
            \Xi &
            \cellsymb(\delta_A)[below right]{1-2}{2-1}
            \cellsymb(\compcell{\xi_A})[above left=-2]{2-2}{3-1}
        \end{tikzcd}
        =
        \begin{tikzcd}[tri]
            FA\ar[d,"\xi_A"{left},bend right=20]\ar[d,"\xi_A"{right},bend left=20] \\
            \Xi
            \cellsymb(\heq){1-1}{2-1}
        \end{tikzcd}\incat{\bL}.
    \end{equation*}
\end{proof}

\begin{example}[Versatile coproducts]
    Consider the \ac{AVDC} $\Rel$ of relations as in \cref{eg:avdc_of_relations}.
    Let $(\zero{X}_0,\zero{X}_1)\colon\Ddbl\zero{2}\to\Rel$ be an \ac{AVD}-functor determined by two (large) sets $\zero{X}_0,\zero{X}_1\in\Rel$, where $\zero{2}$ denotes the two-element set.
    Then, the disjoint union $\zero{X}_0+\zero{X}_1$ gives a versatile colimit of $(\zero{X}_0,\zero{X}_1)$, which is an example of a \emph{versatile coproduct} defined later (\cref{def:specific_versatile_colimits}).
    
    For the reader's convenience, we now explain the concrete meaning of the conditions in the definition of versatile colimits through this simple example.
    A tight cocone from $(\zero{X}_0,\zero{X}_1)$ is precisely a tuple $(\zero{L},l_0,l_1)$ of a (large) set $\zero{L}$ and two maps $\zero{X}_k\arr(l_k)[][1]\zero{L}$ $(k=0,1)$.
    Hence, the condition \ref{axiom:tight_bij} states that for every (large) set $\zero{L}$, composition with the coprojections \mbox{$\zero{X}_k\arr(\xi_k)[][1]\zero{X}_0+\zero{X}_1$} $(k=0,1)$ gives a bijection between the maps $\zero{X}_0+\zero{X}_1\arr(f)[][1]\zero{L}$ and the pairs $(\xi_0\tcomp f,\xi_1\tcomp f)$, which is the standard definition of $\zero{X}_0+\zero{X}_1$ being a binary coproduct in the category of (large) sets.

    A left module is precisely a tuple $(\zero{M},R_0,R_1)$ of a (large) set $\zero{M}$ and two binary relations $R_k\subseteq\zero{X}_k\times\zero{M}$ $(k=0,1)$.
    For a binary relation $R\subseteq(\zero{X}_0+\zero{X}_1)\times\zero{M}$, the restriction $R(\xi_0,\id)$ is defined as
    \[
        R(\xi_0,\id) \coloneq \{ (x,m)\in \zero{X}_0\times\zero{M} \mid (\xi_0(x),m)\in R \},
    \]
    and $R(\xi_1,\id)$ is also defined similarly.
    Then, since any isomorphism between left modules should be identity in this example, the condition \ref{axiom:loose_esssurj_left} states that for every pair of binary relations $R_k\subseteq\zero{X}_k\times\zero{M}$ $(k=0,1)$, there exists $R\subseteq(\zero{X}_0+\zero{X}_1)\times\zero{M}$ satisfying $R(\xi_k,\id)=R_k$ for $k=0,1$.
    This follows from the isomorphism
    \[
        (\zero{X}_0+\zero{X}_1)\times\zero{M} \cong (\zero{X}_0\times\zero{M})+(\zero{X}_1\times\zero{M}).
    \]
    In general, the uniqueness of such relation $R$ is required only up to isomorphism (cf.\ \cref{prop:fully_faithfulness_from_M2_M0}\cref{prop:fully_faithfulness_from_M2_M0-module}), but in this example, $R$ is strictly unique.
    The loosewise dual \ref{axiom:loose_esssurj_right} is similar.

    Finally, we explain only the condition \ref{axiom:cell_type3} here, because the other conditions \ref{axiom:cell_type1_left}\ref{axiom:cell_type1_right}\ref{axiom:cell_type2} are similar and follow from \ref{axiom:cell_type3} by using \cref{cor:simplification_to_M3}.
    For simplicity, we only consider the 1-coary case below.
    Consider (large) sets $\zero{M}_0,\dots,\zero{M}_m,\zero{M},\zero{N}_0,\dots,\zero{N}_n,\zero{N}$, maps $f,g$, and binary relations $S,T,P_1,\dots,P_m,Q_1,\dots,Q_n,R$ of the following forms:
    \begin{equation}\label{eq:boundary_Rel_M3}
        \begin{tikzcd}[scriptsizecolumn]
            \zero{N}_0\ar[d,"g"']\lar[r,path,"\tup{Q}"] & \zero{N}_n\lar[r,"T"] & \zero{X}_0+\zero{X}_1\lar[r,"S"] & \zero{M}_0\lar[r,path,"\tup{P}"] & \zero{M}_m\ar[d,"f"] \\
            \zero{N}\lar[rrrr,"R"'] &&&& \zero{M}
        \end{tikzcd}\incat{\Rel}.
    \end{equation}
    Since $\Rel$ has at most one cell for each boundary, the condition \ref{axiom:cell_type3} states that the existence of the cell fitting into \cref{eq:boundary_Rel_M3} is equivalent to the existence of the cells of the following forms for $k=0,1$:
    \begin{equation*}
        \begin{tikzcd}
            \zero{N}_0\ar[d,"g"']\lar[r,path,"\tup{Q}"] &[-10pt] \zero{N}_n\lar[r,"{T(\id,\xi_k)}"] & \zero{X}_k\lar[r,"{S(\xi_k,\id)}"] & \zero{M}_0\lar[r,path,"\tup{P}"] &[-10pt] \zero{M}_m\ar[d,"f"] \\
            \zero{N}\lar[rrrr,"R"'] &&&& \zero{M}
            \cellsymb(\cdot){1-1}{2-5}
        \end{tikzcd}
        \incat{\Rel}
        \quad (k=0,1).
    \end{equation*}
    This equivalence directly follows from the structure of $\zero{X}_0+\zero{X}_1$ as the disjoint union.
\end{example}

\begin{example}[Versatile initial objects]
    Similarly to the above example, the empty set gives a versatile colimit of the unique \ac{AVD}-functor from the empty \ac{AVDC} to $\Rel$, which is an example of a \emph{versatile initial object} defined later (\cref{def:specific_versatile_colimits}).
\end{example}

\begin{example}[Versatile collages]\label{eg:collage_in_SetProf}
    A \emph{collage}, also called \emph{cograph}, of a profunctor \mbox{$\one{A}\larr(P)[][1.5]\one{B}$} between categories is the category $\one{X}$ whose class of objects is the disjoint union of $\Ob\one{A}$ and $\Ob\one{B}$ and where
    \begin{equation*}
        \one{X}(x,y)\coloneq
        \begin{cases}
            \one{A}(x,y) & \text{if }x,y\in\one{A}; \\
            \one{B}(x,y) & \text{if }x,y\in\one{B}; \\
            P(x,y) & \text{if }x\in\one{A},~y\in\one{B}; \\
            \varnothing & \text{if }x\in\one{B},~y\in\one{A}.
        \end{cases}
    \end{equation*}
    Let $\bJ$ denote the \ac{AVDC} consisting of just two objects $0,1$ and a unique loose arrow $0\larr[][1]1$, and let $\Prof[\Set]$ denote the \ac{AVDC} of locally small (large) categories.
    If the categories $\one{A}$ and $\one{B}$ are large and locally small and the profunctor $P$ is locally small, then $\one{X}$ gives a versatile colimit of $\const{P}$, the \ac{AVD}-functor $\bJ\arr[][1] \Prof[\Set]$ selecting $P$.

    We briefly explain the meaning of the conditions for $\one{X}$ being a versatile colimit. 
    A tight cocone from $\const{P}$ is simply a tuple $(\one{L},L_0,L_1,L_{x,y})$ of the following data: a locally small (large) category $\one{L}$; functors $\one{A}\arr(L_0)[][1]\one{L}\rra(L_1)[][1]\one{B}$; and maps $P(x,y)\arr(L_{x,y}) \one{L}(L_0x,L_1y)$ that are natural in $x\in\one{A}$ and $y\in\one{B}$.
    Then, it can be observed that the tight cocones from $\const{P}$ bijectively correspond to the functors from $\one{X}$, which is what the condition \ref{axiom:tight_bij} requires.

    A left $\const{P}$-module is a tuple $(\one{M},M_0,M_1,M_{x,y,z})$ of the following data: a locally small (large) category $\one{M}$; locally small profunctors $\one{A}\larr(M_0)[][1]\one{M}$ and $\one{B}\larr(M_1)[][1]\one{M}$; and maps
    \[
        P(x,y) \times M_1(y,z) \arr(M_{x,y,z})[][3] M_0(x,z)
    \]
    that are (extra)natural in $x\in\one{A}$, $y\in\one{B}$, and $z\in\one{M}$.
    Given a left $\const{P}$-module $(\one{M},M_0,M_1,M_{x,y,z})$, we can construct a (locally small) profunctor $\one{X}\larr(\hat{M})\one{M}$ as follows:
    \begin{equation*}
        \hat{M}(x,m)\coloneq
        \begin{cases}
            M_0(x,m) & \text{if }x\in\one{A},~m\in\one{M}; \\
            M_1(x,m) & \text{if }x\in\one{B},~m\in\one{M}.
        \end{cases}
    \end{equation*}
    The action of the morphisms in $\one{X}$ on $\hat{M}$ is defined by using the maps $M_{x,y,z}$, together with the actions of the morphisms in $\one{A}$ and $\one{B}$ on $M_0$ and $M_1$, respectively.
    The existence of such profunctor $\hat{M}$ is what the condition \ref{axiom:loose_esssurj_left} requires.
    We omit describing the loosewise dual \ref{axiom:loose_esssurj_right}, since it is similar.

    Consider a left $\const{P}$-module $M=(\one{M},M_0,M_1,M_{x,y,z})$, a right $\const{P}$-module $N=\linebreak(\one{N},N_0,N_1,N_{w,x,y})$, functors $\one{M}\arr(F)[][1]\one{M}'$ and $\one{N}\arr(G)[][1]\one{N}'$, and a locally small profunctor \mbox{$\one{N}'\larr(Q)[][1.5]\one{M}'$}.
    Then, a modulation of type 3 fitting into the boundary
    \begin{equation}\label{eq:SetProf_boundary_M3}
        \begin{tikzcd}
            \one{N}\ar[d,"G"']\lar[r,Rightarrow,"N"] & \const{P}\lar[r,Rightarrow,"M"] & \one{M}\ar[d,"F"] \\
            \one{N}'\lar[rr,"Q"'] && \one{M}'
        \end{tikzcd}
    \end{equation}
    is a pair $(\omega_0,\omega_1)$ of the cells
    \begin{equation*}
        \begin{tikzcd}[scriptsizecolumn]
            \one{N}\ar[d,"G"']\lar[r,"N_0"] & \one{A}\lar[r,"M_0"] & \one{M}\ar[d,"F"] \\
            \one{N}'\lar[rr,"Q"'] && \one{M}'
            \cellsymb(\omega_0){1-1}{2-3}
        \end{tikzcd}
        \quad
        \begin{tikzcd}[scriptsizecolumn]
            \one{N}\ar[d,"G"']\lar[r,"N_1"] & \one{B}\lar[r,"M_1"] & \one{M}\ar[d,"F"] \\
            \one{N}'\lar[rr,"Q"'] && \one{M}'
            \cellsymb(\omega_1){1-1}{2-3}
        \end{tikzcd}
        \incat{\Prof[\Set]}
    \end{equation*}
    such that the following commutes for any $w\in\one{N}$, $x\in\one{A}$, $y\in\one{B}$, and $z\in\one{M}$.
    \begin{equation*}
        \begin{tikzcd}[small]
            & N_0(w,x)\times P(x,y)\times M_1(y,z)\ar[dl,"N_{w,x,y}\times\id"']\ar[dr,"\id\times M_{x,y,z}"] & \\
            N_1(w,y)\times M_1(y,z)\ar[dr,"\omega_1"'] && N_0(w,x)\times M_0(x,z)\ar[dl,"\omega_0"] \\
            & Q(Gw,Fz) &
        \end{tikzcd}
    \end{equation*}
    Then, it can be shown that the modulations fitting into \cref{eq:SetProf_boundary_M3} bijectively correspond to the cells of the following form:
    \begin{equation*}
        \begin{tikzcd}
            \one{N}\ar[d,"G"']\lar[r,"\hat{N}"] & \one{X}\lar[r,"\hat{M}"] & \one{M}\ar[d,"F"] \\
            \one{N}'\lar[rr,"Q"'] && \one{M}'
            \cellsymb(\cdot){1-1}{2-3}
        \end{tikzcd}\incat{\Prof[\Set]}.
    \end{equation*}
    This is a special case of what the condition \ref{axiom:cell_type3} requires, and the general case also follows similarly.
    Since $\Prof[\Set]$ has loose units, the rest conditions \ref{axiom:cell_type1_left}\ref{axiom:cell_type1_right}\ref{axiom:cell_type2} follow from \cref{cor:simplification_to_M3}.
\end{example}

\begin{remark}
    The versatile colimit described in \cref{eg:collage_in_SetProf} also becomes a versatile colimit of another diagram.
    Let $\zero{2}=\{0,1\}$ denote the two-element set.
    Then, the \ac{AVD}-functor $\const{P}\colon\bJ\arr[][1] \Prof[\Set]$ as in \cref{eg:collage_in_SetProf} extends to an \ac{AVD}-functor $\Idimdbl\zero{2}\arr[][1] \Prof[\Set]$ by assigning the loose units to $!_{00}$ and $!_{11}$, the empty profunctor to $!_{10}$.
    Then, the category $\one{X}$ as in \cref{eg:collage_in_SetProf} also becomes a versatile colimit of this extended \ac{AVD}-functor, which provides an example of a \emph{versatile collage} defined later in \cref{def:specific_versatile_colimits}.
\end{remark}

\begin{remark}
    Let $\one{A}\larr(P)\one{B}$ be a locally small profunctor between (not necessarily locally small) large categories.
    Let $\Prof[(\Set,\SET)]$ denote the \ac{AVDC} of large categories and locally small profunctors \cite[{2.6.\ Example}]{Koudenburg2020aug}.
    Then, the category $\one{X}$ as in \cref{eg:collage_in_SetProf} still gives a versatile colimit of the \ac{AVD}-functor $\const{P}\colon\bJ\arr[][1] \Prof[(\Set,\SET)]$ selecting $P$, where $\bJ$ is the same as in \cref{eg:collage_in_SetProf}.
    Since $\one{X}$ is not necessarily locally small, this gives an example of a versatile colimit with no loose unit.
\end{remark}

\begin{remark}
    The disjoint union $\zero{X}_0+\zero{X}_1$ of two sets also gives a versatile coproduct in the diminished \ac{AVDC} $\dim{\Rel}$ as well as in $\Rel$.
    However, the empty set $\varnothing$ fails to be a versatile initial object in $\dim{\Rel}$, despite being so in $\Rel$.
    Indeed, if a versatile initial object exists in $\dim{\Rel}$, the unitality theorem (\cref{thm:unitality_theorem}) implies the existence of a loose unit on the vertex, which is a contradiction because $\dim{\Rel}$ is diminished.
    This can also be seen from the fact that, for instance, the condition \ref{axiom:cell_type1_left} ends up requiring for any map $\zero{X}\arr(f)[][1]\zero{Y}$, the existence of a cell of the following form:
        \begin{equation*}
        \begin{tikzcd}
            \varnothing\ar[d,"!"']\lar[r,"!"] & \zero{X}\ar[dl,"f"] \\
            \zero{Y} &
            \cellsymb(\cdot)[above left=2]{1-2}{2-1}
        \end{tikzcd}\incat{\dim{\Rel}}.
    \end{equation*}
    Here, the symbols ``$!$'' denote the unique map and relation from the empty set.
\end{remark}

To address the issue as explained above, we consider the following definition:
\begin{definition}
    $\xi$ is called a \emph{\ac{VD}-versatile colimit} of $F$ if it satisfies the conditions \ref{axiom:tight_bij}\ref{axiom:loose_esssurj_left}\ref{axiom:loose_esssurj_right}, and the other conditions \ref{axiom:cell_type1_left}\ref{axiom:cell_type1_right}\ref{axiom:cell_type2}\ref{axiom:cell_type3} hold only for 1-coary modulations, i.e., modulations whose bottom boundary is of length 1. 
\end{definition}

\begin{example}
    The empty set is a \ac{VD}-versatile initial object in $\dim{\Rel}$.
\end{example}

We now present sufficient conditions for the two notions to coincide:
\begin{theorem}\label{thm:VDverscolim_unital}
    Suppose that $\bL$ has loose units.
    Then, the following are equivalent:
    \begin{enumerate}
        \item
            $\xi$ is a versatile colimit.
        \item
            $\xi$ is a \ac{VD}-versatile colimit.
        \item
            $\xi$ satisfies \ref{axiom:tight_bij}\ref{axiom:loose_esssurj_left}\ref{axiom:loose_esssurj_right}, and the condition \ref{axiom:cell_type3} holds only for 1-coary modulations.
    \end{enumerate}
\end{theorem}
\begin{proof}
    From the existence of the loose unit, any 0-coary cells reduce to 1-coary cells in $\bL$.
    Combining this with \cref{cor:simplification_to_M3}, we obtain the equivalence.
\end{proof}

\begin{theorem}\label{thm:VDverscolim_diminished}
    Suppose that $\bK$ is non-empty and loosewise discrete, and that $\bL$ is diminished.
    Then, $\xi$ is a versatile colimit if and only if it is a \ac{VD}-versatile colimit.
\end{theorem}
\begin{proof}
    From the assumption that $\bK$ is non-empty, any 0-coary modulation cannot appear in the conditions \ref{axiom:cell_type1_left}\ref{axiom:cell_type1_right}\ref{axiom:cell_type2}\ref{axiom:cell_type3}.
    Indeed, if such a modulation were to exist, its component at some object in $\bK$ would be a non-trivial 0-coary cell in $\bL$, which contradicts the assumption that $\bL$ is diminished.
    Hence, the equivalence follows.
\end{proof}
\subsection{Rigidity of the concept of versatile colimits}
We will now study the uniqueness of (\ac{VD}-)versatile colimits up to an appropriate notion of isomorphism,
as well as their invariance under suitable equivalences between \acp{AVDC}.
\begin{definition}\label{def:admissible_equivalence}\quad
    \begin{enumerate}
        \item
            An invertible tight arrow in an \ac{AVDC} is called \emph{admissible} if it is pulling, and its inverse is also pulling.
            Such a tight arrow is also called an \emph{admissible isomorphism}.
            Two objects are called \emph{admissibly isomorphic} (to each other) if there is an admissible isomorphism between them.
        \item
            An invertible tight \ac{AVD}-transformation is called \emph{admissible} if every component is admissible.
        \item
            An equivalence in the 2-category $\AVDC$ is called \emph{admissible} if the associated invertible tight \ac{AVD}-transformations are admissible.
            Two \acp{AVDC} are called \emph{admissibly equivalent} (to each other) if there is an admissible equivalence between them.\qedhere
    \end{enumerate}
\end{definition}

\begin{definition}\label{def:isofibrancy}
    An \ac{AVDC} $\bL$ is called \emph{iso-fibrant} if every invertible tight arrow in $\bL$ is admissible.
    Clearly, every equivalence between iso-fibrant \acp{AVDC} is admissible.
\end{definition}

\begin{remark}\label{rem:intuition_of_admissible_iso}
    An invertible tight arrow $X\arr(k)[][1]Y$ in an \ac{AVDC} $\bL$ always induces isomorphisms between the categories of tight arrows:
    \begin{gather*}
        \Homcat[\bL]\vvect{L}{X}\cong\Homcat[\bL]\vvect{L}{Y}
        \colon f\mapsto f\tcomp k,\quad g\tcomp k^{-1}\mapsfrom g; \\
        \Homcat[\bL]\vvect{X}{L}\cong\Homcat[\bL]\vvect{Y}{L}
        \colon f\mapsto k^{-1}\tcomp f,\quad k\tcomp g\mapsfrom g.
    \end{gather*}
    If $k$ is admissible, it further induces equivalences between the categories of loose arrows:
    \begin{gather*}
        \Homcat[\bL](L,X)\simeq\Homcat[\bL](L,Y)
        \colon u\mapsto u(\id_L,k^{-1}),\quad v(\id_L,k)\mapsfrom v; \\
        \Homcat[\bL](X,L)\simeq\Homcat[\bL](Y,L)
        \colon u\mapsto u(k^{-1},\id_L),\quad v(k,\id_L)\mapsfrom v.
    \end{gather*}
    Moreover, under the assumption of admissibility, $k$ also induces bijections between the classes of cells.
    We now describe a special case.
    Let us consider the following data:
    \begin{equation*}
        \begin{tikzcd}
            Y\ar[d,"f"']\lar[r,"u"] & A_0\lar[r,path,"\tup{p}"] & A_n\ar[d,"g"] \\
            B\lar[rr,phan,"q"'] && C
        \end{tikzcd}\incat{\bL}.
    \end{equation*}
    Whenever $k$ is admissible, we can obtain the following restriction, denoted by $\comp{k}u$, with the cartesian cell denoted by $\compcell{k}u$:
    \begin{equation*}
        \begin{tikzcd}[largecolumn]
            X\ar[d,"k"']\lar[r,"\comp{k}u"] & A_0\ar[d,equal] \\
            Y\lar[r,"u"'] & A_0
            \cellsymb(\compcell{k}u\colon\cart){1-1}{2-2}
        \end{tikzcd}\incat{\bL}.
    \end{equation*}
    Since $k$ is invertible, the cell $\compcell{k}u$ automatically becomes loosewise invertible.
    Then, pre-composition of $\compcell{k}u$ gives the following bijection between the classes of cells:
    \begin{equation*}
        \Cells[\bL]{k\tcomp f}{g}{\comp{k}u,\tup{p}}{q}\cong\Cells[\bL]{f}{g}{u,\tup{p}}{q}
        \colon (\compcell{k}u,\veq_\tup{p})\tcomp\alpha~\mapsfrom~\alpha.
    \end{equation*}
    The general case also follows similarly.
\end{remark}

The following proposition shows that (\ac{VD}-)versatile colimits of a given diagram are unique up to admissible isomorphism.
\begin{proposition}\label{prop:vers_colim_is_up_to_admissible_iso}
    Let $\xi$ be a versatile (resp.\ \ac{VD}-versatile) colimit of an \ac{AVD}-functor \mbox{$F\colon\bK\to\bL$}, and let $\Xi\in\bL$ be the vertex of $\xi$.
    \begin{enumerate}
        \item\label{prop:vers_colim_is_up_to_admissible_iso-vertex}
            Let $\Xi\arr(k)[][1]L$ be a tight arrow in $\bL$.
            Then, the induced tight cocone $\xi\tcomp k$ is a versatile (resp.\ \ac{VD}-versatile) colimit if and only if $k$ is an admissible isomorphism.
        \item\label{prop:vers_colim_is_up_to_admissible_iso-diagram}
            Let $F'\arr(\rho)[Rightarrow][1]F$ be an invertible tight \ac{AVD}-transformation, and suppose that $\rho$ is admissible.
            Then, the induced tight cocone $\rho\tcomp\xi\coloneq(\rho_A\tcomp\xi_A)_{A\in\bK}$ becomes a versatile (resp.\ \ac{VD}-versatile) colimit of $F'$.
    \end{enumerate}
\end{proposition}
\begin{proof}\quad
    \begin{enumerate}
        \item
            Suppose that $k$ is an admissible isomorphism.
            Then, by \cref{rem:intuition_of_admissible_iso}, the axioms of (\ac{VD}-)versatile colimits for $\xi$ imply those for $\xi\tcomp k$ directly.
            For example, the condition \ref{axiom:loose_esssurj_left} for $\xi\tcomp k$ can be derived as follows:
            \begin{gather*}
                \Homcat[\bL](L,M)
                \simeq\Homcat[\bL](\Xi,M)
                \simeq\Mdl{F}{M}
                \qquad (M\in\bL).
            \end{gather*}
            Here, the first equivalence comes from \cref{rem:intuition_of_admissible_iso}, and the second one comes from \ref{axiom:loose_esssurj_left} for $\xi$.
            The remaining conditions \ref{axiom:loose_esssurj_right}\ref{axiom:cell_type1_left}\ref{axiom:cell_type1_right}\ref{axiom:cell_type2}\ref{axiom:cell_type3} for $\xi\tcomp k$ also follow in a similar way.

            We now show the converse direction.
            Suppose that the tight cocone $\xi\tcomp k$ is a\linebreak (\ac{VD}-)versatile colimit.
            By \ref{axiom:tight_bij} for $\xi\tcomp k$ and $\xi$, the tight arrow $k$ becomes invertible.
            Although we are required to show that both $k$ and $k^{-1}$ are pulling, due to symmetry, it suffices to show that $k$ is left-pulling.
            Let $L\larr(u)M$ be a loose arrow in $\bL$.
            Since every $\xi_A\tcomp k$ is left-pulling, we have the left $F$-module $\comp{(\xi\tcomp k)}u$ as in \cref{const:canonical_functor_loose_to_mdl}.
            By \ref{axiom:loose_esssurj_left-dash} for $\xi$, we have a loose arrow $\Xi\larr(v)M$ in $\bL$ such that $\comp{\xi}v=\comp{(\xi\tcomp k)}u$.
            Then, by \ref{axiom:cell_type1_left} for $\xi$, there is a unique cell $\alpha$ satisfying the following equation:
            \begin{equation*}
                \begin{tikzcd}[hugecolumn]
                    FA\ar[d,"(\xi\tcomp k)_A"']\lar[r,"(\comp{(\xi\tcomp k)}u)_A"] & M\ar[d,equal] \\
                    L\lar[r,"u"'] & M
                    \cellsymb((\compcell{(\xi\tcomp k)}u)_A){1-1}{2-2}
                \end{tikzcd}
                =
                \begin{tikzcd}[hugecolumn]
                    FA\ar[d,"\xi_A"']\lar[r,"(\comp{(\xi\tcomp k)}u)_A"] & M\ar[d,equal] \\
                    \Xi\ar[d,"k"']\lar[r,"v"'] & M\ar[d,equal] \\
                    L\lar[r,"u"'] & M
                    \cellsymb((\compcell{\xi}v)_A){1-1}{2-2}
                    \cellsymb(\alpha){2-1}{3-2}
                \end{tikzcd}\incat{\bL}\quad (A\in\bK).
            \end{equation*}
            By using \ref{axiom:cell_type1_left} for $\xi\tcomp k$, we can obtain a loosewise inverse of $\alpha$.
            In particular, the cell $\alpha$ is cartesian.
        \item
            Since $\rho$ is invertible, it induces isomorphisms of categories $\Cone\vvect{F}{L}\cong\Cone\vvect{F'}{L}$ for any $L\in\bL$.
            Moreover, by the admissibility, $\rho$ induces equivalences of categories $\Mdl{F}{M}\simeq\Mdl{F'}{M}$ for any $M\in\bL$, and bijections among the classes of modulations of the same type.
            Thus, the conditions \ref{axiom:tight_bij}\ref{axiom:loose_esssurj_left}\ref{axiom:loose_esssurj_right}\ref{axiom:cell_type1_left}\ref{axiom:cell_type1_right}\ref{axiom:cell_type2}\ref{axiom:cell_type3} for $\rho\tcomp\xi$ follow directly from those for $\xi$.\qedhere
    \end{enumerate}
\end{proof}

\begin{lemma}\label{lem:replacement_admissible_adj_equiv}
    Every admissible equivalence between \acp{AVDC} can be replaced with an admissible adjoint equivalence.
\end{lemma}
\begin{proof}
    Let
    $
    \begin{tikzcd}
        \bL\ar[r,"\Phi",shift left=0.7] & \bL'\ar[l,"\Psi",shift left=0.7]
    \end{tikzcd}
    $
    form an admissible equivalence with invertible tight \ac{AVD}-transformations $\alpha\colon\Id\Rightarrow\Psi\circ\Phi$ and $\beta\colon\Phi\circ\Psi\Rightarrow\Id$.
    Since $\AVDC$ is a 2-category, we can obtain an adjoint equivalence $(\Phi,\Psi,\eta,\epsilon)$ by defining $\eta\coloneq\alpha$ and $\epsilon\coloneq\beta\circ(\Phi\eta^{-1}\Psi)\circ(\beta^{-1}\Phi\Psi)$.
    Now, we have to show that any components of $\epsilon$ and $\epsilon^{-1}$ are pulling in $\bL'$.
    
    Take an arbitrary loose arrow $X\larr(u)[][1]Y$ in $\bL'$.
    Since $\eta^{-1}_{\Psi X}=\alpha^{-1}_{\Psi X}$ is left-pulling, we obtain a cartesian (loosewise invertible) cell $\lambda$ on the left below.
    By the pullingness of $\beta^{-1}$ and \cref{lem:pulling_and_restriction}, we also obtain the top cartesian cell on the right below.
    \begin{equation*}
        \begin{tikzcd}
            \Psi\Phi\Psi X\ar[d,"\eta^{-1}_{\Psi X}"']\lar[r,"p"] & \Psi Y\ar[d,equal] \\
            \Psi X\lar[r,"\Psi u"'] & \Psi Y
            \cellsymb(\lambda\colon\cart){1-1}{2-2}
        \end{tikzcd}\incat{\bL}
        \qquad\qquad
        \begin{tikzcd}
            \Phi\Psi X\ar[d,"\beta^{-1}_{\Phi\Psi X}"']\lar[r] & Y\ar[d,"\beta^{-1}_Y"] \\
            \Phi\Psi\Phi\Psi X\ar[d,"\Phi\eta^{-1}_{\Psi X}"']\lar[r,"\Phi p"] & \Phi\Psi Y\ar[d,equal] \\
            \Phi\Psi X\ar[d,"\beta_X"']\lar[r,"\Phi\Psi u"] & \Phi\Psi Y\ar[d,"\beta_Y"] \\
            X\lar[r,"u"'] & Y
            \cellsymb(\cart){1-1}{2-2}
            \cellsymb(\Phi\lambda){2-1}{3-2}
            \cellsymb(\beta_u){3-1}{4-2}
        \end{tikzcd}\incat{\bL'}
    \end{equation*}
    Then, any cells on the right above are cartesian.
    Indeed, $\Phi\lambda$ and $\beta_u$ are cartesian since they are loosewise invertible.
    Thus, the composite of the cells on the right above gives the desired restriction $u(\epsilon_X,\id)$, hence $\epsilon_X$ is left-pulling.
    The loosewise dual argument shows that $\epsilon_X$ is right-pulling.
    Furthermore, using the pullingness of $\eta_{\Psi X}$, we can also show that $\epsilon^{-1}_X$ is pulling in a similar way.
\end{proof}

\begin{remark}
    Let $\Phi\colon\bL\to\bL'$ be an equivalence in the 2-category $\AVDC$.
    For any objects $X,Y\in\bL$, $\Phi$ induces an isomorphism between the categories of tight arrows:
    \begin{equation*}
        \Homcat[\bL]\vvect{X}{Y}\cong\Homcat[\bL']\vvect{\Phi X}{\Phi Y}
        \colon f\mapsto\Phi f.
    \end{equation*}
    This follows from 2-functoriality of the assignment $\bL\mapsto\Ttwocat\bL$.

    On the other hand, the functor $\Homcat[\bL](X,Y)\arr(\Phi-)\Homcat[\bL'](\Phi X,\Phi Y)$ is not necessarily an equivalence.
    It is fully faithful, but not essentially surjective in general.
    A counterexample is given by the following \ac{AVD}-functor:
    \begin{equation*}
        \bL\coloneq
        \left\{
            \begin{tikzcd}[small]
                & Y\ar[d,"k","\cong"'] \\
                X\lar[r,"u"'] & Z
            \end{tikzcd}
        \right\}
        \arr(\Phi)[][3]
        \left\{
            \begin{tikzcd}[small]
                0\lar[r,"!"] & 1
            \end{tikzcd}
        \right\}
        =:\bL'.
    \end{equation*}
    Here, both of $\bL$ and $\bL'$ have no non-trivial cells, $k$ is invertible as a tight arrow, and $\Phi$ is a unique \ac{AVD}-functor from $\bL$ to $\bL'$, which is described by $\Phi X=0,\Phi Y=\Phi Z=1$.
    Then, $\Phi$ is an equivalence in $\AVDC$.
    However, the categories $\Homcat[\bL](X,Y)$ and $\Homcat[\bL'](\Phi X,\Phi Y)=\Homcat[\bL'](0,1)$ are far from being equivalent.
    This pathological phenomenon is caused by the fact that the invertible tight arrow $k$ is not admissible.
    In fact, \cref{lem:intuition_of_admissible_equivalence} shows that admissible equivalences are a good enough concept to solve this problem.
\end{remark}

\begin{lemma}\label{lem:intuition_of_admissible_equivalence}
    Let $\Phi\colon\bL\to\bL'$ be an admissible equivalence between \acp{AVDC}.
    Then, the following functor becomes an equivalence of categories for any $X,Y\in\bL$:
    \begin{equation*}
        \Homcat[\bL](X,Y)\arr(\Phi-)\Homcat[\bL'](\Phi X,\Phi Y).
    \end{equation*}
\end{lemma}
\begin{proof}
    The only non-trivial part is essential surjectivity, which is shown as follows.
    Let \mbox{$\Phi X\larr(u)[][1]\Phi Y$} be a loose arrow in $\bL'$.
    By \cref{lem:replacement_admissible_adj_equiv}, $\Phi$ can extend to an admissible adjoint equivalence with $\Psi\colon\bL'\to\bL$, a unit $\eta\colon\Id\Rightarrow\Psi\circ\Phi$, and a counit $\epsilon\colon\Phi\circ\Psi\Rightarrow\Id$.
    By the pullingness of the unit $\eta$ and \cref{lem:pulling_and_restriction}, we can take the following restriction:
    \begin{equation*}
        \begin{tikzcd}
            X\ar[d,"\eta_X"']\lar[r,"v"] & Y\ar[d,"\eta_Y"] \\
            \Psi\Phi X\lar[r,"\Psi u"'] & \Psi\Phi Y
            \cellsymb(\lambda\colon\cart){1-1}{2-2}
        \end{tikzcd}\incat{\bL}.
    \end{equation*}
    By \cref{prop:cart_and_loosewise_inv}, the cell $\lambda$ becomes loosewise invertible.
    Then, the composite of the following cells gives an isomorphism $\Phi v\cong u$ in the category $\Homcat[\bL'](\Phi X,\Phi Y)$:
    \begin{equation*}
        \begin{tikzcd}
            \Phi X\ar[dd,bend right=60,equal]\ar[d,"\Phi\eta_X"]\lar[r,"\Phi v"] & \Phi Y\ar[d,"\Phi\eta_Y"']\ar[dd,bend left=60,equal] \\
            \Phi\Psi\Phi X\ar[d,"\epsilon_{\Phi X}"]\lar[r,"\Phi\Psi u"] & \Phi\Psi\Phi Y\ar[d,"\epsilon_{\Phi Y}"'] \\
            \Phi X\lar[r,"u"'] & \Phi Y
            \cellsymb(\Phi\lambda){1-1}{2-2}
            \cellsymb(\epsilon_u){2-1}{3-2}
            \cellsymb(\heq)[left=13]{1-1}{3-1}
            \cellsymb(\heq)[right=13]{1-2}{3-2}
        \end{tikzcd}\incat{\bL'}.
    \end{equation*}
    This shows the essential surjectivity.
\end{proof}

\begin{lemma}\label{lem:admissible_equiv_preserve_pullingness}
    Every admissible equivalence between \acp{AVDC} preserves left-pullingness and right-pullingness of tight arrows.
\end{lemma}
\begin{proof}
    Let
    $
    \begin{tikzcd}
        \bL\ar[r,"\Phi",shift left=0.7] & \bL'\ar[l,"\Psi",shift left=0.7]
    \end{tikzcd}
    $
    form an admissible equivalence with invertible tight \ac{AVD}-transformations $\eta\colon\Id\Rightarrow\Psi\circ\Phi$ and $\epsilon\colon\Phi\circ\Psi\Rightarrow\Id$.
    By \cref{lem:replacement_admissible_adj_equiv}, we can suppose that these form an (admissible) adjoint equivalence without loss of generality.
    
    Let $A\arr(f)[][1]B$ be a left-pulling tight arrow in $\bL$.
    To show that $\Phi f$ is left-pulling in $\bL'$, let us take an arbitrary loose arrow $\Phi B\larr(u)[][1]L'$ in $\bL'$.
    By assumption, we can take the restrictions on the left below.
    Then, the composite of the cells on the right below gives the desired restriction $u(\Phi f,\id)$ by \cref{prop:radj_preserves_cart}.
    \begin{equation*}
        \begin{tikzcd}
            A\ar[d,"f"']\lar[r,"q"] & \Psi L'\ar[d,equal] \\
            B\ar[d,"\eta_B"']\lar[r,"p"] & \Psi L'\ar[d,equal] \\
            \Psi\Phi B\lar[r,"\Psi u"'] & \Psi L'
            \cellsymb(\mu\colon\cart){1-1}{2-2}
            \cellsymb(\lambda\colon\cart){2-1}{3-2}
        \end{tikzcd}\incat{\bL}
        \qquad\qquad
        \begin{tikzcd}[largecolumn]
            \Phi A\ar[d,equal]\lar[r,"{(\Phi q)(\id,\epsilon^{-1}_{L'})}"] & L'\ar[d,"\epsilon^{-1}_{L'}"] \\
            \Phi A\ar[d,"\Phi f"']\lar[r,"\Phi q"] & \Phi\Psi L'\ar[d,equal] \\
            \Phi B\ar[dd,equal,bend right=60,shift right=1]\ar[d,"\Phi\eta_B"{right=-1}]\lar[r,"\Phi p"] & \Phi\Psi L'\ar[d,equal] \\
            \Phi\Psi\Phi B\ar[d,"\epsilon_{\Phi B}"]\lar[r,"\Phi\Psi u"] & \Phi\Psi L'\ar[d,"\epsilon_{L'}"] \\
            \Phi B\lar[r,"u"'] & L'
            \cellsymb(\cart){1-1}{2-2}
            \cellsymb(\Phi\mu\colon\cart){2-1}{3-2}
            \cellsymb(\Phi\lambda\colon\cart)[right=-8]{3-1}{4-2}
            \cellsymb(\epsilon_u\colon)[below=-3]{4-1}{5-2}
            \cellsymb(\linv)[right=6]{4-1}{5-2}
            \cellsymb(\heq)[left=15]{3-1}{5-1}
        \end{tikzcd}\incat{\bL'}
    \end{equation*}
    This shows that $\Phi$ preserves left-pullingness.
    The preservation of right-pullingness also follows from the loosewise dual argument.
\end{proof}

\begin{corollary}
    The composite of two admissible equivalences is again admissible.
    In particular, admissible equivalences yield an equivalence relation among \acp{AVDC}.
\end{corollary}
\begin{proof}
    Consider two admissible equivalences
    \begin{equation*}
        \begin{tikzcd}
            \bL\ar[rr,"\Phi",shift left=5pt,bend left=10] & \hspace{11pt}\perp_{\eta,\epsilon} & \bL'\ar[ll,"\Psi",shift left=5pt,bend left=10]\ar[rr,"\Phi'",shift left=5pt,bend left=10] & \hspace{11pt}\perp_{\eta',\epsilon'} & \bL''\ar[ll,"\Psi'",shift left=5pt,bend left=10]
        \end{tikzcd}
    \end{equation*}
    with the units $\eta,\eta'$ and the counits $\epsilon,\epsilon'$.
    Then, at each object $L\in\bL$, the unit of the composite equivalence is given by $\eta_L\tcomp\Psi\eta'_{\Phi L}$, which is still admissible as an invertible tight arrow by \cref{lem:admissible_equiv_preserve_pullingness}.
    In other words, the unit of the composite equivalence is admissible.
    The same argument also works for the counit.
\end{proof}

\begin{theorem}\label{thm:admissible_equiv_preserve_versatile_colim}
    (\Ac{VD}-)versatile colimits are preserved by any admissible equivalence.
\end{theorem}
\begin{proof}
    Let
    $
    \begin{tikzcd}
        \bL\ar[r,"\Phi",shift left=0.7] & \bL'\ar[l,"\Psi",shift left=0.7]
    \end{tikzcd}
    $
    be an admissible adjoint equivalence with a unit $\eta\colon\Id\Rightarrow\Psi\circ\Phi$ and a counit $\epsilon\colon\Phi\circ\Psi\Rightarrow\Id$.
    Let $\xi$ be a (\ac{VD}-)versatile colimit of an \ac{AVD}-functor $F\colon\bK\to\bL$, and let $\Xi\in\bL$ be the vertex of $\xi$.
    We have to show that the induced tight cocone $\Phi\xi\coloneq(\Phi\xi_A)_{A\in\bK}$ is a (\ac{VD}-)versatile colimit of $\Phi\circ F$.
    The condition \ref{axiom:tight_bij} for $\Phi\xi$ can be verified by the following isomorphisms for $L'\in\bL'$:
    \begin{equation*}
        \Homcat[\bL']\vvect{\Phi\Xi}{L'}
        \cong\Homcat[\bL]\vvect{\Xi}{\Psi L'}
        \cong\Cone\vvect{F}{\Psi L'}
        \cong\Cone\vvect{\Phi\circ F}{L'}.
    \end{equation*}
    Here, the first and third isomorphisms come from the adjunction, and the second one follows from the universal property of $\xi$.

    Furthermore, the condition \ref{axiom:loose_esssurj_left} for $\Phi\xi$ can be verified by the following equivalences for $L'\in\bL'$:
    \begin{align*}
        \Homcat[\bL'](\Phi\Xi,L')
            &\simeq\Homcat[\bL'](\Phi\Xi,\Phi\Psi L')
                && \text{(by admissibility of $\epsilon_{L'}$ and \cref{rem:intuition_of_admissible_iso})} \\
            &\simeq\Homcat[\bL](\Xi,\Psi L')
                && \text{(by \cref{lem:intuition_of_admissible_equivalence})} \\
            &\simeq\Mdl{F}{\Psi L'}
                && \text{(by the universal property of $\xi$)} \\
            &\simeq\Mdl{\Phi\circ F}{\Phi\Psi L'}
                && \\
            &\simeq\Mdl{\Phi\circ F}{L'}
                && \text{(by admissibility of $\epsilon_{L'}$ and \cref{rem:intuition_of_admissible_iso})}
    \end{align*}
    We now describe the fourth equivalence.
    By \cref{prop:radj_preserves_cart}, $\Phi$ preserves the cartesian cells associated with left modules.
    Thus, $\Phi$ induces a functor $\Mdl{F}{\Psi L'}\arr[][1]\Mdl{\Phi\circ F}{\Phi\Psi L'}$, which becomes an equivalence by \cref{lem:intuition_of_admissible_equivalence}.
    
    The condition \ref{axiom:loose_esssurj_right} for $\Phi\xi$ also follows from the loosewise dual argument, and the remaining conditions directly follow from the argument in \cref{rem:intuition_of_admissible_iso}.
\end{proof}

\begin{theorem}\label{thm:admissible_equiv_and_versatile_cocompleteness}
    Let $\mathcal{S}$ be a class of \acp{AVDC}.
    Let $\bL$ and $\bL'$ be \acp{AVDC} that are admissibly equivalent to each other.
    Then, $\bL$ has (\ac{VD}-)versatile colimits of all shapes in $\mathcal{S}$ if and only if so does $\bL'$.
\end{theorem}
\begin{proof}
    Let
    $
    \begin{tikzcd}
        \bL\ar[r,"\Phi",shift left=0.7] & \bL'\ar[l,"\Psi",shift left=0.7]
    \end{tikzcd}
    $
    be an admissible adjoint equivalence with a unit $\eta\colon\Id\Rightarrow\Psi\circ\Phi$ and a counit $\epsilon\colon\Phi\circ\Psi\Rightarrow\Id$.
    Let $\bK$ be an \ac{AVDC}, and suppose that $\bL$ has (\ac{VD}-)versatile colimits of the shape $\bK$.
    Now, we are required to show that $\bL'$ also has (\ac{VD}-)versatile colimits of the shape $\bK$.
    Let $F\colon\bK\to\bL'$ be an \ac{AVD}-functor.
    By assumption, there is a (\ac{VD}-)versatile colimit $\xi$ of $\Psi\circ F$.
    By \cref{thm:admissible_equiv_preserve_versatile_colim}, $\Phi\xi$ is a (\ac{VD}-)versatile colimit of $\Phi\circ\Psi\circ F$.
    Then, \cref{prop:vers_colim_is_up_to_admissible_iso}\cref{prop:vers_colim_is_up_to_admissible_iso-diagram} implies that the tightwise composite of $\Phi\xi$ with $\epsilon^{-1} F$ gives the desired (\ac{VD}-)versatile colimit of $F$, which finishes the proof.
\end{proof}
\subsection{The case of loosewise indiscrete shapes}\label{subsec:loosewise_indiscrete_shapes}
In this subsection, we study versatile colimits in the special case when the shape is loosewise indiscrete (\cref{def:loosewise_indiscrete}).
Let us fix an \ac{AVD}-functor $F\colon\bK\to\bL$ from a loosewise indiscrete \ac{AVDC} $\bK$.
\begin{proposition}\label{prop:cocones_loosewise_indiscrete}
    A tight cocone from $F$ with a vertex $L\in\bL$ is the same as the following data:
    \begin{itemize}
        \item
            For each object $A\in\bK$, a tight arrow $FA\arr(l_A)L$ in $\bL$.
        \item
            For objects $A,B\in\bK$, a cell $l_{AB}$ of the following form:
            \begin{equation*}
                \begin{tikzcd}[tri]
                    FA\ar[dr,"l_A"']\lar[rr,"F!_{AB}"] &{}& FB\ar[dl,"l_B"] \\
                    & L &
                    \cellsymb(l_{AB})[above=-3]{1-2}{2-2}
                \end{tikzcd}\incat{\bL}.
            \end{equation*}
    \end{itemize}
    These are required to satisfy the following conditions:
    \begin{itemize}
        \item
            For $A\arr(f)B$ in $\bK$, the cell
            \begin{equation*}
                \begin{tikzcd}
                    & FA\ar[dl,"Ff"']\ar[d,equal] \\
                    FB\ar[dr,"l_B"']\lar[r,"F!_{BA}"] & FA\ar[d,"l_A"] \\
                    & L
                    \cellsymb(F!)[below right]{1-2}{2-1}
                    \cellsymb(l_{BA})[above right]{2-1}{3-2}
                \end{tikzcd}
            \end{equation*}
            becomes the tight identity cell.
        \item
            For $A_0,A_1,A_2\in\bK$,
            \begin{equation*}
                \begin{tikzcd}
                    FA_0\ar[d,equal]\lar[r,"F!_{A_0A_1}"] & FA_1\lar[r,"F!_{A_1A_2}"] & FA_2\ar[d,equal] \\
                    FA_1\ar[dr,"l_{A_0}"']\lar[rr,"F!_{A_0A_2}"] &{}& FA_2\ar[dl,"l_{A_2}"] \\
                    & L &
                    \cellsymb(F!){1-1}{2-3}
                    \cellsymb(l_{A_0A_2})[above=-3]{2-2}{3-2}
                \end{tikzcd}
                =
                \begin{tikzcd}[huge]
                    FA_0\ar[dr,"l_{A_0}"']\lar[r,"F!_{A_0A_1}"] & FA_1\ar[d,"l_{A_1}"{right=-1,pos=0.6}]\lar[r,"F!_{A_1A_2}"] & FA_2\ar[dl,"l_{A_2}"] \\
                    & L &
                    \cellsymb(l_{A_0A_1})[above right=0]{1-1}{2-2}
                    \cellsymb(l_{A_1A_2})[above left=-1]{1-3}{2-2}
                \end{tikzcd}\incat{\bL}.
            \end{equation*}
    \end{itemize}
\end{proposition}
\begin{proof}
    The second condition extends to arbitrary finite families $A_0,A_1,\cdots,A_n\in\bK$, rather than being restricted to the case $n=2$.
    Indeed, the case $n=0$ is a special case of the first condition with $f$ taken to be the identity, while the case $n\ge 3$ follows by iterated application of the case $n=2$.
    Furthermore, we have the loosewise dual to the first condition for $A\arr(f)B$ in $\bK$ as follows:
    \begin{equation*}
        \begin{tikzcd}[large]
            & FA\ar[dl,equal]\ar[d,"Ff"] \\
            FA\ar[dr,"l_A"']\lar[r,"F!_{AB}"] & FB\ar[d,"l_B"] \\
            & L
            \cellsymb(F!)[below right]{1-2}{2-1}
            \cellsymb(l_{AB})[above right]{2-1}{3-2}
        \end{tikzcd}
        =
        \begin{tikzcd}[large]
            & FA\ar[dl,equal]\ar[d,"Ff"{right=-1}]\ar[dr,equal] & \\
            FA\ar[dr,"l_A"']\lar[r,"F!_{AB}"] & FB\ar[d,"l_B"]\lar[r,"F!_{BA}"'] & FA\ar[dl,"l_A"] \\
            & L &
            \cellsymb(F!)[below right]{1-2}{2-1}
            \cellsymb(F!)[below left]{1-2}{2-3}
            \cellsymb(l_{AB})[above right]{2-1}{3-2}
            \cellsymb(l_{BA})[above left]{2-3}{3-2}
        \end{tikzcd}
        =
        \begin{tikzcd}
            & FA\ar[dl,equal]\ar[d,"Ff"{right=-1}]\ar[dr,equal] & \\
            FA\ar[d,equal]\lar[r,"F!_{AB}"'] & FB\lar[r,"F!_{BA}"'] & FA\ar[d,equal] \\
            FA\ar[dr,"l_A"']\lar[rr,"F!_{AA}"'] &{}& FA\ar[dl,"l_A"] \\
            & L &
            \cellsymb(F!)[below right]{1-2}{2-1}
            \cellsymb(F!)[below left]{1-2}{2-3}
            \cellsymb(F!){2-1}{3-3}
            \cellsymb(l_{AA}){3-2}{4-2}
        \end{tikzcd}
    \end{equation*}
    \begin{equation*}
        =
        \begin{tikzcd}[tri]
            & FA\ar[dl,equal]\ar[dr,equal] & \\
            FA\ar[dr,"l_A"']\lar[rr,"F!_{AA}"'] &{}& FA\ar[dl,"l_A"] \\
            & L &
            \cellsymb(F!){1-2}{2-2}
            \cellsymb(l_{AA}){2-2}{3-2}
        \end{tikzcd}
        =
        \begin{tikzcd}
            FA\ar[d,bend right=20,"l_A"{left}]\ar[d,bend left=20,"l_A"{right}] \\
            L
            \cellsymb(\heq){1-1}{2-1}
        \end{tikzcd}\incat{\bL}.
    \end{equation*}
    Then, we have
    \begin{equation*}
        \begin{tikzcd}[tri]
            FA_0\ar[d,"Ff"']\lar[rr,path,"F\tup{!}"] & & FA_n\ar[d,"Fg"] \\
            FB\ar[dr,"l_B"']\lar[rr,"F!_{BC}"] &{}& FC\ar[dl,"l_C"] \\
            & L &
            \cellsymb(F!){1-1}{2-3}
            \cellsymb(l_{BC})[above=-3]{2-2}{3-2}
        \end{tikzcd}
        =
        \begin{tikzcd}[large]
            & FA_0\ar[dl,equal]\ar[d,"Ff"{right=-1}]\lar[rr,path,"F\tup{!}"] &[-20pt] &[-20pt] FA_n\ar[d,"Fg"{left=-1}]\ar[dr,equal] & \\
            FA_0\ar[drr,"l_{A_0}"']\lar[r,"F!_{A_0B}"] & FB\ar[dr,"l_B"{right}]\lar[rr,"F!_{BC}"] &{}& FC\ar[dl,"l_C"{left}]\lar[r,"F!_{CA_n}"] & FA_n\ar[dll,"l_{A_n}"] \\
            && L &&
            \cellsymb(F!)[right=8]{1-2}{2-1}
            \cellsymb(F!){1-2}{2-4}
            \cellsymb(F!)[left=9]{1-4}{2-5}
            \cellsymb(l_{A_0B})[above right=-3]{2-1}{3-3}
            \cellsymb(l_{BC})[above]{2-3}{3-3}
            \cellsymb(l_{CA_n})[above=3]{2-5}{3-3}
        \end{tikzcd}
    \end{equation*}
    \begin{equation*}
        =
        \begin{tikzcd}[large]
            & FA_0\ar[dl,equal]\ar[d,"Ff"{right=-1}]\lar[rr,path,"F\tup{!}"] &[-20pt] &[-20pt] FA_n\ar[d,"Fg"{left=-1}]\ar[dr,equal] & \\
            FA_0\ar[d,equal]\lar[r,"F!_{A_0B}"] & FB\lar[rr,"F!_{BC}"] &{}& FC\lar[r,"F!_{CA_n}"] & FA_n\ar[d,equal] \\
            FA_0\ar[drr,"l_{A_0}"']\lar[rrrr,"F!_{A_0A_n}"] &&{}&& FA_n\ar[dll,"l_{A_n}"] \\
            && L &&
            \cellsymb(F!)[right=8]{1-2}{2-1}
            \cellsymb(F!){1-2}{2-4}
            \cellsymb(F!)[left=8]{1-4}{2-5}
            \cellsymb(F!){2-1}{3-5}
            \cellsymb(l_{A_0A_n})[above=-6]{3-3}{4-3}
        \end{tikzcd}
        =
        \begin{tikzcd}[tri]
            FA_0\ar[d,equal]\lar[rr,path,"F\tup{!}"] & & FA_n\ar[d,equal] \\
            FA_0\ar[dr,"l_{A_0}"']\lar[rr,"F!_{A_0A_n}"] &{}& FA_n\ar[dl,"l_{A_n}"] \\
            & L &
            \cellsymb(F!){1-1}{2-3}
            \cellsymb(l_{A_0A_n})[above=-3]{2-2}{3-2}
        \end{tikzcd}
    \end{equation*}
    \begin{equation*}
        =
        \begin{tikzcd}[hugecolumn]
            FA_0\ar[dr,"l_{A_0}"',bend right=20]\lar[r,"F!_{A_0A_1}"] & \cdots\lar[r,"F!_{A_{n-1}A_n}"] & FA_n\ar[dl,"l_{A_n}",bend left=20] \\[20pt]
            & L &
            \cellsymb(l_{A_0A_1})[above=5]{1-1}{2-2}
            \cellsymb(\cdots)[above=5]{1-2}{2-2}
            \cellsymb(l_{A_{n-1}A_n})[above=4]{1-3}{2-2}
        \end{tikzcd}\incat{\bL},
    \end{equation*}
    which shows the compatibility with 1-coary cells.
    The compatibility with 0-coary cells can be shown similarly.
\end{proof}

\begin{proposition}\label{prop:modules_loosewise_indiscrete}
    A left $F$-module with a vertex $M\in\bL$ is the same as the following data:
    \begin{itemize}
        \item
            For each object $A\in\bK$, a loose arrow $FA\larr(m_A)M$ in $\bL$.
        \item
            For objects $A,B\in\bK$, a cell $m_{AB}$ of the following form:
            \begin{equation*}
                \begin{tikzcd}
                    FA\ar[d,equal]\lar[r,"F!_{AB}"] & FB\lar[r,"m_B"] & M\ar[d,equal] \\
                    FA\lar[rr,"m_A"'] && M
                    \cellsymb(m_{AB}){1-1}{2-3}
                \end{tikzcd}\incat{\bL}.
            \end{equation*}
    \end{itemize}
    These are required to satisfy the following:
    \begin{itemize}
        \item
            For each $A\in\bK$,
            \begin{equation*}
                \begin{tikzcd}
                    FA\ar[d,equal]\ar[dr,equal]\lar[rr,"m_A"] && M\ar[d,equal] \\
                    FA\ar[d,equal]\lar[r,"F!_{AA}"'] & FA\lar[r,"m_A"'] & M\ar[d,equal] \\
                    FA\lar[rr,"m_A"'] && M
                    \cellsymb(F!)[below left]{1-1}{2-2}
                    \cellsymb(\veq)[left]{1-3}{2-2}
                    \cellsymb(m_{AA}){2-1}{3-3}
                \end{tikzcd}
                =
                \begin{tikzcd}
                    FA\ar[d,equal]\lar[r,"m_A"] & M\ar[d,equal] \\
                    FA\lar[r,"m_A"'] & M
                    \cellsymb(\veq){1-1}{2-2}
                \end{tikzcd}\incat{\bL}.
            \end{equation*}
        \item
            For $A,B,C\in\bK$,
            \begin{equation*}
                \begin{tikzcd}
                    FA\ar[d,equal]\lar[r,"F!_{AB}"] & FB\lar[r,"F!_{BC}"] & FC\ar[d,equal]\lar[r,"m_C"] & M\ar[d,equal] \\
                    FA\ar[d,equal]\lar[rr,"F!_{AC}"'] && FC\lar[r,"m_C"'] & M\ar[d,equal] \\
                    FA\lar[rrr,"m_A"'] &&& M
                    \cellsymb(F!){1-1}{2-3}
                    \cellsymb(\veq){1-3}{2-4}
                    \cellsymb(m_{AC}){2-1}{3-4}
                \end{tikzcd}
                =
                \begin{tikzcd}
                    FA\ar[d,equal]\lar[r,"F!_{AB}"] & FB\ar[d,equal]\lar[r,"F!_{BC}"] & FC\lar[r,"m_C"] & M\ar[d,equal] \\
                    FA\ar[d,equal]\lar[r,"F!_{AB}"'] & FB\lar[rr,"m_B"'] && M\ar[d,equal] \\
                    FA\lar[rrr,"m_A"'] &&& M
                    \cellsymb(\veq){1-1}{2-2}
                    \cellsymb(m_{BC}){1-2}{2-4}
                    \cellsymb(m_{AB}){2-1}{3-4}
                \end{tikzcd}\incat{\bL}.
            \end{equation*}
    \end{itemize}
\end{proposition}
\begin{proof}
    We have to show that the above data $(m_A,m_{AB})$ uniquely extend to a left $F$-module.
    If such an extension exists, for each tight arrow $f$ in $\bK$, the cell $m_f$ must be defined as follows:
    \begin{equation*}
        \begin{tikzcd}
            FA\ar[d,"Ff"']\lar[r,"m_A"] & M\ar[d,equal] \\
            FB\lar[r,"m_B"'] & M
            \cellsymb(m_f){1-1}{2-2}
        \end{tikzcd}
        \coloneq
        \begin{tikzcd}
            FA\ar[d,"Ff"']\ar[dr,equal]\lar[rr,"m_A"] & & M\ar[d,equal] \\
            FB\ar[d,equal]\lar[r,"F!_{BA}"'] & FA\lar[r,"m_A"'] & M\ar[d,equal] \\
            FB\lar[rr,"m_B"'] && M
            \cellsymb(F!)[below left]{1-1}{2-2}
            \cellsymb(\veq)[left]{1-3}{2-2}
            \cellsymb(m_{BA}){2-1}{3-3}
        \end{tikzcd}\incat{\bL}.
    \end{equation*}
    Let us define several cells in $\bL$ as follows:
    \begin{equation*}
        \beta_0\coloneq
        \begin{tikzcd}
            & FA\ar[dl,equal]\ar[d,"Ff"] \\
            FA\lar[r,"F!_{AB}"'] & FB
            \cellsymb(F!)[below right]{1-2}{2-1}
        \end{tikzcd}
        \qquad
        \delta_0\coloneq
        \begin{tikzcd}
            FA\ar[d,"Ff"']\lar[r,"F!_{AB}"] & FB\ar[d,equal] \\
            FB\lar[r,"F!_{BB}"'] & FB
            \cellsymb(F!){1-1}{2-2}
        \end{tikzcd}
        \qquad
        \eta_0\coloneq
        \begin{tikzcd}[tri]
            & FB\ar[dl,equal]\ar[dr,equal] & \\
            FB\lar[rr,"F!_{BB}"'] &{}& FB
            \cellsymb(F!){1-2}{2-2}
        \end{tikzcd}
    \end{equation*}
    \begin{equation*}
        \gamma\coloneq m_{AB}
        \qquad
        \sigma\coloneq m_{BB}
        \qquad
        \beta_1=\delta_1=\eta_1\coloneq
        \begin{tikzcd}
            M\ar[d,equal,bend right=20]\ar[d,equal,bend left=20] \\
            M
            \cellsymb(\heq){1-1}{2-1}
        \end{tikzcd}
    \end{equation*}
    Since the above cells make $m_f$ split, $m_f$ becomes cartesian by \cref{lem:split_cell_is_cartesian}.
    What remains to be shown is the compatibility of the data $(m_A,m_{AB},m_f)$ with the cells in $\bK$.
    This verification is straightforward and is therefore omitted.
\end{proof}

\begin{proposition}\label{prop:modulations_loosewise_indiscrete}
    When the shape $\bK$ of the diagram \ac{AVD}-functor $F$ is loosewise indiscrete, the condition in each definition of modulations requiring compatibility with tight arrows in $\bK$ automatically follows from the condition requiring compatibility with loose arrows in $\bK$.
\end{proposition}
\begin{proof}
    We can prove this by using \cref{prop:cocones_loosewise_indiscrete}, \cref{prop:modules_loosewise_indiscrete}, and its loosewise dual.
\end{proof}

\begin{theorem}[Strongness theorem]\label{thm:strongness_theorem}
    Let $F\colon\bK\to\bL$ be an \ac{AVD}-functor between \acp{AVDC}, and let $\bK$ be loosewise indiscrete.
    Suppose that we are given a tight cocone $\xi$ from $F$ to a vertex $\Xi\in\bL$ that satisfies the conditions \ref{axiom:loose_esssurj_left}\ref{axiom:cell_type1_left}.
    Then, $\xi_A$ has a conjoint for every $A\in\bK$, and $\xi$ becomes strong.
\end{theorem}
\begin{proof}
    Fix $K\in\bK$.
    Let us define a left $F$-module $m$ with the vertex $FK$ as follows:
    \begin{itemize}
        \item
            For each $A\in\bK$, $m_A\coloneq F!_{AK}\colon FA\larr[][1] FK$ in $\bL$.
        \item
             For $A,B\in\bK$, $m_{AB}$ is defined as the following cell:
             \begin{equation*}
                 \begin{tikzcd}
                     FA\ar[d,equal]\lar[r,"F!_{AB}"] & FB\lar[r,"F!_{BK}"] & FK\ar[d,equal] \\
                     FA\lar[rr,"F!_{AK}"'] & & FK
                     \cellsymb(F!_{ABK}){1-1}{2-3}
                 \end{tikzcd}\incat{\bL}.
             \end{equation*}
             Here, $!_{ABK}$ is a unique cell in $\bK$.
    \end{itemize}
    By \ref{axiom:loose_esssurj_left}, we have a loose arrow $\Xi\larr(q)FK$ in $\bL$ and a modulation $\compcell{\xi}q$ of type 1 whose components are cartesian as follows:
    \begin{equation*}
        \begin{tikzcd}
            F\ar[d,Rightarrow,"\xi"']\lar[r,Rightarrow,"m"] & FK\ar[d,equal] \\
            \Xi\lar[r,"q"'] & FK
            \cellsymb(\compcell{\xi}q){1-1}{2-2}
        \end{tikzcd}
        \Vline
        \begin{tikzcd}[hugecolumn]
            FA\ar[d,"\xi_A"']\lar[r,"m_A=F!_{AK}"] & FK\ar[d,equal] \\
            \Xi\lar[r,"q"'] & FK
            \cellsymb((\compcell{\xi}q)_A\colon\cart){1-1}{2-2}
        \end{tikzcd}\incat{\bL}\quad (A\in\bK).
    \end{equation*}
    We can define a modulation $\sigma$ of type 1 by $\sigma_A\coloneq\xi_{AK}$:
    \begin{equation*}
        \begin{tikzcd}
            F\ar[d,Rightarrow,"\xi"']\lar[r,Rightarrow,"m"] & FK\ar[dl,"\xi_K"] \\
            \Xi &
            \cellsymb(\sigma)[above left=1]{1-2}{2-1}
        \end{tikzcd}
        \Vline
        \begin{tikzcd}
            FA\ar[d,"\xi_A"']\lar[r,"F!_{AK}"] & FK\ar[dl,"\xi_K"] \\
            \Xi &
            \cellsymb(\xi_{AK})[above left=-1]{1-2}{2-1}
        \end{tikzcd}\incat{\bL}\quad (A\in\bK).
    \end{equation*}
    By \ref{axiom:cell_type1_left}, we have a cell $\epsilon$ corresponding to the modulation $\sigma$:
    \begin{equation*}
        \begin{tikzcd}
            \Xi\ar[d,equal]\lar[r,"q"] & FK\ar[dl,"\xi_K"] \\
            \Xi &
            \cellsymb(\epsilon)[above left=1]{1-2}{2-1}
        \end{tikzcd}\incat{\bL}.
    \end{equation*}
    Now, we shall show that $\epsilon$ is cartesian.
    Equivalently, we shall show that $q$ is a conjoint of $\xi_K$.
    To show that, let us consider the following cell $\eta$:
    \begin{equation*}
        \begin{tikzcd}[tri]
            & FK\ar[dl,"\xi_K"']\ar[dr,equal] & \\
            \Xi\lar[rr,"q"'] &{}& FK
            \cellsymb(\eta)[above=-4]{1-2}{2-2}
        \end{tikzcd}
        \coloneq
        \begin{tikzcd}[tri]
            & FK\ar[dl,equal]\ar[dr,equal] & \\
            FK\ar[d,"\xi_K"']\lar[rr,"m_K=F!_{KK}"'] &{}& FK\ar[d,equal] \\
            \Xi\lar[rr,"q"'] && FK
            \cellsymb(F!)[above=-4]{1-2}{2-2}
            \cellsymb((\compcell{\xi}q)_K){2-1}{3-3}
        \end{tikzcd}\incat{\bL}.
    \end{equation*}
    Then, one of the triangle identities can be shown as follows:
    \begin{equation*}
        \begin{tikzcd}
            & FK\ar[dl,"\xi_K"']\ar[d,equal] \\
            \Xi\ar[d,equal]\lar[r,"q"] & FK\ar[dl,"\xi_K"] \\
            \Xi &
            \cellsymb(\eta)[below right]{1-2}{2-1}
            \cellsymb(\epsilon)[above left]{2-2}{3-1}
        \end{tikzcd}
        =
        \begin{tikzcd}
            & FK\ar[dl,equal]\ar[d,equal] \\
            FK\ar[d,"\xi_K"']\lar[r,"F!_{KK}"'] & FK\ar[d,equal] \\
            \Xi\ar[d,equal]\lar[r,"q"] & FK\ar[dl,"\xi_K"] \\
            \Xi &
            \cellsymb(F!)[below right]{1-2}{2-1}
            \cellsymb((\compcell{\xi}q)_K){2-1}{3-2}
            \cellsymb(\epsilon)[above left]{3-2}{4-1}
        \end{tikzcd}
        =
        \begin{tikzcd}
            & FK\ar[dl,equal]\ar[d,equal] \\
            FK\ar[d,"\xi_K"']\lar[r,"F!_{KK}"'] & FK\ar[dl,"\xi_K"] \\
            \Xi &
            \cellsymb(F!)[below right]{1-2}{2-1}
            \cellsymb(\xi_{!_{KK}})[above left=-2]{2-2}{3-1}
        \end{tikzcd}
        \overset{(\xi)}{=}
        \begin{tikzcd}
            FK\ar[d,bend right=30,"\xi_K"{left}]\ar[d,bend left=30,"\xi_K"{right}] \\[2em]
            \Xi
            \cellsymb(\heq){1-1}{2-1}
        \end{tikzcd}\incat{\bL}.
    \end{equation*}
    We next prove the other triangle identity.
    The following calculation shows that a cell $q\to q$, which appears in the triangle identity, is sent to the identity modulation on $m=\comp{\xi}q$ by the functor $\comp{\xi}-\colon\Homcat[\bL](\Xi,FK)\arr\Mdl{F}{FK}$:
    \begin{equation*}
        \begin{tikzcd}[hugecolumn]
            FA\ar[d,"\xi_A"']\lar[r,"m_A=F!_{AK}"] & FK\ar[d,equal] \\
            \Xi\ar[d,equal]\lar[r,"q"] & FK\ar[dl,"\xi_K"{right}]\ar[d,equal] \\
            \Xi\lar[r,"q"'] & FK
            \cellsymb((\compcell{\xi}q)_A){1-1}{2-2}
            \cellsymb(\epsilon)[above left=3]{2-2}{3-1}
            \cellsymb(\eta)[below right=3]{2-2}{3-1}
        \end{tikzcd}
        =
        \begin{tikzcd}
            FA\ar[d,"\xi_A"']\lar[r,"F!_{AK}"] & FK\ar[dl,"\xi_K"{right}]\ar[d,equal] \\
            \Xi\lar[r,"q"'] & FK
            \cellsymb(\xi_{AK})[above left]{1-2}{2-1}
            \cellsymb(\eta)[below right=4]{1-2}{2-1}
        \end{tikzcd}
        =
        \begin{tikzcd}
            FA\ar[d,equal]\lar[r,"F!_{AK}"] & FK\ar[d,equal]\ar[dr,equal] & \\
            FA\ar[dr,"\xi_A"']\lar[r,"F!_{AK}"] & FK\ar[d,"\xi_K"{right=-3}]\lar[r,"F!_{KK}"'] & FK\ar[d,equal] \\
            & \Xi\lar[r,"q"'] & FK
            \cellsymb(\veq){1-1}{2-2}
            \cellsymb(F!)[below left]{1-2}{2-3}
            \cellsymb(\xi_{AK})[above right]{2-1}{3-2}
            \cellsymb((\compcell{\xi}q)_K){2-2}{3-3}
        \end{tikzcd}
    \end{equation*}
    \begin{equation*}
        \overset{(\compcell{\xi}q)}{=}
        \begin{tikzcd}
            FA\ar[d,equal]\lar[r,"F!_{AK}"] & FK\ar[d,equal]\ar[dr,equal] & \\
            FA\ar[d,equal]\lar[r,"F!_{AK}"] & FK\lar[r,"F!_{KK}"'] & FK\ar[d,equal] \\
            FA\ar[dr,"\xi_A"']\lar[rr,"F!_{AK}"] & & FK\ar[d,equal] \\
            & \Xi\lar[r,"q"'] & FK
            \cellsymb(\veq){1-1}{2-2}
            \cellsymb(F!)[below left]{1-2}{2-3}
            \cellsymb(F!_{AKK}){2-1}{3-3}
            \cellsymb((\compcell{\xi}q)_A)[left]{3-3}{4-2}
        \end{tikzcd}
        =
        \begin{tikzcd}
            FA\ar[d,"\xi_A"']\lar[r,"F!_{AK}"] & FK\ar[d,equal] \\
            \Xi\lar[r,"q"'] & FK
            \cellsymb((\compcell{\xi}q)_A){1-1}{2-2}
        \end{tikzcd}\incat{\bL}.
    \end{equation*}
    Since the functor $\comp{\xi}-$ is fully faithful, we have
    \begin{equation*}
        \begin{tikzcd}
            \Xi\ar[d,equal]\lar[r,"q"] & FK\ar[dl,"\xi_K"{right}]\ar[d,equal] \\
            \Xi\lar[r,"q"'] & FK
            \cellsymb(\epsilon)[above left=3]{1-2}{2-1}
            \cellsymb(\eta)[below right=3]{1-2}{2-1}
        \end{tikzcd}
        =
        \begin{tikzcd}
            \Xi\ar[d,equal]\lar[r,"q"] & FK\ar[d,equal] \\
            \Xi\lar[r,"q"'] & FK
            \cellsymb(\veq){1-1}{2-2}
        \end{tikzcd}\incat{\bL}.
    \end{equation*}
    Thus $q=\conj{\xi_K}$, and the cell $\epsilon$ is cartesian.

    Consequently, we have the following for any $A\in\bK$:
    \begin{equation*}
        \begin{tikzcd}
            FA\ar[d,"\xi_A"']\lar[r,"F!_{AK}"] & FK\ar[dl,"\xi_K"] \\
            \Xi &
            \cellsymb(\xi_{AK})[above left=-2]{1-2}{2-1}
        \end{tikzcd}
        =
        \begin{tikzcd}[hugecolumn]
            FA\ar[d,"\xi_A"']\lar[r,"m_A=F!_{AK}"] & FK\ar[d,equal] \\
            \Xi\ar[d,equal]\lar[r,"q"] & FK\ar[dl,"\xi_K"] \\
            \Xi &
            \cellsymb((\compcell{\xi}q)_A\colon\cart){1-1}{2-2}
            \cellsymb(\epsilon\colon\cart)[yshift=8,xshift=-8]{2-2}{3-1}
        \end{tikzcd}\colon\cart\incat{\bL}.
    \end{equation*}
    This proves that $\xi_{AK}$ is cartesian.
\end{proof}

\begin{corollary}\label{cor:unitality}
    Let $F\colon\bK\to\bL$ be an \ac{AVD}-functor between \acp{AVDC}, and let $\bK$ be loosewise indiscrete.
    Then, a vertex of a tight cocone $\xi$ from $F$ has a loose unit in $\bL$ if $\xi$ satisfies the conditions \ref{axiom:loose_esssurj_left}\ref{axiom:loose_esssurj_right}\ref{axiom:cell_type1_left}\ref{axiom:cell_type1_right}\ref{axiom:cell_type2}.
\end{corollary}
\begin{proof}
    Combine the strongness theorem (\cref{thm:strongness_theorem}) and the loosewise dual of the unitality theorem (\cref{thm:unitality_theorem}).
\end{proof}

\begin{example}[Versatile collapses]
    Let $A\coloneq(A^0\larr(A^1)A^0,A^e,A^m)$ be a monoid in an \ac{AVDC} $\bX$.
    Suppose that $A^0$ has a loose unit in $\bX$.
    Let $UA^0$ denote the monoid in $\bX$ induced by the loose unit on $A^0$, let $UA^0\larr(UA^1)UA^0$ denote the module in $\bX$ induced by $A^1$, and let $UA^e$ and $UA^m$ denote the cells in $\Mod(\bX)$ induced by $A^e$ and $A^m$, respectively.
    Now, we have a monoid $UA\coloneq(UA^0,UA^1,UA^e,UA^m)$ in $\Mod(\bX)$ and the corresponding \ac{AVD}-functor $F\colon\Idimdbl\zero{1}\to\Mod(\bX)$, where $\zero{1}$ denotes the singleton.
    Then, the monoid $A$ gives a versatile colimit of $F$, which is strong.
    This is an example of a \emph{versatile collapse} (\cref{def:specific_versatile_colimits}), and is restated in \cref{thm:every_obj_is_versatile_colim}\cref{thm:every_obj_is_versatile_colim-collapse}.
\end{example}

\begin{example}
    Consider the \ac{AVDC} $\Rel$ of relations as in \cref{eg:avdc_of_relations}.
    Let $R\subseteq\zero{X}\times\zero{X}$ be an equivalence relation on a (large) set $\zero{X}$.
    Since a monoid in $\Rel$ is simply a (large) preordered set, we have an \ac{AVD}-functor $F\colon\Idimdbl\zero{1}\to\Rel$ corresponding to $R$.
    Then, the quotient set $\zero{X}/R$ becomes a versatile colimit (collapse) of $F$.
    However, such a versatile colimit does not exist in general unless the relation $R$ is an equivalence relation.
    Indeed, given a preorder $\le$ on $\zero{X}$, we can consider the smallest equivalence relation $\langle\le\rangle$ containing $\le$, but for the quotient $\zero{X}\arr(\pi)[][1]\zero{X}/\langle\le\rangle$ to be a versatile collapse, the strongness theorem (\cref{thm:strongness_theorem}) requires that the following cell is cartesian, which is equivalent to the equality $\le = \langle\le\rangle$:
    \begin{equation*}
        \begin{tikzcd}[tri]
            \zero{X}\ar[dr,"\pi"']\lar[rr,"\le"] &{}& \zero{X}\ar[dl,"\pi"] \\
            & \zero{X}/\langle\le\rangle &
            \cellsymb(\cdot)[above=-2]{1-2}{2-2}
        \end{tikzcd}\incat{\Rel}.
    \end{equation*}
    Hence, a preorder on a (large) set admits a versatile collage in $\Rel$ if and only if it is an equivalence relation.
\end{example}
\section{Axiomatization of double categories of profunctors}\label{sec:axioma}
This section is devoted to our main theorem: the characterization of the \acp{AVDC} of the forms $\Prof[\bX]$, $\Mod(\bX)$, and $\Mat[\bX]$.
After discussing the existence of several versatile colimits in \cref{subsec:formal_const} and ``density'' with respect to them in \cref{subsec:density}, we show in \cref{subsec:characterization} that these two properties characterize the \acp{AVDC} $\Prof[\bX]$, $\Mod(\bX)$, and $\Mat[\bX]$.
Finally, in \cref{subsec:slicing}, we apply the characterization to show that the classes of \acp{AVDC} $\Prof[\bX]$ and $\Mod(\bX)$ are closed under slicing.
\subsection{The formal construction of enriched categories}\label{subsec:formal_const}
\begin{notation}\label{note:avd_functor_induced_by_cat_mon_coloredset}
    Let $\bX$ be an \ac{AVDC} with loose units, and let $\one{A}$ be an $\bX$-enriched large category.
    We now regard $\one{A}$ as an \ac{AVD}-functor $\one{A}\colon\Idimdbl(\Ob\one{A})\to\bX$ as in \cref{prop:enriched_cat_is_AVD_functor}, where $\Ob\one{A}$ denotes the large set of objects in $\one{A}$.
    Then, we obtain an \ac{AVD}-functor $F_\one{A}\colon\Idimdbl(\Ob\one{A})\to\Prof[\bX]$ by composing the embedding $Z$ as in \cref{note:AVD_functor_Z}:
    \begin{equation*}
        \begin{tikzcd}
            \Idimdbl(\Ob\one{A})\ar[dr,"F_\one{A}"']\ar[r,"\one{A}"] & \bX\ar[d,"Z"] \\
            & \Prof[\bX]
        \end{tikzcd}
    \end{equation*}
    Similarly to \cref{prop:enriched_cat_is_AVD_functor}, every object in $\Mod(\bX)$ or $\Mat[\bX]$ can also be regarded as an \ac{AVD}-functor to the \ac{AVDC} $\bX$.
    Indeed, a monoid $M$ in $\bX$ is the same as an \ac{AVD}-functor $M\colon\Idimdbl\zero{1}\to\bX$.
    An $\bX$-colored large set $\zero{A}$ can be regarded as an \ac{AVD}-functor $\abs{\cdot}_\zero{A}\colon\Ddbl\zero{A}\to\bX$, which represents the coloring map.
    Then, we obtain an \ac{AVD}-functor $F_M$ by composing the embedding $U$ if $\bX$ has loose units, and also obtain $F_\zero{A}$ by composing the embedding $Y$:
    \begin{equation*}
        \begin{tikzcd}
            \Idimdbl\zero{1}\ar[dr,"F_M"']\ar[r,"M"] & \bX\ar[d,"U"] \\
            & \Mod(\bX)
        \end{tikzcd}
        \qquad
        \begin{tikzcd}
            \Ddbl\zero{A}\ar[dr,"F_\zero{A}"']\ar[r,"\abs{\cdot}_\zero{A}"] & \dim{\bX}\ar[d,"Y"] \\
            & \Mat[\bX]
        \end{tikzcd}
    \end{equation*}
\end{notation}

It will be shown in \cref{thm:every_obj_is_versatile_colim} that $\bX$-colored sets, monoids in $\bX$, and $\bX$-enriched categories are (\ac{VD}-)versatile colimits of their corresponding diagrams considered in \cref{note:avd_functor_induced_by_cat_mon_coloredset}.
We now give special names to these versatile colimits:
\begin{definition}\label{def:specific_versatile_colimits}\quad
    \begin{enumerate}
        \item
            A \emph{(\ac{VD}-)versatile coproduct} is a (\ac{VD}-)versatile colimit of an \ac{AVD}-functor from $\Ddbl\zero{S}$ for some set $\zero{S}$.
            It is called \emph{large} if the set $\zero{S}$ is large, and it is called a \emph{(\ac{VD}-)versatile initial object} when $\zero{S}$ is the empty set.
        \item
            A \emph{(\ac{VD}-)versatile collapse} is a (\ac{VD}-)versatile colimit of an \ac{AVD}-functor from $\Idimdbl\zero{1}$, where $\zero{1}$ denotes the singleton.
        \item
            A \emph{(\ac{VD}-)versatile collage} is a (\ac{VD}-)versatile colimit of an \ac{AVD}-functor from $\Idimdbl\zero{S}$ for some set $\zero{S}$.
            It is called \emph{large} if the set $\zero{S}$ is large.\qedhere
    \end{enumerate}
\end{definition}

\begin{remark}
    The term ``collapse'' has been used for similar concepts in a virtual equipment.
    For a monoid $M$ in a virtual equipment, a tight cocone from $M$ satisfying \ref{axiom:tight_bij} is called a ``collapse'' in \cite{Schultz2015regular}.
    The same term is also used in \cite{Arkor2024nervetheorem} for a tight cocone from the monoid satisfying a stronger condition, which coincides with our term ``versatile collapse.''
\end{remark}

\begin{lemma}\label{lem:Mat_has_versatile_coproducts}
    For any \ac{AVDC} $\bX$, $\Mat[\bX]$ has all large \ac{VD}-versatile coproducts.
\end{lemma}
\begin{proof}
    Let $(\zero{A}_i)_{i\in\zero{S}}$ be $\bX$-colored large sets indexed by a large set $\zero{S}$.
    Let $\Xi$ be a (large) disjoint union of $(\zero{A}_i)_{i\in\zero{S}}$, and let $\zero{A}_i\arr(\xi_i)[][1]\Xi$ denote the coprojections.
    We write $\ie{i}{x}$ for an element of $\Xi$, where $x\in\zero{A}_i$, and define its color by $\abs{\ie{i}{x}}\coloneq\abs{x}$.

    We have to show that $\Xi$ is a versatile coproduct of $(\zero{A}_i)_{i\in\zero{S}}$.
    The condition \ref{axiom:tight_bij} follows clearly by the construction.
    Since the tight arrow part of $\xi_i(x)$ for each $x\in \zero{A}_i$ is the identity, $\xi_i$ is pulling in $\Mat[\bX]$.
    The remaining conditions \ref{axiom:loose_esssurj_left}\ref{axiom:loose_esssurj_right}\ref{axiom:cell_type1_left}\ref{axiom:cell_type1_right}\ref{axiom:cell_type2}\ref{axiom:cell_type3} follow directly from the structure of $\Xi$ as a disjoint union.
\end{proof}

\begin{theorem}\label{thm:construction_versatile_colim}
    Let $\bX$ be an \ac{AVDC}, and let $\one{C}$ be a category.
    If $\bX$ has \ac{VD}-versatile colimits of all \ac{AVD}-functors $\Ddbl\one{C}\arr[][1]\bX$, then $\Mod(\bX)$ has versatile colimits of all \ac{AVD}-functors $\Idimdbl\one{C}\arr[][1]\Mod(\bX)$.
\end{theorem}
\begin{proof}
    Let $A\colon \Idimdbl\one{C}\to\Mod(\bX)$ be an \ac{AVD}-functor.
    Now, $A$ assigns to each object $i\in\one{C}$, a monoid $A_i=(A_i^0\larr(A_i^1)A_i^0, A_i^e, A_i^m)$ in $\bX$, where $A_i^e$ is the unit and $A_i^m$ is the multiplication.
    $A$ also assigns to each morphism $i\arr(f)[][1]j$ in $\one{C}$, a monoid homomorphism $A_f=(A_i^0\arr(A_f^0)[][1]A_j^0,A_f^1)$; to each pair $(i,j)$ of $i,j\in\one{C}$, a bimodule $A_{ij}=(A_i^0\larr(A_{ij}^1)A_j^0, A_{ij}^l, A_{ij}^r)$ in $\bX$, where $A_{ij}^l$ and $A_{ij}^r$ are the left action and the right action, respectively.

    Let $F\colon\Idimdbl\one{C}\arr[][1]\bX$ denote the \ac{AVD}-functor given by composing $A$ with the forgetful functor $\dim{\Mod(\bX)}\arr[][1]\bX$.
    Let $G\colon\Ddbl\one{C}\arr[][1]\bX$ denote the \ac{AVD}-functor given by composing $F$ with the inclusion $\Ddbl\one{C}\arr[][1]\Idimdbl\one{C}$.
    Let us take a \ac{VD}-versatile colimit $(A_i^0\arr(\xi_i^0)\Xi^0)_i$ in $\bX$ of $G$.
    For each $i\in\one{C}$, the families $(A_i^0\larr(A_{ij}^1)A_j^0)_j$ and $(A_j^0\larr(A_{ji}^1)A_i^0)_j$ yield a right $G$-module and a left $G$-module, respectively.
    Here, \cref{cor:abs_cart_in_loosewise_indiscrete} is used to show that the underlying cells of these modules associated to tight arrows are cartesian.
    By \ref{axiom:loose_esssurj_right} and \ref{axiom:loose_esssurj_left}, there exist two loose arrows $A_i^0\larr(\hq_i)[][1]\Xi^0\larr(\hp_i)[][1]A_i^0$ in $\bX$ and modulations $\hq_i\conjcell{\xi^0}$ and $\compcell{\xi^0}\hp_i$ of type 1 whose components are cartesian:
    \begin{equation*}
        \begin{tikzcd}[xhugecolumn]
            A_i^0\ar[d,equal]\lar[r,"A_{ij}^1"] &[5pt] A_j^0\ar[d,"\xi_j^0"{description}]\lar[r,"A_{ji}^1"] &[5pt] A_i^0\ar[d,equal] \\
            A_i^0\lar[r,"\hq_i"'] & \Xi^0\lar[r,"\hp_i"'] & A_i^0
            \cellsymb((\hq_i\conjcell{\xi^0})_j\colon\cart){1-1}{2-2}
            \cellsymb((\compcell{\xi^0}\hp_i)_j\colon\cart){1-2}{2-3}
        \end{tikzcd}\incat{\bX}\quad (i,j\in\one{C}).
    \end{equation*}
    By \ref{axiom:cell_type0_right} for $\Xi^0$, there exist, for each $i,j\in\one{C}$, a unique cell $\hq_{ij}$ in $\bX$ corresponding to a modulation of type 0 with components on the right below:
    \begin{equation*}
        \begin{tikzcd}
            A_i^0\ar[d,equal]\lar[r,"A_{ij}^1"] & A_j^0\lar[r,"\hq_j"] & \Xi^0\ar[d,equal] \\
            A_i^0\lar[rr,"\hq_i"'] & & \Xi^0
            \cellsymb(\hq_{ij}){1-1}{2-3}
        \end{tikzcd}\incat{\bX}
        \Vline
        \begin{tikzcd}
            A_i^0\ar[d,equal]\lar[r,"A_{ij}^1"] & A_j^0\lar[r,"A_{jk}^1"] & A_k^0\ar[d,equal] \\
            A_i^0\lar[rr,"A_{ik}^1"'] & & A_k^0
            \cellsymb(A!){1-1}{2-3}
        \end{tikzcd}\incat{\bX}\quad (k\in\one{C})
    \end{equation*}
    Then, $(\hq_i,\hq_{ij})$ uniquely extends to a left $F$-module $\hq$ by \cref{prop:modules_loosewise_indiscrete} and \ref{axiom:cell_type0_right} for $\Xi^0$.
    In particular, $\hq$ is also a left $G$-module.
    Thus, by \ref{axiom:loose_esssurj_left} for $\Xi^0$, we obtain a unique loose arrow $\Xi^1$ in $\bX$ and a modulation $\compcell{\xi^0}\Xi^1$ of type 1 whose components are cartesian:
    \begin{equation*}
        \begin{tikzcd}[xhugecolumn]
            A_i^0\ar[d,"\xi_i^0"']\lar[r,"\hq_i"] & \Xi^0\ar[d,equal] \\
            \Xi^0\lar[r,"\Xi^1"'] & \Xi^0
            \cellsymb((\compcell{\xi^0}\Xi^1)_i\colon\cart){1-1}{2-2}
        \end{tikzcd}\incat{\bX}\quad (i\in\one{C}).
    \end{equation*}
    In the same way, we can construct a right $F$-module $\hp=(\hp_i,\hp_{ij})$, a loose arrow ${\Xi^1}'$, and a modulation ${\Xi^1}'\conjcell{\xi^0}$ of type 1 whose components are cartesian.
    By replacing $\hp_i$ appropriately, we can assume $\Xi^1={\Xi^1}'$ without loss of generality, using \cref{prop:pasting_lemma_cartesian}.
    We now have cartesian cells as follows:
    \begin{equation}\label{eq:restriction_of_Xi^1}
        \begin{tikzcd}
            A_i^0\ar[d,"\xi_i^0"']\lar[r,"A_{ij}^1"] & A_j^0\ar[d,"\xi_j^0"] \\
            \Xi^0\lar[r,"\Xi^1"'] & \Xi^0
            \cellsymb(\cart){1-1}{2-2}
        \end{tikzcd}
        =
        \begin{tikzcd}[xhugecolumn]
            A_i^0\ar[d,equal]\lar[r,"A_{ij}^1"] & A_j^0\ar[d,"\xi_j^0"] \\
            A_i^0\ar[d,"\xi_i^0"']\lar[r,"\hq_i"] & \Xi^0\ar[d,equal] \\
            \Xi^0\lar[r,"\Xi^1"'] & \Xi^0
            \cellsymb((\hq_i\conjcell{\xi^0})_j\colon\cart){1-1}{2-2}
            \cellsymb((\compcell{\xi^0}\Xi^1)_i\colon\cart){2-1}{3-2}
        \end{tikzcd}
        =
        \begin{tikzcd}[xhugecolumn]
            A_i^0\ar[d,"\xi_i^0"']\lar[r,"A_{ij}^1"] & A_j^0\ar[d,equal] \\
            \Xi^0\ar[d,equal]\lar[r,"\hp_j"] & A_j^0\ar[d,"\xi_j^0"] \\
            \Xi^0\lar[r,"\Xi^1"'] & \Xi^0
            \cellsymb((\compcell{\xi^0}\hp_j)_i\colon\cart){1-1}{2-2}
            \cellsymb((\Xi^1\conjcell{\xi^0})_j\colon\cart){2-1}{3-2}
        \end{tikzcd}\incat{\bX}\quad (i,j\in\one{C}).
    \end{equation}

    By \ref{axiom:cell_type2} for $\Xi^0$, we have a unique cell $\Xi^e$ below:
    \begin{equation*}
        \begin{tikzcd}[tri]
            & A_i^0\ar[d,"\xi_i^0"{left},bend right=20]\ar[d,"\xi_i^0"{right},bend left=20] & \\
            & \Xi^0\ar[dl,equal]\ar[dr,equal] & \\
            \Xi^0\lar[rr,"\Xi^1"'] &{}& \Xi^0
            \cellsymb(\heq){1-2}{2-2}
            \cellsymb(\Xi^e){2-2}{3-2}
        \end{tikzcd}
        =
        \begin{tikzcd}[tri]
            & A_i^0\ar[dl,equal]\ar[dr,equal] & \\
            A_i^0\ar[d,"\xi_i^0"']\lar[rr,"A_{ii}^1"'] &{}& A_i^0\ar[d,"\xi_i^0"] \\
            \Xi^0\lar[rr,"\Xi^1"'] && \Xi^0
            \cellsymb(A!){1-2}{2-2}
            \cellsymb(\cart){2-1}{3-3}
        \end{tikzcd}\incat{\bX}\quad (i\in\one{C}).
    \end{equation*}
    By \ref{axiom:cell_type3}, \ref{axiom:cell_type0_left}, and \ref{axiom:cell_type0_right} for $\Xi^0$, we have a unique cell $\Xi^m$ below:
    \begin{equation*}
        \begin{tikzcd}
            A_i^0\ar[d,"\xi_i^0"']\lar[r,"A_{ij}^1"] & A_j^0\ar[d,"\xi_j^0"']\lar[r,"A_{jk}^1"] & A_k^0\ar[d,"\xi_k^0"] \\
            \Xi^0\ar[d,equal]\lar[r,"\Xi^1"'] & \Xi^0\lar[r,"\Xi^1"'] & \Xi^0\ar[d,equal] \\
            \Xi^0\lar[rr,"\Xi^1"'] && \Xi^0
            \cellsymb(\cart){1-1}{2-2}
            \cellsymb(\cart){1-2}{2-3}
            \cellsymb(\Xi^m){2-1}{3-3}
        \end{tikzcd}
        =
        \begin{tikzcd}
            A_i^0\ar[d,equal]\lar[r,"A_{ij}^1"] & A_j^0\lar[r,"A_{jk}^1"] & A_k^0\ar[d,equal] \\
            A_i^0\ar[d,"\xi_i^0"']\lar[rr,"A_{ik}^1"'] && A_k^0\ar[d,"\xi_k^0"] \\
            \Xi^0\lar[rr,"\Xi^1"'] && \Xi^0
            \cellsymb(A!){1-1}{2-3}
            \cellsymb(\cart){2-1}{3-3}
        \end{tikzcd}\incat{\bX}\quad (i,j,k\in\one{C}).
    \end{equation*}
    Using the functoriality of $A$ and the universal property of \ac{VD}-versatile colimits, we can verify that $(\Xi^0,\Xi^1,\Xi^e,\Xi^m)$ becomes a monoid $\Xi$ in $\bX$.

    By the naturality axiom of cells in $\Mod(\bX)$, the following two composites of cells coincide:
    \begin{equation*}
        \begin{tikzcd}[tri]
            & A_i^0\ar[dl,equal]\lar[rr,"A_i^1"] & & A_i^0\ar[dl,equal]\ar[dr,equal] & \\
            A_i^0\ar[d,equal]\lar[rr,"A_i^1"'] &{}& A_i^0\lar[rr,"A_{ii}^1"'] &{}& A_i^0\ar[d,equal] \\
            A_i^0\lar[rrrr,"A_{ii}^1"'] &&&& A_i^0
            \cellsymb(\veq){1-2}{2-3}
            \cellsymb(A!){1-4}{2-4}
            \cellsymb(A_{ii}^l){2-1}{3-5}
        \end{tikzcd}
        =
        \begin{tikzcd}[tri]
            & A_i^0\ar[dl,equal]\ar[dr,equal]\lar[rr,"A_i^1"] & & A_i^0\ar[dr,equal] & \\
            A_i^0\ar[d,equal]\lar[rr,"A_{ii}^1"'] &{}& A_i^0\lar[rr,"A_i^1"'] &{}& A_i^0\ar[d,equal] \\
            A_i^0\lar[rrrr,"A_{ii}^1"'] &&&& A_i^0
            \cellsymb(A!){1-2}{2-2}
            \cellsymb(\veq){1-4}{2-3}
            \cellsymb(A_{ii}^r){2-1}{3-5}
        \end{tikzcd}\incat{\bX}.
    \end{equation*}
    Let $\xi_i^1$ be a cell obtained by the tightwise composite of the above cell and the cell \cref{eq:restriction_of_Xi^1} with $i=j$.
    Then, we can verify that $(\xi_i^0,\xi_i^1)$ becomes a tight arrow $A_i\arr(\xi_i)[][1]\Xi$ in $\Mod(\bX)$ for each $i\in\one{C}$.

    Since restrictions in $\Mod(\bX)$ inherit from 1-coary ones in $\bX$, for objects $i,j\in\one{C}$, the cell \cref{eq:restriction_of_Xi^1} yields a cartesian cell $\xi_{ij}$ in $\Mod(\bX)$ of the following form:
    \begin{equation*}
        \begin{tikzcd}[tri]
            A_i\ar[dr,"\xi_i"']\lar[rr,"A_{ij}"] &{}& A_j\ar[dl,"\xi_j"] \\
            & \Xi &
            \cellsymb(\xi_{ij})[above=-5]{1-2}{2-2}
        \end{tikzcd}\colon\cart\incat{\Mod(\bX)}.
    \end{equation*}
    Then, the data $(\xi_i,\xi_{ij})_{i,j}$ yield a tight cocone $\xi$ from $A$ with the vertex $\Xi\in\Mod(\bX)$ by \cref{prop:cocones_loosewise_indiscrete}.
    Indeed, the second condition required by \cref{prop:cocones_loosewise_indiscrete} follows from the construction of the cell $\Xi^m$, and the first one, the compatibility with morphisms in $\one{C}$, follows from the construction of $\Xi^e$ and the compatibility of the modulations $\compcell{\xi^0}\hp_j$ (or $\hq_i\conjcell{\xi^0}$) with them.

    We should show that $\xi$ is a versatile colimit of $A$.
    By \cref{prop:restriction_in_Mod}, $\Mod(\bX)$ has loose units, and we can apply \cref{thm:VDverscolim_unital}.
    Let us begin with the verification of \ref{axiom:tight_bij} for $\xi$.
    Let $l=(l_i,l_{ij})_{i,j}$ be a tight cocone from $A$ with a vertex $L\in\Mod(\bX)$.
    By \ref{axiom:tight_bij} for the versatile colimit $\Xi^0$, there is a unique tight arrow $\Xi^0\arr(k^0)[][1]L^0$ in $\bX$ such that, for all $i$, $\xi_i^0\tcomp k^0=l_i^0$.
    By \ref{axiom:cell_type1_left} and \ref{axiom:cell_type1_right} for $\Xi^0$, there is a unique cell $k^1$ as follows:
    \begin{equation*}
        \begin{tikzcd}
            A_i^0\ar[d,"\xi_i^0"']\lar[r,"A_{ij}^1"] & A_j^0\ar[d,"\xi_j^0"] \\
            \Xi^0\ar[d,"k^0"']\lar[r,"\Xi^1"] & \Xi^0\ar[d,"k^0"] \\
            L^0\lar[r,"L^1"'] & L^0
            \cellsymb(\xi_{ij}\colon\cart){1-1}{2-2}
            \cellsymb(k^1){2-1}{3-2}
        \end{tikzcd}
        =
        \begin{tikzcd}
            A_i^0\ar[d,"l_i^0"']\lar[r,"A_{ij}^1"] & A_j^0\ar[d,"l_j^0"] \\
            L^0\lar[r,"L^1"'] & L^0
            \cellsymb(l_{ij}){1-1}{2-2}
        \end{tikzcd}\incat{\bX}\quad (i,j\in\one{C}).
    \end{equation*}
    Using \ref{axiom:cell_type2}\ref{axiom:cell_type1_left}\ref{axiom:cell_type1_right}\ref{axiom:cell_type3} for $\Xi^0$, we can verify that $(k^0,k^1)$ becomes a tight arrow $\Xi\arr(k)[][1]L$ in $\Mod(\bX)$ and that it is a unique one satisfying $\xi\tcomp k=l$.

    We next show \ref{axiom:loose_esssurj_left} for $\xi$.
    Since $\xi^0_i$ are pulling in $\bX$ and since $\Mod(\bX)$ inherits 1-coary restrictions from $\bX$ by \cref{prop:restriction_in_Mod}, $\xi_i$ become pulling in $\Mod(\bX)$.
    Let $m=(m_i,m_{ij})_{i,j}$ be a left $A$-module with a vertex $M\in\Mod(\bX)$.
    By \ref{axiom:loose_esssurj_left} for $\Xi^0$, there are loose arrow $p^1$ and cartesian cells $\sigma_i$ in $\bX$ being a modulation of type 1:
    \begin{equation*}
        \begin{tikzcd}
            A_i^0\ar[d,"\xi_i^0"']\lar[r,"m_i^1"] & M^0\ar[d,equal] \\
            \Xi^0\lar[r,"p^1"'] & M^0
            \cellsymb(\sigma_i\colon\cart){1-1}{2-2}
        \end{tikzcd}\incat{\bX}\quad (i\in\one{C}).
    \end{equation*}
    By \ref{axiom:cell_type3} and \ref{axiom:cell_type0_left} for $\Xi^0$, there exists a unique cell $p^l$ in $\bX$ satisfying the following:
    \begin{equation*}
        \begin{tikzcd}
            A_i^0\ar[d,"\xi_i^0"']\lar[r,"A_{ij}^1"] & A_j^0\ar[d,"\xi_j^0"]\lar[r,"m_j^1"] & M^0\ar[d,equal] \\
            \Xi^0\ar[d,equal]\lar[r,"\Xi^1"] & \Xi^0\lar[r,"p^1"] & M^0\ar[d,equal] \\
            \Xi^0\lar[rr,"p^1"'] & & M^0
            \cellsymb(\xi_{ij}){1-1}{2-2}
            \cellsymb(\sigma_j){1-2}{2-3}
            \cellsymb(p^l){2-1}{3-3}
        \end{tikzcd}
        =
        \begin{tikzcd}
            A_i^0\ar[d,equal]\lar[r,"A_{ij}^1"] & A_j^0\lar[r,"m_j^1"] & M^0\ar[d,equal] \\
            A_i^0\ar[d,"\xi_i^0"']\lar[rr,"m_i^1"] & & M^0\ar[d,equal] \\
            \Xi^0\lar[rr,"p^1"'] & & M^0
            \cellsymb(m_{ij}){1-1}{2-3}
            \cellsymb(\sigma_i){2-1}{3-3}
        \end{tikzcd}\incat{\bX}\quad (i,j\in\one{C}).
    \end{equation*}
    By \ref{axiom:cell_type0_left} for $\Xi^0$, there exists a unique cell $p^r$ in $\bX$ corresponding to a modulation of type 0 on the right below:
    \begin{equation*}
        \begin{tikzcd}
            \Xi^0\ar[d,equal]\lar[r,"p^1"] & M^0\lar[r,"M^1"] & M^0\ar[d,equal] \\
            \Xi^0\lar[rr,"p^1"'] & & M^0
            \cellsymb(p^r){1-1}{2-3}
        \end{tikzcd}\incat{\bX}
        \Vline
        \begin{tikzcd}
            A_i^0\ar[d,equal]\lar[r,"m_i^1"] & M^0\lar[r,"M^1"] & M^0\ar[d,equal] \\
            A_i^0\lar[rr,"m_i^1"'] & & M^0
            \cellsymb(m_i^r){1-1}{2-3}
        \end{tikzcd}\incat{\bX}\quad (i\in\one{C})
    \end{equation*}
    Then, $p\coloneq(p^1,p^l,p^r)$ and the cells $\sigma_i$ form a loose arrow and cells in $\Mod(\bX)$.
    Then, we can verify that the cells $\sigma_i$ become a modulation (of type 1), which shows \ref{axiom:loose_esssurj_left} for $\xi$.
    The loosewise dual \ref{axiom:loose_esssurj_right} also follows similarly.
    What remains to show is that $\xi$ satisfies the condition \ref{axiom:cell_type3} for 1-coary modulations, which follows from the corresponding condition of $\Xi^0$ directly.
\end{proof}

\begin{corollary}
    For any \ac{AVDC} $\bX$, $\Mod(\bX)$ has all versatile collapses.
\end{corollary}
\begin{proof}
    Since versatile colimits for the shape $\Ddbl\zero{1}$ are trivial, this follows from \cref{thm:construction_versatile_colim}.
\end{proof}

\begin{corollary}\label{cor:prof_has_collage}
    For any \ac{AVDC} $\bX$, $\Prof[\bX]$ has all large versatile collages.
\end{corollary}
\begin{proof}
    Combine \cref{lem:Mat_has_versatile_coproducts,thm:construction_versatile_colim}.
\end{proof}

\begin{theorem}\label{thm:every_obj_is_versatile_colim}
    Let $\bX$ be an \ac{AVDC}.
    \begin{enumerate}
        \item
            Every $\bX$-colored large set $\zero{A}$ is a \ac{VD}-versatile coproduct of \mbox{$F_\zero{A}\colon\Ddbl\zero{A}\to\Mat[\bX]$} in \cref{note:avd_functor_induced_by_cat_mon_coloredset}.
        \item\label{thm:every_obj_is_versatile_colim-collapse}
            If $\bX$ has loose units, then every monoid $M$ in $\bX$ is a versatile collapse of \mbox{$F_M\colon\Idimdbl\zero{1}\to\Mod(\bX)$} in \cref{note:avd_functor_induced_by_cat_mon_coloredset}.
        \item
            If $\bX$ has loose units, then every $\bX$-enriched large category $\one{A}$ is a versatile collage of \mbox{$F_\one{A}\colon\Idimdbl(\Ob\one{A})\to\Prof[\bX]$} in \cref{note:avd_functor_induced_by_cat_mon_coloredset}.
    \end{enumerate}
\end{theorem}
\begin{proof}
    These are special cases of the construction in the proof of \cref{lem:Mat_has_versatile_coproducts,thm:construction_versatile_colim}.
\end{proof}
\subsection{Density}\label{subsec:density}
In this subsection, we study a key property of the embeddings $\dim{\bX}\arr(Y)[hook][1]\Mat[\bX]$, $\bX\arr(U)[hook][1]\Mod(\bX)$, and $\bX\arr(Z)[hook][1]\Prof[\bX]$, which we call ``density.''
This property consists of two conditions: first, every object in each codomain \ac{AVDC} can be written as a corresponding versatile colimit of objects from $\bX$, as shown in the previous subsection; second, every object from $\bX$ satisfies a suitable ``atomicity'' condition with respect to such colimits.
\begin{definition}
    Let $\bL$ be an \ac{AVDC}.
    An object $A\in\bL$ is called \emph{collage-atomic} (resp.\ \emph{coproduct-atomic}) if, for any large versatile collage (resp.\ large \ac{VD}-versatile coproduct) $\Xi\in\bL$ of $F\colon\Idimdbl\zero{S}\to\bL$ (resp.\ $\Ddbl\zero{S}\to\bL$), every tight arrow $A\arr(f)[][1]\Xi$ in $\bL$ uniquely factors through a unique coprojection $Fc\arr(\xi_c)[][1]\Xi$:
    \begin{equation*}
        \begin{tikzcd}[small]
            & A\ar[dl,"\exists !"']\ar[dd,"f"] \\
            Fc\ar[dr,"\xi_c"'] & {} \\
            & \Xi
            \cellsymb(\heq){2-1}{2-2}
        \end{tikzcd}\incat{\bL}\quad (\exists ! c\in\zero{S}).
    \end{equation*}
\end{definition}

\begin{definition}
    Let $\bL$ be an \ac{AVDC}.
    An object $A\in\bL$ is called \emph{collapse-atomic} if, for any versatile collapse $\Xi\in\bL$ of a monoid $B=(B^0,B^1,B^e,B^m)$ in $\bL$, every tight arrow $A\arr(f)[][1]\Xi$ in $\bL$ uniquely factors through the coprojection $B^0\arr(\xi)[][1]\Xi$:
    \begin{equation*}
        \begin{tikzcd}[small]
            & A\ar[dl,"\exists !"']\ar[dd,"f"] \\
            B^0\ar[dr,"\xi"'] & {} \\
            & \Xi
            \cellsymb(\heq){2-1}{2-2}
        \end{tikzcd}\incat{\bL}.
    \end{equation*}
\end{definition}

\begin{proposition}\label{prop:atomic_obj_in_Mat}
    Let $\bX$ be an \ac{AVDC}.
    Then, $\zero{A}\in\Mat[\bX]$ is coproduct-atomic if and only if it is tightwise isomorphic to $\zero{Y}_c$ for some $c\in\bX$.
\end{proposition}
\begin{proof}
    Let $\Xi$ be a \ac{VD}-versatile coproduct of a large family $\zero{A}_i\in\Mat[\bX]$.
    By \cref{lem:Mat_has_versatile_coproducts}, $\Xi$ is a disjoint union of $(\zero{A}_i)_i$.
    Thus, it immediately follows that $\zero{Y}_c$ is coproduct-atomic since the underlying set of $\zero{Y}_c$ is the singleton.

    To prove the converse direction, take a coproduct-atomic $\bX$-colored large set $\zero{A}$ arbitrarily.
    By \cref{thm:every_obj_is_versatile_colim}, $\zero{A}$ can be regarded as a large \ac{VD}-versatile coproduct of objects of the form $\zero{Y}_c$ $(c\in\bX)$.
    Since $\zero{A}$ is coproduct-atomic, the identity tight arrow on $\zero{A}$ factors through some coprojection $\zero{Y}_c\arr(x)[][1]\zero{A}$:
    \begin{equation*}
        \begin{tikzcd}[small]
            & \zero{A}\ar[dd,equal]\ar[dl,"\exists! K"'] \\
            \zero{Y}_c\ar[dr,"x"'] & {} \\
            & \zero{A}
            \cellsymb(\heq){2-1}{2-2}
        \end{tikzcd}\incat{\Mat[\bX]}.
    \end{equation*}
    Since $\zero{Y}_c$ is also coproduct-atomic, the tight arrow $x$ must uniquely factor through itself.
    Thus we have $x\tcomp K=\id$ and $\zero{A}\cong\zero{Y}_c$.
\end{proof}

A similar proof to \cref{prop:atomic_obj_in_Mat} works for the following propositions:
\begin{proposition}\label{prop:atomic_obj_in_Mod}
    Let $\bX$ be an \ac{AVDC} with loose units.
    Then, $A\in\Mod(\bX)$ is collapse-atomic if and only if it is tightwise isomorphic to the trivial monoid $U_c$ as in \cref{note:AVD_functor_U} for some $c\in\bX$.
\end{proposition}

\begin{proposition}\label{prop:atomic_obj_in_Prof}
    Let $\bX$ be an \ac{AVDC} with loose units.
    An $\bX$-enriched large category is collage-atomic in $\Prof[\bX]$ if and only if it is tightwise isomorphic to a preobject classifier $\one{Z}_c$ for some $c\in\bX$.
\end{proposition}

\begin{definition}
    Let $\bL$ be an \ac{AVDC}.
    A full sub-\ac{AVDC} $\bX\subseteq\bL$ is called \emph{collage-dense} (resp.\ \emph{coproduct-dense}; \emph{collapse-dense}) if it satisfies following:
    \begin{itemize}
        \item
            Every object in $\bX$ is collage-atomic (resp.\ coproduct-atomic; collapse-atomic) in $\bL$.
        \item
            Every object in $\bL$ can be written as a large versatile collage (resp.\ a large \ac{VD}-versatile coproduct; a versatile collapse) of objects from $\bX$.\qedhere
    \end{itemize}
\end{definition}

\begin{proposition}\label{prop:dense_fullsub}
    Let $\bX$ be an \ac{AVDC}.
    \begin{enumerate}
        \item
            The full sub-\ac{AVDC} given by $\dim{\bX}\arr(Y)[hook]\Mat[\bX]$ as in \cref{note:AVD_functor_Y} is coproduct-dense.
        \item
            If $\bX$ has loose units, the full sub-\ac{AVDC} given by $\bX\arr(U)[hook]\Mod(\bX)$ as in \cref{note:AVD_functor_U} is collapse-dense.
        \item
            If $\bX$ has loose units, the full sub-\ac{AVDC} given by $\bX\arr(Z)[hook]\Prof[\bX]$ as in \cref{note:AVD_functor_Z} is collage-dense.
    \end{enumerate}
\end{proposition}
\begin{proof}
    These follow from \cref{thm:every_obj_is_versatile_colim,prop:atomic_obj_in_Prof,prop:atomic_obj_in_Mod,prop:atomic_obj_in_Mat}.
\end{proof}

\begin{remark}
    In an \ac{AVDC} with restrictions, the density theorem (\cref{thm:density}) shows that the collage-density can be captured by the ``canonical'' tight cocones defined in \cref{def:canonical_tight_cocone}.
    In particular, if $\bX$ has restrictions, every $\bX$-category can be written as a versatile colimit in $\Prof[\bX]$ of all of its preobjects.
\end{remark}

\begin{lemma}\label{lem:admissible_equiv_preserve_atomicity}
    Coproduct-atomicity, collapse-atomicity, and collage-atomicity of objects are preserved by any admissible equivalence between \acp{AVDC}.
\end{lemma}
\begin{proof}
    This follows from \cref{thm:admissible_equiv_preserve_versatile_colim}.
\end{proof}

\begin{theorem}\label{thm:admissible_equiv_preserve_density}
    Let $\Phi\colon\bL\to\bL'$ be an admissible equivalence, and let $\bX\subseteq\bL$ be a collage-dense (resp.\ coproduct-dense; collapse-dense) full sub-\ac{AVDC}.
    Then, the image of $\bX$ under $\Phi$ is still collage-dense (resp.\ coproduct-dense; collapse-dense).
\end{theorem}
\begin{proof}
    We only show the case of collage-density since the other cases follow in the same way.
    Let $\Phi\bX\subseteq\bL'$ denote the image of $\bX$ under $\Phi$, and let $\Phi\bL\subseteq\bL'$ denote the image of $\bL$ under $\Phi$.
    By \cref{lem:admissible_equiv_preserve_atomicity}, every object in $\Phi\bX$ is collage-atomic.
    We now show the second condition in the definition of collage-density.
    By the collage-density of $\bX$, every object in $\bL$ is a large versatile collage of objects from $\bX$.
    Thus, by \cref{thm:admissible_equiv_preserve_versatile_colim}, every object in $\Phi\bL$ is a large versatile collage of objects from $\Phi\bX$.
    Since every object in $\bL'$ is admissibly isomorphic to some object in $\Phi\bL$, \cref{prop:vers_colim_is_up_to_admissible_iso}\cref{prop:vers_colim_is_up_to_admissible_iso-vertex} concludes that every object in $\bL'$ is a large versatile collage of objects from $\Phi\bX$.
\end{proof}
\subsection{Characterization theorems}\label{subsec:characterization}
We now show the main theorem: the existence of a full sub-\ac{AVDC} with the density property considered in the previous section characterizes the \acp{AVDC} of the forms $\Prof[\bX]$, $\Mod(\bX)$, and $\Mat[\bX]$.
We first introduce the notion of \textit{$C$-discreteness}, which characterizes the shapes of ordinary coproducts up to final functors.
This notion is used to construct the equivalence between \acp{AVDC} in the main theorem.
\begin{definition}\label{def:maximal_objects}
    Let $\one{C}$ be a category.
    An object $m\in\one{C}$ is called \emph{maximal} if every pair of parallel morphisms
    $
    \begin{tikzcd}[small]
        m\ar[r,shift left=1]\ar[r,shift right=1] & \cdot
    \end{tikzcd},
    $
    not necessarily distinct, has a common retraction.
    Let \mbox{$\Max(\one{C})\subseteq\one{C}$} denote the full subcategory of all maximal objects in $\one{C}$.
\end{definition}

\begin{remark}
    The category $\Max(\one{C})$ always becomes a ``simply connected groupoid.''
    That is, $\Max(\one{C})$ has at most one morphism between any two objects, and such a morphism is an isomorphism.
\end{remark}

\begin{definition}
    A category $\one{C}$ is called \emph{$C$-discrete}\footnote{The letter ``$C$'' stands for ``colimit'' and is inspired by the related notion of \textit{$L$-finiteness} in \cite{Pare1990simply}.} if:
    \begin{itemize}
        \item
            The isomorphism classes of $\Max(\one{C})$ form a large set;
        \item
            The inclusion functor $\Max(\one{C})\arr[hook][1]\one{C}$ is final.\qedhere
    \end{itemize}
\end{definition}

\begin{lemma}\label{lem:C_discrete}
    The following are equivalent for a category $\one{C}$:\footnote{The equivalence between \cref{lem:C_discrete-final} and \cref{lem:C_discrete-rightadjoint} is also noted in \cite[3.7.\ Lemma]{Clarke2024lifting}.}
    \begin{enumerate}
        \item\label{lem:C_discrete-def}
            $\one{C}$ is $C$-discrete.
        \item\label{lem:C_discrete-final}
            There is a final functor $\zero{S}\to\one{C}$ from a large discrete category $\zero{S}$.
        \item\label{lem:C_discrete-rightadjoint}
            There is a right adjoint functor $\zero{S}\to\one{C}$ from a large discrete category $\zero{S}$.
        \item\label{lem:C_discrete-concrete}
            There is a large set $\zero{S}$ of objects in $\one{C}$ such that any object in $\one{C}$ has a unique morphism from itself whose codomain lies in $\zero{S}$.
    \end{enumerate}
    Moreover, if these conditions are satisfied, the large set $\zero{S}$ above becomes isomorphic to a skeleton of $\Max(\one{C})$.
\end{lemma}
\begin{proof}
    \proofdirection{\cref{lem:C_discrete-def}}{\cref{lem:C_discrete-final}}
    Since $\Max(\one{C})$ is a simply connected groupoid, the skeleton $\zero{S}$ of $\Max(\one{C})$ is a discrete category.
    Then, the inclusion functor $\zero{S}\arr[hook][1]\Max(\one{C})$ is final because it is an equivalence.
    Since finality is closed under composition, the composite of the inclusions\linebreak \mbox{$\zero{S}\arr[hook][1]\Max(\one{C})\arr[hook][1]\one{C}$} gives the desired final functor.
    
    \proofdirection{\cref{lem:C_discrete-final}}{\cref{lem:C_discrete-concrete}}
    Let $\Phi\colon\zero{S}\to\one{C}$ be a final functor from a large discrete category.
    By the finality, $\Phi$ becomes injective on objects.
    Then, the image of $\Phi$ gives a desired class of objects in $\one{C}$.
    
    \proofdirection{\cref{lem:C_discrete-concrete}}{\cref{lem:C_discrete-def}}
    Let $\zero{S}\subseteq\Ob\one{C}$ be the large set in the condition \cref{lem:C_discrete-concrete}.
    Let $s\in\zero{S}$, and let $f,g\colon s\rightrightarrows c$ be morphisms in $\one{C}$.
    By the assumption, there is a morphism $h\colon c\to s'$ such that $s'\in\zero{S}$.
    By the uniqueness, we have $f\tcomp h=\id=g\tcomp h$, which shows that $s$ is maximal in $\one{C}$.
    Thus, the inclusion $\zero{S}\arr[hook][1]\one{C}$ factors through $\Max(\one{C})\subseteq\one{C}$, where $\zero{S}$ is regarded as a large discrete category.
    Then, $\zero{S}$ gives a large skeleton of $\Max(\one{C})$.
    Since $\zero{S}\arr[hook][1]\one{C}$ is final and the inclusion $\Max(\one{C})\arr[hook][1]\one{C}$ is full (and faithful), the functor $\zero{S}\to\Max(\one{C})$ becomes final.
    Then, the cancellation property shows that $\Max(\one{C})\arr[hook][1]\one{C}$ is final.

    \proofdirection{\cref{lem:C_discrete-concrete}}{\cref{lem:C_discrete-rightadjoint}}
    Let $\zero{S}\subseteq\Ob\one{C}$ be the large set in the condition \cref{lem:C_discrete-concrete}, and let $\Phi\colon\zero{S}\to\one{C}$ be the functor induced from the inclusion.
    Then, the unique morphisms in the condition \cref{lem:C_discrete-concrete} yield a unit of a desired adjunction, whose right adjoint is $\Phi$.

    \proofdirection{\cref{lem:C_discrete-rightadjoint}}{\cref{lem:C_discrete-final}}
    This is immediate since every right adjoint functor is final.
\end{proof}

\begin{notation}\label{note:comma_category}
    Let $\bL$ be an \ac{AVDC}, and let $\bX\subseteq\bL$ be a full sub-\ac{AVDC}.
    For an object $L\in\bL$, let $\slice{\bX}{L}$ denote the comma category:
    \begin{itemize}
        \item
            An object is a pair $(X,x)$ of an object $X\in\bX$ and a tight arrow $X\arr(x)[][1]L$ in $\bL$.
        \item
            A morphism $(X,x)\to (X',x')$ is a tight arrow $X\arr(f)[][1]X'$ in $\bL$ such that $f\tcomp x'=x$.
    \end{itemize}
    Given $(X,x)\in\slice{\bX}{L}$, we write $Dx$ for $X$ and identify $x$ with $(Dx,x)\in\slice{\bX}{L}$.
\end{notation}

\begin{construction}[Nerve construction]\label{const:nerve_construction}
    Let $\bX\subseteq\bL$ be a full sub-\ac{AVDC} of an \ac{AVDC}.
    Suppose that the following conditions hold for every $L\in\bL$:
    \begin{itemize}
        \item
            The category $\slice{\bX}{L}$ is $C$-discrete;
        \item
            $\Max(\slice{\bX}{L})$ has a skeleton whose elements are pulling in $\bL$.
    \end{itemize}
    Then, we can construct an \ac{AVD}-functor $N\colon\dim{\bL}\to\Mat[\bX]$ as follows:
    \begin{enumerate}
        \item
            Fix $L\in\bL$.
            We choose a skeleton $\zero{S}_L$ of $\Max(\slice{\bX}{L})$ whose elements are pulling in $\bL$ and define $\zero{N}L\coloneq\zero{S}_L$. For $x\in\zero{N}L$, its color is defined by $\abs{x}\coloneq Dx$.
        \item
            For a tight arrow $A\arr(f)[][1]B$ in $\bL$, we write $Nf$ for the morphism $\zero{N}A\arr[][1]\zero{N}B$ defined as follows:
            Let $x\in\zero{N}A$; since $\slice{\bX}{B}$ is $C$-discrete, the tight arrow $x\tcomp f$ uniquely factors through a unique element in $\zero{N}B$, denoted by $(Nf)^0x$:
            \begin{equation*}
                \begin{tikzcd}[tinytri]
                    & \abs{x}\ar[dl,"x"']\ar[dr,"(Nf)^1x"] & \\
                    A\ar[dr,"f"'] & & \abs{y}\ar[dl,"(Nf)^0x"]\\
                    & B &
                    \cellsymb(\heq){2-1}{2-3}
                \end{tikzcd}\incat{\bL},
            \end{equation*}
            which gives the assignment $x\mapsto (Nf)x$.
        \item
            For a loose arrow $A\larr(u)[][1]B$ in $\bL$, we write $Nu$ for a matrix $\zero{N}A\larr[][1]\zero{N}B$ over $\bX$ defined as follows:
            For $x\in\zero{N}A$ and $y\in\zero{N}B$, the loose arrow $(Nu)(x,y)$ is defined as a restriction:
            \begin{equation*}
                \begin{tikzcd}[hugecolumn]
                    \abs{x}\ar[d,"x"']\lar[r,"{(Nu)(x,y)}"] & \abs{y}\ar[d,"y"] \\
                    A\lar[r,"u"'] & B
                    \cellsymb(\cart){1-1}{2-2}
                \end{tikzcd}\incat{\bL}.
            \end{equation*}
        \item
            For a cell
            \begin{equation*}
                \begin{tikzcd}
                    A_0\ar[d,"f"']\lar[r,path,"\tup{u}"] & A_n\ar[d,"g"] \\
                    B\lar[r,"v"'] & C
                    \cellsymb(\alpha){1-1}{2-2}
                \end{tikzcd}\incat{\bL},
            \end{equation*}
            we write $N\alpha$ for a cell in $\Mat[\bX]$ defined by the following:
            \begin{equation*}
                \begin{tikzcd}[hugecolumn]
                    \abs{x_0}\ar[d,"(Nf)^1x_0"']\lar[r,"{Nu_1(x_0,x_1)}"] & \abs{x_1}\lar[r,"{Nu_2(x_1,x_2)}"] & \cdots\lar[r,"{Nu_n(x_{n-1},x_n)}"] & \abs{x_n}\ar[d,"(Ng)^1x_n"] \\
                    \abs{(Nf)^0x_0}\ar[d,"(Nf)^0x_0"']\lar[rrr,"{Nv((Nf)^0x_0,(Ng)^0x_n)}"'] & & & \abs{(Ng)^0x_n}\ar[d,"(Ng)^0x_n"] \\
                    B\lar[rrr,"v"'] & & & C
                    \cellsymb((N\alpha)_{x_0x_1\dots x_n}){1-1}{2-4}
                    \cellsymb(\cart){2-1}{3-4}
                \end{tikzcd}
            \end{equation*}
            \begin{equation*}
                =
                \begin{tikzcd}[hugecolumn]
                    \abs{x_0}\ar[d,"x_0"']\lar[r,"{Nu_1(x_0,x_1)}"] & \abs{x_1}\ar[d,"x_1"']\lar[r,"{Nu_2(x_1,x_2)}"] & \cdots\lar[r,"{Nu_n(x_{n-1},x_n)}"] & \abs{x_n}\ar[d,"x_n"] \\
                    A_0\ar[d,"f"']\lar[r,"u_1"'] & A_1\lar[r,"u_2"'] & \cdots\lar[r,"u_n"'] & A_n\ar[d,"g"] \\
                    B\lar[rrr,"v"'] & & & C
                    \cellsymb(\cart){1-1}{2-2}
                    \cellsymb(\cart){1-2}{2-3}
                    \cellsymb(\cdots){1-3}{2-3}
                    \cellsymb(\cart){1-4}{2-3}
                    \cellsymb(\alpha){2-1}{3-4}
                \end{tikzcd}\incat{\bL}.
            \end{equation*}
            Here, $x_0\in\zero{N}A_0,x_1\in\zero{N}A_1,\dots,x_n\in\zero{N}A_n$.
            \qedhere
    \end{enumerate}
\end{construction}

The notions of admissible equivalence and iso-fibrancy appearing in the following statement are defined in \cref{def:admissible_equivalence,def:isofibrancy}, respectively.
\begin{theorem}\label{thm:characterization_prof}
    The following are equivalent for an \ac{AVDC} $\bL$:
    \begin{enumerate}
        \item\label{thm:characterization_prof-1}
            $\bL$ is admissibly equivalent to $\Prof[\bX]$ for some iso-fibrant \ac{AVDC} $\bX$ with loose units.
        \item\label{thm:characterization_prof-2}
            $\bL$ has large versatile collages and an iso-fibrant collage-dense full sub-\ac{AVDC}.\qedhere
    \end{enumerate}
\end{theorem}
\begin{proof}
    \proofdirection{\cref{thm:characterization_prof-1}}{\cref{thm:characterization_prof-2}}
    $\Prof[\bX]$ has large versatile collages and a collage-dense full sub-\ac{AVDC} by \cref{cor:prof_has_collage,prop:dense_fullsub}.
    Furthermore, \cref{thm:admissible_equiv_and_versatile_cocompleteness,thm:admissible_equiv_preserve_density} show that this property is preserved by any admissible equivalence.

    \proofdirection{\cref{thm:characterization_prof-2}}{\cref{thm:characterization_prof-1}}
    Let $\bX\subseteq\bL$ be an iso-fibrant collage-dense full sub-\ac{AVDC}.
    We first show that the conditions of \cref{const:nerve_construction} are satisfied for every $L\in\bL$.
    By the collage-density, there are a large set $\zero{S}_L$, an \ac{AVD}-functor $F_L\colon\Idimdbl\zero{S}_L\to\bL$ factoring through $\bX$, and a tight cocone $\xi^L$ exhibiting $L$ as a versatile colimit of $F_L$.
    Then, by the collage-atomicity, the assignment $s\mapsto\xi^L_s$ yields a final functor $\zero{S}_L\to\slice{\bX}{L}$, which implies $C$-discreteness.
    Moreover, the large set $\zero{S}_L\cong\{\xi^L_s\mid s\in\zero{S}_L\}$ gives a skeleton of $\Max(\slice{\bX}{L})$ whose elements are pulling in $\bL$.
    Thus, we obtain the \ac{AVD}-functor $N\colon\dim{\bL}\to\Mat[\bX]$ of \cref{const:nerve_construction}.
    By \cref{cor:unitality}, $\bL$ has all loose units, hence we have the \ac{AVD}-functor $\N\colon\bL\to\Mod(\Mat[\bX])=\Prof[\bX]$ corresponding to $N$ under \cref{thm:universal_property_of_Mod}.

    Let $L\in\bL$, and consider the $\bX$-enriched category $\one{N}L\coloneq\N(L)$.
    By construction, the hom-loose arrows of $\one{N}L$ are given by the restrictions on the left below for $s,t\in\zero{S}_L$:
    \begin{equation*}
        \begin{tikzcd}[scriptsizecolumn]
            F_L s\ar[d,"\xi^L_s"']\lar[rr,"{\one{N}L(\xi^L_s,\xi^L_t)}"] &{}& F_L t\ar[d,"\xi^L_t"] \\
            L\lar[rr,"\Unit_L"'] &{}& L
            \cellsymb(\cart)[above=-2]{1-2}{2-2}
        \end{tikzcd}
        \qquad
        \begin{tikzcd}[tri]
            F_L s\ar[dr,"\xi^L_s"']\lar[rr,"F_L(!_{st})"] &{}& F_L t\ar[dl,"\xi^L_t"] \\
            & L &
            \cellsymb(\xi^L_{st})[above=-4]{1-2}{2-2}
        \end{tikzcd}\incat{\bL}.
    \end{equation*}
    Here, $\Unit_L$ denotes the loose unit on $L$.
    By the strongness theorem (\cref{thm:strongness_theorem}), the underlying cells $\xi^L_{st}$ of the tight cocone $\xi^L$ are cartesian, hence the loose arrows ${\one{N}L(\xi^L_s,\xi^L_t)}$ and $F_L(!_{st})$ are isomorphic to each other.
    Therefore, we can assume ${\one{N}L(\xi^L_s,\xi^L_t)}=F_L(!_{st})$ without loss of generality.
    In particular, we identify $\one{N}L$ with $F_L$ by using \cref{prop:enriched_cat_is_AVD_functor}.

    Since $\bL$ has all loose units as already seen, $\bL$ can fully be embedded into the \ac{AVDC} $\Prof[\bL]$ through $Z$ as in \cref{note:AVD_functor_Z}.
    In what follows, we will omit $Z$ from the notation and regard $\bL$ as a full sub-\ac{AVDC} of $\Prof[\bL]$.
    Then, under identification of $\bL$-enriched categories with \ac{AVD}-functors, it follows from the universal property of loose units that tight arrows in $\Prof[\bL]$ whose codomain lies in $\bL$ are the same as tight cocones; loose arrows in $\Prof[\bL]$ whose domain or codomain lies in $\bL$ are the same as modules; and a similar correspondence holds between cells and modulations.
    Since, in addition to $\bL$, $\Prof[\bX]$ is also a full sub-\ac{AVDC} of $\Prof[\bL]$, we will work inside $\Prof[\bL]$ below.
    
    To show that $\N$ is an equivalence, we will use \cref{thm:equiv_in_AVDC}.
    We first show that $\N$ is bijective on tight arrows.
    For a tight arrow $A\arr(f)[][1]B$ in $\bL$, the $\bX$-functor $\N f$ makes the following diagram commute by construction of $\N$:
    \begin{equation*}
        \begin{tikzcd}[tinytri]
            & F_A\ar[dl,"\xi^A"']\ar[dr,"\N f"] & \\
            A\ar[dr,"f"'] && F_B\ar[dl,"\xi^B"] \\
            & B &
            \cellsymb(\heq){1-2}{3-2}
        \end{tikzcd}
        \incat{\Prof[\bL]}.
    \end{equation*}
    This induces the following commutative diagram of maps:
    \begin{equation}\label{eq:triangle_bij_on_tightarrows}
        \begin{tikzcd}[tri]
            \Homset[{\Prof[\bL]}]\vvect{A}{B}\ar[rr,"\N"]\ar[rd,"\xi^A\tcomp-"'] &{}& \Homset[{\Prof[\bL]}]\vvect{F_A}{F_B}\ar[dl,"-\tcomp\xi^B"] \\
             & \Homset[{\Prof[\bL]}]\vvect{F_A}{B} &
        \end{tikzcd}
    \end{equation}
    Since $\Homset[{\Prof[\bL]}]\vvect{F_A}{B}$ is isomorphic to the set of tight cocones from $F_A$ with the vertex $B$, the condition \ref{axiom:tight_bij} for $\xi^A$ implies that the map $\xi^A\tcomp-$ above is a bijection.
    In fact, the map $-\tcomp\xi^B$ is also a bijection, which follows from the claim below:
    \begin{claim}\label{claim:bij_from_atomicity}
        Let $\one{A},\one{B}\in\Prof[\bX]$, and let $\xi$ be a versatile collage of $\one{B}$, regarded as an \ac{AVD}-functor, with a vertex $\Xi\in\bL$.
        Then, the following map is a bijection:
        \begin{equation*}
            \Homset[{\Prof[\bL]}]\vvect{\one{A}}{\one{B}}\arr(-\tcomp\xi)\Homset[{\Prof[\bL]}]\vvect{\one{A}}{\Xi}
        \end{equation*}
    \end{claim}
    \begin{since}
        We now construct from a tight cocone $l$ from $\one{A}$ with the vertex $\Xi$, a unique $\bX$-functor $H$ satisfying $H\tcomp\xi=l$ as follows.
        For each object $s\in\Ob\one{A}$, the collage-atomicity of $\abs{s}_\one{A}$ implies that there exist a unique object in $\one{B}$, denoted by $H^0s$, and a unique tight arrow, denoted by $H^1s$, that satisfy $(H^1s) \tcomp \xi_{H^0s} = l_s$, as depicted on the left below.
        Then, since the cell $\xi_{H^0sH^0t}$ is cartesian for $s,t\in\Ob\one{A}$, there is a unique cell $H_{st}$ satisfying the equality on the right below.
        \begin{equation*}
            \begin{tikzcd}[small]
                & \abs{s}_\one{A}\ar[dl,"H^1s"']\ar[dd,"l_s"] \\
                \abs{H^0s}_\one{B}\ar[dr,"\xi_{H^0s}"'] & {} \\
                & \Xi
                \cellsymb(\heq){2-1}{2-2}
            \end{tikzcd}
            \qquad
            \begin{tikzcd}[largetri]
                \abs{s}_\one{A}\ar[d,"H^1s"']\lar[rr,"{\one{A}(s,t)}"] & & \abs{t}_\one{A}\ar[d,"H^1t"] \\
                \abs{H^0s}_\one{B}\ar[dr,"\xi_{H^0s}"']\lar[rr,"{\one{B}(H^0s,H^0t)}"] &{}& \abs{H^0t}_\one{B}\ar[dl,"\xi_{H^0t}"] \\
                & \Xi &
                \cellsymb(H_{st}){1-1}{2-3}
                \cellsymb(\xi_{H^0sH^0t})[above=-4]{2-2}{3-2}
            \end{tikzcd}
            =
            \begin{tikzcd}[tri]
                \abs{s}_\one{A}\ar[dr,"l_s"']\lar[rr,"{\one{A}(s,t)}"] &{}& \abs{t}_\one{A}\ar[dl,"l_t"] \\
                & \Xi &
                \cellsymb(l_{st})[above=-4]{1-2}{2-2}
            \end{tikzcd}\incat{\bL}
        \end{equation*}
        Since the underlying cells of $\xi$ are cartesian by \cref{thm:strongness_theorem}, using universal property of them, we can verify that the tuple $(H^0s,H^1s,H_{st})_{s,t}$ becomes an $\bX$-functor $\one{A}\arr(H)\one{B}$, which satisfies $H\tcomp\xi=l$.
        Moreover, the uniqueness of $H^0s$, $H^1s$, and $H_{st}$ ensures that $H$ is unique, which shows that the map $-\tcomp\xi$ is a bijection.
    \end{since}\noindent
    Combining \cref{claim:bij_from_atomicity} with the commutativity of the diagram \cref{eq:triangle_bij_on_tightarrows}, we conclude that $\N$ is bijective on tight arrows.

    We next show that $\N$ is bijective on cells.
    Since both $\bL$ and $\Prof[\bX]$ have loose units, it suffices to consider 1-coary cells.
    Take arbitrary data on the left below.
    Note that for any loose arrow $A\larr(p)A'$ in $\bL$, the $\bX$-profunctor $\N p$ is a restriction of $p$ along $\xi^A$ and $\xi^{A'}$ on the right below, where the associated cartesian cell is denoted by $\xi^p$.
    \begin{equation*}
        \begin{tikzcd}
            A_0\ar[d,"f"']\lar[r,path,"\tup{u}"] & A_n\ar[d,"g"] \\
            B\lar[r,"v"'] & C
        \end{tikzcd}\incat{\bL}
        \qquad
        \begin{tikzcd}[largecolumn]
            F_A\ar[d,"\xi^A"']\lar[r,"\N p"] & F_{A'}\ar[d,"\xi^{A'}"] \\
            A\lar[r,"p"'] & A'
            \cellsymb(\xi^p\colon\cart){1-1}{2-2}
        \end{tikzcd}
        \incat{\Prof[\bL]}
    \end{equation*}
    Then, by construction of $\N$, we obtain the following commutative diagram of maps:
    \begin{equation*}
        \begin{tikzcd}[tri]
            \Cells[{\Prof[\bL]}]{f}{g}{\tup{u}}{v}\ar[rr,"\N"]\ar[dr,"{(\xi^{u_1},\dots,\xi^{u_n})\tcomp-}"'] && \Cells[{\Prof[\bL]}]{\N f}{\N g}{\N\tup{u}}{\N v}\ar[dl,"-\tcomp\xi^v"] \\
            & \Cells[{\Prof[\bL]}]{\xi^{A_0}\tcomp f}{\xi^{A_n}\tcomp g}{\N\tup{u}}{v} &
        \end{tikzcd}
    \end{equation*}
    Using \ref{axiom:cell_type1_left}\ref{axiom:cell_type1_right}\ref{axiom:cell_type2}\ref{axiom:cell_type3} for the versatile collages $\xi^{A_i}$ $(0\le i\le n)$, we can straightforwardly show that the map $(\xi^{u_1},\dots,\xi^{u_n})\tcomp-$ is a bijection.
    Since the cell $\xi^v$ is cartesian, the map $-\tcomp\xi^v$ is also bijective, which shows that $\N$ is bijective on cells.
    
    Take $\one{A}\in\Prof[\bX]$ arbitrarily.
    Regarding $\one{A}$ as an \ac{AVD}-functor, we can take a versatile collage $\zeta$ with a vertex $Z\in\bL$ from $\one{A}$.
    Applying \cref{claim:bij_from_atomicity} to the versatile collages $\xi^Z$ and $\zeta$, we obtain unique $\bX$-functors $Q$ and $Q'$ satisfying $Q\tcomp\xi^Z=\zeta$ and $Q'\tcomp\zeta=\xi^Z$.
    Using \cref{claim:bij_from_atomicity} again, we can show that these $\bX$-functors are mutually inverses.

    Let $\one{A}\arr(\zeta)Z$ and $\one{B}\arr(\xi)W$ in $\Prof[\bL]$ be versatile collages of $\one{A},\one{B}\in\Prof[\bX]$.
    Let $Q\colon\one{A}\arr(\cong)[][1]F_Z$ and $R\colon\one{B}\arr(\cong)[][1]F_W$ be the invertible $\bX$-functors constructed above.
    Let $\one{A}\larr(P)\one{B}$ be an $\bX$-profunctor.
    Applying \ref{axiom:loose_esssurj_left} to $\zeta$ and the left $\one{A}$-modules $P(-,y)\colon\one{A}\larr[Rightarrow]\abs{y}_\one{B}$ for $y\in\Ob\one{B}$, we obtain a loose arrow $Z\larr(m_y)\abs{y}_\one{B}$ in $\bL$.
    For $y,y'\in\Ob\one{B}$, the underlying cells of the $\bX$-profunctor $P$ induces a modulation $P^r_{y,y'}$ on the left below.
    Applying \ref{axiom:cell_type0_left} to $\zeta$ and the modulation $P^r_{y,y'}$, we obtain a cell $m_{yy'}$ on the right below.
    \begin{equation*}
        \begin{tikzcd}
            \one{A}\ar[d,equal]\lar[r,Rightarrow,"{P(-,y)}"] & \abs{y}_\one{B}\lar[r,"{\one{B}(y,y')}"] & \abs{y'}_\one{B}\ar[d,equal] \\
            \one{A}\lar[rr,Rightarrow,"{P(-,y')}"'] && \abs{y'}_\one{B}
            \cellsymb(P^r_{y,y'}){1-1}{2-3}
        \end{tikzcd}
        \Vline
        \begin{tikzcd}
            Z\ar[d,equal]\lar[r,"m_y"] & \abs{y}_\one{B}\lar[r,"{\one{B}(y,y')}"] & \abs{y'}_\one{B}\ar[d,equal] \\
            Z\lar[rr,"m_{y'}"'] && \abs{y'}_\one{B}
            \cellsymb(m_{yy'}){1-1}{2-3}
        \end{tikzcd}\incat{\bL}.
    \end{equation*}
    Using \ref{axiom:cell_type0_left} for $\zeta$ and \cref{prop:modules_loosewise_indiscrete}, we can verify that the tuple $m\coloneq (m_y,m_{yy'})_{y,y'}$ yields a right $\one{B}$-module with the vertex $Z$.
    Then, applying \ref{axiom:loose_esssurj_right} to $\xi$ and $m$, we obtain a loose arrow $Z\larr(p)W$ in $\bL$.
    By construction, we also obtain a cartesian cell on the left below.
    Since the cell $\xi^p$ is cartesian, there is a unique cell $\theta$ satisfying the following equality.
    \begin{equation*}
        \begin{tikzcd}
            \one{A}\ar[d,"\zeta"']\lar[r,"P"] & \one{B}\ar[d,"\xi"] \\
            Z\lar[r,"p"'] & W
            \cellsymb(\cart){1-1}{2-2}
        \end{tikzcd}
        =
        \begin{tikzcd}[largecolumn]
            \one{A}\ar[d,"Q"',"\cong"]\lar[r,"P"] & \one{B}\ar[d,"\cong"',"R"] \\
            F_Z\ar[d,"\xi^Z"']\lar[r,"\N p"] & F_W\ar[d,"\xi^W"] \\
            Z\lar[r,"p"'] & W
            \cellsymb(\theta){1-1}{2-2}
            \cellsymb(\xi^p\colon\cart){2-1}{3-2}
        \end{tikzcd}\incat{\Prof[\bL]}.
    \end{equation*}
    Since the left and right boundary are invertible, the cell $\theta$ becomes loosewise invertible automatically.

    What remains to be shown is that the equivalence $\N$ is admissible.
    In what follows, we regard $\N$ as a right adjoint part of the equivalence.
    Since $\bX$ is iso-fibrant, so is $\Prof[\bX]$, which follows from the proofs of \cref{prop:Mat_has_1-coary_restrictions} and \cref{prop:restriction_in_Mod}\cref{prop:restriction_in_Mod-1coaryrest}.
    In particular, the invertible $\bX$-functor $Q$ (or $R$) above is clearly admissible, which shows the unit part.
    To show the counit part, let us take an object $L\in\bL$.
    Then, $\N(L)=\one{N}L$ is sent to a versatile collage of $F_L$.
    Since $L$ is also a versatile collage of the same diagram $F_L$, \cref{prop:vers_colim_is_up_to_admissible_iso}\cref{prop:vers_colim_is_up_to_admissible_iso-vertex} induces an admissible isomorphism between them.
    This shows the counit part and finishes the proof.
\end{proof}

We can also prove the following theorems in a similar way to \cref{thm:characterization_prof}:
\begin{theorem}\label{thm:characterization_mat}
    The following are equivalent for an \ac{AVDC} $\bL$:
    \begin{enumerate}
        \item\label{thm:characterization_mat-1}
            $\bL$ is admissibly equivalent to $\Mat[\bX]$ for some iso-fibrant \ac{AVDC} $\bX$.
        \item\label{thm:characterization_mat-2}
            $\bL$ is diminished and has large \ac{VD}-versatile coproducts and an iso-fibrant coproduct-dense full sub-\ac{AVDC}.\qedhere
    \end{enumerate}
\end{theorem}

\begin{theorem}\label{thm:characterization_mod}
    The following are equivalent for an \ac{AVDC} $\bL$:
    \begin{enumerate}
        \item\label{thm:characterization_mod-1}
            $\bL$ is admissibly equivalent to $\Mod(\bX)$ for some iso-fibrant \ac{AVDC} $\bX$ with loose units.
        \item\label{thm:characterization_mod-2}
            $\bL$ has versatile collapses and an iso-fibrant collapse-dense full sub-\ac{AVDC}.\qedhere
    \end{enumerate}
\end{theorem}

\begin{remark}
    In \cref{thm:characterization_mat}\cref{thm:characterization_mat-2}, the assumption that $\bL$ is diminished is necessary.
    Indeed, the \ac{AVDC} $\Rel$ as in \cref{eg:avdc_of_relations} has large versatile coproducts, which are in particular \ac{VD}-versatile coproducts, and the full sub-\ac{AVDC} spanned by the singleton is iso-fibrant and coproduct-dense.
    However, $\Rel$ cannot be equivalent to any $\Mat[\bX]$ because $\Rel$ is not diminished.
\end{remark}

\begin{remark}
    In spite of the fact that the $\Prof$-construction can be split into two constructions as $\Prof[\bX]=\Mod(\Mat[\bX])$, the characterization theorem of $\Prof[\bX]$ (\cref{thm:characterization_prof}) does not directly follow from the characterization theorems of the others (\cref{thm:characterization_mat,thm:characterization_mod}).
    This is because $\Mat[\bX]$ does not have loose units in general.
\end{remark}
\subsection{Closedness under slicing}\label{subsec:slicing}
In this subsection, we prove that the \acp{AVDC} of profunctors are closed under ``slicing'' as a direct consequence of our characterization theorems.
We first generalize to \acp{AVDC}, the notion of slice double categories \cite{Pare2011yoneda}, which has been denoted by the double slash ``$\dblslash$.''
\begin{definition}
    Let $\bL$ be an \ac{AVDC}, and let $L\in\bL$.
    The \emph{slice} \ac{AVDC}, denoted by $\bL/L$, is the \ac{AVDC} defined by the following:
    \begin{itemize}
        \item
            The tight category is $\slice{\bL}{L}$;
        \item
            A loose arrow $x\larr(u) y$ in $\bL/L$ is a pair $(Du,u)$ of a loose arrow $Du$ and a cell $u$
            \begin{equation*}
                \begin{tikzcd}[tri]
                    Dx\ar[dr,"x"']\lar[rr,"Du"] &{}& Dy\ar[dl,"y"] \\
                    & L &
                    \cellsymb(u)[above=-2]{1-2}{2-2}
                \end{tikzcd}\incat{\bL};
            \end{equation*}
        \item
            A cell $\alpha\in\Cells[\bL/L]{f}{g}{\tup{u}}{v}$ is a cell in $\bL$ satisfying the following:
            \begin{equation*}
                \begin{tikzcd}
                    Dx_0\ar[d,"f"']\lar[r,"Du_1"] & \cdots\lar[r,"Du_n"] & Dx_n\ar[d,"g"] \\
                    Dy\ar[dr,"y"']\lar[rr,phan,"Dv"] &{}& Dz\ar[dl,"z"] \\
                    & L &
                    \cellsymb(\alpha){1-1}{2-3}
                    \cellsymb(v)[above=-3]{2-2}{3-2}
                \end{tikzcd}
                =
                \begin{tikzcd}
                    Dx_0\ar[dr,"x_0"']\lar[r,"Du_1"] & \cdots\lar[r,"Du_n"] & Dx_n\ar[dl,"x_n"] \\
                    & L &
                    \cellsymb(u_1)[left=6]{1-2}{2-2}
                    \cellsymb(\cdots){1-2}{2-2}
                    \cellsymb(u_n)[right=6]{1-2}{2-2}
                \end{tikzcd}\incat{\bL}.
            \end{equation*}
    \end{itemize}
    We write $D_L\colon\bL/L\to\bL$ for the canonical \ac{AVD}-functor defined by $x\mapsto Dx$.
    For a full sub-\ac{AVDC} $\bX\subseteq\bL$ and $L\in\bL$, we write $\bX/L\subseteq\bL/L$ for the full sub-\ac{AVDC} consisting of objects $x\in\bL/L$ such that $Dx\in\bX$. 
\end{definition}

\begin{remark}
    The slice \ac{AVDC} $\bL/L$ defined above is not a comma object in the 2-category $\AVDC$ with respect to the \ac{AVD}-functor $\Ddbl\zero{1}\arr(\const{L})\bL$ that chooses the object $L\in\bL$.
    This is because the ``walking object'' $\Ddbl\zero{1}$ is not a 2-terminal object in $\AVDC$.
    On the other hand, if the object $L\in\bL$ is unital, i.e., admits a loose unit, then the slice \ac{AVDC} $\bL/L$ is a comma object with respect to the \ac{AVD}-functor $\Idbl\zero{1}\arr(\const{L})\bL$, because $\Idbl\zero{1}$ is a ``walking unital object'' and is a 2-terminal object in $\AVDC$.
\end{remark}

\begin{lemma}
    Let $F\colon\bK\to\bL$ be an \ac{AVD}-functor between \acp{AVDC}.
    Then, a tight cocone from $F$ with a vertex $L\in\bL$ is the same as an \ac{AVD}-functor $\bK\to\bL/L$ whose composite with $D_L\colon\bL/L\to\bL$ is $F$.
    \begin{equation*}
        \begin{tikzcd}[hugecolumn]
            \bK\ar[dr,"F"']\ar[r] & \bL/L\ar[d,"D_L"] \\
            & \bL
        \end{tikzcd}
    \end{equation*}
\end{lemma}

\begin{lemma}\label{lem:cocompleteness_of_slice}
    Let $\bL$ be an \ac{AVDC}, and let $L\in\bL$.
    Let $G\colon\bK\to\bL/L$ be an \ac{AVD}-functor from an \ac{AVDC}.
    Suppose that we are given a (\ac{VD}-)versatile colimit $\xi$ of $D_LG$ with a vertex $\Xi\in\bL$.
    Then, there is a (\ac{VD}-)versatile colimit of $G$, which is sent to $\xi$ by $D_L$.
\end{lemma}
\begin{proof}
    Let $l$ denote the tight cocone from $D_LG$ associated with $G$, and let $L\in\bL$ be its vertex.
    By \ref{axiom:tight_bij} for the (\ac{VD}-)versatile colimit $\xi$, we obtain the canonical tight arrow $\Xi\arr(k)[][1]L$ in $\bL$.
    Then, the \ac{AVD}-functor $H\colon\bK\to\bL/\Xi$ corresponding to $\xi$ makes the following diagram commute:
    \begin{equation*}
        \begin{tikzcd}[hugecolumn]
            \bK\ar[r,"H"]\ar[drr,"G"'] & \bL/\Xi\ar[r,phantom,"\cong"] &[-30pt] (\bL/L)/k\ar[d,"D_k"] \\
            && \bL/L
        \end{tikzcd}
    \end{equation*}
    This gives a tight cocone from $G$ with the vertex $k$, which becomes a (\ac{VD}-)versatile colimit of $G$ straightforwardly.
\end{proof}

\begin{lemma}
    Let $\bX\subseteq\bL$ be a collage-dense (resp.\ collapse-dense) full sub-\ac{AVDC} of an \ac{AVDC}, and let $L\in\bL$.
    Then, $\bX/L\subseteq\bL/L$ also becomes collage-dense (resp.\ collapse-dense).
\end{lemma}
\begin{proof}
    This follows from \cref{lem:cocompleteness_of_slice} directly.
\end{proof}

By the characterization theorems (\cref{thm:characterization_prof,thm:characterization_mod}), we now have the following:
\begin{corollary}\label{cor:slices_of_prof_and_mod}
    Let $\bX$ be an iso-fibrant \ac{AVDC} with loose units.
    \begin{enumerate}
        \item\label{cor:slices_of_prof_and_mod-prof}
            For an $\bX$-enriched category $\one{A}$, there is an admissible equivalence $\Prof[\bX]/\one{A}\simeq\Prof[(\bX/\one{A})]$.
        \item\label{cor:slices_of_prof_and_mod-mod}
            For a monoid $M$ in $\bX$, there is an admissible equivalence $\Mod(\bX)/M\simeq\Mod(\bX/M)$.
    \end{enumerate}
\end{corollary}

\begin{remark}
    \cref{cor:slices_of_prof_and_mod}\cref{cor:slices_of_prof_and_mod-prof} is a double categorical refinement of \cite[4.5.\ Theorem]{FujiiLack2024oplax}, which treats the (strict) slice 2-category of the 2-category of categories and functors enriched in a bicategory.
\end{remark}

\appendix
\section{Comparison with proarrow equipments}\label[appendix]{sec:proarrow_equipments}
For readers who are not familiar with the notion of \textit{proarrow equipment}, we recall it based on \cite{Wood1985proarrowII}.
Note that for the sake of notational consistency with our terminology, we replace the codomain bicategory with its 1-cell dual.
\begin{definition}[\cite{Wood1985proarrowII}]
    A \emph{proarrow equipment} is a pseudo-functor $\conj{(\cdot)}\colon\bi{K}\to\bi{M}^\op$ between bicategories that satisfies the following conditions:
    \begin{itemize}
        \labeleditem{(Ax.1)}
            $\bi{K}$ and $\bi{M}$ have the same class of objects, and $\conj{(\cdot)}$ is the identity on objects.
        \labeleditem{(Ax.2)}\label{proarrow:locally_ff}
            $\conj{(\cdot)}$ is locally fully faithful.
        \labeleditem{(Ax.3)}
            For every 1-cell $f$ in $\bi{K}$, $\conj{f}$ has a left adjoint $\comp{f}$ in $\bi{M}$.
    \end{itemize}
    A 1-cell $A\arr(u)[][1]B$ in $\bi{M}$ is called \emph{representable} if $u\cong\conj{f}$ for some 1-cell $B\arr(f)[][1]A$ in $\bi{K}$.
    Note that the assignment $f\mapsto\comp{f}$ yields an identity-on-objects pseudo-functor $\comp{(\cdot)}\colon\bi{K}\to\bi{M}^\co$, where $\bi{M}^\co$ denotes the 2-cell dual of $\bi{M}$.
\end{definition}

\begin{remark}
    For a bicategory $\bi{W}$, a monoid in the diminished \ac{AVDC} $\Ldbl\bi{W}$ is the same as a \textit{monad} $t=(t^0,t^1,t^e,t^m)$ in $\bi{W}$ in the sense of \cite{Benabou1967introduction}.
    \begin{equation*}
        \begin{tikzcd}[hugecolumn]
            t^0\ar[r,bend left=30,equal]\ar[r,bend right=30,"t^1"{below}] & t^0 \\
            \dtwocell(t^e){1-1}{1-2}
        \end{tikzcd}
        \qquad
        \begin{tikzcd}
            & t^0\ar[rd,"t^1"] & \\
            t^0\ar[ru,"t^1"]\ar[rr,"t^1"'] &{}& t^0
            \dtwocell(t^m){1-2}{2-2}
        \end{tikzcd}\incat{\bi{W}}
    \end{equation*}
    For an \ac{AVD}-functor $\const{t}\colon\Idimdbl\zero{1}\to\Ldbl\bi{W}$ corresponding to a monad $t$ in $\bi{W}$, a right (resp.\ left) $\const{t}$-module coincides with what has historically been called \textit{$t$-algebra} (resp.\ \textit{$t$-opalgebra}) \cite{Street1972formal}.
    Thus, for an object $X\in\bi{W}$, we also write $\Alg(X,t)\coloneq\Mdl{X}{\const{t}}$ and $\Alg(t,X)\coloneq\Mdl{\const{t}}{X}$.
\end{remark}

\begin{notation}[The mate $t$-opalgebra]
    Let $t=(t^0,t^1,t^e,t^m)$ be a monad in a bicategory $\bi{W}$.
    Let $m=(m^1,m^2)\in\Alg(X,t)$ be a $t$-algebra whose underlying 1-cell $m^1$ has a left adjoint $\mate{m}^1$ in $\bi{W}$.
    Then, the left adjoint curries a canonical $t$-opalgebra structure; the resulting $t$-opalgebra will be denoted by $\mate{m}=(\mate{m}^1,\mate{m}^2)$, where $\mate{m}^2$ is obtained by taking the mate of the 2-cell $m^2$.\vspace{-0.6em}
    \begin{equation*}
        \begin{tikzcd}[smallcolumn]
            X\ar[rr,"m^1"',shift right=1,bend right=10] & {\scriptstyle\perp} & t^0\ar[ll,"\mate{m}^1"',shift right=1,bend right=10]
        \end{tikzcd}
        \qquad
        \begin{tikzcd}
            & t^0\ar[rd,"t^1"] & \\
            X\ar[ru,"m^1"]\ar[rr,"m^1"'] &{}& t^0
            \dtwocell(m^2){1-2}{2-2}
        \end{tikzcd}
        \qquad
        \begin{tikzcd}
            & t^0\ar[dr,"\overline{m}^1"] & \\
            t^0\ar[ru,"t^1"]\ar[rr,"\mate{m}^1"'] &{}& X
            \dtwocell(\mate{m}^2){1-2}{2-2}
        \end{tikzcd}\incat{\bi{W}}
    \end{equation*}
\end{notation}

In \cite{Wood1985proarrowII}, the following additional conditions on a proarrow equipment $\conj{(\cdot)}\colon\bi{K}\to\bi{M}^\op$ are considered:
\begin{itemize}
    \labeleditem{(Ax.4)}\label{proarrow:coproduct}
        $\bi{K}$ has finite bicoproducts, and the pseudo-functors $\conj{(\cdot)}\colon\bi{K}\to\bi{M}^\op$ and $\comp{(\cdot)}\colon\bi{K}\to\bi{M}^\co$ preserve them.
    \labeleditem{(Ax.5)}\label{proarrow:EM}
        For every monad $t=(t^0,t^1,t^e,t^m)$ in the bicategory $\bi{M}$, there are an object $\Xi\in\bi{K}$ and a $t$-algebra $e=(e^1,e^2)\in\Alg(\Xi,t)$ with $e^1$ representable that satisfies the following conditions:
        \begin{itemize}
            \item
                For every $X\in\bi{M}$, the functor $\Homcat[\bi{M}](X,\Xi)\arr(-\lcomp e) \Alg(X,t)$ induced by composition with $e$ is an equivalence of categories.
            \item
                For every $X\in\bi{M}$, the functor $\Homcat[\bi{M}](\Xi,X)\arr(\mate{e}\lcomp -) \Alg(t,X)$ induced by composition with $\mate{e}$ is an equivalence of categories.
            \item
                If the composite $u\lcomp e^1$ with a 1-cell $u$ is representable, then $u$ is representable.
                \begin{equation*}
                    \begin{tikzcd}
                        &[10pt] &[-10pt] t^0\ar[dr,"t^1"] &[-10pt] \\
                        \cdot\ar[r,"u"] & \Xi\ar[ru,"e^1"]\ar[rr,"e^1"'] &{}& t^0
                        \dtwocell(e^2){1-3}{2-3}
                    \end{tikzcd}\incat{\bi{M}}
                \end{equation*}
        \end{itemize}
\end{itemize}

\begin{notation}
    As shown in \cite[Appendix C]{Shulman2008framed}, a proarrow equipment whose domain is a (strict) 2-category is essentially the same concept as a pseudo double category with restrictions, and the latter can be regarded as an \ac{AVDC} with loose composites and restrictions as described in \cref{rem:pseudo_double_category}.
    Given a proarrow equipment $\conj{(\cdot)}\colon\bi{K}\to\bi{M}^\op$ such that $\bi{K}$ is a 2-category, we write $\dbl{F}_\conj{(\cdot)}$ for the corresponding \ac{AVDC} with loose composites and restrictions.
    For clarity, we describe the \ac{AVDC} $\dbl{F}_\conj{(\cdot)}$ explicitly as follows:
    \begin{itemize}
        \item
            Objects in $\dbl{F}_\conj{(\cdot)}$ are those of $\bi{K}$ (and $\bi{M}$).
        \item
            Tight arrows in $\dbl{F}_\conj{(\cdot)}$ are the 1-cells of $\bi{K}$.
        \item
            Loose arrows $\dbl{F}_\conj{(\cdot)}$ are the 1-cells of $\bi{M}$.
        \item
            Cells
            \begin{equation*}
                \begin{tikzcd}
                    A\ar[d,"f"']\lar[r,path,"\tup{u}"] & B\ar[d,"g"] \\
                    X\lar[r,phan,"v"'] & Y
                    \cellsymb(\alpha){1-1}{2-2}
                \end{tikzcd}\incat{\dbl{F}_\conj{(\cdot)}}
            \end{equation*}
            are the 2-cells $\conj{f}\lcomp(\lcomp\tup{u})\arr(\alpha)[Rightarrow] (\lcomp v)\lcomp\conj{g}$ in $\bi{M}$.
            Here, for a path of 1-cells $\tup{p}=(p_1,\dots,p_n)$, we use the notation $\lcomp\tup{p}\coloneq(\cdots((p_1\lcomp p_2)\lcomp p_3)\cdots)\lcomp p_n$.
            If $n=0$, $\lcomp\tup{p}$ is defined as the identity 1-cell.\qedhere
    \end{itemize}
\end{notation}

\begin{definition}
    An \ac{AVDC} $\bK$ is called \emph{tightwise discrete} if the tight category $\Tcat\bK$ is discrete.
\end{definition}

\begin{lemma}\label{lem:representability_of_module}
    Let $F\colon\bK\to\bL$ be an \ac{AVD}-functor between \acp{AVDC} with $\bK$ tightwise discrete.
    For a left $F$-module $F\larr(m)[Rightarrow]M$, the following are equivalent:
    \begin{enumerate}
        \item\label{lem:representability_of_module-global}
            $m\cong\comp{l}$ in $\Mdl{F}{M}$ for some tight cocone $F\arr(l)[Rightarrow]M$ such that the companion $\comp{l_A}$ exists for every $A\in\bK$.
            Here, $\comp{l}$ is the left $F$-module described in \cref{rem:companion_of_tight_cocone}.
        \item\label{lem:representability_of_module-componentwise}
            For every $A\in\bK$, $m_A$ is a companion of some tight arrow.
    \end{enumerate}
\end{lemma}
\begin{proof}
    \proofdirection{\cref{lem:representability_of_module-global}}{\cref{lem:representability_of_module-componentwise}}
    This follows from the construction of $\comp{l}$.

    \proofdirection{\cref{lem:representability_of_module-componentwise}}{\cref{lem:representability_of_module-global}}
    For each $A\in\bK$, we can take a tight arrow $FA\arr(l_A)M$ in $\bL$ and suppose $m_A=\comp{l_A}$ without loss of generality.
    For $A\larr(u)B$ in $\bK$, composition with the associated cells to the companions $\comp{l_A}$ and $\comp{l_B}$ induces a bijective correspondence between the cells of the following forms:\vspace{-1.0em}
    \begin{equation*}
        \begin{tikzcd}[tri]
            FA\ar[dr,"l_A"']\lar[rr,"Fu"] &{}& FB\ar[dl,"l_B"] \\
            & M &
            \cellsymb(\cdot)[above=-2]{1-2}{2-2}
        \end{tikzcd}
        \Vline
        \begin{tikzcd}[largecolumn]
            FA\ar[d,equal]\lar[r,"Fu"] & FB\lar[r,"m_B (=\comp{l_B})"] &[20pt] M\ar[d,equal] \\
            FA\lar[rr,"m_A (=\comp{l_A})"'] & & M
            \cellsymb(\cdot){1-1}{2-3}
        \end{tikzcd}\incat{\bL}.
    \end{equation*}
    Define $l_u$ as the cell that corresponds to the cell $m_u$ under this correspondence.
    Then it also follows directly from this bijective correspondence that the tuple $l=(l_A,l_u)_{A,u}$ is compatible with cells in $\bK$, which is sufficient for $l$ to become a tight cocone from $F$ since $\bK$ has no non-trivial tight arrow.
    The isomorphism $m\cong\comp{l}$ is obvious.
\end{proof}

\begin{definition}
    Let $F\colon\bK\to\bL$ be an \ac{AVD}-functor between \acp{AVDC} with $\bK$ tightwise discrete.
    For a tight cocone $\xi$ from $F$ with a vertex $\Xi\in\bL$, we consider a weakened version of the condition \ref{axiom:tight_bij}:\vspace{-0.2em}
    \begin{itemize}
        \labeleditem{(wT)}\label{axiom:tight_esssurj}
            The canonical functor $\Homcat[\bL]\vvect{\Xi}{L}\arr(\xi\tcomp-)\Cone\vvect{F}{L}$ of \cref{const:canonical_functor_tight_to_cocone} is essentially surjective on objects for any $L\in\bL$.
    \end{itemize}
    Then, the tight cocone $\xi$ is called a \emph{versatile bicolimit} of $F$ if it satisfies the conditions \ref{axiom:tight_esssurj}\ref{axiom:loose_esssurj_left}\ref{axiom:loose_esssurj_right}\ref{axiom:cell_type1_left}\ref{axiom:cell_type1_right}\ref{axiom:cell_type2}\ref{axiom:cell_type3}.
\end{definition}

\begin{remark}
    For a general shape $\bK$, which is not necessarily tightwise discrete, we refrain from defining the notion of versatile bicolimit here.
    The reason is that, in the general case, we have to replace tight cocones in \cref{lem:representability_of_module} with \textit{pseudo tight cocones}---tight cocones whose compatibility with tight arrows in $\bK$ is weakened up to isomorphism---and so, to define general versatile bicolimits, it would be necessary to reformulate the definition of versatile colimits in terms of pseudo tight cocones.
\end{remark}

\begin{theorem}\label{thm:special_case_versbicolim}
    Let $\bL$ be an \ac{AVDC} with extensions, lifts, and loose composites.
    Let $F\colon\bK\to\bL$ be an \ac{AVD}-functor with $\bK$ tightwise discrete.
    Then, a tight cocone $\xi$ from $F$ with a vertex $\Xi\in\bL$ becomes a versatile bicolimit if and only if it satisfies the following conditions:
    \begin{itemize}
        \item
            The functor $\Homcat[\bL]\vvect{\Xi}{L}\arr(\xi\tcomp-)\Cone\vvect{F}{L}$ is an equivalence of categories for any $L\in\bL$;
        \item
            The functors $\Homcat[\bL](\Xi,L)\arr(\comp{\xi}-)\Mdl{F}{L}$ and $\Homcat[\bL](L,\Xi)\arr(-\conj{\xi})\Mdl{L}{F}$ are equivalences of categories for any $L\in\bL$.
    \end{itemize}
\end{theorem}
\begin{proof}
    \cref{prop:fully_faithfulness_from_M2_M0} implies that these two conditions are necessary.
    Sufficiency also follows from an argument similar to the proof of \cref{thm:strongest_simplification}.
\end{proof}

\begin{definition}\quad
    \begin{enumerate}
        \item
            A \emph{finite versatile bicoproduct} is a versatile bicolimit of an \ac{AVD}-functor from $\Ddbl\zero{S}$ for some finite set $\zero{S}$.
        \item
            A \emph{versatile bicollapse} is a versatile bicolimit of an \ac{AVD}-functor from $\Idimdbl\zero{1}$, where $\zero{1}$ denotes the singleton.\qedhere
    \end{enumerate}
\end{definition}

In \cite{Wood1982proarrowI} (but not in \cite{Wood1985proarrowII}), the codomain bicategories of proarrow equipments are required to be \textit{biclosed}, i.e., to have right Kan extensions and lifts.
Under this extra hypothesis, Wood's axioms \ref{proarrow:coproduct} and \ref{proarrow:EM} can be interpreted as asserting the existence of specific versatile bicolimits:
\begin{theorem}\label{thm:proarr_equip_in_terms_of_versbicolim}
    Let $\conj{(\cdot)}\colon\bi{K}\to\bi{M}^\op$ be a proarrow equipment such that $\bi{K}$ is a 2-category and $\bi{M}$ is biclosed.
    \begin{enumerate}
        \item\label{thm:proarr_equip_in_terms_of_versbicolim-axiom4}
            $\conj{(\cdot)}$ satisfies \ref{proarrow:coproduct} if and only if $\dbl{F}_\conj{(\cdot)}$ has finite versatile bicoproducts.
        \item\label{thm:proarr_equip_in_terms_of_versbicolim-axiom5}
            $\conj{(\cdot)}$ satisfies \ref{proarrow:EM} if and only if $\dbl{F}_\conj{(\cdot)}$ has versatile bicollapses.
    \end{enumerate}
\end{theorem}
\begin{proof}\quad
    \begin{enumerate}
        \item
            The biclosedness of $\bi{M}$ implies that $\dbl{F}_\conj{(\cdot)}$ has extensions and lifts, hence we can use \cref{thm:special_case_versbicolim}.
            Since the pseudo-functors $\conj{(\cdot)}\colon\bi{K}\to\bi{M}^\op$ and $\comp{(\cdot)}\colon\bi{K}\to\bi{M}^\co$ are compatible with conjunctions and companions in $\dbl{F}_\conj{(\cdot)}$, the equivalence \cref{thm:proarr_equip_in_terms_of_versbicolim-axiom4} follows.
        \item
            Note that monads in $\bi{M}$ are the same as \ac{AVD}-functors $\Idimdbl\zero{1}\to\dbl{F}_\conj{(\cdot)}$.
            Let $t$ be a monad in $\bi{M}$, and let $\const{t}\colon\Idimdbl\zero{1}\to\dbl{F}_\conj{(\cdot)}$ be the corresponding \ac{AVD}-functor.
            By the dual of \cref{lem:representability_of_module}, giving a $t$-algebra in $\bi{M}$ whose underlying 1-cell is representable is equivalent to giving a tight cocone from $\const{t}$ in $\dbl{F}_\conj{(\cdot)}$.
            Thus, it suffices to show that a tight cocone $\xi$ from $\const{t}$ with a vertex $\Xi$ is a versatile bicolimit if and only if the induced $t$-algebra $\conj{\xi}$, which is obtained by the dual of the construction explained in \cref{rem:companion_of_tight_cocone}, satisfies the three conditions in \ref{proarrow:EM}.
            The first and second conditions in \ref{proarrow:EM} are clearly equivalent to that the functors
            \[
                \Homcat[\dbl{F}_\conj{(\cdot)}](X,\Xi)\arr(-\conj{\xi})\Mdl{X}{\const{t}},
                \quad
                \Homcat[\dbl{F}_\conj{(\cdot)}](\Xi,X)\arr(\comp{\xi}-)\Mdl{\const{t}}{X}
            \]
            are equivalences for any $X\in\dbl{F}_\conj{(\cdot)}$.

            We now suppose the first and second conditions in \ref{proarrow:EM} for $\conj{\xi}$.
            Take $X\in\dbl{F}_\conj{(\cdot)}$ arbitrarily, and consider the following diagram of functors:
            \begin{equation}\label{eq:pseudo_comm_proarrow_equip}
                \begin{tikzcd}[hugecolumn]
                    \Homcat[\bi{K}](\Xi,X) & \Cone\vvect{\const{t}}{X} \\
                    \Homcat[\bi{M}](X,\Xi)_\mathrm{rep} & \Alg(X,t)_\mathrm{rep} \\
                    \Homcat[\bi{M}](X,\Xi) & \Alg(X,t)
                    \arrow[from=1-1,to=1-2,"\xi\tcomp-"]
                    \arrow[from=2-1,to=2-2,"-\lcomp\conj{\xi}"']
                    \arrow[from=3-1,to=3-2,"-\lcomp\conj{\xi}"',"\simeq"]
                    \arrow[from=1-1,to=2-1,"\conj{(\cdot)}"',"\simeq"]
                    \arrow[from=1-2,to=2-2,"\conj{(\cdot)}","\simeq"']
                    \arrow[from=2-1,to=3-1,hook]
                    \arrow[from=2-2,to=3-2,hook]
                    \cellsymb(\rotatebox{45}{$\cong$}){1-1}{2-2}
                    \cellsymb(\rotatebox{45}{$=$}){2-1}{3-2}
                \end{tikzcd}
            \end{equation}
            Here, $\Homcat[\bi{M}](X,\Xi)_\mathrm{rep}$ denotes the full subcategory of $\Homcat[\bi{M}](X,\Xi)$ spanned by the representable 1-cells, and $\Alg(X,t)_\mathrm{rep}$ denotes the full subcategory of $\Alg(X,t)$ spanned by the $t$-algebras whose underlying 1-cell is representable.
            The functor $\conj{(\cdot)}$ in the left column of \cref{eq:pseudo_comm_proarrow_equip} is an equivalence by \ref{proarrow:locally_ff}.
            The functor $\conj{(\cdot)}$ in the right column is essentially surjective by the dual of \cref{lem:representability_of_module}; since full faithfulness is immediate, it is an equivalence.
            Since the functor $-\lcomp\conj{\xi}$ in the bottom row of \cref{eq:pseudo_comm_proarrow_equip} is an equivalence, it can be observed that the functor $\xi\tcomp-$ in the top row is an equivalence if and only if the lower commutative square in \cref{eq:pseudo_comm_proarrow_equip} exhibits a (strict) pullback, and the latter is equivalent to the third condition in \ref{proarrow:EM}.
            Combining this with \cref{thm:special_case_versbicolim} shows that $\xi$ is a versatile bicolimit if and only if the $t$-algebra $\conj{\xi}$ satisfies the three conditions in \ref{proarrow:EM}, which finishes the proof.\qedhere
    \end{enumerate}
\end{proof}

\begin{remark}
    A similar statement to \cref{thm:proarr_equip_in_terms_of_versbicolim}\cref{thm:proarr_equip_in_terms_of_versbicolim-axiom5} can be found in \cite[Theorem 5.8]{Schultz2015regular}, where the codomain bicategory $\bi{M}$ is not necessarily biclosed.
\end{remark}

\section{Finality of augmented virtual double functors}\label[appendix]{sec:final_functors}
\begin{definition}\label{def:coslice_category}
    Let $\Phi\colon\bJ\to\bK$ be an \ac{AVD}-functor between \acp{AVDC}.
    For a loose path $A\larr(\tup{u})[path]B$ in $\bK$, we define a category $\coslice{\tup{u}}{\Phi}$ as follows:
    \begin{itemize}
        \item
            An object in $\coslice{\tup{u}}{\Phi}$ is a tuple $(X^0,X^1,X,\phi^0,\phi^1,\phi)$ of the following form:\vspace{-0.5em}
            \begin{equation}\label{eq:object_in_coslice}
                \begin{tikzcd}
                    A\ar[d,"\phi^0"']\lar[r,path,"\tup{u}"] & B\ar[d,"\phi^1"] \\
                    \Phi X^0\lar[r,phan,"\Phi X"'] & \Phi X^1
                    \cellsymb(\phi){1-1}{2-2}
                \end{tikzcd}\incat{\bK}.
            \end{equation}
            We also write $(X,\phi)$ for such an object $(X^0,X^1,X,\phi^0,\phi^1,\phi)$.
        \item
            A morphism $(X,\phi)\arr(\theta)[][1](Y,\psi)$ in $\coslice{\tup{u}}{\Phi}$ is a tuple $(\theta^0,\theta^1,\theta)$ such that
            \begin{equation*}
                \begin{tikzcd}
                    A\ar[d,"\phi^0"']\lar[r,path,"\tup{u}"] & B\ar[d,"\phi^1"] \\
                    \Phi X^0\ar[d,"\Phi\theta^0"']\lar[r,phan,"\Phi X"'] & \Phi X^1\ar[d,"\Phi\theta^1"] \\
                    \Phi Y^0\lar[r,phan,"\Phi Y"'] & \Phi Y^1
                    \cellsymb(\phi){1-1}{2-2}
                    \cellsymb(\Phi\theta){2-1}{3-2}
                \end{tikzcd}
                =
                \begin{tikzcd}
                    A\ar[d,"\psi^0"']\lar[r,path,"\tup{u}"] & B\ar[d,"\psi^1"] \\
                    \Phi Y^0\lar[r,phan,"\Phi Y"'] & \Phi Y^1
                    \cellsymb(\psi){1-1}{2-2}
                \end{tikzcd}\incat{\bK}.
            \end{equation*}
    \end{itemize}
    When $A=B$ and $\tup{u}$ is of length 0, the category $\coslice{\tup{u}}{\Phi}$ is also denoted by $\coslice{A}{\Phi}$.
\end{definition}

\begin{remark}\label{rem:ordinary_coslise_is_retract}
    In the situation of \cref{def:coslice_category}, the assignments $(X,\phi)\mapsto (X^i,\phi^i)$ \mbox{$(i=0,1)$} yield two functors to the comma categories: \mbox{$(-)^0\colon\coslice{\tup{u}}{\Phi}\to A/(\Tcat\Phi)$} and \mbox{$(-)^1\colon\coslice{\tup{u}}{\Phi}\to B/(\Tcat\Phi)$}.
    If $A=B$ and $\tup{u}$ is of length 0, both functors $(-)^0$ and $(-)^1$ have a common section:
    \begin{equation*}
        \begin{tikzcd}[scriptsize]
            & A/(\Tcat\Phi)\ar[dl,equal]\ar[d]\ar[dr,equal] & \\
            A/(\Tcat\Phi) & \ar[l,"(-)^0"]\coslice{A}{\Phi}\ar[r,"(-)^1"'] & A/(\Tcat\Phi)
        \end{tikzcd}
    \end{equation*}
    Indeed, the assignment 
    \begin{equation*}
        \begin{tikzcd}[scriptsize]
            A\ar[d,"p"'] \\
            \Phi X
        \end{tikzcd}
        \quad\mapsto\quad
        \begin{tikzcd}[scriptsize]
            A\ar[d,bend right=20,"p"{left}]\ar[d,bend left=20,"p"{right}] \\
            \Phi X
            \cellsymb(\heq){1-1}{2-1}
        \end{tikzcd}
    \end{equation*}
    gives such a common section $A/(\Tcat\Phi)\to\coslice{A}{\Phi}$.
\end{remark}

As in \cite{Pare1990simply}, we use the following terminology:
\begin{definition}
    For a category $\one{C}$, we write $\fgrpd\one{C}$ for the strict localization of $\one{C}$ by all morphisms.
    The groupoid $\fgrpd\one{C}$ is called the \emph{fundamental groupoid} of $\one{C}$.
    A category $\one{C}$ is called \emph{simply connected} if the fundamental groupoid $\fgrpd\one{C}$ has at most one morphism between any two objects.
\end{definition}

\begin{definition}
    An \ac{AVD}-functor $\Phi\colon\bJ\to\bK$ between \acp{AVDC} is called \emph{naively final} if:
    \begin{itemize}
        \item
            For every object $A\in\bK$, the comma category $A/(\Tcat\Phi)$ is simply connected.
        \item
            For every loose path $\tup{u}$ in $\bK$, the category $\coslice{\tup{u}}{\Phi}$ is connected.
        \item
            For every loose path $A_0\larr(\tup{u})[path]A_n$ in $\bK$, there exist data of the following form:
            \begin{equation}\label{eq:diagram_finality}
                \begin{tikzcd}
                    A_0\ar[d,"p_0"']\lar[r,"u_1"] & A_1\ar[d,"p_1"]\lar[r,"u_2"] & \cdots\lar[r,"u_n"] & A_n\ar[d,"p_n"] \\
                    \Phi X_0\ar[d,"\Phi f"']\lar[r,phan,"\Phi v_1"] & \Phi X_1\lar[r,phan,"\Phi v_2"] & \cdots\lar[r,phan,"\Phi v_n"] & \Phi X_n\ar[d,"\Phi g"] \\
                    \Phi Y\lar[rrr,phan,"\Phi w"'] &&& \Phi Z
                    \cellsymb(\phi_1){1-1}{2-2}
                    \cellsymb(\phi_2){1-2}{2-3}
                    \cellsymb(\phi_n){1-4}{2-3}
                    \cellsymb(\Phi\theta){2-1}{3-4}
                \end{tikzcd}\incat{\bK}.
            \end{equation}
    \end{itemize}
\end{definition}

\begin{example}
    For a large set $\zero{S}$, the inclusion \ac{AVD}-functor $\Idimdbl\zero{S}\to\Idbl\zero{S}$ is always naively final.
    On the other hand, the inclusion $\Idimdbl\one{C}\to\Idbl\one{C}$ for a category $\one{C}$ is not necessarily naively final due to the lack of simple connectedness of the coslice categories $c/\one{C}$.
\end{example}

\begin{lemma}\label{lem:tightwise_functor_of_final_AVDfunctor}
    Let $\Phi\colon\bJ\to\bK$ be a naively final \ac{AVD}-functor between \acp{AVDC}.
    Then, for every $A\in\bK$, the comma category $A/(\Tcat\Phi)$ is connected (and simply connected).
\end{lemma}
\begin{proof}
    This follows from that $A/(\Tcat\Phi)$ is a retract of the category $\coslice{A}{\Phi}$ for any $A\in\bK$ (\cref{rem:ordinary_coslise_is_retract}).
\end{proof}

\begin{proposition}\label{prop:final_loosewise_VD_indiscrete}
    Let $\one{C},\one{D}$ be categories, and let $\Phi\colon\Idimdbl\one{C}\to\Idimdbl\one{D}$ be an \ac{AVD}-functor, which is the same data as the functor $\Tcat\Phi\colon\one{C}\to\one{D}$.
    Then, the following are equivalent:
    \begin{enumerate}
        \item\label{prop:final_loosewise_VD_indiscrete-1}
            For every object $d\in\one{D}$, the comma category $d/(\Tcat\Phi)$ is connected and simply connected.
        \item\label{prop:final_loosewise_VD_indiscrete-2}
            The \ac{AVD}-functor $\Phi$ is naively final.
    \end{enumerate}
\end{proposition}
\begin{proof}
    \proofdirection{\cref{prop:final_loosewise_VD_indiscrete-2}}{\cref{prop:final_loosewise_VD_indiscrete-1}} This follows from \cref{lem:tightwise_functor_of_final_AVDfunctor}.

    \proofdirection{\cref{prop:final_loosewise_VD_indiscrete-1}}{\cref{prop:final_loosewise_VD_indiscrete-2}}
    The first and third conditions for naive finality are trivial.
    We will show the second condition.
    Let $a\larr(\tup{u})[path]b$ in $\Idimdbl\one{D}$ be a loose path.
    The following shows that every object $(x,\phi)$ in $\coslice{\tup{u}}{\Phi}$ on the left below is connected with an object such that $X$ is of length 1 in \cref{eq:object_in_coslice}:
    \begin{equation*}
        \begin{tikzcd}
            a\ar[d,"\phi^0"']\lar[r,path,"\tup{u}"] & b\ar[d,"\phi^1"] \\
            \Phi x^0\ar[d,equal]\lar[r,phan,"\Phi x"] & \Phi x^1\ar[d,equal] \\
            \Phi x^0\lar[r,"\Phi !"'] & \Phi x^1
            \cellsymb(\phi){1-1}{2-2}
            \cellsymb(\Phi !){2-1}{3-2}
        \end{tikzcd}
        =
        \begin{tikzcd}
            a\ar[d,"\phi^0"']\lar[r,path,"\tup{u}"] & b\ar[d,"\phi^1"] \\
            \Phi x^0\lar[r,"\Phi !"'] & \Phi x^1
            \cellsymb(!)[right=-2]{1-1}{2-2}
        \end{tikzcd}\incat{\Idimdbl\one{D}}
    \end{equation*}
    The full subcategory of $\coslice{\tup{u}}{\Phi}$ consists of objects where $X$ is of length 1 in \cref{eq:object_in_coslice} is isomorphic to a product $a/(\Tcat\Phi)\times b/(\Tcat\Phi)$ of comma categories, which are connected by the assumption.
    Therefore, $\coslice{\tup{u}}{\Phi}$ is connected.
\end{proof}

We now present a slight generalization of cartesian cells.
While this may seem somewhat technical, we introduce it here since it will be used later.
\begin{definition}
    Let $A\larr(\tup{u})[path]B$ be a loose path in an \ac{AVDC} $\bL$.
    Let $\one{C}$ be a category, and let $F\colon\one{C}\to\Lphancat\bL$ be a functor.
    A \emph{cone} over $F$ with the vertex $\tup{u}$ is a family of cells $\alpha_c$ for $c\in\one{C}$ satisfying the following equality for any morphism $c\arr(s)[][1]d$ in $\one{C}$:
    \begin{equation*}
        \begin{tikzcd}
            A\ar[d,"\alpha^0_c"']\lar[r,path,"\tup{u}"] & B\ar[d,"\alpha^1_c"] \\
            F^0c\ar[d,"F^0s"']\lar[r,phan,"Fc"] & F^1c\ar[d,"F^1s"] \\
            F^0d\lar[r,phan,"Fd"'] & F^1d
            \cellsymb(\alpha_c){1-1}{2-2}
            \cellsymb(Fs){2-1}{3-2}
        \end{tikzcd}
        =
        \begin{tikzcd}
            A\ar[d,"\alpha^0_d"']\lar[r,path,"\tup{u}"] & B\ar[d,"\alpha^1_d"] \\
            F^0d\lar[r,phan,"Fd"'] & F^1d
            \cellsymb(\alpha_d){1-1}{2-2}
        \end{tikzcd}\incat{\bL}.
    \end{equation*}
\end{definition}

\begin{definition}[Jointly cartesian cells]
    Let $\bL$ be an \ac{AVDC}, let $\one{C}$ be a category, and let $F\colon\one{C}\to\Lphancat\bL$ be a functor.
    A cone over $F$
    \begin{equation*}
        \begin{tikzcd}
            X^0\ar[d,"\alpha^0_c"']\lar[r,phan,"X"] & X^1\ar[d,"\alpha^1_c"] \\
            F^0c\lar[r,phan,"Fc"'] & F^1c
            \cellsymb(\alpha_c){1-1}{2-2}
        \end{tikzcd}\incat{\bL}\quad (c\in\one{C})
    \end{equation*}
    is called \emph{jointly cartesian} in $\bL$ if it satisfies the following condition:
    Suppose that we are given a loose path $A\larr(\tup{u})[path]B$, tight arrows $A\arr(f)[][1]X^0$ and $B\arr(g)[][1]X^1$, and a cone $\beta$ over $F$ on the right below; then there uniquely exists a cell $\gamma$ satisfying the following equality for any $c\in\one{C}$.
    \begin{equation*}
        \begin{tikzcd}
            A\ar[d,"f"']\lar[r,path,"\tup{u}"] & B\ar[d,"g"] \\
            X^0\ar[d,"\alpha^0_c"']\lar[r,phan,"X"] & X^1\ar[d,"\alpha^1_c"] \\
            F^0c\lar[r,phan,"Fc"'] & F^1c
            \cellsymb(\gamma){1-1}{2-2}
            \cellsymb(\alpha_c){2-1}{3-2}
        \end{tikzcd}
        =
        \begin{tikzcd}
            A\ar[d,"f"']\lar[r,path,"\tup{u}"] & B\ar[d,"g"] \\
            X^0\ar[d,"\alpha^0_c"'] & X^1\ar[d,"\alpha^1_c"] \\
            F^0c\lar[r,phan,"Fc"'] & F^1c
            \cellsymb(\beta_c){1-1}{3-2}
        \end{tikzcd}\incat{\bL}
    \end{equation*}
\end{definition}

\begin{notation}
    Let $\Phi\colon\bJ\to\bK$ and $F\colon\bK\to\bL$ be \ac{AVD}-functors between \acp{AVDC}.
    Then, a tight cocone $l$ from $F$ yields a tight cocone from $F\Phi$, denoted by $l_\Phi$, in a natural way.
    We also use such a notation for modules and modulations.
\end{notation}

\begin{theorem}\label{thm:invariance_by_finality}
    Let $\Phi\colon\bJ\to\bK$ be a naively final \ac{AVD}-functor.
    Then, the following hold for any \ac{AVD}-functor $F\colon\bK\to\bL$.
    \begin{enumerate}
        \item\label{thm:invariance_by_finality-cocones}
            The assignment $l\mapsto l_\Phi$ yields isomorphisms of categories
            \begin{equation*}
                -_\Phi\colon\Cone\vvect{F}{L}\arr(\cong)\Cone\vvect{F\Phi}{L}
                \qquad
                (L\in\bL).
            \end{equation*}
        \item\label{thm:invariance_by_finality-modules}
            Assume the following additional condition: for any $A\in\bK$, there exists an object $(X,p)\in A/(\Tcat\Phi)$ such that $Fp$ is left-pulling in $\bL$.
            Then, the assignment $m\mapsto m_\Phi$ yields equivalences of categories
            \begin{equation*}
                -_\Phi\colon\Mdl{F}{M}\arr(\simeq)\Mdl{F\Phi}{M}
                \qquad
                (M\in\bL).
            \end{equation*}
        \item\label{thm:invariance_by_finality-modulations}
            The assignment $\rho\mapsto\rho_\Phi$ yields bijections among the classes of modulations of the same type.\qedhere
    \end{enumerate}
\end{theorem}
\begin{proof}
    We first show \cref{thm:invariance_by_finality-modulations} for modulations of type 1.
    Let $\sigma$ be a modulation of type 1 exhibited by the following:
    \begin{equation*}
        \begin{tikzcd}
            F\Phi\ar[d,Rightarrow,"l_\Phi"']\lar[r,Rightarrow,"m_\Phi"] & M\lar[r,path,"\tup{p}"] & M'\ar[d,"j"] \\
            L\lar[rr,phan,"q"'] & & L'
            \cellsymb(\sigma){1-1}{2-3}
        \end{tikzcd}
    \end{equation*}
    Here, $m$ is a left $F$-module, and $l$ is a tight cocone from $F$.
    We have to construct a modulation $\fs$ such that $\fs_\Phi=\sigma$.
    For each $A\in\bK$, let us take a tight arrow $A\arr(a)[][1]\Phi X$ in $\bK$ by using the ordinary finality of $\Tcat\Phi$ and define $\fs_A$ as the following cell:
    \begin{equation*}
        \fs_A\coloneq
        \begin{tikzcd}
            FA\ar[d,"Fa"']\lar[r,"m_A"] & M\ar[d,equal]\lar[r,path,"\tup{p}"] & M'\ar[d,equal] \\
            F\Phi X\ar[d,"l_{\Phi X}"']\lar[r,"m_{\Phi X}"'] & M\lar[r,path,"\tup{p}"'] & M'\ar[d,"j"] \\
            L\lar[rr,phan,"q"'] && L'
            \cellsymb(m_a){1-1}{2-2}
            \cellsymb(\veq){1-2}{2-3}
            \cellsymb(\sigma_X){2-1}{3-3}
        \end{tikzcd}\incat{\bL}.
    \end{equation*}
    By using the ordinary finality of $\Tcat\Phi$ again, we can show that the cells $\fs_A$ are independent of the choice of $A\arr(a)[][1]\Phi X$.
    Then, from the independence of $\fs_A$ and the second condition in the definition of naive finality, it easily follows that the cells $\fs$ form a desired modulation $\fs$.
    The uniqueness of $\fs$ is trivial.
    The same argument works in the case of modulations of the other types.

    We next show \cref{thm:invariance_by_finality-cocones}.
    Since the functor $-_\Phi\colon\Cone\vvect{F}{L}\arr[][1]\Cone\vvect{F\Phi}{L}$ is fully faithful by \cref{thm:invariance_by_finality-modulations}, it suffices to show that the functor $-_\Phi$ is bijective on objects.
    Let $l$ be a tight cocone from $F\Phi$ to $L$.
    Since $A/(\Tcat\Phi)$ is connected for each $A\in\bK$, we can define $\fl_A$ as $(Fp)\tcomp l_X$ independently of a choice of $A\arr(p)[][1]\Phi X$ in $\bK$.
    Since $\coslice{\tup{u}}{\Phi}$ is connected for $A_0\larr(\tup{u})[path][1]A_n$ in $\bK$, we can also define a cell $\fl_\tup{u}$ as follows independently of a choice of an object $(X,\phi)\in\coslice{\tup{u}}{\Phi}$:
    \begin{equation*}
        \begin{tikzcd}[tri]
            FA_0\ar[dr,"\fl_{A_0}"']\lar[rr,path,"F\tup{u}"] &{}& FA_n\ar[dl,"\fl_{A_n}"] \\
            & L &
            \cellsymb(\fl_\tup{u})[above=-3]{1-2}{2-2}
        \end{tikzcd}
        \coloneq
        \begin{tikzcd}[tri]
            FA_0\ar[d,"F\phi^0"']\lar[rr,path,"F\tup{u}"] && FA_n\ar[d,"F\phi^1"] \\
            F\Phi X^0\ar[dr,"l_{X^0}"']\lar[rr,phan,"F\Phi X"] &{}& F\Phi X^1\ar[dl,"l_{X^1}"] \\
            & L &
            \cellsymb(F\phi){1-1}{2-3}
            \cellsymb(l_X)[above=-3]{2-2}{3-2}
        \end{tikzcd}\incat{\bL}.
    \end{equation*}
    Taking data $(\tup{X},Y,Z,\tup{p},f,g,\tup{v},w,\tup{\phi},\theta)$ as in \cref{eq:diagram_finality}, we can show that the cell $\fl_\tup{u}$ is the composite of the cells $(\fl_{u_1},\dots,\fl_{u_n})$:
    \begin{equation*}
        \begin{tikzcd}[tri]
            FA_0\ar[dr,"\fl_{A_0}"']\lar[rr,path,"F\tup{u}"] &{}& FA_n\ar[dl,"\fl_{A_n}"] \\
            & L &
            \cellsymb(\fl_\tup{u})[above=-3]{1-2}{2-2}
        \end{tikzcd}
        =
        \begin{tikzcd}[tri]
            FA_0\ar[d,"Fp_0"']\lar[rr,path,"F\tup{u}"] && FA_n\ar[d,"Fp_n"] \\
            F\Phi X_0\ar[d,"F\Phi f"']\lar[rr,path,"F\Phi\tup{v}"] && F\Phi X_n\ar[d,"F\Phi g"] \\
            F\Phi Y\ar[dr,"l_Y"']\lar[rr,phan,"F\Phi w"] &{}& F\Phi Z\ar[dl,"l_Z"] \\
            & L &
            \cellsymb(F\tup{\phi}){1-1}{2-3}
            \cellsymb(F\Phi\theta){2-1}{3-3}
            \cellsymb(l_w)[above=-3]{3-2}{4-2}
        \end{tikzcd}
    \end{equation*}
    \begin{equation*}
        =
        \begin{tikzcd}
            FA_0\ar[d,"Fp_0"']\lar[r,"Fu_1"] & FA_1\ar[d,"Fp_1"{description}]\lar[r,"Fu_2"] & \cdots\lar[r,"Fu_{n-1}"] & FA_{n-1}\ar[d,"Fp_{n-1}"{description}]\lar[r,"Fu_n"] & FA_n\ar[d,"Fp_n"] \\
            F\Phi X_0\ar[drr,"l_{X_0}"',bend right=10]\lar[r,phan,"F\Phi v_1"] & F\Phi X_1\ar[dr,"l_{X_1}"{right}]\lar[r,phan,"F\Phi v_2"] & \cdots\lar[r,phan,"F\Phi v_{n-1}"] & F\Phi X_{n-1}\ar[dl,"l_{X_{n-1}}"{left}]\lar[r,phan,"F\Phi v_n"] & F\Phi X_n\ar[dll,"l_{X_n}",bend left=10] \\[10pt]
            && L &&
            \cellsymb(F\phi_1){1-1}{2-2}
            \cellsymb(F\phi_2){1-2}{2-3}
            \cellsymb(F\phi_{n-1}){1-4}{2-3}
            \cellsymb(F\phi_n){1-4}{2-5}
            \cellsymb(l_{v_1})[above]{2-1}{3-3}
            \cellsymb(l_{v_n})[above]{2-5}{3-3}
        \end{tikzcd}
    \end{equation*}
    \begin{equation*}
        =
        \begin{tikzcd}[scriptsizecolumn]
            FA_0\ar[drr,"\fl_{A_0}"',bend right=10]\lar[r,"Fu_1"] & FA_1\ar[dr,"\fl_{A_1}"{right}]\lar[r,"Fu_2"] & \cdots\lar[r,"Fu_{n-1}"] & FA_{n-1}\ar[dl,"\fl_{A_{n-1}}"{left}]\lar[r,"Fu_n"] & FA_n\ar[dll,"\fl_{A_n}",bend left=10] \\[10pt]
            && L &&
            \cellsymb(\fl_{u_1})[above]{1-1}{2-3}
            \cellsymb(\fl_{u_n})[above]{1-5}{2-3}
        \end{tikzcd}\incat{\bL}.
    \end{equation*}
    To show that $\fl$ is a tight cocone, take an arbitrary cell
    \begin{equation}\label{eq:arbitrary_cell_in_K}
        \begin{tikzcd}
            A_0\ar[d,"b"']\lar[r,path,"\tup{u}"] & A_n\ar[d,"c"] \\
            B\lar[r,phan,"v"'] & C
            \cellsymb(\alpha){1-1}{2-2}
        \end{tikzcd}\incat{\bK}.
    \end{equation}
    Taking an object $(Z,\chi)\in\coslice{v}{\Phi}$, we have the following:
    \begin{equation*}
        \begin{tikzcd}[tri]
            FA_0\ar[d,"Fb"']\lar[rr,path,"F\tup{u}"] & & FA_n\ar[d,"Fc"] \\
            FB\ar[dr,"\fl_B"']\lar[rr,phan,"Fv"] &{}& FC\ar[dl,"\fl_C"] \\
            & L &
            \cellsymb(F\alpha){1-1}{2-3}
            \cellsymb(\fl_v)[above=-3]{2-2}{3-2}
        \end{tikzcd}
        =
        \begin{tikzcd}[tri]
            FA_0\ar[d,"Fb"']\lar[rr,path,"F\tup{u}"] & & FA_n\ar[d,"Fc"] \\
            FB\ar[d,"F\chi^0"']\lar[rr,phan,"Fv"] & & FC\ar[d,"F\chi^1"] \\
            F\Phi Z^0\ar[dr,"l_{Z^0}"']\lar[rr,phan,"F\Phi Z"] &{}& F\Phi Z^1\ar[dl,"l_{Z^1}"] \\
            & L &
            \cellsymb(F\alpha){1-1}{2-3}
            \cellsymb(F\chi){2-1}{3-3}
            \cellsymb(l_Z)[above=-3]{3-2}{4-2}
        \end{tikzcd}
        =
        \begin{tikzcd}[tri]
            FA_0\ar[dr,"\fl_{A_0}"']\lar[rr,path,"F\tup{u}"] &{}& FA_n\ar[dl,"\fl_{A_n}"] \\
            & L &
            \cellsymb(\fl_\tup{u})[above=-3]{1-2}{2-2}
        \end{tikzcd}\incat{\bL}.
    \end{equation*}
    Therefore, $\fl$ becomes a tight cocone.
    The uniqueness of $\fl$ is trivial.

    We next show \cref{thm:invariance_by_finality-modules} under the additional assumption on left-pullingness.
    Since the functor $-_\Phi\colon\Mdl{F}{M}\arr[][1]\Mdl{F\Phi}{M}$ is fully faithful by \cref{thm:invariance_by_finality-modulations}, it suffices to show that the functor $-_\Phi$ is essentially surjective.
    Let $m$ be a left $F\Phi$-module with a vertex $M$.
    Consider a functor $G_A\colon A/(\Tcat\Phi)\to\Larcat\bL$ defined by the following assignment:
    \begin{equation*}
        \begin{tikzcd}[small]
            A\ar[d,"p"] \\
            \Phi X
        \end{tikzcd}\incat{\bK}
        \qquad\mapsto\qquad
        \begin{tikzcd}
            F\Phi X\lar[r,"m_X"] & M
        \end{tikzcd}\incat{\bL}.
    \end{equation*}
    Note that $G_A$ can be decomposed into two functors $A/(\Tcat\Phi)\to\Tcat\bJ\arr(m_{(-)})\Larcat\bL$, where the first one is the forgetful functor and the second one is induced by the left module $m$.
    By the assumption, there are an object $A\arr(p_0)[][1]\Phi X_0$ in $A/(\Tcat\Phi)$ and a restriction, denoted by $\fm_A$, of the following form:
    \begin{equation}\label{eq:restriction_given_by_assumption}
        \begin{tikzcd}
            FA\ar[d,"Fp_0"']\lar[r,"\fm_A"] & M\ar[d,equal] \\
            F\Phi X_0\lar[r,"m_{X_0}"'] & M
            \cellsymb(\cart){1-1}{2-2}
        \end{tikzcd}\incat{\bL}.
    \end{equation}
    Since $A/(\Tcat\Phi)$ is connected and simply connected, the above cell \cref{eq:restriction_given_by_assumption} uniquely extends to a cone over $G_A$ of the following form:
    \begin{equation}\label{eq:cone_over_GA}
        \begin{tikzcd}
            FA\ar[d,"Fp"']\lar[r,"\fm_A"] & M\ar[d,equal] \\
            F\Phi X\lar[r,"m_X"'] & M
            \cellsymb(\rho^p_X\colon\cart){1-1}{2-2}
        \end{tikzcd}\incat{\bL}\text{, where }(X,p)\in A/(\Tcat\Phi).
    \end{equation}
    Note that $\rho^p_X$ automatically becomes cartesian since the cell \cref{eq:restriction_given_by_assumption} (=$\rho^{p_0}_{X_0}$) is cartesian.
    Since $A/(\Tcat\Phi)$ is connected, the cone \cref{eq:cone_over_GA} over $G_A$ becomes jointly cartesian.
    Furthermore, since $\coslice{\tup{u}}{\Phi}$ is connected for $A\larr(\tup{u})[path]B$ in $\bK$, a cone over $\coslice{\tup{u}}{\Phi}\arr((-)^0) A/(\Tcat\Phi)\arr(G_A)\Larcat\bL$ obtained by composing $(-)^0$ with the cone \cref{eq:cone_over_GA} also becomes jointly cartesian.

    Let $A\arr(f)[][1]B$ be a tight arrow in $\bK$.
    Then, the assignment to $(X,p)\in B/(\Tcat\Phi)$, the cell $\rho^{f\tcomp p}_X$ gives a cone over $G_B$.
    Using the joint cartesianness of ``$\rho$,'' we have a unique cell $\fm_f$ satisfying the following for any $(X,p)\in B/(\Tcat\Phi)$:
    \begin{equation*}
        \begin{tikzcd}
            FA\ar[d,"Ff"']\lar[r,"\fm_{A}"] & M\ar[d,equal] \\
            FB\ar[d,"Fp"'] & M\ar[d,equal] \\
            F\Phi X\lar[r,"m_X"'] & M
            \cellsymb(\rho^{f\tcomp p}_X){1-1}{3-2}
        \end{tikzcd}
        =
        \begin{tikzcd}
            FA\ar[d,"Ff"']\lar[r,"\fm_A"] & M\ar[d,equal] \\
            FB\ar[d,"Fp"']\lar[r,"\fm_B"] & M\ar[d,equal] \\
            F\Phi X\lar[r,"m_X"'] & M
            \cellsymb(\fm_f){1-1}{2-2}
            \cellsymb(\rho^p_X){2-1}{3-2}
        \end{tikzcd}\incat{\bL}.
    \end{equation*}
    It easily follows that the assignment $f\mapsto \fm_f$ is functorial.
    
    Let $A_0\larr(\tup{u})[path]A_n$ be a loose path in $\bK$.
    Then, the assignment to $(X,\phi)\in\coslice{\tup{u}}{\Phi}$, a cell on the left below gives a cone over $\coslice{\tup{u}}{\Phi}\arr((-)^0) A_0/(\Tcat\Phi)\arr(G_{A_0})\Larcat\bL$.
    Using the joint cartesianness of ``$\rho$,'' we have a unique cell, denoted by $\fm_\tup{u}$, such that the following holds for every object $(X,\phi)\in\coslice{\tup{u}}{\Phi}$:
    \begin{equation*}
        \begin{tikzcd}
            FA_0\ar[d,"F\phi^0"']\lar[r,path,"F\tup{u}"] & FA_n\ar[d,"F\phi^1"]\lar[r,"\fm_{A_n}"] & M\ar[d,equal] \\
            F\Phi X^0\ar[d,equal]\lar[r,phan,"F\Phi X"'] & F\Phi X^1\lar[r,"m_{X^1}"'] & M\ar[d,equal] \\
            F\Phi X^0\lar[rr,"m_{X^0}"'] && M
            \cellsymb(F\phi){1-1}{2-2}
            \cellsymb(\rho^{\phi^1}_{X^1}){1-2}{2-3}
            \cellsymb(m_X){2-1}{3-3}
        \end{tikzcd}
        =
        \begin{tikzcd}
            FA_0\ar[d,equal]\lar[r,path,"F\tup{u}"] & FA_n\lar[r,"\fm_{A_n}"] & M\ar[d,equal] \\
            FA_0\ar[d,"F\phi^0"']\lar[rr,"\fm_{A_0}"] & & M\ar[d,equal] \\
            F\Phi X^0\lar[rr,"m_{X^0}"'] & & M
            \cellsymb(\fm_\tup{u}){1-1}{2-3}
            \cellsymb(\rho^{\phi^0}_{X^0}){2-1}{3-3}
        \end{tikzcd}\incat{\bL}.
    \end{equation*}
    Taking data $(\tup{X},Y,Z,\tup{p},f,g,\tup{v},w,\tup{\phi},\theta)$ as in \cref{eq:diagram_finality}, we can decompose the cell $\fm_\tup{u}$ into the cells $(\fm_{u_1},\dots,\fm_{u_n})$ as follows:
    \begin{equation*}
        \begin{tikzcd}
            FA_0\ar[d,"Fp_0"']\lar[r,path,"F\tup{u}"] & FA_n\ar[d,"Fp_n"']\lar[r,"\fm_{A_n}"] & M\ar[dd,equal] \\
            F\Phi X_0\ar[d,"F\Phi f"']\lar[r,path,"F\Phi\tup{v}"] & F\Phi X_n\ar[d,"F\Phi g"'] & \\
            F\Phi Y\ar[d,equal]\lar[r,phan,"F\Phi w"'] & F\Phi Z\lar[r,"m_Z"'] & M\ar[d,equal] \\
            F\Phi Y\lar[rr,"m_Y"'] && M
            \cellsymb(F\tup{\phi}){1-1}{2-2}
            \cellsymb(F\Phi\theta){2-1}{3-2}
            \cellsymb(\rho^{p_n\tcomp\Phi g}_Z)[right=-8]{1-2}{3-3}
            \cellsymb(m_w){3-1}{4-3}
        \end{tikzcd}
        =
        \begin{tikzcd}
            FA_0\ar[d,"Fp_0"']\lar[r,path,"F\tup{u}"] & FA_n\ar[d,"Fp_n"']\lar[r,"\fm_{A_n}"] & M\ar[d,equal] \\
            F\Phi X_0\ar[d,"F\Phi f"']\lar[r,path,"F\Phi\tup{v}"] & F\Phi X_n\ar[d,"F\Phi g"']\lar[r,"m_{X_n}"] & M\ar[d,equal] \\
            F\Phi Y\ar[d,equal]\lar[r,phan,"F\Phi w"'] & F\Phi Z\lar[r,"m_Z"'] & M\ar[d,equal] \\
            F\Phi Y\lar[rr,"m_Y"'] && M
            \cellsymb(F\tup{\phi}){1-1}{2-2}
            \cellsymb(F\Phi\theta){2-1}{3-2}
            \cellsymb(\rho^{p_n}_{X_n}){1-2}{2-3}
            \cellsymb(m_g){2-2}{3-3}
            \cellsymb(m_w){3-1}{4-3}
        \end{tikzcd}
    \end{equation*}
    \begin{equation*}
        =
        \begin{tikzcd}
            FA_0\ar[d,"Fp_0"']\lar[r,path,"F\tup{u}"] & FA_n\ar[d,"Fp_n"']\lar[r,"\fm_{A_n}"] & M\ar[d,equal] \\
            F\Phi X_0\ar[d,equal]\lar[r,path,"F\Phi\tup{v}"] & F\Phi X_n\lar[r,"m_{X_n}"] & M\ar[d,equal] \\
            F\Phi X_0\ar[d,"F\Phi f"']\lar[rr,"m_{X_0}"] & & M\ar[d,equal] \\
            F\Phi Y\lar[rr,"m_Y"'] && M
            \cellsymb(F\tup{\phi}){1-1}{2-2}
            \cellsymb(\rho^{p_n}_{X_n}){1-2}{2-3}
            \cellsymb(m_\tup{v}){2-1}{3-3}
            \cellsymb(m_f){3-1}{4-3}
        \end{tikzcd}
        =
        \begin{tikzcd}
            FA_0\ar[d,equal]\lar[r,path,"{F(u_1,\dots,u_{n-1})}"] &[20pt] FA_{n-1}\ar[d,equal]\lar[r,"Fu_n"] & FA_n\lar[r,"\fm_{A_n}"] & M\ar[d,equal] \\
            FA_0\ar[d,"Fp_0"']\lar[r,path] & FA_{n-1}\ar[d,"Fp_{n-1}"]\lar[rr,"\fm_{A_{n-1}}"] & & M\ar[d,equal] \\
            F\Phi X_0\ar[d,equal]\lar[r,path,"{F\Phi(v_1,\dots,v_{n-1})}"] & F\Phi X_{n-1}\lar[rr,"m_{X_{n-1}}"] & & M\ar[d,equal] \\
            F\Phi X_0\ar[d,"F\Phi f"']\lar[rrr,"m_{X_0}"] & & & M\ar[d,equal] \\
            F\Phi Y\lar[rrr,"m_Y"'] & & & M
            \cellsymb(\veq){1-1}{2-2}
            \cellsymb(\fm_{u_n}){1-2}{2-4}
            \cellsymb(F(\phi_1,\dots,\phi_{n-1})){2-1}{3-2}
            \cellsymb(\rho^{p_{n-1}}_{X_{n-1}}){2-2}{3-4}
            \cellsymb(m_{(v_1,\dots,v_{n-1})}){3-1}{4-4}
            \cellsymb(m_f){4-1}{5-4}
        \end{tikzcd}
    \end{equation*}
    \begin{equation*}
        =\cdots =
        \begin{tikzcd}
            FA_0\ar[d,equal]\lar[r,path,"F\tup{u}"] & FA_n\lar[r,"\fm_{A_n}"] & M\ar[d,equal] \\
            FA_0\ar[d,"Fp_0"']\lar[rr,"\fm_{A_0}"] & & M\ar[d,equal] \\
            F\Phi X_0\ar[d,"F\Phi f"']\lar[rr,"m_{X_0}"] & & M\ar[d,equal] \\
            F\Phi Y\lar[rr,"m_Y"'] & & M
            \cellsymb((\fm_{u_1},\dots,\fm_{u_n})){1-1}{2-3}
            \cellsymb(\rho^{p_0}_{X_0}){2-1}{3-3}
            \cellsymb(m_f){3-1}{4-3}
        \end{tikzcd}
        =
        \begin{tikzcd}
            FA_0\ar[d,equal]\lar[r,path,"F\tup{u}"] & FA_n\lar[r,"\fm_{A_n}"] & M\ar[d,equal] \\
            FA_0\ar[d,"Fp_0"']\lar[rr,"\fm_{A_0}"] & & M\ar[dd,equal] \\
            F\Phi X_0\ar[d,"F\Phi f"'] & & \\
            F\Phi Y\lar[rr,"m_Y"'] & & M
            \cellsymb((\fm_{u_1},\dots,\fm_{u_n})){1-1}{2-3}
            \cellsymb(\rho^{p_0\tcomp\Phi f}_Y\colon\cart){2-1}{4-3}
        \end{tikzcd}\incat{\bL}.
    \end{equation*}
    To show that $\fm$ is a left $F$-module, let us take an arbitrary cell $\alpha$ in $\bK$ as in \cref{eq:arbitrary_cell_in_K}.
    Taking an object $(Y,\psi)\in\coslice{v}{\Phi}$, we have the following:
    \begin{equation*}
        \begin{tikzcd}
            FA_0\ar[d,"Fb"']\lar[r,path,"F\tup{u}"] & FA_n\ar[d,"Fc"]\lar[r,"\fm_{A_n}"] & M\ar[d,equal] \\
            FB\ar[d,equal]\lar[r,phan,"Fv"] & FC\lar[r,"\fm_C"] & M\ar[d,equal] \\
            FB\ar[d,"F\psi^0"']\lar[rr,"\fm_B"] & & M\ar[d,equal] \\
            F\Phi Y^0\lar[rr,"m_{Y^0}"'] & & M
            \cellsymb(F\alpha){1-1}{2-2}
            \cellsymb(\fm_c){1-2}{2-3}
            \cellsymb(\fm_v){2-1}{3-3}
            \cellsymb(\rho^{\psi^0}_{Y^0}\colon\cart){3-1}{4-3}
        \end{tikzcd}
        =
        \begin{tikzcd}
            FA_0\ar[d,"Fb"']\lar[r,path,"F\tup{u}"] & FA_n\ar[d,"Fc"]\lar[r,"\fm_{A_n}"] & M\ar[d,equal] \\
            FB\ar[d,"F\psi^0"']\lar[r,phan,"Fv"] & FC\ar[d,"F\psi^1"]\lar[r,"\fm_C"] & M\ar[d,equal] \\
            F\Phi Y^0\ar[d,equal]\lar[r,phan,"F\Phi Y"] & F\Phi Y^1\lar[r,"m_{Y^1}"] & M\ar[d,equal] \\
            F\Phi Y^0\lar[rr,"m_{Y^0}"'] & & M
            \cellsymb(F\alpha){1-1}{2-2}
            \cellsymb(\fm_c){1-2}{2-3}
            \cellsymb(F\psi){2-1}{3-2}
            \cellsymb(\rho^{\psi^1}_{Y^1}){2-2}{3-3}
            \cellsymb(m_Y){3-1}{4-3}
        \end{tikzcd}
    \end{equation*}
    \begin{equation*}
        =
        \begin{tikzcd}
            FA_0\ar[d,"Fb"']\lar[r,path,"F\tup{u}"] & FA_n\ar[d,"Fc"]\lar[r,"\fm_{A_n}"] & M\ar[dd,equal] \\
            FB\ar[d,"F\psi^0"'] & FC\ar[d,"F\psi^1"] & \\
            F\Phi Y^0\ar[d,equal]\lar[r,phan,"F\Phi Y"] & F\Phi Y^1\lar[r,"m_{Y^1}"] & M\ar[d,equal] \\
            F\Phi Y^0\lar[rr,"m_{Y^0}"'] & & M
            \cellsymb(F(\alpha\tcomp\psi)){1-1}{3-2}
            \cellsymb(\rho^{c\tcomp\psi^1}_{Y^1})[right=-8]{1-2}{3-3}
            \cellsymb(m_Y){3-1}{4-3}
        \end{tikzcd}
        =
        \begin{tikzcd}
            FA_0\ar[d,equal]\lar[r,path,"F\tup{u}"] & FA_n\lar[r,"\fm_{A_n}"] & M\ar[d,equal] \\
            FA_0\ar[d,"Fb"']\lar[rr,"\fm_{A_0}"] & & M\ar[dd,equal] \\
            FB\ar[d,"F\psi^0"'] & & \\
            F\Phi Y^0\lar[rr,"m_{Y^0}"'] & & M
            \cellsymb(\fm_\tup{u}){1-1}{2-3}
            \cellsymb(\rho^{b\tcomp\psi^0}_{Y^0}){2-1}{4-3}
        \end{tikzcd}
    \end{equation*}
    \begin{equation*}
        =
        \begin{tikzcd}
            FA_0\ar[d,equal]\lar[r,path,"F\tup{u}"] & FA_n\lar[r,"\fm_{A_n}"] & M\ar[d,equal] \\
            FA_0\ar[d,"Fb"']\lar[rr,"\fm_{A_0}"] & & M\ar[d,equal] \\
            FB\ar[d,"F\psi^0"']\lar[rr,"\fm_B"] & & M\ar[d,equal] \\
            F\Phi Y^0\lar[rr,"m_{Y^0}"'] & & M
            \cellsymb(\fm_\tup{u}){1-1}{2-3}
            \cellsymb(\fm_b){2-1}{3-3}
            \cellsymb(\rho^{\psi^0}_{Y^0}\colon\cart){3-1}{4-3}
        \end{tikzcd}\incat{\bL},
    \end{equation*}
    which shows that $\fm$ becomes a left $F$-module.
    We can easily verify that the cells $\rho^\id_X$ for $X\in\bJ$ form an invertible modulation $\fm_\Phi\cong m$ of type 0, which finishes the proof.
\end{proof}

\begin{example}
    Let $\bJ$ be the \ac{AVDC} consisting of two objects $0,1$ and a unique loose arrow $0\larr[][1]1$.
    Let $\bK$ be an \ac{AVDC} defined by the following:
    \begin{itemize}
        \item
            $\bK$ has just two objects $0,1$;
        \item
            $\bK$ has no non-trivial tight arrow;
        \item
            $\bK$ has just three loose arrows $0\larr[][1] 0\larr[][1] 1\larr[][1] 1$;
        \item
            For any boundary for cells, which include 0-coary ones, $\bK$ has a unique cell filling it.
    \end{itemize}
    Then, the inclusion $\bJ\to\bK$ gives a naively final \ac{AVD}-functor.
    An \ac{AVD}-functor $F\colon\bK\to\bL$ is the same as a choice of a loose arrow $F0\larr[][1] F1$ and loose units on $F0$ and $F1$.
    By \cref{thm:invariance_by_finality}, we can ignore the loose units when we regard $F$ as a diagram for tight cocones, modules, and modulations.
\end{example}

\begin{corollary}\label{cor:versatile_colim_and_finality}
    Let $\Phi\colon\bJ\to\bK$ be a naively final \ac{AVD}-functor.
    Let $F\colon\bK\to\bL$ be an \ac{AVD}-functor such that $Ff$ is pulling in $\bL$ for any tight arrow $f$ in $\bK$.
    Then, a tight cocone $\xi$ from $F$ is a (\ac{VD}-)versatile colimit if and only if so is the tight cocone $\xi_\Phi$ from $F\Phi$.
\end{corollary}
\begin{proof}
    This follows from \cref{thm:invariance_by_finality}.
\end{proof}

\begin{remark}
    The definition of naively final \ac{AVD}-functors is too strong for the purpose of proving \cref{thm:invariance_by_finality} and \cref{cor:versatile_colim_and_finality}.
    For example, identity \ac{AVD}-functors are not necessarily naively final, but \cref{thm:invariance_by_finality} and \cref{cor:versatile_colim_and_finality} clearly hold for them.
\end{remark}

\section{Density in the case of virtual equipments}\label[appendix]{sec:density_ve}
\begin{definition}\label{def:canonical_tight_cocone}
    Let $\bE$ be an \ac{AVDC} with restrictions.
    Let $\bX\subseteq\bE$ be a full sub-\ac{AVDC}.
    Fix an object $E\in\bE$.
    \begin{enumerate}
        \item
            We define an \ac{AVD}-functor $K_E\colon\Idimdbl(\slice{\bX}{E})\to\bE$, factoring through $\bX$, as follows:
            \begin{itemize}
                \item
                    For $x\in\slice{\bX}{E}$, $K_E(x)\coloneq Dx$.
                \item
                    For $x,y\in\slice{\bX}{E}$, $K_E(!_{xy})$ is defined to be the following restriction:
                    \begin{equation}\label{eq:cartcell_K_X}
                        \begin{tikzcd}[tri]
                            Dx\ar[dr,"x"']\lar[rr,"{K_E(!_{xy})}"] & {} & Dy\ar[dl,"y"] \\
                            & E &
                            \cellsymb(\cart)[above=-3]{1-2}{2-2}
                        \end{tikzcd}\incat{\bE}.
                    \end{equation}
                \item
                    For $x_0,\dots,x_n\in\slice{\bX}{E}$ and $x_0\arr(f)[][1]y,x_n\arr(g)[][1]z$ in $\slice{\bX}{E}$, the assignment to the unique cell $!$ in $\Idimdbl(\slice{\bX}{E})$ is defined using the universality of restrictions:
                    \begin{equation*}
                        \begin{tikzcd}
                            Dx_0\ar[d,"f"']\lar[r,"K_E(!_{x_0x_1})"] & \cdots\lar[r,"K_E(!_{x_{n-1}x_n})"] & Dx_n\ar[d,"g"] \\
                            Dy\ar[dr,"y"']\lar[rr,"K_E(!_{yz})"] & {} & Dz\ar[dl,"z"] \\
                            & E &
                            \cellsymb(K_E(!)){1-1}{2-3}
                            \cellsymb(\cart)[above=-3]{2-2}{3-2}
                        \end{tikzcd}
                        =
                        \begin{tikzcd}[scriptsizecolumn]
                            Dx_0\ar[ddrr,"x_0"']\lar[r,"K_E(!)"] & Dx_1\ar[ddr,"x_1"{right}]\lar[r,"K_E(!)"] & \cdots\lar[r,"K_E(!)"] & Dx_{n-1}\ar[ddl,"x_{n-1}"{left}]\lar[r,"K_E(!)"] & Dx_n\ar[ddll,"x_n"] \\
                            &&&& \\
                            && E &&
                            \cellsymb(\cart)[above=5]{1-1}{3-3}
                            \cellsymb(\cart)[above=5]{1-5}{3-3}
                            \cellsymb(\cdots)[above=5]{1-3}{3-3}
                        \end{tikzcd}\incat{\bE}.
                    \end{equation*}
            \end{itemize}
        \item
            Furthermore, the cells \cref{eq:cartcell_K_X} yield a tight cocone $K_E\Rightarrow E$, which is denoted by $\kappa_E$.
            \qedhere
    \end{enumerate}
\end{definition}

\begin{theorem}[Density theorem]\label{thm:density}
    Let $\bE$ be an \ac{AVDC} with restrictions.
    For a full sub-\ac{AVDC} $\bX\subseteq\bE$ whose objects are collage-atomic in $\bE$, the following are equivalent:
    \begin{enumerate}
        \item\label{thm:density-1}
            $\bX\subseteq\bE$ is collage-dense.
        \item\label{thm:density-2}
            For every $E\in\bE$, the tight cocone $\kappa_E$ of \cref{def:canonical_tight_cocone} is a versatile colimit and the category $\slice{\bX}{E}$ is $C$-discrete.\qedhere
    \end{enumerate}
\end{theorem}
\begin{proof}
    \proofdirection{\cref{thm:density-2}}{\cref{thm:density-1}}
    Since $\slice{\bX}{E}$ is $C$-discrete, there is a final functor $\Phi\colon\zero{S}\arr[][1.4]\slice{\bX}{E}$ from a large discrete category $\zero{S}$.
    By \cref{prop:final_loosewise_VD_indiscrete}, $\Phi$ induces a naively final \ac{AVD}-functor\linebreak \mbox{$\Idimdbl\Phi\colon\Idimdbl\zero{S}\to\Idimdbl(\slice{\bX}{E})$}.
    Then, \cref{cor:versatile_colim_and_finality} makes $(\kappa_E)_{\Idimdbl\Phi}$ be a versatile collage.
    
    \proofdirection{\cref{thm:density-1}}{\cref{thm:density-2}}
    Fix $E\in\bE$.
    By assumption, there is a large set $\zero{S}$, an \ac{AVD}-functor $F\colon\Idimdbl\zero{S}\to\bE$ factoring through $\bX$, and a tight cocone $\xi$ from $F$ that exhibits $E$ as a large versatile collage.
    Then, the following assignment yields a functor $\Phi\colon\zero{S}\to\slice{\bX}{E}$:
    \begin{equation*}
        i\in\zero{S} \qquad\arr(\Phi)[mapsto]\qquad
        \begin{tikzcd}[smallrow]
            Fi\ar[d,"\xi_i"] \\
            E
        \end{tikzcd}\incat{\slice{\bX}{E}}.
    \end{equation*}
    By the definition of collage-atomic objects, the functor $\Phi$ becomes final, hence $\slice{\bX}{E}$ is $C$-discrete.
    By virtue of the strongness theorem (\cref{thm:strongness_theorem}), we have an (admissible) invertible tight \ac{AVD}-transformation of the following form:
    \begin{equation*}
        \begin{tikzcd}[scriptsizerow]
            \Idimdbl\zero{S}\ar[dr,"\Idimdbl\Phi"']\ar[rr,"F"] & {} & \bE \\
            & \Idimdbl(\slice{\bX}{E})\ar[ru,"K_E"'] &
            \dtwocell(\cong){1-2}{2-2}
        \end{tikzcd}\incat{\AVDC}.
    \end{equation*}
    By \cref{prop:final_loosewise_VD_indiscrete}, the induced \ac{AVD}-functor $\Idimdbl\Phi$ is naively final.
    Then, \cref{prop:vers_colim_is_up_to_admissible_iso}\cref{prop:vers_colim_is_up_to_admissible_iso-diagram} and \cref{cor:versatile_colim_and_finality} imply that the canonical tight cocone $\kappa_E$ becomes a versatile colimit.
\end{proof}

\printbibliography

\end{document}